\documentclass{amsart}
\usepackage{amsxtra,amscd}
\usepackage{graphicx}
\usepackage{amsmath}
\usepackage{amsfonts}
\usepackage{amssymb}

\newtheorem{theorem}{Theorem}[section]
\newtheorem{lemma}[theorem]{Lemma}
\newtheorem{claim}[theorem]{Claim}
\theoremstyle{definition}
\newtheorem{definition}[theorem]{Definition}
\newtheorem{example}[theorem]{Example}

\newtheorem{proposition}[theorem]{Proposition}
\newtheorem{remark}[theorem]{Remark}
\newtheorem{question}[theorem]{Question}

\numberwithin{equation}{section}

\begin{document}
\title{Transcendental Galois Theory and Noether's Problem}
\author{Feng-Wen An}
\address{School of Mathematics and Statistics, Wuhan University, Wuchang,
Wuhan, Hubei 430072, People's Republic of China}
\email{fwan@amss.ac.cn}

\begin{abstract}
In 1918, Noether published a paper where she studied such a problem, now
called {Noether's problem on rationality}: Let $L=K\left( t_{1},t_{2},\cdots
,t_{n}\right) $ be a purely transcendental extension over a field $K$ and $G$
a finite subgroup acting transitively on $t_{1},t_{2},\cdots ,t_{n}$ in an
evident manner. Is it true that the invariant subfield $L^{G}$ of $L$ under $%
G$ is still purely transcendental over $K$? The problem has been open in
general except for minor particular cases. In this paper we will attempt to
understand a general theory for Noether's problem on rationality by
transcendental Galois theory. Then new particular cases will be obtained. We will
also give a generalisation for the remarkable counter-example given by Swan
in 1969.
\end{abstract}

\maketitle

{\tiny {\ }}

\begin{center}
{\tiny {Contents} }
\end{center}

{\tiny \qquad {Introduction} }

{\tiny \qquad {1. Main Theorems of The Paper} }

{\tiny \qquad {2. Preliminaries} }

{\tiny \qquad {3. Transcendental Galois Theory, I. Tame Galois} }

{\tiny \qquad {4. Transcendental Galois Theory, II. Algebraic Galois
Subgroups} }

{\tiny \qquad {5. Transcendental Galois Theory, III. Decomposition Groups} }

{\tiny \qquad {6. Noether Solutions and Transcendental Galois Theory} }

{\tiny \qquad {7. $G$-Symmetric Functions and Noether Solutions} }

{\tiny \qquad {8. Lifting of Isomorphisms and Noether Solutions} }

{\tiny \qquad {9. Proofs of Main Theorems} }

{\tiny \qquad {References}}

\section*{Introduction}

\subsection{Noether's problem on rationality}

Let $K$ be a field and let $t_{1},t_{2},\cdots ,t_{n}$ be algebraically
independent variables over $K$. Consider the purely transcendental extension
$L=K\left( t_{1},t_{2},\cdots ,t_{n}\right) $ over $K$ of transcendence
degree $n$.

In 1918 Noether published a paper (\emph{[}Noether 1918\emph{]}), where she
studied such a problem for the case that $K=\mathbb{Q}$, now called the
Noether's problem such as the following. (See \emph{[}Wiki\emph{]} for a
brief introduction).

\begin{remark}
\textbf{(Noether's problem on rationality)} (\emph{[}Noether 1918\emph{]})
Let $G$ be a finite subgroup which permutates the variables $%
t_{1},t_{2},\cdots ,t_{n}$. Is it true that the $G$-invariant subfield $%
L^{G} $ of $L$ is a purely transcendental extension of $K$? Here, $L^{G}$
denotes the subfield of $L$ consisting of all the elements $x\in L$ such
that $\sigma \left( x\right) =x$ holds for every $\sigma \in G$.
\end{remark}

Noether's problem on rationality, needless to say, is of paramount
importance for the inverse Galois theory, modern Galois theory and
transcendental number theory. From a geometric aspect, such a problem arises
from the problem on rationality and parametrization of algebraic varieties
in algebraic geometry. For instance, see \emph{[}Kollar 2002\emph{]} for a
conjectural description and question on unirationality of algebraic varieties.

The first counter-example to Noether's problem, given by Swan in 1969, is
that for the case that $K=\mathbb{Q}$, $n=47$ and $G$ is the cyclic group of
order $47$, the $G$-invariant subfield $L^{G}$ is not purely transcendental
over $K$ (See \emph{[}Swan 1969\emph{]}). His approach, after \emph{[}Fisher
1915\emph{]}, is called the \emph{Fisher-Swan method}.

For the case that $G$ is a finite abelian group, an equivalent condition is
given in \emph{[}Lenstra 1974\emph{]}. For the case that $G$ is a finite $p$%
-group and $K$ is an algebraically closed field, there is a counter-example
in \emph{[}Saltman 1984\emph{]}.

Among these approaches to Noether's problem, there is a common ingredient
involved in, a specified element in a $\mathbb{Z}_{p}$-extension. In deed,
in \emph{[}An 2019\emph{]} we attempt to obtain an explicit explanation for
it, i.e., there is a connection between Iwasawa theory and Noether's problem
by $K$-theory.

There are many others who have studied Noether's problem. For instance,
\emph{[}Chevalley 1955\emph{]} and \emph{[}Masuda 1955\emph{]}. In recent
years, there are studies for the case that $G$ is a finite non-abelian
group. For instance, see \emph{[}Kang-Prokhorov 2010\emph{]}.

However, the problem has been still open in most particular cases and in a
general theory.

\subsection{Main results in the paper}

The following \emph{Theorems 0.4}, \emph{0.7}, \emph{0.11} and
\emph{0.13} form the main results  in this paper.

Let's fix notation and symbols before we give the statements of the main
results. Suppose $L$ is an arbitrary extension over a field $K$ (algebraic
or transcendental).

\begin{definition}
(\textbf{Definition 1.1}) The \textbf{Galois group} $Aut\left( L/K\right) $
of the extension $L/K$ is the group of all automorphisms $\sigma $ of the
field $L$ with $\sigma (x)=x$ for all $x\in K$.

Let $H$ be a subgroup (or subset) of the Galois group $Aut\left( L/K\right) $%
. The \textbf{invariant subfield} of $L$ under $H$ is defined to be the
subfield
\begin{equation*}
L^{H}=\{x\in L:\sigma \left( x\right) =x\text{ for each }\sigma \in H\}.
\end{equation*}%
The subfield $L^{H}$ is also called the $H$\textbf{-invariant subfield} of $%
L $ in the paper.
\end{definition}

\begin{definition}
(\textbf{Definition 1.2}) The field $L$ is said to be \textbf{Galois} over a
subfield $K$ if $K=L^{Aut\left( L/K\right) }$ holds, i.e., $K$ is the
invariant subfield of $L$ under the Galois group $Aut\left( L/K\right) $.
\end{definition}

Let $\Sigma _{n}$ denotes the \emph{full permutation group} (i.e., \emph{%
symmetric group}) on $n$ letters and let $C_{n}$ and $A_{n}$ denote the
\emph{cyclic group} of order $n$ and the \emph{alternating subgroup} in $%
\Sigma _{n}$, respectively.

For the purely transcendental extension $L=K\left( t_{1},t_{2},\cdots
,t_{n}\right) $ over a field $K$, the full permutation group $\Sigma _{n}$ acts on the
variables $t_{1},t_{2},\cdots ,t_{n}$ in an evident manner (See \emph{Definition 7.7}). Hence, $A_{n}$, $%
C_{n}$ and $\Sigma _{n}$ all are taken as subgroups of the Galois group $%
Aut\left( L/K\right) $. See \emph{\S 7} for a full discussion.

For the invariant subfields of $L$ under $\Sigma _{n}$ and its subgroups,
here there is the following theorem.

\begin{theorem}
\emph{(\textbf{Lemmas 7.16-8})} Let $L=K\left( t_{1},t_{2},\cdots
,t_{n}\right) $ be a purely transcendental extension over a field $K$ of
finite transcendence degree $n$. For the invariant subfields of $L$, there
are the following statements.

$\left( i\right) $ If $n\geq 3$, then the $C_{n}$-invariant subfield $%
L^{C_{n}}$ of $L$ is not purely transcendental over $K$.

$\left( ii\right) $ If $n\geq 3$, then the $A_{n}$-invariant subfield $%
L^{A_{n}}$ of $L$ is not purely transcendental over $K$.

$\left( iii\right) $ Let $n\geq 1$ and let $G$ be a finite subgroup
contained in $\Sigma _{n}\subseteq Aut\left( L/K\right) $. Then the $G$%
-invariant subfield $L^{G}$ of $L$ is purely transcendental over $K$ if and
only if $G=\Sigma _{n}$ holds.

In particular, for any $n\geq 3$, the groups $A_{n}$, $C_{n}$ and $\Sigma
_{n}$ have a transitive action on the variables $t_{1},t_{2},\cdots ,t_{n}$,
respectively.
\end{theorem}

\begin{remark}
\bigskip The above $\left( i\right) $ of \emph{Theorem 0.4} is a
generalisation of the main theory in \emph{[}Swan 1969\emph{]}. But in $%
\left( iii\right) $ of \emph{Theorem 0.4}, the \textquotedblleft
if\textquotedblright\ is in deed known and hidden in many textbook; only the
\textquotedblleft only if\textquotedblright\ is ours.
\end{remark}

For sake of convenience, we have the following definition.

\begin{definition}
(\textbf{Definition 1.4}) Let $L$ be a transcendental extension over a field
$K$. A \textbf{Noether solution} of $L/K$ is a subgroup $G$ of the Galois
group $Aut\left( L/K\right) $ satisfying the two properties: $\left(
a\right) $ $L$ is algebraic over the $G$-invariant subfield $L^{G}$ of $L$; $%
\left( b\right) $ The $G$-invariant subfield $L^{G}$ is purely
transcendental over $K$.
\end{definition}

From \emph{Theorem 0.4} there is a theorem on the coexistence of Noether
solutions and non-solutions in a fixed purely transcendental extension.

\begin{theorem}
\emph{(\textbf{Theorem 1.8})} \emph{(Co-existence of solutions and
non-solutions)} Let $L=K\left( t_{1},t_{2},\cdots ,t_{n}\right) $ be a
purely transcendental extension over a field $K$ of transcendence degree $n$%
. There are the following statements.

$\left( i\right) $ \emph{(Existence of Noether solutions)} For any $n\geq 2$%
, there exist at least two finite subgroups $G$ of $Aut\left( L/K\right) $
such that each $G$ satisfies the two properties: $\left( a\right) $ $G$ has
a transitive action on $t_{1},t_{2},\cdots ,t_{n}$; $\left( b\right) $ The $%
G $-invariant subfield $L^{G}$ is purely transcendental over $K$, i.e., $G$
is a Noether solution of $L/K$.

$\left( ii\right) $ \emph{(Existence of non-solutions)} For any $n\geq 3$,
there exist at least six finite subgroups $G$ of $Aut\left( L/K\right) $
such that each $G$ satisfies the two properties: $\left( a\right) $ $G$ has
a transitive action on $t_{1},t_{2},\cdots ,t_{n}$; $\left( b\right) $ The $%
G $-invariant subfield $L^{G}$ is not purely transcendental over $K$.
\end{theorem}

\begin{remark}
Let $L=K\left( t_{1},t_{2},\cdots ,t_{n}\right) $ be a purely transcendental
extension over a field $K$ of transcendence degree $n\geq 3$.

$\left( i\right) $ Let $G$ be the Galois group of $L$ over the subfield $%
K\left( t_{1}^{2},t_{2},\cdots ,t_{n}\right) $. Then $G$ is a Noether
solution of $L/K$, i.e., the $G$-invariant subfield $L^{G}$ is purely
transcendental over $K$. The subgroup $G$ is of order $\sharp G=2$.

$\left( ii\right) $ Let $H$ be the subgroup in $Aut\left( L/K\right) $
generated by a linear involution of $L/K$. (See \emph{Definition 8.1}). Then
$H$ satisfies the above conditions $\left( a\right) -\left( b\right) $ in
\emph{Theorem 0.7}. The subgroup $H$ is of order $\sharp H=2$. The two
groups $G$ and $H$ are isomorphic.

So, for a subgroup $G$ of the Galois group $Aut\left( L/K\right) $, in order
to obtain enough information for a possible Noether solution of $L/K$, we
need to take a full consideration of the action of the subgroup $G$ on the
extension $L/K$.
\end{remark}

As there can be many Noether solutions in a purely transcendental extension,
we need subdivisions of Noether solutions such as the following.

\begin{definition}
(\textbf{Definition 1.4}) Let $L$ be a transcendental extension over a field
$K$. Let $G$ be a Noether solution of $L/K$. A \textbf{Noether }$G$\textbf{%
-solution} of $L/K$ is a Noether solution $P$ of $L/K$ such that $P$ and $G$
are conjugate subgroups in the Galois group $Aut\left( L/L\right) $.
\end{definition}

\begin{definition}
(\textbf{Definition 1.5}) Let $L$ be a transcendental extension over a field
$K$ and $G$ a Noether solution of $L/K$. Denote by
\begin{equation*}
\Psi _{L/K}\left( G\right) :=\{G_{\lambda }:\lambda \in I\}
\end{equation*}%
the set of all the Noether solutions $G_{\lambda }$ of $L/K$ which contain $%
G $.

$\left( i\right) $ A subset $\Phi \subseteq \Psi _{L/K}\left( G\right) $ is
a \textbf{Zorn-component} of $\Psi _{L/K}\left( G\right) $ if the two
conditions are satisfied: $\left( a\right) $ $\Phi $ is a totally ordered
subset; $\left( b\right) $ There is no other element $P_{0}\in \Psi
_{L/K}\left( G\right) $ such that the union $\Phi \cup \{P_{0}\}$ is still
totally ordered. Here, via set-inclusion, $\Psi _{L/K}\left( G\right) $ is a
partially ordered set.

$\left( ii\right) $ The \textbf{height group} $H_{L/K}\left( G\right) $ of $%
G $ in $L/K$ is the subgroup in the Galois group $Aut\left( L/K\right) $
generated by the subset $\bigcup\limits_{P\in \Psi _{L/K}\left( G\right) }P.$

$\left( iii\right) $ Consider the general linear group $\Pi :=GL_{K}\left(
L\right) $ of $L$ over $K$ and the quotient $\Pi ^{ab}:=\Pi /\left[ \Pi ,\Pi %
\right] $ of $\Pi $ by the commutator subgroup $\left[ \Pi ,\Pi \right] $.
The \textbf{class height group} of $G$ in $L/K$ is the image
\begin{equation*}
CH_{L/K}\left( G\right) :=\tau _{L/K}\left( H_{L/K}\left( G\right) \right)
\end{equation*}%
of the height subgroup $H_{L/K}\left( G\right) $ under the natural map $\tau
_{L/K}:\Pi \rightarrow \Pi ^{ab}$.
\end{definition}

Here there is a theorem on properties of Noether solutions, where for the
centralizer $Z_{Aut\left( L/K\right) }\left( G\right) $ of a subgroup $G$ in
$Aut\left( L/K\right) $, let $ind_{L/K}\left( G\right) $ denote the index of
the centralizer $Z_{Aut\left( L/K\right) }\left( G\right) $ in the Galois
group $Aut\left( L/K\right) $.

\begin{theorem}
\emph{(\textbf{Theorem 1.6})} Assume $L$ is a purely transcendental
extension over a field $K$ of finite transcendence degree. Then there are
the following properties for Noether solutions of $L/K$.

$\left( i\right) $ Let $G$ be a subgroup of $Aut\left( L/K\right) $. Then $G$
is a Noether solution of $L/K$ if and only if each conjugate subgroup of $G$
in the Galois group $Aut\left( L/K\right) $ is a Noether solution of $L/K$.

In such a case, we have
\begin{equation*}
Con_{L/K}\left( G\right) =\{\phi ^{-1}\cdot G\cdot \phi :\phi \in Aut\left(
L/K\right) \}
\end{equation*}
for the set $Con_{L/K}\left( G\right) $ of Noether $G$-solutions of $L/K$.

$\left( ii\right) $ For each Noether solution $G$ of $L/K$, there is
\begin{equation*}
con_{L/K}\left( G\right) \leq ind_{L/K}\left( G\right)
\end{equation*}%
for the number $con_{L/K}\left( G\right) $ of Noether $G$-solutions of $L/K$.

$\left( iii\right) $ Let $G$ be a Noether solution of $L/K$. If there exists
only one Zorn-component in $\Psi _{L/K}\left( G\right) $, then the class
height group $CH_{L/K}\left( G\right) $ of $G$ is isomorphic to a subgroup
of the multiplicative subgroup $\left( L^{G}\right) ^{\ast }$ of the $G$%
-invariant subfield $L^{G}$.
\end{theorem}

As \emph{Remark 0.8} indicates, for a subgroup $G$ of the Galois group $%
Aut\left( L/K\right) $, we have to consider the action of the group $G$ on
the extension $L/K$ to obtain enough information for a Noether solution. So
here we have such a definition.

\begin{definition}
(\textbf{Definition 1.3}) Let $L$ and $M$ be two extensions over a field $K$%
. Two subgroups $G\subseteq Aut\left( L/K\right) $ and $H\subseteq Aut\left(
M/K\right) $ are $\left( L,M\right) $\textbf{-spatially isomorphic} if the
three conditions are satisfied:

$\left( i\right) $ There is an isomorphism $\psi :H\rightarrow G$ of groups.

$\left( ii\right) $ There is a $K$-isomorphism $\phi :L\rightarrow M$ of
fields such that the actions of $G$ on $L$ and $H$ on $M$ are $\left( \psi
,\phi \right) $\emph{-compatible}, i.e., $g\left( x\right) =\phi ^{-1}\circ
\psi \left( g\right) \circ \phi \left( x\right) $ holds for any $g\in G$ and
$x\in L$.

$\left( iii\right) $ Fixed any $g\in G$ and $x\in L$. There is $g\left(
x\right) =x$ in $L$ if and only if $\psi \left( g\right) \circ \phi \left(
x\right) =\phi \left( x\right) $ holds in $M$.
\end{definition}

By \emph{Definition 0.12} we have the following theorem on criterions for
Noether solutions of $L/K$, where for a subgroup $P$ of $Aut\left(
L/K\right) $, we denote the index of $P$ in the Galois group $Aut\left(
L/K\right) $ by
\begin{equation*}
\left[ Aut\left( L/K\right) :P\right] :=\sharp \frac{Aut\left( L/K\right) }{P%
}
\end{equation*}%
and the isotropy subgroup of $P$ under the conjugate action of $Aut\left(
L/K\right) $ by
\begin{equation*}
Aut\left( L/K\right) ^{P}:=\{\phi \in Aut\left( L/K\right) :\phi ^{-1}\cdot
P\cdot \phi =P\}.
\end{equation*}

Note that here for a purely transcendental extension $L/K$ of finite
transcendence degree, a subgroup $G$ of the Galois group $Aut\left(
L/K\right) $ is a finite group if and only if $L$ is algebraic over the $G$%
-invariant subfield $L^{G}$.

\begin{theorem}
\emph{(\textbf{Theorem 1.7})} Let $L\ $be a purely transcendental extension
over a field $K$ of finite transcendence degree. Suppose $G$ is a finite
subgroup of the Galois group $Aut\left( L/K\right) $. Then there are the
following criterions for Noether solutions of $L/K$.

$\left( i\right) $ $G$ is a Noether solution of $L/K$ if and only if there
are two transcendence bases $\Delta $ and $\Lambda $ of $L/K$ with $K\left(
\Lambda \right) \supseteq K\left( \Delta \right) $ satisfying the two
conditions: $\left( a\right) $ $K\left( \Lambda \right) $ is Galois over $%
K\left( \Delta \right) $; $\left( b\right) $ $G$ and $Aut\left( K\left(
\Lambda \right) /K\left( \Delta \right) \right) $ are $\left( L,K\left(
\Lambda \right) \right) $-spatially isomorphic.

In such a case, we have
\begin{equation*}
con_{L/K}\left( G\right) =\left[ Aut\left( L/K\right) :Aut\left( L/K\right)
^{G}\right]
\end{equation*}%
for the number $con_{L/K}\left( G\right) $ of Noether $G$-solutions of $L/K$.

$\left( ii\right) $ $G$ is a Noether solution of $L/K$ if and only if there
is a subgroup $H$ of the Galois group $Aut\left( L/K\right) $ satisfying the
three conditions:

$\ \left( N1\right) $ There is $G\cap H=\{1\}$ and $\sigma \cdot \delta
=\delta \cdot \sigma $ holds in $Aut\left( L/K\right) $ for any $\sigma \in
G $ and $\delta \in H$.

$\ \left( N2\right) $ $K$ is the invariant subfield $L^{\left\langle G\cup
H\right\rangle }$ of $L$ under the subgroup $\left\langle G\cup
H\right\rangle $ generated in $Aut\left( L/K\right) $ by the subset $G\cup H$%
.

$\ \left( N3\right) $ There are two transcendence bases $\Delta $ and $%
\Lambda $ of $L/K$ with $K\left( \Delta \right) \subseteq K\left( \Lambda
\right) $ and a subgroup $H_{\Lambda }$ of the Galois group $Aut\left(
K\left( \Lambda \right) /K\right) $ having the three properties: $\left(
a\right) $ $K\left( \Lambda \right) $ is Galois over $K\left( \Delta \right)
$; $\left( b\right) $ $H$ and $H_{\Lambda }$ are $\left( L,K\left( \Lambda
\right) \right) $-spatially isomorphic; $\left( c\right) $ $Aut\left(
K\left( \Delta \right) /K\right) $ is the set of restrictions $\sigma
|_{K\left( \Delta \right) }$ of $\sigma $ in $H_{\Lambda }$.

In such a case, $H$ is unique (up to $G$) and there is
\begin{equation*}
con_{L/K}\left( G\right) =\left[ Aut\left( L/K\right) :Aut\left( L/K\right)
^{\left\langle G\cup H\right\rangle }\right]
\end{equation*}%
for the number $con_{L/K}\left( G\right) $ of Noether $G$-solutions of the
extension $L/K$.
\end{theorem}

\begin{remark}
(\textbf{Remark 6.18}) The above two conditions $\left( N1\right) -\left(
N2\right) $ in $\left( ii\right) $ of \emph{Theorem 0.13} are an essential
condition for a finite subgroup $G\subseteq Aut\left( L/K\right) $ to be a
Noether solution of $L/K$. In \emph{\S 6.6} there is a full discussion for
the conditions $\left( N1\right) -\left( N2\right) $ such as linear
decompositions of automorphisms of $L/K$ and Galois-complemented groups in $%
L/K$ focusing on the interplay between Noether solutions and transcendental
Galois theory, where we obtain a key lemma, \emph{Lemma 6.26}, for the proof
of \emph{Theorem 0.4}.
\end{remark}

\subsection{Outline of the paper}

The above main theorems will be proved mainly in \emph{\S 9} (and partly in
\emph{\S \S 7-8}) after in \emph{\S 6} we obtain the technical results on
the interplay between transcendental Galois theory and Noether solutions,
which form the key ingredients of the proofs.

Let's explain what's the transcendental Galois theory and why there exists
such a theory. Here we have several remarks for them.

\begin{remark}
(\emph{Algebraic Galois theory v.s. transcendental Galois theory}) Let $L$
be an extension over a field $K$. Let $H$ be a subgroup of the Galois group $%
Aut\left( L/K\right) $. Then $L$ is Galois over the $H$-invariant subfield $%
M:=L^{H}$.

$\left( i\right) $ For the subextension $L/M$ in $L/K$, there are two cases:

If $L$ is algebraic over $M$, then this is the (usual) \emph{Galois theory
of algebraic extensions}, called the \textbf{algebraic Galois theory} in the
paper.

If $L$ is transcendental over $M$, then there is a counterpart, \emph{Galois
theory of transcendental extensions}, called the \textbf{transcendental
Galois theory} in the paper (See \emph{[}An 2010\emph{]} for preliminary
studies).

$\left( ii\right) $ The transcendental Galois theory is of its own
characteristics distinct from algebraic Galois theory. (See the following \emph{Theorem 0.16} and \emph{Remarks 0.17-9}).
\end{remark}

\begin{theorem}
\emph{(\textbf{Theorem 3.1})} Let $L$ be a purely transcendental extension
over a field $K$ of transcendence degree $n\leq +\infty $. Then $L$ is
Galois over $K$, i.e., there is
\begin{equation*}
L^{Aut\left( L/K\right) }=K.
\end{equation*}

In particular, there is a subgroup $H$ in $Aut\left( L/K\right) $ of order $%
\sharp H=2$ such that
\begin{equation*}
L^{Aut\left( L/K\right) }=L^{H}=K
\end{equation*}%
hold for the invariant subfields of $L$.
\end{theorem}

\begin{remark}
Let $L$ be an extension over a field $K$ of degree $\left[ L:K\right]
=+\infty $. Suppose $L$ is Galois over $K$.

$\left( i\right) $ If $L$ is algebraic over $K$, then there is no finite
subgroup $G$ in $Aut\left( L/K\right) $ such that $L^{G}=K$ holds.

$\left( ii\right) $ Assume $L$ is a transcendental extension over $K$ of
finite transcendence degree. Let $M$ be a purely transcendental extension $M$
over $K$ such that $M$ is finite and Galois over $L$. Under certain
conditions, the finite group $G:=Aut\left( M/L\right) $ can fix the $L$ to
be $K$, i.e., $L^{G}=K$. See \emph{Lemma 6.20} and \emph{Example 6.22}.
\end{remark}

\begin{remark}
For a transcendental extension $M/K$, there can be finite groups $G$ of the
Galois group $Aut\left( M/K\right) $ such that $K$ is the $G$-invariant
subfield $M^{G}$ of $M$. It is called \textbf{the first characteristic }of
transcendental extensions in the paper (See \emph{Remark 6.21}). In deed, a
transcendental extension $M/K$ can have other characteristics (See \emph{%
Remarks 6.25} \& \emph{8.4}).
\end{remark}

\begin{remark}
Let $L$ be a purely transcendental extension over a field $K$ of
transcendence degree $n$. Then the general linear group $GL\left( n,K\right)
$ is isomorphic to a subgroup of the Galois group $Aut\left( L/K\right) $.
(See $\left( iii\right) $ of \emph{Example 1.9}).
\end{remark}

\begin{remark}
(\emph{Remark 8.4}) (\emph{Patterns producing roots}) Let $L=K\left(
t_{1},t_{2},\cdots ,t_{n}\right) $ be a purely transcendental extension over
a field $K$. Let $\Sigma _{n}\subseteq Aut\left( L/K\right) $ denote the
full permutation subgroup of $L/K$ (relative to the variables $%
t_{1},t_{2},\cdots ,t_{n}$).

Fixed a subgroup $G\subseteq \Sigma _{n}$. For instance, $G$ is a subgroup
in $Aut\left( L/K\right) $ generated by a linear involution of $L/K$.

$\left( i\right) $ The $G$-invariant subfield $L^{G}$ of $L$ is the $G$%
-symmetric functions $\Lambda ^{G}\left( L\right) $ of $L$; the function
field $L$ is algebraic Galois over the $G$-invariant subfield $L^{G}$. As $%
L=L^{G}\left[ \theta \right] $ is a simple extension over the subfield $%
L^{G} $ for some $\theta \in L$, there is a minimal polynomial $f\left(
X\right) $ of $\theta $ over the subfield $L^{G}$ such that $f\left( \theta
\right) =0$.

$\left( ii\right) $ Each permutation subgroup $G$ of $\Sigma _{n}$
establishes a \textbf{pattern} for the variables $t_{1},t_{2},\cdots ,t_{n}$
during the mapping and translating in the extension $L/K$. So, \emph{%
patterns }$G$\emph{\ produce roots }$\theta $. This gives to us an essential
comparison between algebraic Galois theory and transcendental Galois theory.

$\left( iii\right) $ The unusual complexities of Galois groups $H$ in
transcendental Galois theory are due to the entanglements of the permutation
subgroups $G$ (relative to the given variables) with the usual Galois groups
$P$ in algebraic Galois theory. In deed, there are Galois groups of separated type (See \emph{Definition 8.11}).
\end{remark}

For the transcendental Galois theory, there are several types of Galois
extensions such as \emph{Galois extensions}, \emph{tame Galois extensions},
\emph{$\sigma$-tame Galois extensions} and \emph{complete Galois extensions}
(See \emph{[}An 2010\emph{]} or \emph{\S 2.2} for definitions). For the case
of algebraic extensions, these Galois extensions will be reduced to the same
one, the Galois extensions (See \emph{Remark 3.15}). In deed, unlike the
algebraic Galois theory, the transcendental Galois theory has many exotic
and unusual properties. For the part of the transcendental Galois theory in
the paper, we proceed step by step with extreme caution and there are many
places in the exposition of the paper which appear copious and are overdone, where there are almost no advanced techniques which are in need of the proofs, although the key observation was obtained from arithmetic fundamental groups.

Here there is a brief outline of the paper. In \emph{\S 1} we will give the
statements of the main theorems of the paper. The main theorems will be
proved in \emph{\S 9}.

In \emph{\S \S 2-5} we will give a systematic exposition for our \emph{%
transcendental Galois theory}, including several types of Galois extensions,
where the Galois group of a transcendental extension has two functional
parts, i.e., the \emph{full algebraic Galois subgroup} and the \emph{highest
transcendental Galois subgroup}. Like the blocks of a diagonal matrix, the
algebraic and transcendental parts form the \emph{decomposition subgroup} of
the extension under the Galois action. Several results such as Galois
correspondence and Dedekind independence of automorphisms of fields will be obtained.

In \emph{\S 6} based on the transcendental Galois theory in \emph{\S \S 2-5}%
, a passage between the transcendental Galois theory and Noether's problem
on rationality will be established, where we will focus on the properties of
the interplay between Noether solutions and transcendental Galois theory.

In \emph{\S 7} the $G$-symmetric functions will be studied under the
techniques obtained in \emph{\S 6}. In \emph{\S 8} there will be a full
discussion on spatial isomorphisms and Noether solutions. In particular,
\emph{$K$-theory} will come into Noether solutions in a very natural way. In
\emph{\S 9} we will give the proofs of the main theorems stated in \emph{\S 1%
}.

The paper is part of our long-term project of \emph{geometric class field
theory}, where by arithmetic fundamental groups (including motivic Galois
groups) as Galois groups, automorphic representations and $L$-functions, we
attempt to understand a theory for class fields over arithmetic function
fields relative to a fixed set of ground variables. In such a case, we need
to make a comparison between different sets of ground variables, where there
are difficulties involving Noether's problem on rationality, which triggered
our studies on Noether's problem couples of years ago. Finally, the author
would like to express his acknowledgements to Professor Li Bang-He for his
long-standing support.

\bigskip

\section{Main Theorems of The Paper}

In this section we will give the statements of the main theorems of the
present paper which form a general theory on Noether's problem on
rationality.

\subsection{Some definitions in transcendental Galois theory}

Let $L$ be an arbitrary extension of a field $K$ (algebraic or
transcendental).

\begin{definition}
(See \emph{[}An 2010\emph{]}) The \textbf{Galois group} of the extension $%
L/K $, denoted by $Aut\left( L/K\right) $, is the group of all automorphisms
$\sigma $ of the field $L$ such that $\sigma (x)=x$ holds for each $x\in K$.

Let $H$ be a subgroup (or subset) of the Galois group $Aut\left( L/K\right) $%
. The \textbf{invariant subfield} of $L$ under $H$ is defined to be the
subfield
\begin{equation*}
L^{H}=\{x\in L:\sigma \left( x\right) =x\text{ for each }\sigma \in H\}.
\end{equation*}%
The subfield $L^{H}$ is also called the $H$\textbf{-invariant subfield} of $%
L $ in the paper.
\end{definition}

\begin{definition}
(See \emph{[}An 2010\emph{]}) The field $L$ is said to be \textbf{Galois}
over a subfield $K$ if $K=L^{Aut\left( L/K\right) }$ holds, i.e., $K$ is the
invariant subfield of $L$ under the Galois group $Aut\left( L/K\right) $.
\end{definition}

\begin{definition}
Let $L$ and $M$ be two extensions over a field $K$. Two subgroups $%
G\subseteq Aut\left( L/K\right) $ and $H\subseteq Aut\left( M/K\right) $ are
said to be $\left( L,M\right) $\textbf{-spatially isomorphic} if the three
conditions are satisfied:

$\left( i\right) $ There is an isomorphism $\psi :H\rightarrow G$ of groups.

$\left( ii\right) $ There is a $K$-isomorphism $\phi :L\rightarrow M$ of
fields such that the actions of $G$ on $L$ and $H$ on $M$ are $\left( \psi
,\phi \right) $\emph{-compatible}, i.e., $g\left( x\right) =\phi ^{-1}\circ
\psi \left( g\right) \circ \phi \left( x\right) $ holds for any $g\in G$ and
$x\in L$.

$\left( iii\right) $ Fixed any $g\in G$ and $x\in L$. There is $g\left(
x\right) =x$ in $L$ if and only if $\psi \left( g\right) \circ \phi \left(
x\right) =\phi \left( x\right) $ holds in $M$.

In such a case, the pair $\left( \psi ,\phi \right) $ are called an $\left(
L,M\right) $\textbf{-spatial isomorphism} from $G$ to $H$.
\end{definition}

\begin{definition}
Let $L$ be a transcendental extension over a field $K$.

$\left( i\right) $ A \textbf{Noether solution} of $L/K$ is a subgroup $G$ of
the Galois group $Aut\left( L/K\right) $ satisfying the two properties: $%
\left( a\right) $ $L$ is algebraic over the $G$-invariant subfield $L^{G}$
of $L$; $\left( b\right) $ The $G$-invariant subfield $L^{G}$ is purely
transcendental over $K$.

$\left( ii\right) $ Let $G$ be a Noether solution of $L/K$. A \textbf{%
Noether }$G$\textbf{-solution} of $L/K$ is a Noether solution $P$ of $L/K$
such that $P$ and $G$ are conjugate subgroups in the Galois group $Aut\left(
L/L\right) $.
\end{definition}

\begin{definition}
Let $L$ be a transcendental extension over a field $K$ and $G$ a Noether
solution of $L/K$. Denote by
\begin{equation*}
\Psi _{L/K}\left( G\right) :=\{G_{\lambda }:\lambda \in I\}
\end{equation*}%
the set of all the Noether solutions $G_{\lambda }$ of $L/K$ which contain $%
G $.

$\left( i\right) $ A subset $\Phi \subseteq \Psi _{L/K}\left( G\right) $ is
said to be a \textbf{Zorn-component} of $\Psi _{L/K}\left( G\right) $ if the
two conditions are satisfied: $\left( a\right) $ $\Phi $ is a totally
ordered subset; $\left( b\right) $ There is no other element $P_{0}\in \Psi
_{L/K}\left( G\right) $ such that the union $\Phi \cup \{P_{0}\}$ is still
totally ordered. Here, via set-inclusion, $\Psi _{L/K}\left( G\right) $ is a
partially ordered set.

$\left( ii\right) $ The \textbf{height group} of $G$ in $L/K$, denoted by $%
H_{L/K}\left( G\right) $, is the subgroup in the Galois group $Aut\left(
L/K\right) $ generated by the subset
\begin{equation*}
\bigcup\limits_{P\in \Psi _{L/K}\left( G\right) }P.
\end{equation*}

$\left( iii\right) $ Consider the general linear group $\Pi :=GL_{K}\left(
L\right) $ of $L$ over $K$ and its maximal abelian quotient $\Pi ^{ab}:=\Pi /%
\left[ \Pi ,\Pi \right] $ of $\Pi $ by the commutator subgroup $\left[ \Pi
,\Pi \right] $. The \textbf{class height group} of $G$ in $L/K$ is defined
to be the image
\begin{equation*}
CH_{L/K}\left( G\right) :=\tau _{L/K}\left( H_{L/K}\left( G\right) \right)
\end{equation*}%
of the height subgroup $H_{L/K}\left( G\right) $ under the natural
homomorphism $\tau _{L/K}:\Pi \rightarrow \Pi ^{ab}$ of groups.
\end{definition}

\subsection{Notation and symbols}

Now let's fix notation and symbols. For a field $F$, let $F^{\ast }$ denote
the multiplicative group $\left( F\setminus \{0\},\times \right) $ of $F$.

Let $L$ be a transcendental extension over a field $K$ and let $G$ be a
Noether solution $G$ of the extension $L/K$. Denote by
\begin{equation*}
Con_{L/K}\left( G\right)
\end{equation*}%
the set of all the Noether $G$-solutions of $L/K$ and by
\begin{equation*}
con_{L/K}\left( G\right) :=\sharp Con_{L/K}\left( G\right)
\end{equation*}%
the number of all the Noether $G$-solutions of $L/K$.

For the centralizer $Z_{Aut\left( L/K\right) }\left( G\right) $ of $G$ in $%
Aut\left( L/K\right) $, denote by
\begin{equation*}
ind_{L/K}\left( G\right) :=\sharp \frac{Aut\left( L/K\right) }{Z_{Aut\left(
L/K\right) }\left( G\right) }
\end{equation*}%
the index of the centralizer $Z_{Aut\left( L/K\right) }\left( G\right) $ in
the Galois group $Aut\left( L/K\right) $.

For a subgroup $P$ of $Aut\left( L/K\right) $, we denote the index of $P$ in
the Galois group $Aut\left( L/K\right) $ by
\begin{equation*}
\left[ Aut\left( L/K\right) :P\right] :=\sharp \frac{Aut\left( L/K\right) }{P%
}
\end{equation*}%
and the isotropy subgroup of $P$ under the conjugate action of $Aut\left(
L/K\right) $ by
\begin{equation*}
Aut\left( L/K\right) ^{P}:=\{\phi \in Aut\left( L/K\right) :\phi ^{-1}\cdot
P\cdot \phi =P\}.
\end{equation*}

For a purely transcendental extension $M=K\left( t_{1},t_{2},\cdots
,t_{n}\right) $ over a field $K$, as usual, a subgroup $P$ of $Aut\left(
M/K\right) $ is said to \textbf{have a transitive action on the variables} $%
t_{1},t_{2},\cdots ,t_{n}$ if for any $1\leq i<j\leq n$, there is some $%
\sigma \in P$ such that $\sigma \left( t_{i}\right) =t_{j}$ holds.

\subsection{The main theorems of the paper}

Here there are the main theorems of the paper, which form a general theory
for solutions of Noether's problem on rationality. We will be prove these
theorems in \emph{\S 9 }at the end of the paper after a journey of
transcendental Galois theory.

\begin{theorem}
\emph{(Main Theorem for Noether solutions: Properties)} Assume $L$ is a
purely transcendental extension over a field $K$ of finite transcendence
degree. Then there are the following properties for Noether solutions of $%
L/K $.

$\left( i\right) $ Fixed a subgroup $G$ of $Aut\left( L/K\right) $. Then $G$
is a Noether solution of $L/K$ if and only if each conjugate subgroup of $G$
in the Galois group $Aut\left( L/K\right) $ is a Noether solution of $L/K$.

In such a case, we have
\begin{equation*}
Con_{L/K}\left( G\right) =\{\phi ^{-1}\cdot G\cdot \phi :\phi \in Aut\left(
L/K\right) \}
\end{equation*}
for the set $Con_{L/K}\left( G\right) $ of Noether $G$-solutions of $L/K$.

$\left( ii\right) $ For each Noether solution $G$ of $L/K$, there is
\begin{equation*}
con_{L/K}\left( G\right) \leq ind_{L/K}\left( G\right)
\end{equation*}%
for the number $con_{L/K}\left( G\right) $ of Noether $G$-solutions of $L/K$.

In particular, if $ind_{L/K}\left( G\right) =m<+\infty $, then there are at
most $m$ Noether $G$-solutions of $L/K$.

$\left( iii\right) $ Let $G$ be a Noether solution of $L/K$. If there exists
only one Zorn-component in $\Psi _{L/K}\left( G\right) $, then the class
height group $CH_{L/K}\left( G\right) $ of $G$ is isomorphic to a subgroup
of the multiplicative subgroup $\left( L^{G}\right) ^{\ast }$ of the $G$%
-invariant subfield $L^{G}$.
\end{theorem}

\begin{theorem}
\emph{(Main Theorem for Noether solutions: Criterions)} Let $L\ $be a purely
transcendental extension over a field $K$ of finite transcendence degree.
Suppose $G$ is a subgroup of the Galois group $Aut\left( L/K\right) $ and $L$
is algebraic over the $G$-invariant subfield $L^{G}$. Then there are the
following criterions for Noether solutions of $L/K$.

$\left( i\right) $ $G$ is a Noether solution of $L/K$ if and only if there
are two transcendence bases $\Delta $ and $\Lambda $ of $L/K$ with $K\left(
\Lambda \right) \supseteq K\left( \Delta \right) $ satisfying the two
conditions: $\left( a\right) $ $K\left( \Lambda \right) $ is Galois over $%
K\left( \Delta \right) $; $\left( b\right) $ $G$ and $Aut\left( K\left(
\Lambda \right) /K\left( \Delta \right) \right) $ are $\left( L,K\left(
\Lambda \right) \right) $-spatially isomorphic.

In such a case, we have
\begin{equation*}
con_{L/K}\left( G\right) =\left[ Aut\left( L/K\right) :Aut\left( L/K\right)
^{G}\right]
\end{equation*}%
for the number $con_{L/K}\left( G\right) $ of Noether $G$-solutions of $L/K$.

$\left( ii\right) $ $G$ is a Noether solution of $L/K$ if and only if there
is a subgroup $H$ of the Galois group $Aut\left( L/K\right) $ satisfying the
three conditions:

$\ \left( N1\right) $ There is $G\cap H=\{1\}$ and $\sigma \cdot \delta
=\delta \cdot \sigma $ holds in $Aut\left( L/K\right) $ for any $\sigma \in
G $ and $\delta \in H$.

$\ \left( N2\right) $ $K$ is the invariant subfield $L^{\left\langle G\cup
H\right\rangle }$ of $L$ under the subgroup $\left\langle G\cup
H\right\rangle $ generated in $Aut\left( L/K\right) $ by the subset $G\cup H$%
.

$\ \left( N3\right) $ There are two transcendence bases $\Delta $ and $%
\Lambda $ of $L/K$ with $K\left( \Delta \right) \subseteq K\left( \Lambda
\right) $ and a subgroup $H_{\Lambda }$ of the Galois group $Aut\left(
K\left( \Lambda \right) /K\right) $ having the three properties: $\left(
a\right) $ $K\left( \Lambda \right) $ is Galois over $K\left( \Delta \right)
$; $\left( b\right) $ $H$ and $H_{\Lambda }$ are $\left( L,K\left( \Lambda
\right) \right) $-spatially isomorphic; $\left( c\right) $ $Aut\left(
K\left( \Delta \right) /K\right) $ is the set of restrictions $\sigma
|_{K\left( \Delta \right) }$ of $\sigma $ in $H_{\Lambda }$.

In such a case, $H$ is unique (up to $G$) and there is
\begin{equation*}
con_{L/K}\left( G\right) =\left[ Aut\left( L/K\right) :Aut\left( L/K\right)
^{\left\langle G\cup H\right\rangle }\right]
\end{equation*}%
for the number $con_{L/K}\left( G\right) $ of Noether $G$-solutions of the
extension $L/K$.
\end{theorem}

\begin{theorem}
\emph{(Main Theorem for Noether solutions: Co-existence of solutions and
non-solutions)} Let $L=K\left( t_{1},t_{2},\cdots ,t_{n}\right) $ be a
purely transcendental extension over a field $K$ of transcendence degree $n$%
. There are the following statements.

$\left( i\right) $ \emph{(Existence of Noether solutions)} For any $n\geq 2$%
, there exist at least two finite subgroups $G$ of $Aut\left( L/K\right) $
such that each $G$ satisfies the two properties: $\left( a\right) $ $G$ has
a transitive action on $t_{1},t_{2},\cdots ,t_{n}$; $\left( b\right) $ $G$
is a Noether solution of $L/K$.

$\left( ii\right) $ \emph{(Existence of non-solutions)} For any $n\geq 3$,
there exist at least six finite subgroups $G$ of $Aut\left( L/K\right) $
such that each $G$ satisfies the two properties: $\left( a\right) $ $G$ has
a transitive action on $t_{1},t_{2},\cdots ,t_{n}$; $\left( b\right) $ The $%
G $-invariant subfield $L^{G}$ is not purely transcendental over $K$.
\end{theorem}

Note that \emph{Lemmas 7.16-8} in \emph{\S 7} are complementary to the above
\emph{Theorem 1.8}.

\subsection{The underlying mystery of Noether's problem}

Here there is a heuristic example for us to understand the underlying
mystery of Noether's problem on rationality for a purely transcendental
extension.

\begin{example}
Let $L=\mathbb{C}\left( t_{1},t_{2},\cdots ,t_{n}\right) $ be a purely
transcendental extension over the complex field $\mathbb{C}$ of
transcendence degree $n\geq 3$.

$\left( i\right) $ Fixed prime numbers $p_{1}\geq p_{2}\geq \cdots \geq
p_{n}\geq 29.$ Denote by
\begin{equation*}
\omega _{j}^{k_{j}}:=t_{j}^{p_{j}-k_{j}}
\end{equation*}%
a power of the variable $t_{j}$ for any $0\leq k_{j}\leq p_{j}-1$ and $1\leq
j\leq n.$ Consider a tower of sub-extensions in $L/\mathbb{C}$
\begin{equation*}
\mathbb{C}\subsetneqq M_{0}\subsetneqq M_{1}\subsetneqq M_{2}\subsetneqq
M_{3}\subsetneqq \mathbb{C}\left( t_{1},t_{2},\cdots ,t_{n}\right) =L
\end{equation*}%
with a tower of subgroups in the Galois group $Aut\left( L/\mathbb{C}\right)
$
\begin{equation*}
Aut\left( L/\mathbb{C}\right) \supsetneqq G_{0}\supsetneqq G_{1}\supsetneqq
G_{2}\supsetneqq G_{3}\supsetneqq \{1\}
\end{equation*}%
where
\begin{eqnarray*}
M_{0} &:&=\mathbb{C}\left( \omega _{1}^{3},\omega _{2}^{3},\cdots ,\omega
_{n}^{3}\right) ; \\
M_{1} &:&=M_{0}\left[ \omega _{1}^{7},\omega _{2}^{7},\cdots ,\omega _{n}^{7}%
\right] ; \\
M_{2} &:&=M_{1}\left[ \prod_{i=1}^{n}\omega _{i}^{11}\right] ; \\
M_{3} &:&=M_{1}\left[ \omega _{1}^{11},\omega _{2}^{11},\cdots ,\omega
_{n}^{11}\right] ; \\
G_{i} &:&=Aut\left( L/M_{i}\right) \text{ is the Galois group of }L/M_{i}%
\text{ for each }0\leq i\leq 3.
\end{eqnarray*}

Then $L$ is Galois over $M_{i}$ and $M_{i}=L^{G_{i}}$ for each $0\leq i\leq
3 $; $M_{0}$, $M_{1}$ and $M_{3}$ all are purely transcendental over $%
\mathbb{C}$; $M_{2}$ is not purely transcendental over $%
\mathbb{C}$.

$\left( ii\right) $ Consider a map $\sigma :L\rightarrow L$ of \emph{%
homogeneous linear type} (See \emph{Definition 8.1} and \emph{Remark 1.11}
for details)

\begin{equation*}
\sigma :\left(
\begin{array}{c}
t_{1} \\
t_{2} \\
\vdots \\
t_{n}%
\end{array}%
\right) \longmapsto A\cdot \left(
\begin{array}{c}
t_{1} \\
t_{2} \\
\vdots \\
t_{n}%
\end{array}%
\right)
\end{equation*}%
where
\begin{equation*}
A=\left(
\begin{array}{cccc}
&  &  & 1 \\
&  & 1 &  \\
& \cdots &  &  \\
1 &  &  &
\end{array}%
\right)
\end{equation*}%
is the $n\times n$ matrix. Let $G_{4}=\left\langle \sigma \right\rangle $ be
the subgroup in $Aut\left( L/\mathbb{C}\right) $ generated by $\sigma $.

Then $G_{4}$ is a finite cyclic subgroup of order $\sharp G_{4}=2$ and $%
G_{4} $ acts transitively on the variables $t_{1},t_{2},\cdots ,t_{n}$ over $%
\mathbb{C}$.\ The $G_{4}$-invariant subfield $L^{G_{4}}$ is not purely
transcendental over $\mathbb{C}$ (from \emph{Lemma 7.18}).

$\left( iii\right) $ By replacing the matrix $A$ in the above $\left(
ii\right) $ by any invertible matrix in $GL\left( n,\mathbb{C}\right) $, it
is seen that the general linear group $GL\left( n,\mathbb{C}\right) $ is
taken as a subgroup of the Galois group $Aut\left( L/\mathbb{C}\right) $,
i.e., we have
\begin{equation*}
GL\left( n,\mathbb{C}\right) \subseteq Aut\left( L/\mathbb{C}\right) .
\end{equation*}
So, the Galois group of a purely transcendental extension are very big.
\end{example}

\begin{proof}
$\left( i\right) $ Trivial that $M_{0}$ is purely transcendental over $%
\mathbb{C}$ since $\omega _{1}^{j_{1}},\omega _{2}^{j_{2}},\cdots ,\omega
_{n}^{j_{n}}$ are algebraically independent over $\mathbb{C}$ for any $0\leq
j_{k}\leq p_{k}-1$ and $0\leq k\leq n$.

From $p_{i}-c_{i}=GCD\left( p_{i}-3,p_{i}-7\right) $, the greatest common
divisor of the positive integers for each $0\leq i\leq n$, it is seen that $%
M_{1}=\mathbb{C}\left( \omega _{1}^{c_{1}},\omega _{2}^{c_{2}},\cdots
,\omega _{n}^{c_{n}}\right) $ is purely transcendental over $\mathbb{C}$.

In the same way, $M_{3}=\mathbb{C}\left( \omega _{1}^{d_{1}},\omega
_{2}^{d_{2}},\omega _{3}^{d_{3}},\cdots ,\omega _{n}^{d_{n}}\right) $ is
purely transcendental over $\mathbb{C}$, where
\begin{equation*}
p_{i}-d_{i}=GCD\left( p_{i}-3,p_{i}-7,p_{i}-11\right) \geq 2
\end{equation*}%
is the greatest common divisor of the positive integers for each $0\leq
i\leq n$.

Prove $M_{2}$ is not purely transcendental over $\mathbb{C}$. In deed, the
sub-extension $M_{1}/\mathbb{C}$ has the same transcendence degree as $L/%
\mathbb{C}$; hence, the element $\prod_{i=1}^{n}\omega _{i}^{11}$ of the
subfield $M_{2}$ must be algebraic over $M_{1}$. On the other hand, we have
\begin{equation*}
M_{1}\subsetneqq M_{2}=M_{1}\left[ \prod_{i=1}^{n}\omega _{i}^{11}\right]
\subsetneqq L
\end{equation*}%
since by comparing the powers it is seen that the element $%
\prod_{i=1}^{n}\omega _{i}^{11}$ can not be expressed by an $M_{0}$-linear
combination of any powers of the generators
\begin{equation*}
\omega _{1}^{c_{1}},\omega _{2}^{c_{2}},\cdots ,\omega _{n}^{c_{n}}
\end{equation*}%
of the subfield $M_{1}$ over $M_{0}$.

$\left( ii\right) $ Trivial.

$\left( iii\right) $ Fixed an invertible matrix $A\in GL\left( n,\mathbb{C}%
\right) $. Put%
\begin{equation*}
A\cdot \left(
\begin{array}{c}
t_{1} \\
t_{2} \\
\vdots \\
t_{n}%
\end{array}%
\right) =\left(
\begin{array}{c}
s_{1} \\
s_{2} \\
\vdots \\
s_{n}%
\end{array}%
\right) .
\end{equation*}

We have $s_{1},s_{2},\cdots ,s_{n}\in L$ since each $s_{i}$ is a linear
combination of $t_{1},t_{2},\cdots ,t_{n}$. Hence, $\mathbb{C}\left(
s_{1},s_{2},\cdots ,s_{n}\right) \subseteq \mathbb{C}\left(
t_{1},t_{2},\cdots ,t_{n}\right) $. On the other hand, by takeing the
inverse $A^{-1}$ of the matrix $A$, we have $\mathbb{C}\left(
t_{1},t_{2},\cdots ,t_{n}\right) \subseteq \mathbb{C}\left(
s_{1},s_{2},\cdots ,s_{n}\right) .$ This proves $\mathbb{C}\left(
t_{1},t_{2},\cdots ,t_{n}\right) =\mathbb{C}\left( s_{1},s_{2},\cdots
,s_{n}\right) .$It follows that $s_{1},s_{2},\cdots ,s_{n}\in L$ are
algebraically independent over $\mathbb{C}$ since the transcendence degree
of $L/\mathbb{C}$ is $n$.

From every matrix $A\in GL\left( n,\mathbb{C}\right) $, we obtain an
automorphism $\tau _{A}\in Aut\left( L/\mathbb{C}\right) $ given by $%
t_{1}\mapsto s_{1},t_{2}\mapsto s_{2},\cdots ,t_{n}\mapsto s_{n}.$This
completes the proof.
\end{proof}

Enlightened by the above \emph{Example 1.9} and under the above \emph{%
Theorems 1.6-8} and \emph{Lemmas 7.16-8}, we now try to understand the
underlying mystery of Noether's problem on rationality for a purely
transcendental extension (See \emph{Remarks 1.10-1}).

\begin{remark}
(\emph{The underlying mystery of Noether's problem, I. From Galois
subgroups of the extension}) Hidden in the above example is the following
fact. Let $L=K\left( t_{1},t_{2},\cdots ,t_{n}\right) $ be a purely
transcendental extension over a field $K$ of transcendence degree.
For a technical reason, suppose $L$ is $\sigma $-tame Galois over $K$ (See \emph{\S 2.2} for definition).

Fixed a subgroup $G$ of the Galois group $Aut\left( L/K\right) $ with $%
L $ being algebraic over the $G$-invariant subfield $L^{G}$ and $G$ transitively acting on the variables $t_{1},t_{2},\cdots ,t_{n}$. Then $L$ is
algebraic Galois over $L^{G}$ and $G=Aut\left( L/L^{G}\right) $ is a finite
group.

For the action of the subgroup $G$ on the extension $L/K$, there are two
cases:

$Case$ $\left( i\right) $. When the subgroup $G$ is acting on the variables $%
t_{1},t_{2},\cdots ,t_{n}$, there are no other redundant elements (up to
transcendence bases of the sub-extension $L^{G}/K$) which are simultaneously
fixed by $G$, that is, the Galois group $G$ of the sub-extension $L/L^{G}$
is \textquotedblleft relatively big enough\textquotedblright . In fact, in
such a case, $G$ will be a \emph{full algebraic Galois subgroup} of the
extension $L/K$ (See \emph{Remark 9.7} and \emph{\S 4.1} for definitions).

However, if the Galois group $G$ of $L/L^{G}$ is bigger, then $G$ can not be
necessarily relatively big enough; but there is another relatively big
enough Galois group $H$ of the finite Galois extension $L/L^{H}$ such that $%
H\supsetneqq G.$ From \emph{Remarks 9.7-8} it is seen that there is no
biggest relatively big enough Galois group in the extension $L/K$. So, in
\emph{\S 4}, we used the terminology \textquotedblleft
full\textquotedblright\ to account for such a phenomenon.

$Case$ $\left( ii\right) $. When transitively acting on the variables $%
t_{1},t_{2},\cdots ,t_{n}$, the subgroup $G$ still possesses other redundant
elements $\omega _{1}^{k_{1}},\omega _{2}^{k_{2}},\cdots ,\omega
_{n}^{k_{n}} $ (up to transcendence bases of the sub-extension $L^{G}/K$) in
$L$ which also are simultaneously fixed by $G$, that is, the Galois group $%
G=Aut\left( L/L^{G}\right) $ of the sub-extension $L/L^{G}$ is not
relatively big enough (See \emph{Remarks 9.7-8}). In such a case, $G$ will
not be a \emph{full algebraic Galois subgroup} of $L/K$; from \emph{Remarks
9.7-8} it is seen that there is a relatively big enough one $H$ containing $%
G $ such that by the enlarged subgroup $H$, the redundant elements are
exactly rules out of the invariant subfield $L^{H}$. All such relatively big
enough subgroups for $L/K$ are finite subgroups in $Aut\left( L/K\right) $.
\end{remark}

Let $L$ be a purely transcendental extension of a field $K$ of finite
transcendence degree. Fixed two transcendence bases $\Delta _{1}$ and $%
\Delta _{2}$ of $L/K$, i.e., $K\left( \Delta _{1}\right) ,K\left( \Delta
_{2}\right) \subseteq L$. A $K$-isomorphism $\delta :K\left( \Delta
_{1}\right) \rightarrow K\left( \Delta _{2}\right) $ of subfields of $L$ is
\textbf{extendable} in $L/K$ if there exists some $\sigma \in
Aut\left( L/K\right) $  whose restriction $\sigma
|_{K\left( \Delta _{1}\right) }$  to $K\left( \Delta _{1}\right)
$ is $\delta $.

In such a case, we must have $\left[ L:K\left( \Delta _{1}\right) \right] =%
\left[ L:K\left( \Delta _{2}\right) \right] $. It follows that a $K$%
-isomorphism $\delta :K\left( \Delta _{1}\right) \rightarrow K\left( \Delta
_{2}\right) $ is not necessarily extendable in $L/K$. The construction for
such an automorphism $\sigma \in Aut\left( L/K\right) $ of an extendable $%
\delta $ can be obtained from \emph{Lemma 2.5} in \emph{\S 2}. (See \emph{\S %
8.4-6} for further discussions).

\begin{remark}
(\emph{The underlying mystery of Noether's problem, II. From patterns of
automorphisms of the extension}) Let $L=K\left( t_{1},t_{2},\cdots
,t_{n}\right) $ be a purely transcendental extension over a field $K$ of
transcendence degree $n$. Let $G$ be a subgroup of the Galois group $%
Aut\left( L/K\right) $.

By extendable $K$-isomorphisms $\delta :K\left( \Delta _{1}\right)
\rightarrow K\left( \Delta _{2}\right) $ between subfields of $L$ for two
transcendence bases $\Delta _{1},\Delta _{2}$ of $L/K$, we have several
particular patterns for automorphisms $\sigma $ of $L/K$ (relative to the
variables $t_{1},t_{2},\cdots ,t_{n}$).

$Case$ $\left( i\right) $. The map $\sigma \in Aut\left( L/K\right) $  is said to be of \textbf{%
vertical type} (relative to the variables $t_{1},t_{2},\cdots ,t_{n}$) if $%
\sigma $ is the composite of some extensions to $L$ of extendable $K$%
-isomorphisms of the form $\sigma _{i_{0}}:K\left( t_{1},\cdots
,t_{i_{0}-1},P\left( t_{i_{0}}\right) ,t_{i_{0}+1},\cdots ,t_{n}\right)
\rightarrow K\left( t_{1},\cdots ,t_{i_{0}-1},Q\left( t_{i_{0}}\right)
,t_{i_{0}+1},\cdots ,t_{n}\right) $ which is a $K$-isomorphism between
subfields in $L$ for some $1\leq i_{0}\leq n$ satisfying the property: $%
\sigma _{i_{0}}\left( t_{i}\right) =t_{i}$ for each $1\leq i\leq n$ with $%
i\not=i_{0}$; $\sigma _{i_{0}}\left( P\left( t_{i_{0}}\right) \right)
=Q\left( t_{i_{0}}\right) $, where $P\left( X\right) ,Q\left( X\right) \in
K\left( X\right) $ are nonconstant polynomials of an indeterminate $X$ with
coefficients in the field $K$. In such a case, the degree $\deg P\left(
X\right) =\deg Q\left( X\right) $ is called the $i_{0}$\textbf{-th height}
of the map $\sigma _{0}$.

If every map $\sigma \in G$ is of vertical type, then $L=L^{G}\left( \beta
\right) $ can have a generator $\beta $ over the invariant subfield $L^{G}$
such that any $L^{G}$-conjugate of $\beta $ must involve the elements in the
roots of unity in $K$; hence, $L$ is required to have \textquotedblleft
enough roots\textquotedblright over $K$ to guarantee that $L$ is Galois over
the invariant subfield $L^{G}$ (See \emph{[}Fisher 1915\emph{]}). In such a
case, $G$ is a Noether solution of $L/K$.

$Case$ $\left( ii\right) $. The map $\sigma \in Aut\left( L/K\right)$ is said to be of
\textbf{horizontal type} (relative to the variables $t_{1},t_{2},\cdots
,t_{n}$) if $\sigma $ is the composite of some maps $\sigma
_{0}:L\rightarrow L$ defined by $\sigma _{0}:t_{j_{0}}\longmapsto
t_{j_{0}+m_{0}} $ for some $j_{0}\in \left( \mathbb{Z}/n\mathbb{Z};+\right) $%
; $\sigma _{0}:t_{j}\longmapsto t_{j} $ for any other $j\in \left( \mathbb{Z}%
/n\mathbb{Z};+\right) $ with $j\not=j_{0}$.

Here $m_{0}\in \left( \mathbb{Z}/n\mathbb{Z};+\right) $ is called the $j_{0}$%
-th \textbf{step} of the map $\sigma _{0}\in Aut\left( L/K\right) $.

If every map $\sigma \in G$ is of horizontal type, then (under certain
assumption, for instance, $G$ is a finite cyclic group) $L=L^{G}\left( \beta
\right) $ can have a generator $\beta $ over the invariant subfield $L^{G}$
such that any $L^{G}$-conjugate of $\beta $ does not involve the elements in
the roots of unity in $K$. In such a case, $G$ is not necessarily a Noether
solution of $L/K$ in general.

$Case$ $\left( iii\right) $. The map $\sigma \in Aut\left( L/K\right)$ is said to be of
\textbf{slanted type} (relative to the variables $t_{1},t_{2},\cdots ,t_{n}$%
) if $\sigma $ is a finite number of composites of the maps $\delta
:L\rightarrow L$ of horizontal type and of vertical type.

In such a case, if all the heights of $\sigma $ are equal to some non-zero
integer $m_{0}$, the map $\sigma $ is said to be with the \textbf{same level}
$m_{0}$; otherwise, if $\sigma $ has at least two distinct heights (greater
than one), the map $\sigma $ is said to be with \textbf{mixed levels}.

If $G$ contains maps of slanted type with mixed levels, then $G$ is not
necessarily a Noether solution of the extension $L/K$; it is not true that
the invariant subfield $L^{G}$ is purely transcendental over $K$.

$Case$ $\left( iv\right) $. The map $\sigma \in Aut\left( L/K\right)$ is said to be of
\textbf{linear type} (relative to the variables $t_{1},t_{2},\cdots ,t_{n}$)
if there is an invertible $n\times n$ matrix $A$ and an $n\times 1$ matrix $%
B $ both with coefficients in $K$ such that
\begin{equation*}
\sigma :\left(
\begin{array}{c}
t_{1} \\
t_{2} \\
\vdots \\
t_{n}%
\end{array}%
\right) \longmapsto A\cdot \left(
\begin{array}{c}
t_{1} \\
t_{2} \\
\vdots \\
t_{n}%
\end{array}%
\right) +B.
\end{equation*}

In such a case, if $B$ is the zero matrix, the map $\sigma $ is said to be
of \textbf{homogeneous linear type}.

$Case$ $\left( v\right) $. The map $\sigma \in Aut\left( L/K\right)$ is said to be of \textbf{%
monomial type} (relative to the variables $t_{1},t_{2},\cdots ,t_{n}$) if
there is an invertible $n\times n$ matrix $A=\left( a_{ij}\right) _{1\leq
i,j\leq n}$ with coefficients in $\mathbb{Z}$ such that $\sigma $ is an
extension to $L$ of some extendable $K$-isomorphism $\sigma _{0}:K\left(
\Delta _{1}\right) \rightarrow K\left( \Delta _{2}\right) $ where

\begin{equation*}
\sigma _{0}:\left(
\begin{array}{c}
P_{1}\left( t_{1}\right) \\
P_{2}\left( t_{2}\right) \\
\vdots \\
P_{n}\left( t_{n}\right)%
\end{array}%
\right) \longmapsto \left(
\begin{array}{c}
\prod\limits_{j=1}^{n}t_{j}^{a_{1j}} \\
\prod\limits_{j=1}^{n}t_{j}^{a_{2j}} \\
\vdots \\
\prod\limits_{j=1}^{n}t_{j}^{a_{nj}}%
\end{array}%
\right) ;
\end{equation*}
\begin{equation*}
P_{1}\left( X\right) ,P_{2}\left( X\right) ,\cdots ,P_{n}\left( X\right) \in
K\left[ X\right]
\end{equation*}%
are polynomials of the indeterminate $X$ with coefficients in $K$;
\begin{equation*}
\Delta _{1}=\{P_{1}\left( t_{1}\right) ,P_{2}\left( t_{2}\right) ,\cdots
,P_{n}\left( t_{n}\right) \}
\end{equation*}%
and
\begin{equation*}
\Delta
_{2}=\{\prod\limits_{j=1}^{n}t_{j}^{a_{1j}},\prod%
\limits_{j=1}^{n}t_{j}^{a_{2j}},\cdots
,\prod\limits_{j=1}^{n}t_{j}^{a_{nj}}\}
\end{equation*}%
are two transcendence bases of $L/K$.

$Case$ $\left( vi\right) $. The map $\sigma \in Aut\left( L/K\right)$ is said to be of
\textbf{geometric type} if $\sigma $ arises from a geometric group, a Lie
group or a pro-Lie group.

$Case$ $\left( vii\right) $. The map $\sigma \in Aut\left( L/K\right)$ is said to be of $%
\infty $ \textbf{type} if $\sigma $ is none of the above types.
\end{remark}

See \emph{\S 8.1} \& \emph{\S 8.3} for further results.

\section{Preliminaries}

In this section we will fix notation and terminology and review some
preliminary facts for transcendental Galois theory (See \emph{[}An 2010\emph{%
]}). These facts will provide us a useful picture for construction of fields
in the transcendental Galois theory. In fact, the phenomenon for \emph{tame
Galois} occurs through the whole paper.

\subsection{Notation and terminology}

For a finite group (or a finite set) $X$, let $\sharp X$ denote the number
of elements contained in $X$.

Let $K$ be a field. Denote by $\overline{K}$ a given algebraic closure of
the field $K$.

Let $L$ be an arbitrary extension of a field $K$ (algebraic or
transcendental). Suppose $\Delta $ and $A$ are two subsets of $L$ such that $%
L=K(\Delta ,A)$ is algebraic over $K\left( \Delta \right) $ and $K\left(
\Delta \right) $ is purely transcendental over $K$. Here, $\Delta $ can be
possibly empty.

Recall that a \textbf{nice basis} of $L/K$ is a pair $\left( \Delta
,A\right) $ such that $\Delta $ is a transcendence base of $L$ over $K$ and $%
A$ is a $K\left( \Delta \right) $-linear basis of $L$ as a vector space over
$K(\Delta )$.

Recall that the \textbf{Galois group} of the extension $L/K$ is the group $%
Aut\left( L/K\right) $ consisting of all the automorphisms $\sigma $ of the
field $L$ such that $\sigma (x)=x$ holds for all $x\in K$.

Let $H$ be a subset of the Galois group $Aut\left( L/K\right) $. The \textbf{%
invariant subfield} of $L$ under $H$ is the subfield $L^{H}$ consisting of
any element $x\in L$ such that $\sigma \left( x\right) =x$ holds for every $%
\sigma \in H$. (See \emph{Definition 1.4}).

\subsection{Transcendental Galois extensions}

Let $L$ be an arbitrary extension of a field $K$. Here, for a nice basis $%
\left( \Delta ,A\right) $ of $L/K$, $\Delta $ can be possibly empty; in such
a case, $L$ is algebraic over $K$. There are several types of transcendental
Galois extensions. (See \emph{Remark 3.15} for a comparison).

Recall that $L$ is said to be \textbf{Galois} over $K$ if $K$ is the
invariant subfield $L^{Aut\left( L/K\right) }$ of $L$ under the Galois group
$Aut\left( L/K\right) $. (See \emph{[}An 2010\emph{]}).

$L$ is said to be \textbf{tame Galois} over $K$ if there is a transcendence
base $\Delta $ of the extension $L/K$ such that $L$ is algebraic Galois over
the subfield $K\left( \Delta \right) $. (See \emph{[}An 2010\emph{]}).

$L$ is said to be $\sigma $\textbf{-tame Galois} over $K$ if $L$ is
algebraic Galois over $K\left( \Delta \right) $ for every transcendence base
$\Delta $ of $L/K$. (See \emph{[}An 2010\emph{]}).

Let $\Omega _{L/K}$ be the set of all the transcendence bases $\Delta $ of
the extension $L/K$. Fixed a non-void subset $\Omega _{0}$ of $\Omega _{L/K}$%
. $L$ is said to be $\sigma $\textbf{-tame Galois} over $K$ \textbf{inside} $%
\Omega _{0}$ if $L$ is algebraic Galois over $K\left( \Delta \right) $ for
each transcendence base $\Delta $ contained in $\Omega _{0}$. (See \emph{%
Definition 2.12}).

$L$ is said to be \textbf{complete Galois} over $K$ if $L$ is Galois over
any intermediate subfield $M$ in $L/K$, i.e., $K\subseteq M\subseteq L$.
(See \emph{Definition 3.13}). Note that \emph{complete Galois} is also
called \emph{absolute Galois} in \emph{[}An 2010\emph{]}.

\subsection{Conjugation and quasi-Galois}

There is a conjugation of a field, as a conjugate of a field for algebraic
Galois theory, which play the same role for transcendental Galois theory.

\begin{definition}
Let $L$ and $M$ be two extensions of a field $K$. Suppose $M\subseteq
\overline{L}$, that is, $M$ is contained in a fixed algebraic closure $%
\overline{L}$ of $L$. Then $M$ is said to be a $K$\textbf{-conjugation} of $%
L $ if there is a transcendence base $\Delta $ of the extension $L/K$ such
that $M$ is a $K(\Delta )$-conjugate of $L$ (in the sense of algebraic
extensions of fields). In such a case, $M$ is also said to be a $K$\textbf{%
-conjugation} of $L$ with respect to $\Delta $ (or $\left( \Delta ,A\right) $%
). Here, $A$ is a $K\left( \Delta \right) $-linear basis of the vector space
$L$.
\end{definition}

Also there are several types of quasi-Galois extensions in transcendental
Galois theory, as the counter-parts in algebraic Galois theory.

\begin{definition}
Let $L$ be an extension of a field $K$.

$\left( i\right) $ $L$ is said to be \textbf{quasi-Galois} over $K$ if there
is a transcendence base $\Delta $ of the extension $L/K$ such that each
irreducible polynomial $f(X)$ over the subfield $K(\Delta )$ that has a root
in $L$ factors completely in $L\left[ X\right] $ into linear factors.

In such a case, $L$ is also said to be \textbf{quasi-Galois} over $K$
\textbf{with respect to }$\Delta $ (or $(\Delta ,A)$).

$\left( ii\right) $ $L$ is said to be $\sigma $-\textbf{quasi-Galois} over $%
K $ (respectively, \textbf{inside} $\Omega _{0}$) if for each transcendence
base $\Delta $ of the extension $L/K$ (respectively, with $\Delta \in \Omega
_{0}$), each irreducible polynomial $f(X)$ over the subfield $K(\Delta )$
that has a root in $L$ factors completely in $L\left[ X\right] $ into linear
factors.
\end{definition}

\begin{proposition}
Let $L$ be an arbitrary extension of a field $K$. The following statements
are equivalent:

$(i)$ $L$ is quasi-Galois over $K$.

$(ii)$ There is a transcendence base $\Delta $ of $L/K$ such that $L$ is an
algebraic normal extension of the subfield $K\left( \Delta \right) $.

$(iii)$ There is a transcendence base $\Delta $ of $L/K$ such that $L$
contains every $F$-conjugate of $F\left( x\right) $ for any $x\in L\setminus
K(\Delta )$ and subfield $K(\Delta )\subseteq F\subseteq L$.

$\left( iv\right) $ There is a transcendence base $\Delta $ of $L/K$ such
that each $K$-conjugation of $L$ with respect to $\Delta $ must be contained
in $L$.

$\left( v\right) $ There is a transcendence base $\Delta $ of $L/K$ such
that $L$ is the unique $K$-conjugation of $L$ with respect to $\Delta $.
\end{proposition}

\begin{proof}
Prove $\left( i\right) \Leftrightarrow \left( ii\right) \Leftrightarrow
\left( iii\right) $. It is immediately from the fact that $L$ is
quasi-Galois over $K$ with respect to a transcendence base $\Delta $ of $L/K$
if and only if $L$ is an algebraic normal extension of the subfield $K\left(
\Delta \right) $.

Prove $\left( ii\right) \Leftrightarrow \left( iv\right) $ and $\left(
ii\right) \Leftrightarrow \left( v\right) $. Trivial from the algebraic
normal extension $L$ of the subfield $K\left( \Delta \right) $.
\end{proof}

\begin{proposition}
Let $L$ be an arbitrary extension of a field $K$. The following statements
are equivalent:

$(i)$ $L$ is $\sigma $-quasi-Galois over $K$ (respectively, inside $\Omega
_{0}$).

$(ii)$ For any nice basis $\Delta $ of $L/K$ (respectively, with $\Delta \in
\Omega _{0}$), $L$ is an algebraic normal extension of the subfield $K\left(
\Delta \right) $.

$(iii)$ For any nice basis $\Delta $ of $L/K$ (respectively, with $\Delta
\in \Omega _{0}$), $L$ contains every $F$-conjugate of $F\left( x\right) $
for any $x\in L\setminus K(\Delta )$ and subfield $K(\Delta )\subseteq
F\subseteq L$.

$\left( iv\right) $ For any nice basis $\Delta $ of $L/K$ (respectively,
with $\Delta \in \Omega _{0}$), each $K$-conjugation of $L$ with respect to $%
\Delta $ must be contained in $L$.

$\left( v\right) $ For any nice basis $\Delta $ of $L/K$ (respectively, with
$\Delta \in \Omega _{0}$), $L$ is the unique $K$-conjugation of $L$ with
respect to $\Delta $.
\end{proposition}

\begin{proof}
Prove $\left( i\right) \Leftrightarrow \left( ii\right) \Leftrightarrow
\left( iii\right) \Leftrightarrow \left( iv\right) \Leftrightarrow \left(
v\right) $. It is immediately from the fact that for each transcendence base
$\Delta $ of the extension $L/K$, the field $L$ is quasi-Galois over $K$
with respect to $\Delta $ if and only if $L$ is an algebraic normal
extension of the subfield $K\left( \Delta \right) $.

Prove $\left( v\right) \iff \left( vii\right) $; $\left( vi\right) \iff
\left( vii\right) $. Trivial.
\end{proof}

\subsection{Conjugation by adjunction}

Let $L$ be a finitely generated extension of a field $K$. Let $\overline{L}$
denote an algebraic closure of $L$.

Assume $\left( \Delta ,A\right) $ is a \textbf{nice basis} of the extension $%
L/K$. If $\Delta $ and $A$ both are finite sets, we say $\left( \Delta
,A\right) $ is a \textbf{finite nice basis} of $L/K$; in particular, if $%
\Delta =\{w_{1},w_{2},\cdots ,w_{r}\}$ and $A=\{w_{r+1},w_{r+2},\cdots
,w_{n}\}$, we also say $w_{1},w_{2},\cdots ,w_{n}\in L$ are an \textbf{$%
(r,n) $-nice basis} of $L/K$. In general, the types $\left( r,n\right) $ of
its nice bases can be very diverse.

\begin{lemma}
Let $F$ be a subfield with $K\subseteq F\subsetneqq L$ and let $L/K$ has a
nice basis of type $\left( r,n\right) $ with $0\leq r\leq n$. Suppose $x\in
L $ is algebraic over $F$ and $z\in \overline{L}$ is a conjugate of $x$ over
$F $.

Then for each $(s,m)$-nice basis $v_{1},v_{2},\cdots ,v_{m}$ of the
extension $L/F\left( x\right) $ there exists an $F$-isomorphism $\tau $ from
the field
\begin{equation*}
L=F\left( x,v_{1},v_{2},\cdots ,v_{s},v_{s+1},\cdots ,v_{m}\right)
\end{equation*}%
onto the field of the form
\begin{equation*}
\tau \left( L\right) =F\left( z,v_{1},v_{2},\cdots ,v_{s},w_{s+1},\cdots
,w_{m}\right)
\end{equation*}%
such that
\begin{equation*}
x\longmapsto z,v_{1}\longmapsto v_{1},\cdots ,v_{s}\longmapsto
v_{s},v_{s+1}\longmapsto w_{s+1},\cdots ,v_{m}\longmapsto w_{m}
\end{equation*}%
where
\begin{equation*}
w_{s+1},w_{s+2},\cdots ,w_{m}
\end{equation*}%
are elements contained in an algebraic closure $\overline{L}$ of the field $%
L $ and $0\leq s\leq r,m\leq n$.

Moreover, we can choose
\begin{equation*}
w_{s+1}=v_{s+1},w_{s+2}=v_{s+2},\cdots ,w_{m}=v_{m}
\end{equation*}%
for the case that $z\notin F(v_{1},v_{2},\cdots ,v_{m})$ and $v_{s+1},\cdots
,v_{m}\notin \overline{F}$.
\end{lemma}

\begin{proof}
If $x\in F$, then $z\in F$ since $z$ is an $F$-conjugate of $x$. Without
loss of generality, we suppose $x\notin F$ and $z\notin F$.

Fixed an $(s,m)$-nice basis $v_{1},v_{2},\cdots ,v_{m}$ of the extension $%
L/F\left( x\right) $. If $s=0$, go to \emph{Step 2}.

For the case $s\not=0$, we will proceed in several steps in the following.

\emph{Step 1}. Suppose $s\not=0$. That is, at least $v_{1}$ is a variable
over the field $F\left( x\right) $.

Let $\sigma _{0}$ be the $F-$isomorphism between fields $F(x)$ and $F(z)$
with $\sigma _{0}(x)=z$. Via $\sigma _{0}$, we have got an isomorphism $%
\sigma _{1}$ of the subfield $F\left( x,v_{1}\right) $ onto the subfield $%
F\left( z,v_{1}\right) $ given by
\begin{equation*}
\sigma _{1}:\frac{f(v_{1})}{g(v_{1})}\mapsto \frac{\sigma _{0}(f)(v_{1})}{%
\sigma _{0}(g)(v_{1})}
\end{equation*}%
for any polynomials
\begin{equation*}
f[X_{1}],g[X_{1}]\in F\left( x\right) [X_{1}]
\end{equation*}%
with
\begin{equation*}
g[X_{1}]\neq 0,
\end{equation*}
where $\sigma _{0}(g)\left( X_{1}\right) $ denotes the polynomial
\begin{equation*}
\sigma _{0}(g)\left( X_{1}\right) =\sum_{i=0}^{r}\sigma _{0}\left(
a_{i}\right) X_{1}^{i}
\end{equation*}
with coefficients $\sigma _{0}\left( a_{i}\right) \in F\left( z\right) $ for
a polynomial
\begin{equation*}
g\left( X_{1}\right) =\sum_{i=0}^{r}a_{i}X_{1}^{i}
\end{equation*}
with coefficients $a_{i}\in F\left( x\right) $.

Evidently, the map $\sigma _{1}$ is well-defined from the fact that
\begin{equation*}
g(v_{1})=0
\end{equation*}%
if and only if
\begin{equation*}
{\sigma _{0}(g)(v_{1})}=0.
\end{equation*}

In the same way, for the elements
\begin{equation*}
v_{1},v_{2},\cdots ,v_{s}\in L
\end{equation*}%
which are variables over $F(x)$, there is an isomorphism
\begin{equation*}
\sigma _{s}:F\left( x,v_{1},v_{2},\cdots ,v_{s}\right) \longrightarrow
F\left( z,v_{1},v_{2},\cdots ,v_{s}\right)
\end{equation*}%
of fields defined by
\begin{equation*}
x\longmapsto z\text{ and }v_{i}\longmapsto v_{i}
\end{equation*}%
for $1\leq i\leq s$, where for the restrictions we have
\begin{equation*}
\sigma _{s}|_{F\left( x,v_{1},v_{2},\cdots ,v_{i}\right) }=\sigma _{i}.
\end{equation*}

If $s=m$, we have
\begin{equation*}
L=F\left( x,v_{1},v_{2},\cdots ,v_{s}\right) .
\end{equation*}%
Then $\tau =\sigma _{s}$ and $F\left( z,v_{1},v_{2},\cdots ,v_{s}\right)
=\tau \left( L\right) $ is the desired field.

\emph{Step 2}. Suppose $1\leq s\leq m-1$. Consider the element $v_{s+1}\in L$%
.

As the element $x$ is algebraic over $F$ and $v_{s+1}$ is algebraic over the
subfield $F\left( x,v_{1},v_{2},\cdots ,v_{s}\right) $, it is seen that $%
v_{s+1}$ is algebraic over the subfield
\begin{equation*}
F(v_{1},v_{2},\cdots ,v_{s})\subseteq L.
\end{equation*}

There are several cases for the two elements $v_{s+1}$ and $z$.

\emph{Case (i)}. Suppose $z\notin F(v_{1},v_{2},\cdots ,v_{s+1})$ and $%
v_{s+1}\in \overline{F(v_{1},v_{2},\cdots ,v_{s})}\setminus \overline{F}$.

There is an isomorphism
\begin{equation*}
\sigma _{s+1}:F\left( x,v_{1},v_{2},\cdots ,v_{s+1}\right) \cong F\left(
z,v_{1},v_{2},\cdots ,v_{s+1}\right)
\end{equation*}%
of fields given by
\begin{equation*}
x\longmapsto z,v_{1}\longmapsto v_{1},\cdots ,v_{s}\longmapsto v_{s}.
\end{equation*}

Prove the map $\sigma _{s+1}$ is well-defined. In deed, by the below \emph{%
Claim 2.6} it is seen that
\begin{equation*}
f\left( v_{s+1}\right) =0
\end{equation*}%
holds if and only if
\begin{equation*}
\sigma _{s}\left( f\right) \left( v_{s+1}\right) =0
\end{equation*}%
holds for any polynomial
\begin{equation*}
f\left( X_{s+1}\right) \in F\left( x,v_{1},v_{2},\cdots ,v_{s}\right) \left[
X_{s+1}\right] .
\end{equation*}

\emph{Case (ii)}. Suppose $z\notin F(v_{1},v_{2},\cdots ,v_{s+1})$ and $%
v_{s+1}\in \overline{F}$.

By Claim 2.2 there is some $v_{s+1}^{\prime }\in \overline{F}$ and an $F$%
-isomorphism
\begin{equation*}
\sigma _{s+1}:F\left( x,v_{1},v_{2},\cdots ,v_{s},v_{s+1}\right)
\longrightarrow F\left( z,v_{1},v_{2},\cdots ,v_{s},v_{s+1}^{\prime }\right)
\end{equation*}%
of fields defined by
\begin{equation*}
x\longmapsto z,v_{1}\longmapsto v_{1},\cdots ,v_{s}\longmapsto
v_{s},v_{s+1}\longmapsto v_{s+1}^{\prime }.
\end{equation*}

\emph{Case (iii)}. Suppose that $z$ is contained in the field $%
F(v_{1},v_{2},\cdots ,v_{s+1})$.

By the below \emph{Claim 2.7} we have an element $v_{s+1}^{\prime }$
contained in an extension of $F$ such that the fields $F(x,v_{s+1})$ and $%
F(z,v_{s+1}^{\prime })$ are isomorphic over $F$.

Then by the same procedure as in \emph{Case (i)} of \emph{Claim 2.6} it is
seen that the two fields
\begin{equation*}
F(x,v_{s+1},v_{1},v_{2},\cdots ,v_{s})\cong F(z,v_{s+1}^{\prime
},v_{1},v_{2},\cdots ,v_{s})
\end{equation*}%
are isomorphic over $F$.

In the same way, for the elements $v_{s+2},\cdots ,v_{m}$ and we have an $F$%
-isomorphism $\tau $ from the field
\begin{equation*}
F\left( x,v_{1},v_{2},\cdots ,v_{s},v_{s+1},\cdots ,v_{m}\right)
\end{equation*}%
onto the field of the form
\begin{equation*}
F\left( z,v_{1},v_{2},\cdots ,v_{s},w_{s+1},\cdots ,w_{m}\right)
\end{equation*}%
such that
\begin{equation*}
x\longmapsto z,v_{1}\longmapsto v_{1},\cdots ,v_{s}\longmapsto
v_{s},v_{s+1}\longmapsto w_{s+1},\cdots ,v_{m}\longmapsto w_{m}.
\end{equation*}%
where
\begin{equation*}
w_{s+1},w_{s+2},\cdots ,w_{m}
\end{equation*}%
are elements contained in an extension of $F$. This completes the proof.
\end{proof}

\begin{claim}
Take an element $f$ in the polynomial ring $F\left[ X,X_{1},X_{2},\cdots
,X_{s+1}\right] $ over the field $F$. Suppose that $z$ is not contained in
the field $F(v_{1},v_{2},\cdots ,v_{s+1})$.

$\left( i\right) $ Let $v_{s+1}\in \overline{F(v_{1},v_{2},\cdots ,v_{s})}%
\setminus \overline{F}$. Then
\begin{equation*}
f\left( x,v_{1},v_{2},\cdots ,v_{s},v_{s+1}\right) =0
\end{equation*}%
holds if and only if
\begin{equation*}
f\left( z,v_{1},v_{2},\cdots ,v_{s},v_{s+1}\right) =0
\end{equation*}%
holds.

$\left( ii\right) $ Let $v_{s+1}\in \overline{F}$. Then there is some $%
v_{s+1}^{\prime }\in \overline{F}$ such that
\begin{equation*}
f\left( x,v_{1},v_{2},\cdots ,v_{s},v_{s+1}\right) =0
\end{equation*}%
holds if and only if
\begin{equation*}
f\left( z,v_{1},v_{2},\cdots ,v_{s},v_{s+1}^{\prime }\right) =0
\end{equation*}%
holds.
\end{claim}

\begin{proof}
Here we use \emph{Weil's algebraic theory of specializations} (See \emph{[}%
Weil 1946\emph{]}) to give the proof.

\emph{Case (i)}. Let $v_{s+1}\in \overline{F(v_{1},v_{2},\cdots ,v_{s})}%
\setminus \overline{F}$.

By \emph{Theorem 1 }(on \emph{generic specializations}) on \emph{Page 28},
\emph{[}Weil 1946\emph{]}, it is clear that $\left( z\right) $ is a generic
specialization of $\left( x\right) $ over $F$ since $z$ and $x$ are
conjugates over $F$. From \emph{Proposition 1} on \emph{Page 3}, \emph{[}%
Weil 1946\emph{]}, it is seen that $F\left( v_{1},v_{2},\cdots
,v_{s+1}\right) $ and the field $F(x)$ are free with respect to each other
over $F$ since $x$ is algebraic over $F$. Hence, $F\left( v_{1},v_{2},\cdots
,v_{s+1}\right) $ is a free field over $F$ with respect to $(x)$.

By \emph{Proposition 3} on \emph{Page 4}, \emph{[}Weil 1946\emph{]}, it is
seen that $F\left( v_{1},v_{2},\cdots ,v_{s+1}\right) $ and the algebraic
closure $\overline{F}$ are linearly disjoint over $F$. Then $F\left(
v_{1},v_{2},\cdots ,v_{s+1}\right) $ is a regular extension of $F$ (see
\emph{[}Weil 1946\emph{]}, \emph{Page 18}).

From \emph{Theorem 5} (on \emph{extensions of specializations}) on \emph{%
Page 29}, \emph{[}Weil 1946\emph{]}, it is seen that
\begin{equation*}
\left( z,v_{1},v_{2},\cdots ,v_{s+1}\right)
\end{equation*}%
is a generic specialization of
\begin{equation*}
\left( x,v_{1},v_{2},\cdots ,v_{s+1}\right)
\end{equation*}%
over the field $F$ since we have the two generic specializations over $F$:

\qquad $\left( z\right) $ is a generic specialization of $\left( x\right) $
over $F$;

\qquad $\left( v_{1},v_{2},\cdots ,v_{s+1}\right) $ is a generic
specialization of $\left( v_{1},v_{2},\cdots ,v_{s+1}\right) $ itself over $%
F $.

Now for any polynomial $f\left( X,X_{1},X_{2},\cdots ,X_{s+1}\right) $ over
the field $F$, it is seen that
\begin{equation*}
f\left( x,v_{1},v_{2},\cdots ,v_{s+1}\right) =0
\end{equation*}%
holds if and only if
\begin{equation*}
f\left( z,v_{1},v_{2},\cdots ,v_{s+1}\right) =0
\end{equation*}%
holds from the preliminary facts of generic specializations.

\emph{Case (ii)}. Let $v_{s+1}\in \overline{F}$. By \emph{Step 1} in the
proof of \emph{Lemma 2.5}, we have an $F$-isomorphism
\begin{equation*}
\sigma _{s}:F\left( x,v_{1},v_{2},\cdots ,v_{s}\right) \longrightarrow
F\left( z,v_{1},v_{2},\cdots ,v_{s}\right)
\end{equation*}%
of fields which maps%
\begin{equation*}
x\longmapsto z,v_{1}\longmapsto v_{1},\cdots ,v_{s}\longmapsto v_{s}.
\end{equation*}

From the below \emph{Claim 2.7} we have an element $v_{s+1}^{\prime }$
contained in an algebraic extension of $F$ such that there is an $F$%
-isomorphism
\begin{equation*}
\sigma _{s+1}:F\left( x,v_{1},v_{2},\cdots ,v_{s},v_{s+1}\right)
\longrightarrow F\left( z,v_{1},v_{2},\cdots ,v_{s},v_{s+1}^{\prime }\right)
\end{equation*}%
of fields which maps%
\begin{equation*}
x\longmapsto z,v_{1}\longmapsto v_{1},\cdots ,v_{s}\longmapsto
v_{s},v_{s+1}\longmapsto v_{s+1}^{\prime }.
\end{equation*}

$F(x,v_{s+1})$ and $F(z,v_{s+1}^{\prime })$ are isomorphic over $F$.

As in the above \emph{Case (i)} it is seen that
\begin{equation*}
\left( x,v_{1},v_{2},\cdots ,v_{s},v_{s+1}\right)
\end{equation*}%
is a (generic) specialization of
\begin{equation*}
\left( z,v_{1},v_{2},\cdots ,v_{s},v_{s+1}^{\prime }\right)
\end{equation*}%
over $F$. Hence, for any polynomial $f\left( X,X_{1},X_{2},\cdots
,X_{s+1}\right) $ over the field $F$, it is seen that
\begin{equation*}
f\left( x,v_{1},v_{2},\cdots ,v_{s},v_{s+1}\right) =0
\end{equation*}%
holds if and only if
\begin{equation*}
f\left( z,v_{1},v_{2},\cdots ,v_{s},v_{s+1}^{\prime }\right) =0.
\end{equation*}
This completes the proof.
\end{proof}

\begin{claim}
Assume that $F(u)$ and $F(u^{\prime })$ are $F$-isomorphic fields by the map
\begin{equation*}
u\mapsto u^{\prime }.
\end{equation*}
Let $w$ be an element contained in an extension of $F $. Then there is an
element $w^{\prime }$ contained in some extension of $F $ such that the
fields $F(u,w)$ and $F(u^{\prime },w^{\prime })$ are isomorphic over $F$.
\end{claim}

\begin{proof}
It is immediate from \emph{Proposition 4} on \emph{Page 30}, \emph{[}Weil
1946\emph{]}.
\end{proof}

\subsection{Separability for transcendental extensions}

Let ${L}$ be a transcendental extension over a field $K$.

Recall that $L$ is \textbf{separably generated} over $K$ if for each
transcendence base $\Delta $ of $L$ over $K$, the field $L$ is an algebraic
separable extension of $K\left( \Delta \right) $.

\begin{proposition}
Let $L$ be a finitely generated extension of a field $K$. The following
statements are equivalent.

$\left( i\right) $ $L$ is separably generated over $K$.

$\left( ii\right) $ For each transcendence base $\Delta $ of $L$ over $K$, $%
L $ is separable over $K\left( \Delta \right) $.

$\left( iii\right) $ For some transcendence base $\Lambda $ of $L\ $over $K$%
, $L$ is separable over $K\left( \Lambda \right) $.
\end{proposition}

\begin{proof}
It reduces to prove $\left( iii\right) \implies \left( ii\right) $. Let $%
\Delta $ be a transcendence base of $L$ over $K$ and $L$ separable over $%
K\left( \Delta \right) $. By taking a linear basis $A$ of the vector space $%
L $ over the subfield $K\left( \Delta \right) $, we obtain a nice basis $%
\left( \Delta ,A\right) $ of the extension $L/K$. Given any transcendence
base $\Lambda $ of $L$ over $K$. In the same way, we have a nice basis $%
\left( \Lambda ,B\right) $ of $L/K$.

We prove $L$ is separable over $K\left( \Lambda \right) $. In deed, take any
element $\beta \in L$. As $\beta $ is algebraic over $K\left( \Delta \right)
$ and $K\left( \Lambda \right) $, we consider the minimal polynomials $%
f\left( t\right) $, $g\left( t\right) $ and $h\left( t\right) $ of $\beta $
over the subfields $K\left( \Delta \right) $, $K\left( \Lambda \right) $ and
$K\left( \Delta ,\Lambda \right) $, respectively. As the subfield $K\left(
\Delta ,\Lambda \right) $ is a finite separable extension of $K\left( \Delta
\right) $, we have
\begin{equation*}
f\left( t\right) =h\left( t\right) \cdot \left( t-a_{1}\right) \cdots \left(
t-a_{m}\right)
\end{equation*}%
with distinct roots
\begin{equation*}
a_{1},\cdots ,a_{m}\in K\left( \Delta ,\Lambda \right)
\end{equation*}%
and the derivative
\begin{equation*}
\frac{d}{dt}f\left( t\right) \not=0;
\end{equation*}%
then the derivative
\begin{equation*}
\frac{d}{dt}h\left( t\right) \not=0
\end{equation*}%
and it follows that $h\left( t\right) $ and $\frac{d}{dt}h\left( t\right) $
are coprime polynomials of variable $t$ with coefficients in the subfield $%
K\left( \Delta ,\Lambda \right) $.

In the same way, as $K\left( \Delta ,\Lambda \right) $ is a finite extension
of $K\left( \Lambda \right) $, we have a nonzero polynomial $b\left(
t\right) $ with coefficients in the subfield $K\left( \Delta ,\Lambda
\right) $ such that
\begin{equation*}
g\left( t\right) =h\left( t\right) \cdot b\left( t\right) ;
\end{equation*}
for the derivative of $g\left( t\right) $, we have
\begin{equation*}
\frac{d}{dt}g\left( t\right) =\frac{d}{dt}h\left( t\right) \cdot b\left(
t\right) +h\left( t\right) \cdot \frac{d}{dt}b\left( t\right) ;
\end{equation*}
then we must have
\begin{equation*}
\frac{d}{dt}g\left( t\right) \not=0;
\end{equation*}
otherwise, if
\begin{equation*}
\frac{d}{dt}g\left( t\right) =0,
\end{equation*}
that is,
\begin{equation*}
\frac{d}{dt}h\left( t\right) \cdot b\left( t\right) +h\left( t\right) \cdot
\frac{d}{dt}b\left( t\right) =0,
\end{equation*}
then the two polynomials $h\left( t\right) $ and $\frac{d}{dt}h\left(
t\right) $ will not be co-prime over the subfield $K\left( \Delta ,\Lambda
\right) $, which will be in contradiction. Hence, $g\left( t\right) $ is a
separable polynomial over $K\left( \Lambda \right) $. This proves every
element $\beta \in L$ is separable over $K\left( \Lambda \right) $.
\end{proof}

\subsection{Finiteness for nice bases of extensions}

To end the section, we prove that every nice basis of a finitely generated
extension is finite.

\begin{proposition}
Let $L$ be a finitely generated extension over a field $K$. Then each nice
basis $\left( \Delta ,A\right) $ of the extension $L/K$ is a finite nice
basis, i.e., $\Delta $ and $A$ both are finite sets.
\end{proposition}

\begin{proof}
Suppose $L$ is generated over $K$ by a finite number of elements
\begin{equation*}
t_{1},t_{2},\cdots ,t_{n}\in L.
\end{equation*}%
Let $\left( \Delta ,A\right) $ be a nice basis of the extension $L/K$.
Immediately, the cardinality $\sharp \Delta $ is finite. It is seen that the
elements $t_{1},t_{2},\cdots ,t_{n}\in L$ are algebraic over the subfield $%
M=K\left( \Delta \right) $. As $L=K\left( t_{1},t_{2},\cdots ,t_{n}\right)
\subseteq M\left( t_{1},t_{2},\cdots ,t_{n}\right) \subseteq L$, one has a
finitely generated and algebraic extension $L=M\left[ t_{1},t_{2},\cdots
,t_{n}\right] $ over $M$ and it follows that $L$ is a finite extension of $M$%
; hence, the linear basis $A$ of the $M$-vector space $L$ is a finite subset
of $L$.
\end{proof}

\subsection{Galois action on subsets of transcendence bases}

Let $L$ be a transcendental extension over a field $K$. Let $\Omega _{L/K}$
denote the set of all the transcendence bases of $L/K$.

\begin{definition}
Let $\Omega _{0}\subseteq \Omega _{L/K}$ be a subset.

$\left( i\right) $ A transcendence base $\Delta _{\max }$ of $L/K$ is said
to be \textbf{Zorn-maximal} in the extension $L/K$ (or in $\Omega _{L/K}$)
if there must be $K\left( \Delta \right) =K\left( \Delta _{\max }\right) $
for any transcendence base $\Delta $ of $L/K$ with $K\left( \Delta \right)
\supseteq K\left( \Delta _{\max }\right) $.

$\left( ii\right) $ Let $G\subseteq Aut\left( L/K\right) $ be a subgroup.
The subset $\Omega _{0}$ is said to be $G$\textbf{-invariant} (or \textbf{%
invariant} under the subgroup $G$) if $\sigma \left( \Omega _{0}\right)
\subseteq \Omega _{0}$ holds for any $\sigma \in G$.

$\left( iii\right) $ The subset $\Omega _{0}$ is said to be \textbf{dense}
in the extension $L/K$ (or in $\Omega _{L/K}$) if $\Omega _{0}$ contains a
Zorn-maximal transcendence base $\Delta _{\max }$ of $L/K$ satisfying the
two properties: $\left( a\right) $ For any $\Delta \in \Omega _{0}$, $%
K\left( \Delta \right) \subseteq K\left( \Delta _{\max }\right) $ holds; $%
\left( b\right) $ For any $\Delta \in \Omega _{0}$, there always exists some
$\Lambda \in \Omega _{0}$ with $K\left( \Lambda \right) \subseteq K\left(
\Delta \right) $. In such a case, $\Omega _{0}$ is also said to be \textbf{%
dense} in $L/K$ (relative to $\Delta _{\max }$).
\end{definition}

\begin{remark}
(\emph{Partial order and Galois action}) Consider the set $\Omega _{L/K}$ of
all the transcendence bases of $L/K$.

$\left( i\right) $ There is a partial order $\leq $ in the set $\Omega
_{L/K} $ given in an evident manner: For any $\Delta ,\Lambda \in \Omega
_{L/K}$, we say $\Delta \leq \Lambda $ if $K\left( \Delta \right) $ is a
subfield of $K\left( \Lambda \right) $. Furthermore, the partial order $\leq
$ in $\Omega _{L/K}$ is preserved by the Galois action of $Aut\left(
L/K\right) $. In deed, let $\Delta ,\Lambda \in \Omega _{L/K}$. Then $\Delta
\leq \Lambda $ holds if and only if $\sigma \left( \Delta \right) \leq
\sigma \left( \Lambda \right) $ holds for any $\sigma \in Aut\left(
L/K\right) $. Particularly, each Zorn-maximal $\Delta _{\max }$ in $\Omega
_{L/K}$ is invariant under any subgroup of $Aut\left( L/K\right) $.

The Galois invariance of the partial order $\leq $ in the set $\Omega _{L/K}$
is one of the key ingredients of transcendental Galois theory discussed in
the paper.

$\left( ii\right) $ Let $\Omega _{ver}$ be the set of all the vertical
transcendence bases of the extension $L/K$ (See \emph{Definition 3.8} for
definition). From the above $\left( i\right) $ it is seen that $\Omega
_{ver} $ is dense in $L/K$ and is invariant under any subgroup of the Galois
group $Aut\left( L/K\right) $.
\end{remark}

\begin{definition}
Let $\Omega _{0}\subseteq \Omega _{L/K}$ be a subset. The field $L$ is said
to be $\sigma $\textbf{-tame Galois} over $K$ \textbf{inside} $\Omega _{0}$
if $L$ is algebraic Galois over $K\left( \Delta \right) $ for each
transcendence base $\Delta \in \Omega _{0}$. In such a case, for simplicity,
$Aut\left( L/K\left( \Delta \right) \right) $ is called a $\Omega _{0}$%
\textbf{-subgroup} of the Galois group $Aut\left( L/K\right) $ for any $%
\Delta \in \Omega _{0}$.

In particular, $L$ is said to be $\sigma $\textbf{-tame Galois} over $K$ if $%
L$ is $\sigma $-tame Galois over $K$ inside $\Omega _{L/K}$, i.e., $L$ is
algebraic Galois over $K\left( \Delta \right) $ for each transcendence base $%
\Delta $ of $L/K$.
\end{definition}

\begin{remark}
Consider $L=\mathbb{C}$ and $K=\mathbb{Q}$. Let $\Delta $ be a transcendence
base of the complex field $\mathbb{C}$ over the rational field $\mathbb{Q}$.
Then $\mathbb{C}$ is algebraic Galois over $\mathbb{Q}\left( \Delta \right) $%
; $\mathbb{C}$ is tame Galois over $\mathbb{Q}$; $\mathbb{C}$ is $\sigma $%
-tame Galois over $\mathbb{Q}$ (See \emph{Proposition 3.7}). This is a case
for a $\sigma $-tame Galois extension of infinite transcendence degree.

On the other hand, it is not easy for one to find a $\sigma $-tame Galois
extension of finite transcendence degree. Theoretically, for the complex
field $\mathbb{C}$ there is a subfield $M$ of $\mathbb{C}$ such that $%
\mathbb{C}$ is a transcendental extension over $M$ of finite transcendence
degree. Then $\mathbb{C}$ is $\sigma $-tame Galois over the subfield $M$.
\end{remark}

\section{Transcendental Galois Theory, I. Tame Galois}

\emph{Transcendental Galois extensions} of fields, as the counterpart to the
(usual) \emph{algebraic Galois extensions}, are much more complicated. In
this section we will discuss several types of transcendental Galois
extensions which involve purely transcendental extensions and then we will
draw a comparison between these Galois extensions.

\subsection{Purely transcendental is Galois}

Let's start with the simplest case for transcendental Galois theory, purely
transcendental extensions of fields. Here the characteristic of a field $K$
is either positive or zero.

Let $L$ be an arbitrary extension of a field $K$. Recall that $L$ is said to
be \textbf{Galois} over $K$ if $K$ is the invariant subfield of $L$ under
the Galois group $Aut(L/K)$, i.e., $K=L^{Aut\left( L/K\right) }$, where
\begin{equation*}
L^{Aut\left( L/K\right) }=\{x\in L:\sigma (x)=x\text{ for any }\sigma \in
Aut(L/K)\}.
\end{equation*}

\begin{theorem}
\emph{(Purely transcendental: Galois)} Let $\Delta $ be any set of
algebraically independent variables over an arbitrary field $K$. Then the
purely transcendental extension $K\left( \Delta \right) $ is Galois over $K$%
. In particular, there is a subgroup $G$ contained in the Galois group $%
Aut\left( K\left( \Delta \right) /K\right) $ of order $\sharp G=2$ such that
\begin{equation*}
K\left( \Delta \right) ^{Aut\left( K\left( \Delta \right) /K\right)
}=K\left( \Delta \right) ^{G}=K
\end{equation*}%
hold for the invariant subfields of $L$ under the subgroup $G$ and the
Galois group $Aut\left( K\left( \Delta \right) /K\right) $, respectively.
\end{theorem}

\begin{proof}
Prove the subfield $K$ is the invariant subfield of $K(\Delta )$ under the
Galois group $Aut(K(\Delta )/K)$. In deed, consider a $K$-automorphism $\tau
_{L/K}$ of $L=K(\Delta )$, called the \textbf{reciprocal transformation} of
the field $K(\Delta )$, defined by
\begin{equation*}
\tau _{L/K}:v\mapsto \frac{1}{v}
\end{equation*}%
for any $v\in \Delta $. That is, the map $\tau _{L/K}:K(\Delta )\rightarrow
K(\Delta )$ is given by
\begin{equation*}
\frac{f(v_{1},v_{2},\cdots ,v_{n})}{g(v_{1},v_{2},\cdots ,v_{n})}\in
K(\Delta )\mapsto \frac{f(\tau _{L/K}(v_{1}),\tau _{L/K}(v_{2}),\cdots ,\tau
_{L/K}(v_{n}))}{g(\tau _{L/K}(v_{1}),\tau _{L/K}(v_{2}),\cdots ,\tau
_{L/K}(v_{n}))}\in K(\Delta )
\end{equation*}%
for any distinct variables
\begin{equation*}
v_{1},v_{2},\cdots ,v_{n}\in \Delta
\end{equation*}%
and for any polynomials
\begin{equation*}
f(X_{1},X_{2},\cdots ,X_{n})
\end{equation*}%
and
\begin{equation*}
g(X_{1},X_{2},\cdots ,X_{n})\not=0
\end{equation*}%
of variables $X_{1},X_{2},\cdots ,X_{n}$ with coefficients in the field $K$.
Then we have
\begin{equation*}
g(v_{1},v_{2},\cdots ,v_{n})\not=0
\end{equation*}%
if and only if
\begin{equation*}
g(X_{1},X_{2},\cdots ,X_{n})\not=0.
\end{equation*}%
Hence, $\tau _{L/K}$ is well-defined since
\begin{equation*}
g(v_{1},v_{2},\cdots ,v_{n})=0
\end{equation*}%
if and only if
\begin{equation*}
g(\tau _{L/K}(v_{1}),\tau _{L/K}(v_{2}),\cdots ,\tau _{L/K}(v_{n}))=0.
\end{equation*}

Put
\begin{equation*}
K(\Delta )^{\{\tau _{L/K}\}}=\{x\in K\left( \Delta \right) :\tau
_{L/K}\left( x\right) =x\}.
\end{equation*}%
We have
\begin{equation*}
K=K(\Delta )^{\{\tau \}}.
\end{equation*}%
Then
\begin{equation*}
K=K\left( \Delta \right) ^{Aut(K(\Delta )/K)}\subseteq K(\Delta )^{\{\tau
_{L/K}\}}=K.
\end{equation*}%
This proves that $K(\Delta )$ is Galois over $K$ and the subgroup $%
\left\langle \tau _{L/K}\right\rangle $ is of order two.
\end{proof}

\subsection{Purely transcendental is tame Galois}

Let $L$ be an arbitrary extension of a field $K$. Recall that $L$ is said to
be \textbf{tame Galois} over $K$ if there is a transcendence base $\Delta $
of the extension $L/K$ such that $L$ is algebraic Galois over the subfield $%
K(\Delta )$. In such a case, $L/K$ is also said to be \textbf{tame Galois}
with respect to $\Delta $ (See \emph{[}An 2010\emph{]}).

\begin{proposition}
Let $L$ be a purely transcendental extension over a field $K$ of finite
transcendence degree. For each transcendence base $\Delta $ of the extension
$L/K$, the field $L$ is a finite and separable extension over the subfield $%
K\left( \Delta \right) $.
\end{proposition}

\begin{proof}
Let $L=K\left( \Lambda \right) $ with a transcendence base $\Lambda
=\{t_{1},t_{2},\cdots ,t_{n}\}$. The extension $L/K$ has a specific nice
basis $\left( \Lambda ,\{1\}\right) $. Fixed any nice basis $\left( \Delta
,A\right) $ of the extension $L/K$.

Prove $L$ is an algebraic and separable extension over $K\left( \Delta
\right) $. In deed, consider the specific nice basis $\left( \Lambda
,\{1\}\right) $ of $L/K$, $L$ is separable over $L=K\left( \Lambda \right) $
itself; from \emph{Proposition 2.8} it is seen that $L=K\left( \Delta
\right) \left[ A\right] $ is algebraic and separable over $K\left( \Delta
\right) $ since $L$ is separable over $L$ itself.

Prove $L$ is a finite extension over $K\left( \Delta \right) $. In fact,
from \emph{Proposition 2.9} it is seen that the subsets $\Delta $ and $A$
both are finite sets and hence $L=K\left( \Delta \right) \left[ A\right] $
is a finite extension over $K\left( \Delta \right) $.
\end{proof}

In the same way for algebraic Galois, we have the following criteria for
tame Galois by conjugation and quasi-Galois.

\begin{lemma}
Let $L$ be an arbitrary extension of a field $K$. Fixed a transcendence base
$\Delta $ of $L/K$. Then $L$ is tame Galois over $K$ with respect to $\Delta
$ if and only if $L$ satisfies the two properties: $\left( i\right) $ $L$ is
quasi-Galois over $K$ with respect to $\Delta $; $\left( ii\right) $ $L$ is
separable over $K\left( \Delta \right) $.
\end{lemma}

\begin{theorem}
\emph{(Purely transcendental: tame Galois)} Let $\Delta $ be a set of
algebraically independent variables over an arbitrary field $K$.

$\left( i\right) $ If $\Delta $ is a finite set, then the purely
transcendental extension $K\left( \Delta \right) $ is tame Galois over $K$.

$\left( ii\right) $ If the characteristic of $K$ is zero, then the purely
transcendental extension $K\left( \Delta \right) $ is tame Galois over $K$.
\end{theorem}

\begin{proof}
$\left( i\right) $ Let $\Delta =\{t_{1},t_{2},\cdots ,t_{n}\}$ and $\Lambda
=\{t_{1}^{2},t_{2},\cdots ,t_{n}\}$. There is a nice basis $\left( \Lambda
,B\right) $ of $L=K\left( \Delta \right) $ over $K$ where $B=\{1,t_{1}\}$.
Then $L$ is tame quasi-Galois over $K$ with respect to $\left( \Lambda
,B\right) $. From \emph{Proposition 3.2} it is seen that $L$ is separable
over $K\left( \Lambda \right) $; hence, $L$ is tame Galois over $K$ with
respect to $\left( \Lambda ,B\right) $ due to \emph{Lemma 3.3}.

$\left( ii\right) $ In the same way it is seen that $L=K\left( \Delta
\right) $ is quasi-Galois over $K\left( \Lambda \right) $ and then $L$ is
Galois over $K\left( \Lambda \right) $ since $K$ is of zero characteristic.
Here, $\Lambda =\{t_{1}^{2}\}\cup \Delta ^{\prime }$ and $\Delta ^{\prime
}=\Delta \setminus \{t_{1}\}$.
\end{proof}

\begin{remark}
Let $L$ be a purely transcendental extension over $K$. In fact, there at
least two types of nice bases which $L$ is tame Galois over $K$ with respect
to. One is given by a quadratic extension of a variable such as above. The
other is given by the symmetric group $S$ of $n$ letters, where $n$ is the
transcendence degree of $L/K$.
\end{remark}

\subsection{Purely transcendental and $\sigma $-tame Galois}

Let ${L}$ be an arbitrary extension of a field $K$. Let $\Omega _{L/K}$
denote the set of all the transcendence bases $\Delta $ of the extension $%
L/K $. Fixed a non-void subset $\Omega _{0}\subseteq \Omega _{L/K}$. Recall
that $L$ is said to be $\sigma $-\textbf{tame Galois} over $K$
(respectively, \textbf{inside} $\Omega _{0}$) if $L$ is tame Galois over $K$
with respect to each transcendence base $\Delta $ of $L/K$ (respectively,
with $\Delta \in \Omega _{0}$), i.e., if $L$ is algebraic Galois over $%
K\left( \Delta \right) $ for each transcendence base $\Delta $ of $L/K$
(respectively, with $\Delta \in \Omega _{0}$).

From \emph{Proposition 2.4}, for $\sigma $-tame Galois we have the following
criteria by conjugations and quasi-Galois extensions.

\begin{lemma}
Let $L$ be an arbitrary extension of a field $K$. Then the field $L$ is $%
\sigma $-tame Galois over $K$ (respectively, inside $\Omega _{0}$) if and
only if $L$ satisfies the two properties: $\left( i\right) $ $L$ is
separably generated over $K$; $\left( ii\right) $ $L$ is $\sigma $%
-quasi-Galois over $K$ (respectively, inside $\Omega _{0}$).
\end{lemma}

\begin{theorem}
\emph{(}$\sigma $\emph{-tame Galois) }Let $\mathbb{Q}\subseteq K\subseteq
\mathbb{C}$ be any intermediate subfield in the extension $\mathbb{C}/%
\mathbb{Q}$. Then $\mathbb{C}$ is $\sigma $-tame Galois over $K$.
\end{theorem}

\begin{proof}
Fixed any transcendence base $\Delta $ of the extension $\mathbb{C}/K$. As $%
\mathbb{C}$ is algebraic and separable over $K\left( \Delta \right) $, it
reduces to prove $\mathbb{C}$ is a normal extension over $K\left( \Delta
\right) $, i.e., to prove $\mathbb{C}$ contains every $F$-conjugate $\beta $
of an $\alpha \in \mathbb{C}\setminus K(\Delta )$ for any subfield $K(\Delta
)\subseteq F\subseteq \mathbb{C}$ from \emph{Proposition 2.3}. In deed, let $%
f\left( X\right) \in F\left[ X\right] $ be the minimal polynomial of the
element $\alpha $ over the subfield $F$ of indeterminate $X$. It is seen
that $f\left( X\right) $, as a complex polynomial, has a factorisation of
linear factors over $\mathbb{C}$ in virtue of the fundamental theorem of
algebra. Hence, $\beta \in \mathbb{C}$. This proves $\mathbb{C}$ is $\sigma $%
-tame Galois over $K$.
\end{proof}

\begin{definition}
Let $t_{1},t_{2},\cdots ,t_{n}$ be algebraically independent variables over
a field $K$. Consider the purely transcendental extension of $L=K\left(
t_{1},t_{2},\cdots ,t_{n}\right) $ over $K$.

$\left( i\right) $ A transcendence base $\Delta $ of the extension $L/K$ is
said to be \textbf{vertical} (relative to the variables $t_{1},t_{2},\cdots
,t_{n}$) if $\Delta =\{t_{1}^{m_{1}},t_{2}^{m_{2}},\cdots ,t_{n}^{m_{n}}\}$
for some positive integers $m_{1},m_{2},\cdots ,m_{n}\in \mathbb{Z}$. In
such a case, the Galois group $Aut\left( L/K\left( \Delta \right) \right) $
is said to be \textbf{of vertical type} or a \textbf{vertical subgroup} of
the Galois group $Aut\left( L/K\right) $ (relative to the variables $%
t_{1},t_{2},\cdots ,t_{n}$).

Let $\Omega _{ver}\subseteq \Omega _{L/K}$ denote the subset consisting of
all the vertical transcendence bases $\Delta $ of $L/K$ (relative to the
variables $t_{1},t_{2},\cdots ,t_{n}$).

$\left( ii\right) $ If $L$ is $\sigma $-tame Galois over $K$ inside $\Omega
_{ver}$, the field $L=K\left( t_{1},t_{2},\cdots ,t_{n}\right) $ is said to
be $\sigma $\textbf{-tame Galois of vertical type} over $K$ or $\sigma $%
\textbf{-tame Galois} over $K$ \textbf{of vertical type}.
\end{definition}

In general, for a purely transcendental extension $L=K\left(
t_{1},t_{2},\cdots ,t_{n}\right) $ over a field $K$, there are infinitely
many subsets in $\Omega _{L/K}$, such as $\Omega _{ver}$, consisting of
vertical transcendence bases of $L/K$ relative to the other variables
\begin{equation*}
t_{1}+c_{1},t_{2}+c_{2},\cdots ,t_{n}+c_{n}
\end{equation*}%
of $L$ over $K$ (See $\left( iii\right) $ in the proof of \emph{Example 1.9}%
), where
\begin{equation*}
c_{1},c_{2},\cdots ,c_{n}\in K
\end{equation*}%
are any given elements since we have
\begin{equation*}
L=K\left( t_{1},t_{2},\cdots ,t_{n}\right) =K\left(
t_{1}+c_{1},t_{2}+c_{2},\cdots ,t_{n}+c_{n}\right) .
\end{equation*}

\begin{theorem}
\emph{(}$\sigma $\emph{-tame Galois of vertical type)} Let $%
t_{1},t_{2},\cdots ,t_{n}$ be algebraically independent variables over a
field $K$. Let $\Omega _{ver}$ be the set of all the vertical transcendence
bases $\Delta $ of $L=K\left( t_{1},t_{2},\cdots ,t_{n}\right) $ over $K$
(relative to $t_{1},t_{2},\cdots ,t_{n}$).

$\left( i\right) $ The set $\Omega _{ver}$ is dense in the extension $L/K$
(relative to $t_{1},t_{2},\cdots ,t_{n}$) and is invariant under the Galois
group $Aut\left( L/K\right) $.

$\left( ii\right) $ Let $\Delta =\{t_{1}^{m_{1}},t_{2}^{m_{2}},\cdots
,t_{n}^{m_{n}}\}\in \Omega _{ver}$ and $G=Aut\left( L/K\left( \Delta \right)
\right) $. If $K$ contains all the $m_{i}$-th roots of unity for any $1\leq
i\leq n$, then the vertical group $G$ is a Noether solution of $L/K$.

$\left( iii\right) $ If $K$ is algebraically closed, then $K\left(
t_{1},t_{2},\cdots ,t_{n}\right) $ is $\sigma $-tame Galois of vertical type
over $K$.

In such a case, for any $\Delta =\{t_{1}^{m_{1}},t_{2}^{m_{2}},\cdots
,t_{n}^{m_{n}}\}\in \Omega _{ver}$, the vertical subgroup $G=Aut\left(
L/K\left( \Delta \right) \right) $ is a Noether solution of $L/K$; the order
$\sharp G\rightarrow +\infty $ as $m_{i}\rightarrow +\infty $; the Noether
integers of $L/K$ have no upper bound.
\end{theorem}

\begin{proof}
$\left( i\right) $ Trivial.

$\left( ii\right) $ Let $\Delta =\{t_{1}^{m_{1}},t_{2}^{m_{2}},\cdots
,t_{n}^{m_{n}}\}$ be a vertical transcendence base of the field $L=K\left(
t_{1},t_{2},\cdots ,t_{n}\right) $ over $K$. Here, $m_{1},m_{2},\cdots
,m_{n}\in \mathbb{Z}$ are positive integers. Put $M=K\left( \Delta \right) $
and $L_{i}=K\left[ t_{i}\right] $ for $1\leq i\leq n$. Suppose $K$ is
algebraically closed. It follows that all the $m_{i}$-th roots of $%
s_{i}=t_{i}^{m_{i}}$ are contained in $L_{i}$ for each $1\leq i\leq n$ and
then each $L_{i}$ is a normal extension over $M$. As the compositum of
normal sub-extensions is still a normal extension, it is seen that the
compositum $L=L_{1}L_{2}\cdots L_{n}$ is normal over $M$.

On the other hand, from \emph{Proposition 2.8} it is seen that $L$ is
separable over $M$ and then $L$ is algebraic Galois over the intermediate
subfield $M=K\left( \Delta \right) $. This proves that $L$ is algebraic
Galois over $M=L^{G}$ and $M$ is purely transcendental over $K$, i.e., $G$
is a Noether solution of $L/K$.

$\left( iii\right) $ Immediately from $\left( ii\right) $.
\end{proof}

\begin{remark}
Let $t_{1},t_{2},\cdots ,t_{n}$ be algebraically independent variables over
a field $K$. Let $L=K\left( t_{1},t_{2},\cdots ,t_{n}\right) $. Consider the
Galois group
\begin{equation*}
G=Aut\left( L/K\left( t_{1}^{m_{1}},t_{2}^{m_{2}},\cdots
,t_{n}^{m_{n}}\right) \right)
\end{equation*}%
for a set of given positive integers $m_{1},m_{2},\cdots ,m_{n}\in \mathbb{Z}
$. The above \emph{Theorem 3.9} says that if $K$ contains the $m_{i}$-th
roots of unity for each $1\leq i\leq n$, then finite group $G$ is a Noether
solution of $L/K$. In \emph{[}Fisher 1915\emph{]} $K$ is assumed to contain
"enough roots of unity" for the case of an abelian group. Here in \emph{%
Theorem 3.9}, $G$ is not necessarily an abelian group. $\left( i\right) $ of
\emph{Theorem 3.9} generalises the result in \emph{[}Fisher 1915\emph{]}.
\end{remark}

\subsection{Tame Galois is Galois}

There is a comparison lemma between Galois and tame Galois.

\begin{proposition}
Let $L$ be a purely transcendental extension over a field $K$ of
transcendence degree $n$. Here $n\in \mathbb{N}$ or $n=+\infty $. If $M$ is
an algebraic Galois extension over $L$, then $M$ is Galois over $K$.
\end{proposition}

\begin{proof}
Assume $M$ is algebraic Galois over $L$.

$\left( i\right) $ Let $L/K$ and $M/L$ be finitely generated extensions.
Fixed an $\left( m,n\right) $-nice basis
\begin{equation*}
v_{1},v_{2},\cdots ,v_{m},\cdots ,v_{n}\in M
\end{equation*}%
of the extension $M/K$ such that
\begin{equation*}
L=K\left( v_{1},v_{2},\cdots ,v_{m}\right)
\end{equation*}%
and
\begin{equation*}
M=L\left[ v_{m+1},\cdots ,v_{n}\right] .
\end{equation*}

Suppose $M$ is algebraic Galois over $L$. The Galois group $Aut\left(
M/L\right) $ has $\left( n-m\right) $ elements $\sigma _{1},\sigma
_{2},\cdots ,\sigma _{n-m}$. Take the reciprocity transformation $\tau \in
Aut\left( L/K\right) $ of the transcendental extension $L/K$ defined by
\begin{equation*}
\tau \left( v_{1}\right) =\frac{1}{v_{1}},\tau \left( v_{2}\right) =\frac{1}{%
v_{2}},\cdots ,\tau \left( v_{m}\right) =\frac{1}{v_{m}}.
\end{equation*}

For the extension $M=K\left( v_{1},v_{2},\cdots ,v_{m},\cdots ,v_{n}\right) $
over $K$, we construct $n-m$ $K$-automorphisms $\delta _{i}:M\rightarrow M$,
called the \emph{joint transformation} of $M/K$ by $\sigma _{i}$ and $\tau $%
, in such a manner:

For each $1\leq i\leq n-m$, the map $\delta _{i}:M\rightarrow M$ is given by
\begin{equation*}
\delta _{i}:\frac{f(v_{1},v_{2},\cdots ,v_{m},\cdots ,v_{n})}{%
g(v_{1},v_{2},\cdots ,v_{m},\cdots ,v_{n})}\in M
\end{equation*}%
\begin{equation*}
\mapsto \frac{f(\tau (v_{1}),\tau (v_{2}),\cdots ,\tau (v_{m}),\sigma
_{i}(v_{m+1}),\cdots ,\sigma _{i}(v_{n}))}{g(\tau (v_{1}),\tau
(v_{2}),\cdots ,\tau (v_{m}),\sigma _{i}(v_{m+1}),\cdots ,\sigma _{i}(v_{n}))%
}\in M
\end{equation*}%
for any polynomials
\begin{equation*}
f(X_{1},X_{2},\cdots ,X_{n})
\end{equation*}%
and
\begin{equation*}
g(X_{1},X_{2},\cdots ,X_{n})\not=0
\end{equation*}%
over the field $K$ with
\begin{equation*}
g(v_{1},v_{2},\cdots ,v_{n})\not=0.
\end{equation*}

By taking common solutions of a set of $n-m$ equations
\begin{equation*}
\left\{
\begin{array}{c}
\delta _{1}\left( x\right) -x=0 \\
\delta _{2}\left( x\right) -x=0 \\
\cdots \\
\delta _{n-m}\left( x\right) -x=0%
\end{array}%
\right.
\end{equation*}
in the field $M$, we have
\begin{equation*}
K=\{x\in M:\delta _{i}(x)=x,i=1,2,\cdots ,n-m\}.
\end{equation*}

As each $\delta _{i}\in Aut\left( M/K\right) $, it is seen that $%
K=M^{Aut\left( M/K\right) }$ is the invariant subfield of $M$ under the
Galois group $Aut(M/K)$. This proves that $M$ is Galois over $K$.

$\left( ii\right) $ Suppose $L/K$ and $M/L$ are infinitely generated
extensions. Fixed a nice basis $\left( \Delta ,A\right) $ of the extension $%
M/K$ such that $L=K\left( \Delta \right) $ and $M=L\left[ A\right] $. Take
any $v_{1},v_{2},\cdots ,v_{m},\cdots ,v_{n}\in M$ with $v_{1},v_{2},\cdots
,v_{m}\in \Delta $ and $v_{m+1},\cdots ,v_{n}\in A.$

By taking the reciprocity transformation $\tau \in Aut\left( L/K\right) $
defined by $\tau \left( v\right) =\frac{1}{v}$ for each $v\in \Delta $, for
each $\sigma _{i}\in Aut\left( M/L\right) $ (with $i\in I$) we have the
joint transformation $\delta _{i}\in Aut\left( M/K\right) $ given in an
evident manner:
\begin{equation*}
\delta _{i}:\frac{f(v_{1},v_{2},\cdots ,v_{m},\cdots ,v_{n})}{%
g(v_{1},v_{2},\cdots ,v_{m},\cdots ,v_{n})}\in M
\end{equation*}%
\begin{equation*}
\mapsto \frac{f(\tau (v_{1}),\tau (v_{2}),\cdots ,\tau (v_{m}),\sigma
_{i}(v_{m+1}),\cdots ,\sigma _{i}(v_{n}))}{g(\tau (v_{1}),\tau
(v_{2}),\cdots ,\tau (v_{m}),\sigma _{i}(v_{m+1}),\cdots ,\sigma _{i}(v_{n}))%
}\in M
\end{equation*}

for any $v_{1},v_{2},\cdots ,v_{m}\in \Delta $ and $v_{m+1},\cdots ,v_{n}\in
A$ and for any polynomials $f(X_{1},X_{2},\cdots ,X_{n})$ and $%
g(X_{1},X_{2},\cdots ,X_{n})\not=0$ with coefficients in $K$. In the same
way as in $\left( i\right) $, we have $M^{\{\delta _{i}:i\in I\}}=K$.
\end{proof}

\begin{lemma}
\emph{(Comparison Lemma)} Let $L$ be a purely transcendental extension over
a field $K$ of transcendence degree $n\leq +\infty $. If $L$ is tame Galois
over $K$, then $L$ must be Galois over $K$.
\end{lemma}

\begin{proof}
Suppose $L$ is tame Galois over $K$ with respect to its nice basis $\left(
\Delta ,A\right) $. Then $L$ is an algebraic Galois extension over the
subfield $K\left( \Delta \right) $. From \emph{Proposition 3.11} it is seen
that $L$ is a transcendental Galois extension over $K$.
\end{proof}

\subsection{Purely transcendental and complete Galois}

There are another type of transcendental Galois extensions which involve
Noether's problem on rationality.

\begin{definition}
Let ${L}$ be an arbitrary extension of a field $K$. The field $L$ is said to
be \textbf{complete Galois} over $K$ if $L$ is Galois over $M$ for each
intermediate subfield $M$ with $K\subseteq M\subseteq L$.
\end{definition}

\begin{theorem}
\emph{(Complete Galois) }Let $\mathbb{Q}\subseteq K\subseteq \mathbb{C}$ be
any intermediate subfield in the extension $\mathbb{C}/\mathbb{Q}$. Then $%
\mathbb{C}$ is complete Galois over $K$.

In particular, $\mathbb{C}$ is complete Galois over $\mathbb{Q}$.
\end{theorem}

\begin{proof}
Let $L$ be an intermediate subfield with $\mathbb{Q}\subseteq K\subseteq
L\subseteq \mathbb{C}$. It is seen that $\mathbb{C}$ is $\sigma $-tame
Galois over $L$ from \emph{Proposition 3.7} and hence $\mathbb{C}$ is Galois
over $L\left( \Delta \right) $ for each nice basis $\left( \Delta ,A\right) $
of the sub-extension $\mathbb{C}/K$. Applying \emph{Proposition 3.11} to the
case here, we come to a conclusion that $\mathbb{C}$ is a transcendental
Galois extension over the subfield $L$. This completes the proof.
\end{proof}

Now we have a comparison remark for these transcendental Galois extensions.

\begin{remark}
For an arbitrary extension of a field, we have

complete Galois $\implies $ $\sigma $-tame Galois $\implies $ $\sigma $-tame
Galois inside $\Omega _{0}\ \implies $ tame Galois $\implies $ Galois.

Particularly, for the case of algebraic extensions, all these Galois
coincide with each other.
\end{remark}

\subsection{Galois correspondence for complete Galois}

There is a transcendental version of Galois correspondence for extensions of
fields. Here we only discuss the finite Galois correspondence for the
purpose of Noether's problem on rationality. In general, for an infinite
Galois correspondence, we need some certain additional structures, such as
topological structures, established on the Galois group of a transcendental
extension.

\begin{theorem}
\emph{(Galois correspondence for complete Galois)} Let $L$ be an extension
of a field $K$ of finite transcendence degree. Suppose $L$ is complete
Galois over $K$. Then there exists a one-to-one correspondence between
finite subgroups $H$ of the Galois group $Aut\left( L/K\right) $ and
intermediate subfield $M$ in the extension $L/K$ of finite degree $\left[ L:M%
\right] <+\infty $, given by
\begin{equation*}
H=Aut\left( L/M\right) \text{ and }M=L^{H}.
\end{equation*}
\end{theorem}

\begin{proof}
Let $H$ be a finite subgroup of the Galois group $Aut\left( L/K\right) $.
Put $M=L^{H}$, i.e., the $H$-invariant subfield of $L$. From the assumption
that $L$ is complete Galois over $K$, it is seen that $L$ is Galois over $M$
and $M=L^{Aut\left( L/M\right) }$ is the $Aut\left( L/M\right) $-invariant
subfield of $L$; it follows that
\begin{equation*}
L^{H}=L^{Aut\left( L/M\right) }=L^{Aut\left( L/L^{H}\right) }=M.
\end{equation*}

We will proceed in several steps.

\emph{Step 1.} Prove $L$ is a finite extension over $M$. Otherwise, if $L$
is an infinite extension over $M$, that is, if $L$ is infinite algebraic or
transcendental over $M$, then there is no finite subgroup $\Gamma $ of the
Galois group $Aut\left( L/K\right) $ such that $M=L^{\Gamma }$ is a subfield
fixed by $\Gamma $ (See any standard textbook on fields and Galois theory);
so there will be a contradiction.

\emph{Step 2.} Prove $H=Aut\left( L/L^{H}\right) $. As $L$ is finite over $%
M=L^{H}$ from \emph{Step 1}, it is seen that $L$ is an algebraic extension
over $M$; from the assumption that $L$ is Galois over $M$, it is seen that $%
L $ is algebraic Galois over $M$ and hence we have $H=Aut\left( L/M\right) $
from the standard theory of algebraic Galois extension of a field.

\emph{Step 3.} Prove the existence of the Galois correspondence. In deed,
given a finite subgroup $H$ of the Galois group $Aut\left( L/K\right) $.
There is a sub-extension $M=L^{H}$ in the extension $L/K$ with $H=Aut\left(
L/M\right) $.

On the other hand, suppose $M^{\prime }$ is a sub-extension in $L/K$ with $%
H=Aut\left( L/M^{\prime }\right) $. From \emph{Step 1} it is seen that $L$
is a finite Galois extension over $M$ and $M^{\prime }$, respectively; as $%
H=Aut\left( L/M\right) =Aut\left( L/M^{\prime }\right) $, we have
\begin{equation*}
M=L^{Aut\left( L/M\right) }=L^{H}=L^{Aut\left( L/M^{\prime }\right)
}=M^{\prime }.
\end{equation*}

This proves for each finite subgroup $H$ of the Galois group $Aut\left(
L/K\right) $ there is one and only one sub-extension $M=L^{H}$ in the
extension $L/K$ with $H=Aut\left( L/M\right) $.

Conversely, for each finite sub-extension $L/M$ in $L/K$ there is a unique
finite subgroup $H=Aut\left( L/M\right) $ in $Aut\left( L/K\right) $ with $%
M=L^{H}$ from \emph{Step 2}. This completes the proof.
\end{proof}

It follows that we have the following theorems of Galois correspondence for
purely transcendental extensions.

\begin{theorem}
\emph{(Galois correspondence for purely transcendental: }$\sigma $\emph{%
-tame Galois of vertical type)} Let $L$ be a purely transcendental extension
over a field $K$ of transcendence degree $n<+\infty $, saying $L=K\left(
t_{1},t_{2},\cdots ,t_{n}\right) $. Suppose $K$ is separably closed.

Then there exists a one-to-one correspondence between vertical subgroups $H$
of the Galois group $Aut\left( L/K\right) $ (relative to $t_{1},t_{2},\cdots
,t_{n}$) and intermediate subfields $M$ in the extension $L/K$ of finite
degree $\left[ L:M\right] <+\infty $ such that $H=Aut\left( L/M\right) $, $%
M=L^{H}$ and $M$ has a vertical transcendence base over $K$ (relative to $%
t_{1},t_{2},\cdots ,t_{n}$).
\end{theorem}

\begin{proof}
Immediately from \emph{Theorem 3.9} with the same procedure in the proof of
\emph{Theorem 3.16}.
\end{proof}

\begin{theorem}
\emph{(Galois correspondence for purely transcendental: }$\sigma $\emph{%
-tame Galois inside }$\Omega _{0}$\emph{)} Let $L$ be a purely
transcendental extension over a field $K$ of transcendence degree $n<+\infty
$, saying $L=K\left( t_{1},t_{2},\cdots ,t_{n}\right) $. Fixed a subset $%
\Omega _{0}$ of some certain transcendence bases of $L/K$. Suppose $L$ is $%
\sigma $-tame Galois inside $\Omega _{0}$.

Then there exists a one-to-one correspondence between subgroups $H$ of the
Galois group $Aut\left( L/K\right) $ and intermediate subfields $M$ in the
extension $L/K$ of finite degree $\left[ L:M\right] <+\infty $ satisfying
the property: $H=Aut\left( L/M\right) $, $M=L^{H}$ and $M$ has a
transcendence base $\Delta $ of $L/K$ with $\Delta \in \Omega _{0}$.
\end{theorem}

\begin{proof}
Immediately from \emph{Theorem 3.9} with the same procedure in the proof of
\emph{Theorem 3.16}.
\end{proof}

\subsection{Dedekind independence theorem in transcendental Galois theory}

Here, we also have a transcendental version of Dedekind independence theorem
for the Galois group, which will be used to give a proof for \emph{Lemma 8.16%
}.

\begin{theorem}
\emph{(Dedekind Independence)} Let $L$ be an algebraic or transcendental
extension of a field $K$. Let $Aut\left( L/K\right) $ denote the Galois
group of $L/K$.

$\left( i\right) $ Any finite number of distinct elements of the Galois
group $Aut\left( L/K\right) $ are $L$-linearly independent. That is, fixed
any $m$ distinct $\sigma _{1},\sigma _{2},\cdots ,\sigma _{m}$ of the Galois
group $Aut\left( L/K\right) $ and any elements $x_{1},x_{2},\cdots ,x_{m}\in
L$. Consider the $K$-linear map on $L$ (as an $L$-linear combination of the $%
\sigma _{i}$'s)
\begin{equation*}
\sum_{j=1}^{m}x_{j}\cdot \sigma _{j}:L\rightarrow L
\end{equation*}%
given by
\begin{equation*}
z\in L\longmapsto \sum_{j=1}^{m}x_{j}\cdot \sigma _{j}\left( z\right) \in L.
\end{equation*}

If the $L$-linear combination of the elements $\sigma _{1},\sigma
_{2},\cdots ,\sigma _{m}$
\begin{equation*}
\sum_{j=1}^{m}x_{j}\cdot \sigma _{j}
\end{equation*}%
is the zero map on $L$, i.e., for each $z\in L$ there is
\begin{equation*}
\sum_{j=1}^{m}x_{j}\cdot \sigma _{j}\left( z\right) =0,
\end{equation*}%
then we must have
\begin{equation*}
x_{1}=x_{2}=\cdots =x_{m}=0.
\end{equation*}%
$\left( ii\right) $ Any non-void subset $A$ of $Aut\left( L/K\right) $ is $L$%
-linearly independent. In particular, the Galois group $Aut\left( L/K\right)
$ is an $L$-linearly independent set.

$\left( iii\right) $ Fixed any intermediate subfield $K\subseteq M\subseteq
L $. Then the Galois group $Aut\left( L/K\right) $ is an $M$-linearly
independent set.
\end{theorem}

\begin{proof}
$\left( i\right) $ Fixed any finite distinct elements $\sigma _{1},\sigma
_{2},\cdots ,\sigma _{m}$ in the Galois group $Aut\left( L/K\right) $.
Suppose $x_{1},x_{2},\cdots ,x_{m}\in L$ such that the $L$-linear
combination of the $m$ elements $\sigma _{1},\sigma _{2},\cdots ,\sigma
_{m}\in Aut\left( L/K\right) $
\begin{equation*}
\sum_{j=1}^{m}x_{j}\cdot \sigma _{j}
\end{equation*}%
is the zero map on $L$. Induction on $m$.

For $m=1$, if $x_{1}\cdot \sigma _{1}\left( z\right) =0$ holds for any $z\in
L$, then $x_{1}=0$ by taking $z=1$.

Suppose it is true for $m$, that is, any $m$ distinguished elements of $%
Aut\left( L/K\right) $ are $L$-linearly independent.

Now consider $m+1$. Take any subset $\Gamma $ of $\left( m+1\right) $
distinct elements
\begin{equation*}
\sigma _{1},\sigma _{2},\cdots ,\sigma _{m+1}\in Aut\left( L/K\right) .
\end{equation*}

Without loss of generality, suppose $m+1$ is the least integer in such a
sense: $m+1$ is the minimum of the set
\begin{equation*}
\{r\in \mathbb{N}:\sum_{j=1}^{r}w_{j}\cdot \delta _{j}=0\text{ for some }
w_{1},w_{2},\cdots ,w_{r}\in L^{\times }\text{ and }\delta _{1},\delta
_{2},\cdots ,\delta _{r}\in \Gamma \}
\end{equation*}
where $L^{\times }=L\setminus \{0\}.$ Otherwise, if not, there will be at
least one element of $\Gamma $ which can be an $L$-linear combination of the
other $m$ elements of $\Gamma $; then it will reduce to the case for $m$.

Suppose there are $x_{1},x_{2},\cdots ,x_{m+1}\in L$ such that the $L$%
-linear combination
\begin{equation*}
\sum_{j=1}^{m+1}x_{j}\cdot \sigma _{j}
\end{equation*}%
is the zero map on $L$. We prove $x_{1},x_{2},\cdots ,x_{m+1}$ are all zero.

In deed, for any $x,y\in L$ we have
\begin{equation*}
\begin{array}{l}
0=\left( \sum_{j=1}^{m+1}x_{j}\cdot \sigma _{j}\right) \left( xy\right) \\
=\sum_{j=1}^{m+1}x_{j}\cdot \sigma _{j}\left( xy\right) \\
=\sum_{j=1}^{m+1}x_{j}\cdot \sigma _{j}\left( x\right) \cdot \sigma
_{j}\left( y\right) ; \\
\\
0=\left( \sum_{j=1}^{m+1}x_{j}\cdot \sigma _{j}\right) \left( x\right) \cdot
\sigma _{m+1}\left( y\right) \\
=\sum_{j=1}^{m+1}x_{j}\cdot \sigma _{j}\left( x\right) \cdot \sigma
_{m+1}\left( y\right) .%
\end{array}%
\end{equation*}

Then
\begin{equation*}
\begin{array}{l}
0=\sum_{j=1}^{m+1}x_{j}\cdot \sigma _{j}\left( x\right) \cdot \left( \sigma
_{j}\left( y\right) -\sigma _{m+1}\left( y\right) \right) \\
=\sum_{j=1}^{m}x_{j}\cdot \sigma _{j}\left( x\right) \cdot \left( \sigma
_{j}\left( y\right) -\sigma _{m+1}\left( y\right) \right) \\
=\sum_{j=1}^{m}\left( x_{j}\cdot \left( \sigma _{j}\left( y\right) -\sigma
_{m+1}\left( y\right) \right) \right) \cdot \sigma _{j}\left( x\right)%
\end{array}%
\end{equation*}
hold for any $x,y\in L$.

It follows that for any fixed $y_{0}\in L$, the $L$-linear combination of
the $m$ elements in $Aut\left( L/K\right) $
\begin{equation*}
\sum_{j=1}^{m}\left( x_{j}\cdot \left( \sigma _{j}\left( y_{0}\right)
-\sigma _{m+1}\left( y_{0}\right) \right) \right) \cdot \sigma _{j}=0
\end{equation*}
is the zero map on $L$.

From the assumption for the case of $m$ that any subset of $m$ distinct
elements in $Aut\left( L/K\right) $ is $L$-linearly independent, we have
\begin{equation*}
x_{j}\cdot \left( \sigma _{j}\left( y_{0}\right) -\sigma _{m+1}\left(
y_{0}\right) \right) =0
\end{equation*}
for each $1\leq j\leq m$.

On the other hand, since $\sigma _{1},\sigma _{2},\cdots ,\sigma _{m+1}$ are
distinct elements in the Galois group $Aut\left( L/K\right) $, we must have
some nonzero $y_{0}\in L$ such that
\begin{equation*}
\sigma _{j}\left( y_{0}\right) -\sigma _{m+1}\left( y_{0}\right) \not=0
\end{equation*}
for each $1\leq j\leq m$.

We have $x_{1}=x_{2}=\cdots =x_{m}=0$ and $x_{m+1}=0.$ It follows that the $%
\left( m+1\right) $ distinct elements
\begin{equation*}
\sigma _{1},\sigma _{2},\cdots ,\sigma _{m+1}\in Aut\left( L/K\right)
\end{equation*}
are $L$-linearly independent.

This proves any finite number of distinct elements of the Galois group $%
Aut\left( L/K\right) $ are $L$-linearly independent.

$\left( ii\right) $ Take a non-void subset $A\subseteq Aut\left( L/K\right) $%
. From $\left( i\right) $ it is seen that any finite number of distinct
elements of $A$ are $L$-linearly independent; hence, $A$ is an $L$-linearly
independent set.

$\left( iii\right) $ Immediately from $\left( ii\right) $. This completes
the proof.
\end{proof}

\section{Transcendental Galois Theory, II. Algebraic Galois Subgroups}

In this section we will discuss that for a transcendental Galois extension $%
L/K$, there are infinitely many coupled parts in the Galois group $Aut\left(
L/K\right) $, i.e., the algebraic parts and the corresponding transcendental
parts.

\subsection{Full algebraic Galois groups}

Let $L$ be an arbitrary extension of a field $K$. Consider the Galois group $%
Aut\left( L/K\right) $ of the extension $L/K$.

\begin{definition}
A subgroup $G$ of $Aut\left( L/K\right) $ is called an \textbf{algebraic
Galois subgroup} of the extension $L/K$ if there is an intermediate subfield
$K\subseteq M\subseteq L$ such that $L$ is an algebraic extension over $M$
and $G$ is the Galois group $Aut\left( L/M\right) $ of the sub-extension $%
L/M $.

In such a case, if the subfield $M$ is purely transcendental over $K$, we
say $G$ is a \textbf{full algebraic Galois subgroup} of the extension $L/K$.
A full algebraic Galois subgroup of $L/K$ is denoted by $\pi _{a}\left(
L/K\right) $.
\end{definition}

\begin{remark}
A Noether solution must be a full algebraic Galois subgroup. In deed, let $L$
be a purely transcendental extension of a field $K$ and let $G$ be a
subgroup of $Aut\left( L/K\right) $. If $G$ is a Noether solution of $L/K$,
then $G$ must be a full algebraic Galois subgroup of $L/K$.
\end{remark}

\begin{remark}
Let $L$ be a transcendental extension of a field $K$. In general, a full
algebraic Galois subgroup $\pi _{a}\left( L/K\right) $ of the extension $L/K$
is never biggest (by set-inclusion) as a subgroup in the Galois group $%
Aut\left( L/K\right) $. See \emph{Example 4.5}.
\end{remark}

\begin{definition}
A subgroup $G$ of $Aut\left( L/K\right) $ is called a \textbf{transcendental
Galois subgroup} of the extension $L/K$ if there is an intermediate subfield
$K\subseteq M\subseteq L$ satisfying the three conditions:

$\left( i\right) $ $M$ is a purely transcendental extension over $K$;

$\left( ii\right) $ $\sigma \left( x\right) =x$ holds for any $\sigma \in G$
and for any $x\in L\smallsetminus M$;

$\left( iii\right) $ For each $\sigma \in G$, the restriction of $\sigma $
to $M$ belongs to the Galois group $Aut\left( M/K\right) $. Conversely, for
each $\delta \in Aut\left( M/K\right) $, there is one unique $\sigma \in G$
such that $\delta $ is the restriction of $\sigma $ to $M$.

In such a case, if $L$ is algebraic over the subfield $M$, we say $G$ is a
\textbf{highest transcendental Galois subgroup} of the extension $L/K$. A
highest transcendental Galois subgroup of $L/K$ is denoted by $\pi
_{t}\left( L/K\right) $.
\end{definition}

In the following it is seen that for a given transcendental extension of a
field, its full algebraic Galois subgroups can vary heavily while its
highest transcendental Galois subgroups are all isomorphic.

\begin{example}
Let $u,v,w$ be three algebraically independent variables over the rational
field $\mathbb{Q}$. Suppose $L_{n}=\mathbb{Q}\left(
u^{2n},v^{3n},w^{4n}\right) $ for $n\geq 0$. Then the Galois groups $%
G_{n}=Aut\left( L_{0}/L_{n}\right) $ all are full algebraic subgroups of the
extension $L_{0}/\mathbb{Q}$ for each $n\geq 0$ which form an increasing
sequence of subgroups
\begin{equation*}
\{1\}=G_{0}\subsetneqq G_{1}\subsetneqq \cdots \subsetneqq G_{n}\subsetneqq
G_{n+1}\subsetneqq \cdots
\end{equation*}
in the Galois group $Aut\left( L_{0}/\mathbb{Q}\right) $.
\end{example}

\begin{proposition}
(\emph{The algebraic Galois subgroup at a nice basis}) Let $L$ be a
transcendental extension over a field $K$ of finite transcendence degree.
Then the subgroup
\begin{equation*}
\pi _{a}\left( L/K\right) \left( \Delta ,A\right) :=Aut\left( L/K\left(
\Delta \right) \right)
\end{equation*}%
is a full algebraic Galois subgroup of the extension $L/K$ for each nice
basis $\left( \Delta ,A\right) $ of the extension $L/K$. Conversely, each
full algebraic Galois subgroup of $L/K$ is given in such a manner.
\end{proposition}

\begin{proof}
Immediately from the assumption that $L$ is algebraic over $K\left( \Delta
\right) $ and $K\left( \Delta \right) $ is purely transcendental over $K$.
\end{proof}

\begin{proposition}
(\emph{The transcendental Galois subgroup at a nice basis}) Let $L$ be a
transcendental extension over a field $K$ of finite transcendence degree.

$\left( i\right) $ Fixed a nice basis $\left( \Delta ,A\right) $ of the
extension $L/K$. Then each element of the Galois group $Aut\left( K\left(
\Delta \right) /K\right) $ can be extended to a unique element of a highest
transcendental Galois subgroup of $L/K$. Conversely, every element of a
highest transcendental Galois subgroup of $L/K$ is given in such a manner.

$\left( ii\right) $ There is a unique highest transcendental Galois subgroup
of the extension $L/K$, denoted by $\pi _{t}\left( L/K\right) \left( \Delta
,A\right) $, satisfying the property
\begin{equation*}
Aut\left( K\left( \Delta \right) /K\right) =\{\sigma |_{M}\in Aut\left(
K\left( \Delta \right) /K\right) :\sigma \in \pi _{t}\left( L/K\right)
\left( \Delta ,A\right) \}
\end{equation*}%
for each nice basis $\left( \Delta ,A\right) $ of the extension $L/K$.

$\left( iii\right) $ For every highest transcendental Galois subgroup $H$ of
the extension $L/K$, there is a nice basis $\left( \Delta ,A\right) $ of $%
L/K $ such that $H=\pi _{t}\left( L/K\right) \left( \Delta ,A\right) $.

$\left( iv\right) $ All the highest transcendental Galois subgroups of the
extension $L/M$ are isomorphic groups.
\end{proposition}

\begin{proof}
$\left( i\right) $ Consider the $K$-linear spaces $E:=K\left( \Delta \right)
,$ $F:=span_{K}\left( L\backslash K\left( \Delta \right) \right) ,$ and $%
V:=L=E\oplus F.$ Here, $E$ and $F$ are called the $E$\emph{-component} and $%
F $\emph{-component} of the linear space $L$, respectively. \emph{[}See
\emph{Proposition 5.3} in \emph{\S 5.2} for details\emph{]}.

Fixed an element $\sigma _{E}\in Aut\left( K\left( \Delta \right) /K\right) $%
. It is seen that $\sigma _{E}$ can be extended to an element $\sigma \in
Aut\left( L/K\right) $. In deed, it is immediately from \emph{Lemma 2.5} for
the case that $L=K\left( \Delta \right) \left[ A\right] $ is a finite
extension over $K\left( \Delta \right) $.

Assume $L$ is an infinite extension over $K\left( \Delta \right) $, then
consider a tower of sub-extensions
\begin{equation*}
K\left( \Delta \right) \subsetneqq L^{\prime }\subsetneqq L
\end{equation*}
with $\left[ L^{\prime }:K\left( \Delta \right) \right] <+\infty $. In the
same way, by \emph{Lemma 2.5} again we have an element $\sigma ^{\prime }\in
Aut\left( L^{\prime }/K\right) $ with $\sigma _{E}$ is the restriction of $%
\sigma ^{\prime }$ to $K\left( \Delta \right) $. Choose a $K\left( \Delta
\right) $-linear basis $B^{\prime }$ of the vector space $L^{\prime }$ over $%
K\left( \Delta \right) $ and an $L^{\prime } $-linear basis $B^{\prime
\prime }$ of $L$ over $L^{\prime }$. By adjunction, we obtain an element $%
\sigma \in Aut\left( L/K\right) $ satisfying the property: $\sigma _{E}$ is
the restriction of $\sigma $ to $K\left( \Delta \right) $; there are two
elements $\sigma ^{\prime }\in Aut\left( L^{\prime }/K\left( \Delta \right)
\right) $ and $\sigma ^{\prime \prime }\in Aut\left( L/L^{\prime }\right) $
such that $\sigma \left( x^{\prime }\right) =\sigma ^{\prime }\left(
x^{\prime }\right) $ and $\sigma \left( x^{\prime \prime }\right) =\sigma
^{\prime \prime }\left( x^{\prime \prime }\right) $ hold for any $x^{\prime
}\in B^{\prime }$ and $x^{\prime \prime }\in B^{\prime \prime }$. Here, $%
\sigma ^{\prime }$ and $\sigma ^{\prime \prime }$ can be assumed to be such
that $\sigma ^{\prime }\left( x^{\prime }\right) =x^{\prime }$ and $\sigma
^{\prime \prime }\left( x^{\prime \prime }\right) =x^{\prime \prime }$ hold
for any $x^{\prime }\in B^{\prime }$ and $x^{\prime \prime }\in B^{\prime
\prime }$. This gives the extension of $\sigma _{E}$ to $\sigma \in
Aut\left( L/K\right) $.

From the below \emph{Claim 4.8} it is seen that each element $\sigma _{E}\in
Aut\left( K\left( \Delta \right) /K\right) \subseteq GL_{K}\left( E\right) $
with the identity map $1_{F}$ on $F$ defines a unique element $\sigma
=\left( \sigma _{E},1_{F}\right) \in GL_{K}\left( V\right) $. On the other
hand, for any $x\in V$, we have $x=x_{E}+x_{F}$ with $x_{E}\in E$ and $%
x_{F}\in F$; then
\begin{equation*}
\sigma \left( x\right) =\left(
\begin{array}{cc}
\sigma _{E} & 0 \\
0 & 1_{F}%
\end{array}%
\right) \left(
\begin{array}{c}
x_{E} \\
x_{F}%
\end{array}%
\right) =\sigma _{E}\left( x_{E}\right) +x_{F}.
\end{equation*}%
It follows that $\sigma \in Aut\left( L/K\right) $ and from $\left(
ii\right) $ it is seen that $\sigma $ is contained in a highest
transcendental Galois subgroup of the extension $L/K$. The converse is also
true. \emph{[}See \emph{Proposition 5.3} in \emph{\S 5.2} for further
properties for linear decompositions of automorphisms at nice bases\emph{]}.

$\left( ii\right) $ Immediately from definition since the transcendence
degree $tr.\deg K\left( \Delta \right) /K$ is exactly equal to $tr.\deg L/K$.

$\left( iii\right) $ For a highest transcendental Galois subgroup $H$ of $%
L/K $, there is a purely transcendental sub-extension $M=K\left( \Delta
\right) $ in $L/K$ with $tr.\deg M/K=tr.\deg L/K$, where $\Delta $ is a
transcendence base of $M/K$. Taking a linear basis $A$ of the vector space $%
L $ over $M$, we obtain a nice basis $\left( \Delta ,A\right) $ of $L/K$.
From $\left( i\right) $ we have
\begin{equation*}
H=\pi _{t}\left( L/K\right) \left( \Delta ,A\right) .
\end{equation*}

$\left( iv\right) $ For any two nice bases $\left( \Delta ,A\right) $ and $%
\left( \Lambda ,B\right) $ of the extension $L/K$, we have isomorphic groups
\begin{equation*}
\pi _{t}\left( L/K\right) \left( \Delta ,A\right) \cong Aut\left( K\left(
\Delta \right) /K\right) \cong Aut\left( K\left( \Lambda \right) /K\right)
\cong \pi _{t}\left( L/K\right) \left( \Lambda ,B\right)
\end{equation*}%
from the below \emph{Claim 4.9}. This completes the proof.
\end{proof}

\begin{claim}
Let $V$ be a linear space over a field $K$. Let $GL_{K}\left( V\right) $
denote the group of $K$-linear isomorphisms of $V$. Then for any direct sum
of $K$-subspaces
\begin{equation*}
V=E\oplus F
\end{equation*}%
there is a direct sum of subgroups
\begin{equation*}
\{\sigma \in GL_{K}\left( V\right) :\sigma \left( E\right) \subseteq
E,\sigma \left( F\right) \subseteq F\}=GL_{K}\left( E\right) \oplus
GL_{K}\left( F\right) .
\end{equation*}
\end{claim}

\begin{proof}
Let $H=\{\sigma \in GL_{K}\left( V\right) :\sigma \left( E\right) \subseteq
E,\sigma \left( F\right) \subseteq F\}.$ Then every $\sigma \in H$ has a
unique decomposition $\sigma =\left( \sigma _{E},\sigma _{F}\right) $ with $%
\sigma _{E}\in GL_{K}\left( E\right) $ and $\sigma _{F}\in GL_{K}\left(
F\right) $ being as coordinate components of $\sigma $, which is immediately
from the diagonal action on the $2$-column vectors.
\end{proof}

\begin{claim}
Suppose $L=K\left( s_{1},s_{2},\cdots ,s_{n}\right) $ and $M=K\left(
t_{1},t_{2},\cdots ,t_{n}\right) $ are two purely transcendental extensions
where $s_{1},s_{2},\cdots ,s_{n}$ and $t_{1},t_{2},\cdots ,t_{n}$ are two
sets of algebraically independent variables over $K$. Then there is an
isomorphism of the Galois groups
\begin{equation*}
Aut\left( L/K\right) \cong Aut\left( M/K\right) .
\end{equation*}
\end{claim}

\begin{proof}
Immediately $L$ and $M$ are isomorphic $K$-algebras via a map $\tau
:L\rightarrow M$ given by $s_{i}\mapsto t_{i}$ for any $1\leq i\leq n$.
Consider the map
\begin{equation*}
\lambda :Aut\left( L/K\right) \rightarrow Aut\left( M/K\right) ,t\mapsto
\tau \circ t\circ \tau ^{-1}
\end{equation*}%
between the Galois groups. From $\lambda $ it is seen that
\begin{equation*}
Aut\left( L/K\right) \cong Aut\left( M/K\right)
\end{equation*}%
are isomorphic groups.
\end{proof}

\begin{proposition}
Let $L$ be a transcendental extension over a field $K$ of finite
transcendence degree. Then for any nice basis $\left( \Delta ,A\right) $ of
the extension $L/K$, the full algebraic Galois subgroup
\begin{equation*}
\pi _{a}\left( L/K\right) \left( \Delta ,A\right)
\end{equation*}%
and highest transcendental Galois subgroup
\begin{equation*}
\pi _{t}\left( L/K\right) \left( \Delta ,A\right)
\end{equation*}%
intersect at a single element, i.e.,
\begin{equation*}
\pi _{a}\left( L/K\right) \left( \Delta ,A\right) \bigcap \pi _{t}\left(
L/K\right) \left( \Delta ,A\right) =\{1\}.
\end{equation*}
\end{proposition}

\begin{proof}
Immediately from \emph{Propositions 4.6-4.7}.
\end{proof}

From \emph{Propositions 4.6-4.7} we have the concluding remark for this
subsection.

\begin{remark}
(\emph{The algebraic and transcendental Galois subgroups at a nice basis})
Let $L$ be a transcendental extension over a field $K$ of finite
transcendence degree.

$\left( i\right) $ Each full algebraic Galois subgroup $G$ of the extension $%
L/M$ has a nice basis $\left( \Delta ,A\right) $ of $L/K$ such that $G=\pi
_{a}\left( L/K\right) \left( \Delta ,A\right) $ holds; conversely, each nice
basis $\left( \Delta ,A\right) $ of $L/K$ has a unique full algebraic Galois
subgroup $G$ of $L/M$ such that $G=\pi _{a}\left( L/K\right) \left( \Delta
,A\right) $. In such a case, $G$ is said to be the \textbf{full algebraic
Galois subgroup at a nice basis} $\left( \Delta ,A\right) $ of $L/K$.

$\left( ii\right) $ Each highest transcendental Galois subgroup $H$ of the
extension $L/M$ has a nice basis $\left( \Delta ,A\right) $ of $L/K$ such
that $H=\pi _{t}\left( L/K\right) \left( \Delta ,A\right) $ holds;
conversely, each nice basis $\left( \Delta ,A\right) $ of $L/K$ has a unique
highest transcendental Galois subgroup $H$ of $L/M$ such that $H=\pi
_{t}\left( L/K\right) \left( \Delta ,A\right) $. In such a case, $H$ is said
to be the \textbf{highest transcendental Galois subgroup at a nice basis} $%
\left( \Delta ,A\right) $ of $L/K$.
\end{remark}

\subsection{Galois correspondence for $\sigma$-tame Galois, (I)-(II)}

For an extension of a field, to make preparations for the Galois action on
its nice bases that will be discussed in \emph{\S 5}, we need a
transcendental version of Galois correspondence for $\sigma $-tame Galois
extensions, which generalises the result in \emph{Theorem 3.16}.

\begin{lemma}
\emph{(Galois correspondence for $\sigma $-tame Galois, I)} Let $L$ be a
transcendental extension over a field $K$ of finite transcendence degree.
Suppose $L$ is $\sigma $-tame Galois over $K$ (respectively, inside $\Omega
_{0}$). Here, $\Omega _{0}$ is a given subset of some certain transcendence
bases of $L/K$.

$\left( i\right) $ There exists a one-to-one correspondence between finite
subgroups $H$ of the Galois group $Aut\left( L/K\right) $ and intermediate
subfields $M$ in $L/K$ of finite degree $\left[ L:M\right] <+\infty $
satisfying the property: $M=L^{H}$ and $H=Aut\left( L/M\right) $
(respectively, such that $K\left( \Delta \right) \subseteq M$ for some $%
\Delta \in \Omega _{0}$).

$\left( ii\right) $ There exists a one-to-one correspondence between \emph{%
full algebraic Galois subgroups} $H$ of the extension $L/K$ and \emph{%
highest transcendental subfields} $M$ in $L/K$, i.e., $M$ is purely
transcendental over $K$ and $L$ is algebraic over $M$, satisfying the
property: $M=L^{H}$ and $H=Aut\left( L/M\right) $ (respectively, such that $%
K\left( \Delta \right) \subseteq M$ for some $\Delta \in \Omega _{0}$).
\end{lemma}

\begin{proof}
Without loss of generality, assume $\Omega _{0}$ is the set of all the
transcendence bases of the extension $L/K$.

$\left( i\right) $ Fixed a finite subgroup $H\subseteq Aut\left( L/K\right) $%
. Let $M=L^{H}$ be the $H$-invariant subfield of $L$. There is a nice basis $%
\left( \Delta ,A\right) $ of the extension $L/K$ such that
\begin{equation*}
K\left( \Delta \right) \subseteq M\subseteq K\left( \Delta \right) \left[ A%
\right] =L
\end{equation*}
where $M$ is an algebraic extension over $K\left( \Delta \right) $.

From the assumption that $L$ is $\sigma $-tame Galois over $K$, it is seen
that $L$ is Galois over $K\left( \Delta \right) $ and then over $M$; hence, $%
L^{H}=L^{Aut\left( L/L^{H}\right) }$ and $H=Aut\left( L/L^{H}\right) $.
Immediately, $L$ is a finite extension over $M$.

Now prove the existence of the Galois correspondence. In deed, given a
finite subgroup $H$ of the Galois group $Aut\left( L/K\right) $. There is a
unique subfield $M=L^{H}$ such that $H=Aut\left( L/M\right) $ and $\left[ L:M%
\right] =\sharp H<+\infty $ from the standard Galois theory for a tower of
algebraic extensions
\begin{equation*}
K\left( \Delta \right) \subseteq M\subseteq K\left( \Delta \right) \left[ A%
\right] =L
\end{equation*}
where $M$ is independent of the choice of a nice basis $\left( \Delta
,A\right) $ of $L/M$.

Conversely, for any finite subfield $M$ of $L$ with $\left[ L:M\right]
<+\infty $ there is a unique finite subgroup
\begin{equation*}
H=Aut\left( L/M\right) \subseteq Aut\left( L/K\right)
\end{equation*}%
such that $M=L^{H}.$

$\left( ii\right) $ Fixed a full algebraic Galois subgroup $H\subseteq
Aut\left( L/K\right) $. From \emph{Proposition 4.6} there is a nice basis $%
\left( \Delta ,A\right) $ of the extension $L/K$ such that $H=Aut\left(
L/K\left( \Delta \right) \right) $; as $L$ is $\sigma $-tame Galois over $K$%
, $L$ is an algebraic Galois extension over $K\left( \Delta \right) $ and
hence $L^{H}=K\left( \Delta \right) $ and $H=Aut\left( L/L^{H}\right) $.
This gives the existence for $M=K\left( \Delta \right) $.

On the other hand, let $\left( \Lambda ,B\right) $ be another nice basis of $%
L/K$ such that
\begin{equation*}
H=Aut\left( L/K\left( \Lambda \right) \right) .
\end{equation*}
We have
\begin{equation*}
K\left( \Delta \right) =L^{H}=K\left( \Lambda \right) .
\end{equation*}
This gives the uniqueness for $M=K\left( \Delta \right) $.

Conversely, take any highest transcendental subfield $M$ in the extension $%
L/K$, that is, $M$ is a subfield with $K\subseteq M\subseteq L$ such that $L$
is algebraic over $M$ and $M$ is purely transcendental over $K$ with
\begin{equation*}
tr.\deg M/K=tr.\deg L/K.
\end{equation*}
As $L$ is $\sigma $-tame Galois over $K$, we have the unique subgroup $%
H=Aut\left( L/M\right) $ of $Aut\left( L/K\right) $ such that $%
M=L^{Aut\left( L/M\right) }$. This completes the proof.
\end{proof}

Also there is a Galois correspondence between full algebraic Galois
subgroups and highest transcendental Galois subgroups in a $\sigma $-tame
Galois extension of a field of finite transcendence degree.

\begin{lemma}
\emph{(Galois correspondence for $\sigma $-tame Galois, II)} Let $L$ be a
transcendental extension over a field $K$ of finite transcendence degree.
Suppose $L$ is $\sigma $-tame Galois over $K$ (respectively, inside $\Omega
_{0}$). Here, $\Omega _{0}$ is a given subset of some certain transcendence
bases of $L/K$.

Then for any full algebraic Galois subgroup $G$ of the extension $L/K$
(respectively, with $L^{G}=K\left( \Delta \right) $ for some $\Delta \in
\Omega _{0}$), there is a unique highest transcendental Galois subgroup $H$
of $L/K$ such that $G\cap H=\{1\}.$

Conversely, for any highest transcendental Galois subgroup $H$ of $L/K$
(respectively, with $H=K\left( \Delta \right) $ for some $\Delta \in \Omega
_{0}$), there is a unique full algebraic Galois subgroup $G$ of $L/K$ such
that $G\cap H=\{1\}.$
\end{lemma}

\begin{proof}
For sake of convenience, suppose $\Omega _{0}$ is the set of all the
transcendence bases of the extension $L/K$. There are two cases for the
extension $L/K$.

\emph{Case }$\left( i\right) $. Assume $L$ is a finitely generated field
over $K$.

$\left( a\right) $ Fixed any full algebraic Galois subgroup $G$ of the
extension $L/K$. From \emph{Propositions 4.6-7} we have a nice basis $\left(
\Delta ,A\right) $ of $L/K$ such that
\begin{equation*}
G=\pi _{a}\left( L/K\right) \left( \Delta ,A\right)
\end{equation*}%
and that
\begin{equation*}
\pi _{t}\left( L/K\right) \left( \Delta ,A\right)
\end{equation*}%
is a highest transcendental Galois subgroup of $L/K$. Set
\begin{equation*}
H=\pi _{t}\left( L/K\right) \left( \Delta ,A\right) .
\end{equation*}%
By \emph{Propositions 4.10} we have
\begin{equation*}
G\cap H=\{1\}.
\end{equation*}

Suppose $H^{\prime }$ is a highest transcendental Galois subgroup of $L/K$
with $G\cap H^{\prime }=\{1\}$. From \emph{Proposition 4.7} again there is a
nice basis $\left( \Lambda ,B\right) $ of $L/K$ such that
\begin{equation*}
G^{\prime }=\pi _{a}\left( L/K\right) \left( \Lambda ,B\right)
\end{equation*}%
and
\begin{equation*}
H^{\prime }=\pi _{t}\left( L/K\right) \left( \Lambda ,B\right) .
\end{equation*}%
As $L$ is $\sigma $-tame Galois over $K$, from the above \emph{Lemma 4.12}
we have a unique subfield $M=L^{G}$ with $K\subseteq M\subseteq L$
satisfying the property: $L$ is algebraic over $M$ and $M$ is purely
transcendental over $K$ with
\begin{equation*}
tr.\deg M/K=tr.\deg L/K.
\end{equation*}

It follows that
\begin{equation*}
K\left( \Delta \right) =L^{G}=M=L^{G^{\prime }}=K\left( \Lambda \right)
\end{equation*}
and then
\begin{equation*}
H=Aut\left( K\left( \Delta \right) /K\right) =Aut\left( M/K\right)
=Aut\left( K\left( \Lambda \right) /K\right) =H^{\prime }.
\end{equation*}

This proves for each full algebraic Galois subgroup $G$ of $L/K$ there is
one and only one highest transcendental Galois subgroup $H$ of $L/K$ such
that
\begin{equation*}
G\cap H=\{1\}.
\end{equation*}

$\left( b\right) $ Conversely, given any highest transcendental Galois
subgroup $H$ of $L/K$. There is a nice basis $\left( \Delta ,A\right) $ of $%
L/K$ such that
\begin{equation*}
H=\pi _{t}\left( L/K\right) \left( \Delta ,A\right) .
\end{equation*}

Consider the subfield $M=K\left( \Delta \right) $. From \emph{Lemma 4.12}
again we have a unique full algebraic Galois subgroup $G$ of $L/K$ such that
\begin{equation*}
M=L^{G}\text{ and }G=Aut\left( L/M\right) =\pi _{a}\left( L/K\right) \left(
\Delta ,A\right) .
\end{equation*}

\emph{Case }$\left( ii\right) $. Suppose $L$ is an infinitely generated
field over $K$ of finite transcendence degree. Let $G$ be a full algebraic
Galois subgroup of $L/K$. From $\left( ii\right) $ of \emph{Lemma 4.12} it
is seen that there is a unique highest transcendental subfield $M=L^{G}$. As
$M$ is purely transcendental over $K$, there is a nice basis $\left( \Delta
,A\right) $ of $L/K$ such that $M=K\left( \Delta \right) $ and $L=M\left[ A%
\right] $. Consider the full algebraic and highest transcendental Galois
subgroups $\pi _{a}\left( L/K\right) \left( \Delta ,A\right) $ and $\pi
_{t}\left( L/K\right) \left( \Delta ,A\right) $ of $L/K$, respectively. We
have
\begin{equation*}
G=\pi _{a}\left( L/K\right) \left( \Delta ,A\right)
\end{equation*}
and
\begin{equation*}
H=\pi _{t}\left( L/K\right) \left( \Delta ,A\right)
\end{equation*}
with $G\cap H=\{1\}$.

Conversely, let $H$ be a highest transcendental Galois subgroup of $L/K$. By
\emph{Proposition 4.7} we have a nice basis $\left( \Delta ,A\right) $ of $%
L/K$ such that
\begin{equation*}
H=\pi _{t}\left( L/K\right) \left( \Delta ,A\right) .
\end{equation*}
As $L$ is $\sigma $-tame Galois over $K$, it is seen that $L$ is algebraic
Galois over $K\left( \Delta \right) $. Hence, $G=\pi _{a}\left( L/K\right)
\left( \Delta ,A\right) $ is the desired group. If $G^{\prime }$ is another
full algebraic Galois subgroup of $L/K$ with $G^{\prime }\cap H=\{1\}$, we
must have
\begin{equation*}
L^{G^{\prime }}=L^{G}=K\left( \Delta \right) ;
\end{equation*}
hence, $G^{\prime }=G$ holds since $L$ is algebraic Galois over $K\left(
\Delta \right) $. This completes the proof.
\end{proof}

\section{Transcendental Galois Theory, III. Decomposition Groups}

In this section we will discuss the Galois action on nice bases of a
transcendental extension and then introduce the decomposition group of the
extension at a nice basis. Further properties involving decomposition
groups will be also obtained.

\subsection{Galois action on nice bases}

Let $L$ be a transcendental extension over a field $K$ of finite
transcendence degree. Recall that $\left( \Delta ,A\right) $ is a \textbf{%
nice basis} of $L/K$ if $\Delta $ is a transcendence base of $L$ over $K$
and $A$ is a $K\left( \Delta \right) $-linear basis of $L$ as a vector space
over $K(\Delta )$.

\begin{definition}
Two nice bases $\left( \Delta ,A\right) $ and $\left( \Lambda ,B\right) $ of
the extension $L/K$ are said to be $\pi $\textbf{-equivalent}, denoted by $%
\left( \Delta ,A\right) \sim _{\pi }\left( \Lambda ,B\right) $, if there are
equalities
\begin{equation*}
\pi _{a}\left( L/K\right) \left( \Delta ,A\right) =\pi _{a}\left( L/K\right)
\left( \Lambda ,B\right)
\end{equation*}%
and
\begin{equation*}
\pi _{t}\left( L/K\right) \left( \Delta ,A\right) =\pi _{t}\left( L/K\right)
\left( \Lambda ,B\right)
\end{equation*}%
for the full algebraic Galois subgroups and highest transcendental Galois
subgroups of $L/K$ at $\left( \Delta ,A\right) $ and $\left( \Lambda
,B\right) $, respectively.

We denote by $\left[ \Delta ,A\right] _{\pi }$ the $\pi $-equivalence class
of a nice basis $\left( \Delta ,A\right) $ of the extension $L/K$ and let $%
\mathfrak{N}\left( L/K\right) $ denote the set of all the nice basis classes
$\left[ \Delta ,A\right] _{\pi }$ of $L/K$.
\end{definition}

For the extension $L/K$ there is a natural action of the Galois group $%
Aut\left( L/K\right) $ on its nice bases.

\begin{remark}
(\emph{Galois action on nice bases}) For an element $\sigma $ of the Galois
group $Aut\left( L/K\right) $ and for a nice basis $\left( \Delta ,A\right) $
of the extension $L/K$, it is seen that $\left( \sigma \left( \Delta \right)
,\sigma \left( A\right) \right) $ is also a nice basis of $L/K$ since $%
\sigma \left( \Delta \right) $ is a transcendence of $L/K$ and $L=K\left(
\sigma \left( \Delta \right) \right) \left[ \sigma \left( A\right) \right] .$
It follows that the Galois group $Aut\left( L/K\right) $ has a natural
action on the nice basis classes $\left[ \Delta ,A\right] _{\pi }$ of $L/K$.

Take a nice basis $\left( \Delta ,A\right) $ of the extension $L/K$. The set
\begin{equation*}
D_{L/K}\left( \Delta ,A\right) :=\{\sigma \in Aut\left( L/K\right) :\left[
\Delta ,A\right] _{\pi }=\left[ \sigma \left( \Delta \right) ,\sigma \left(
A\right) \right] _{\pi }\}
\end{equation*}%
is a subgroup of the Galois group $Aut\left( L/K\right) $, called the
\textbf{decomposition group} of the extension $L/K$ at $\left( \Delta
,A\right) $ (or at $\left[ \Delta ,A\right] _{\pi }$).
\end{remark}

\subsection{Linear decompositions of automorphisms of fields}

In order to obtain further properties on decomposition groups, we need some
preparatory facts for a $\sigma $-tame Galois extension.

Let $L$ be a transcendental extension over a field $K$. For a $K$-linear
subspace $V$ of $L$, let $GL_{K}\left( V\right) $ denote the group of $K$%
-linear isomorphisms of $V$.

Each automorphism $\sigma \in Aut\left( L/K\right) $ of the extension $L/K$
is a $K$-linear isomorphism of $L$ as a $K$-linear space, i.e., the Galois
group $Aut\left( L/K\right) $ is a subgroup of the general linear group $%
GL_{K}\left( L\right) $.

\begin{proposition}
(\emph{The linear decomposition of automorphisms at a nice basis}) Let $L$
be a transcendental extension over a field $K$ of finite transcendence
degree. Fixed a basis $\left( \Delta ,A\right) $ of the extension $L/K$.

$\left( i\right) $ As a $K$-linear space, the field $L$ has a direct sum of $%
K$-linear subspaces relative to $\left( \Delta ,A\right) $ (respectively,
with $\Delta \in \Omega _{0}$)
\begin{equation*}
L=E\oplus F
\end{equation*}%
where
\begin{equation*}
E:=K\left( \Delta \right)
\end{equation*}%
and
\begin{equation*}
F:=span_{K}\left( L\backslash K\left( \Delta \right) \right)
\end{equation*}%
are $K$-linear subspaces of $L$. In particular, $1\not\in F$; $x^{-1}\in F$
holds for any $0\not=x\in F$; but $F$ is not an algebra over $K$.

Moreover, each element $x\in L$ has a unique expression
\begin{equation*}
x=x_{E}+x_{F}
\end{equation*}%
in $L$ with $x_{E}\in E$ and $x_{F}\in F$. In this sense, each $%
x=x_{E}+x_{F}\in L$ will be denoted by
\begin{equation*}
x=\left(
\begin{array}{c}
x_{E} \\
x_{F}%
\end{array}%
\right)
\end{equation*}%
as a column $2$-vector, where $x_{E}$ and $x_{F}$ are called the $E$\textbf{%
-component} and $F$\textbf{-component} of $x$ in the double decomposition,
respectively.

$\left( ii\right) $ Under the direct sum of $L$, each $\sigma \in \pi
_{a}\left( L/K\right) \left( \Delta ,A\right) $ has a unique expression by a
square matrix
\begin{equation*}
\sigma \left( x\right) =\left(
\begin{array}{cc}
1_{E} & 0 \\
0 & \sigma _{F}%
\end{array}%
\right) \left(
\begin{array}{c}
x_{E} \\
x_{F}%
\end{array}%
\right) =\left(
\begin{array}{c}
x_{E} \\
\sigma _{F}\left( x_{F}\right)%
\end{array}%
\right)
\end{equation*}%
for any
\begin{equation*}
x=\left(
\begin{array}{c}
x_{E} \\
x_{F}%
\end{array}%
\right) \in L,
\end{equation*}%
where $1_{E}$ denotes the identity mapping on $E$ and $\sigma _{F}:=\sigma
|_{F}\in GL_{K}\left( F\right) $ denotes the restriction of $\sigma \in
GL_{K}\left( L\right) $ to the linear subspace $F$. In particular, $\sigma
_{F}\left( F\right) \subseteq F$. We denote
\begin{equation*}
\sigma =\left(
\begin{array}{cc}
1_{E} & 0 \\
0 & \sigma _{F}%
\end{array}%
\right)
\end{equation*}%
as a $2\times 2$ matrix, where $1_{E}$ and $\sigma _{F}$ are called the $E$%
\textbf{-component} and $F$\textbf{-component} of $\sigma $ in the double
decomposition, respectively.

$\left( iii\right) $ Under the direct sum of $L$, each $\sigma \in \pi
_{t}\left( L/K\right) \left( \Delta ,A\right) $ has a unique expression by a
square matrix
\begin{equation*}
\sigma \left( x\right) =\left(
\begin{array}{cc}
\sigma _{E} & 0 \\
0 & 1_{F}%
\end{array}%
\right) \left(
\begin{array}{c}
x_{E} \\
x_{F}%
\end{array}%
\right) =\left(
\begin{array}{c}
\sigma _{E}\left( x_{E}\right) \\
x_{F}%
\end{array}%
\right)
\end{equation*}%
for any
\begin{equation*}
x=\left(
\begin{array}{c}
x_{E} \\
x_{F}%
\end{array}%
\right) \in L,
\end{equation*}%
where $1_{F}$ denotes the identity mapping on $F$ and $\sigma _{E}:=\sigma
|_{E}\in GL_{K}\left( E\right) $ denotes the restriction of $\sigma \in
GL_{K}\left( L\right) $ to the linear subspace $E$. In particular, $\sigma
_{E}\left( E\right) \subseteq E$. We denote
\begin{equation*}
\sigma =\left(
\begin{array}{cc}
\sigma _{E} & 0 \\
0 & 1_{F}%
\end{array}%
\right)
\end{equation*}%
as a $2\times 2$ matrix, where $\sigma _{E}$ and $1_{F}$ are called the $E$%
\textbf{-component} and $F$\textbf{-component} of $\sigma $ in the double
decomposition, respectively.

$\left( iv\right) $ Suppose $L$ is tame Galois with repect to $\left( \Delta
,A\right) $. Under the direct sum of $L$, each $\delta \in D_{L/K}\left(
\Delta ,A\right) $ has a unique expression by a square matrix
\begin{equation*}
\delta \left( x\right) =\left(
\begin{array}{cc}
\delta _{E} & 0 \\
0 & \delta _{F}%
\end{array}%
\right) \left(
\begin{array}{c}
x_{E} \\
x_{F}%
\end{array}%
\right) =\left(
\begin{array}{c}
\delta _{E}\left( x_{E}\right) \\
\delta _{F}\left( x_{F}\right)%
\end{array}%
\right)
\end{equation*}%
for any
\begin{equation*}
x=\left(
\begin{array}{c}
x_{E} \\
x_{F}%
\end{array}%
\right) \in L,
\end{equation*}%
where $\delta _{E}:=\delta |_{E}\in GL_{K}\left( E\right) $ and $\delta
_{F}:=\delta |_{F}\in GL_{K}\left( F\right) $ denote the restrictions of $%
\delta \in GL_{K}\left( L\right) $ to the linear subspaces $E$ and $F$,
respectively. In particular, $\delta _{E}\left( E\right) \subseteq E$ and $%
\delta _{F}\left( F\right) \subseteq F$ hold.\ We denote%
\begin{equation*}
\delta =\left(
\begin{array}{cc}
\delta _{E} & 0 \\
0 & \delta _{F}%
\end{array}%
\right)
\end{equation*}%
as a $2\times 2$ matrix, where $\delta _{E}$ and $\delta _{F}$ are called
the $E$\textbf{-component} and $F$\textbf{-component} of $\delta $ in the
double decomposition, respectively.
\end{proposition}

\begin{proof}
$\left( i\right) $ Trivial from the one-by-one correspondence
\begin{equation*}
x=x_{E}+x_{F}=\left(
\begin{array}{cc}
1 & 1%
\end{array}%
\right) \left(
\begin{array}{c}
x_{E} \\
x_{F}%
\end{array}%
\right) \longleftrightarrow \left(
\begin{array}{c}
x_{E} \\
x_{F}%
\end{array}%
\right)
\end{equation*}%
for every $x\in L$.

$\left( ii\right) $ Consider the Galois group
\begin{equation*}
\pi _{a}\left( L/K\right) \left( \Delta ,A\right) =Aut\left( L/K\left(
\Delta \right) \right)
\end{equation*}%
of the algebraic extension $L/K\left( \Delta \right) $. Fixed any $\sigma
\in \pi _{a}\left( L/K\right) \left( \Delta ,A\right) $. It reduces to prove
that $E$ and $F$ both are invariant subspaces of $\sigma $, i.e.,
\begin{equation*}
\sigma \left( E\right) \subseteq E\text{ and }\sigma \left( F\right)
\subseteq F.
\end{equation*}%
Immediately, $\sigma \left( E\right) =E$ since $\sigma \in Aut\left(
L/K\left( \Delta \right) \right) $ and $\sigma |_{E}$ is the identity map on
$E$.

In the following, we prove
\begin{equation*}
\sigma \left( F\right) \subseteq F
\end{equation*}%
holds. We proceed in several steps.

\emph{Step 1}. From the nice basis $\left( \Delta ,A\right) $ of the
extension $L/K$, we have
\begin{equation*}
L=span_{K\left( \Delta \right) }A
\end{equation*}%
and
\begin{equation*}
A=A_{0}\cup \{e_{0}\}
\end{equation*}%
with $e_{0}\in K\left( \Delta \right) $ and $A_{0}=A\setminus \{e_{0}\}$,
i.e., $A_{0}\cap K\left( \Delta \right) =\varnothing $.

Under the $K\left( \Delta \right) $-linear basis $A$ of the linear space $L$%
, every element $x\in L$ has a unique expression as a finite $K\left( \Delta
\right) $-linear combination of some elements $e_{0},e_{1},\cdots ,e_{m}\in
A $, that is,
\begin{equation*}
x=\sum_{i=0}^{m}c_{i}\cdot e_{i}
\end{equation*}%
with
\begin{equation*}
c_{i}=\frac{f_{i}\left( t_{1},t_{2},\cdots ,t_{n}\right) }{g_{i}\left(
t_{1},t_{2},\cdots ,t_{n}\right) }\in K\left( \Delta \right)
\end{equation*}%
and
\begin{equation*}
e_{i}\in A
\end{equation*}%
for every $0\leq i\leq m$, where
\begin{equation*}
\Delta =\{t_{1},t_{2},\cdots ,t_{n}\}
\end{equation*}%
is the given transcendence base of $L/K$ and
\begin{equation*}
f_{i}\left( X_{1},X_{2},\cdots ,X_{n}\right) ,g_{i}\left( X_{1},X_{2},\cdots
,X_{n}\right) \in K\left[ X_{1},X_{2},\cdots ,X_{n}\right]
\end{equation*}%
are polynomials with coefficients in $K$ for every $0\leq i\leq m$.

\emph{Step 2}. As the subset $A_{0}$ of $A$ is linearly independent over $%
K\left( \Delta \right) $ and then over $K,$ there is a subset $%
A_{1}\subseteq F$ such that the union $A_{0}\cup A_{1}$ is a $K$-linear
basis of the linear space $F$, i.e.,
\begin{equation*}
F=span_{K}\left( A_{0}\cup A_{1}\right) .
\end{equation*}%
Obviously, $1\notin F$; none of elements $\omega $ in $A_{1}$ is contained
in $K\left( \Delta \right) $ although $\omega $ can be a finite $K$-linear
combination of some products $u\cdot v$ with $u\in K\left( \Delta \right) $
and $v\in L\setminus K\left( \Delta \right) $.

From \emph{Step 1} it is seen that every element $x\in A_{1}$ is of the form
\begin{equation*}
\sum_{i=0}^{m}\frac{f_{i}\left( t_{1},t_{2},\cdots ,t_{n}\right) }{%
g_{i}\left( t_{1},t_{2},\cdots ,t_{n}\right) }\cdot e_{i}
\end{equation*}%
where
\begin{equation*}
e_{i}\in A
\end{equation*}%
and
\begin{equation*}
f_{i}\left( X_{1},X_{2},\cdots ,X_{n}\right) ,g_{i}\left( X_{1},X_{2},\cdots
,X_{n}\right) \in K\left[ X_{1},X_{2},\cdots ,X_{n}\right]
\end{equation*}%
for each $0\leq i\leq m$.

\emph{Step 3}. Fixed any $\sigma \in Aut\left( L/K\left( \Delta \right)
\right) $. As $K\left( \Delta \right) $ is contained in the invariant
subfield $L^{Aut\left( L/K\right) }$ of $L$ under $Aut\left( L/K\left(
\Delta \right) \right) $, it is seen that $\sigma \left( z\right) =z$ holds
for any $z\in K\left( \Delta \right) $. It follows that for an $x\in L$, $%
x\in K\left( \Delta \right) $ holds if and only if $\sigma \left( x\right)
\in K\left( \Delta \right) $. Then we have
\begin{equation*}
L\ni \sigma \left( x\right) \notin K\left( \Delta \right)
\end{equation*}%
for any
\begin{equation*}
x=\sum_{i=0}^{m}k_{i}\cdot e_{i}
\end{equation*}%
with $m\geq 1$, where
\begin{equation*}
e_{0}\in K\left( \Delta \right) ,k_{0}\in K
\end{equation*}
and%
\begin{equation*}
e_{i}\in A_{0}\cup A_{1},k_{i}\in K
\end{equation*}%
for $1\leq i\leq m$.

Particularly, for any $\sigma \in Aut\left( L/K\left( \Delta \right) \right)
$ and $x\in A_{0}\cup A_{1}$, we have
\begin{equation*}
L\ni \sigma \left( x\right) \notin K\left( \Delta \right)
\end{equation*}%
and

$\sigma \left( x\right) \in L\setminus K\left( \Delta \right) .$

Then
\begin{equation*}
\sigma \left( A_{0}\cup A_{1}\right) \subseteq F
\end{equation*}%
holds for each $\sigma \in Aut\left( L/K\left( \Delta \right) \right) $.
Hence,
\begin{equation*}
\sigma \left( F\right) \subseteq F.
\end{equation*}

This proves $\sigma _{F}=\sigma |_{F}\in GL_{K}\left( F\right) $ is
well-defined.

$\left( iii\right) $ Take an element $\sigma \in \pi _{t}\left( L/K\right)
\left( \Delta ,A\right) $. For the restriction $\sigma _{E}$ of $\sigma $ to
$E$, we have
\begin{equation*}
\sigma _{E}\in Aut\left( K\left( \Delta \right) /K\right)
\end{equation*}%
from definition and then
\begin{equation*}
\sigma \left( E\right) =E.
\end{equation*}%
This proves $\sigma _{E}=\sigma |_{E}\in GL_{K}\left( E\right) $ is
well-defined.

$\left( iv\right) $ Fixed an element $\delta $ of the decomposition group $%
D_{L/K}\left( \Delta ,A\right) $ of the extension $L/K$ at the nice basis $%
\left( \Delta ,A\right) $. As
\begin{equation*}
\left( \Delta ,A\right) \sim _{\pi }\left( \delta \left( \Delta \right)
,\delta \left( A\right) \right)
\end{equation*}%
are $\pi $-equivalent nice bases of $L/K$, we have
\begin{equation*}
\begin{array}{l}
\pi _{t}\left( L/K\right) \left( \Delta ,A\right) \\
=\pi _{t}\left( L/K\right) \left( \delta \left( \Delta \right) ,\delta
\left( A\right) \right)%
\end{array}%
\end{equation*}%
and
\begin{equation*}
\begin{array}{l}
Aut\left( L/K\left( \Delta \right) \right) \\
=\pi _{a}\left( L/K\right) \left( \Delta ,A\right) \\
=\pi _{a}\left( L/K\right) \left( \delta \left( \Delta \right) ,\delta
\left( A\right) \right) \\
=Aut\left( L/K\left( \delta \left( \Delta \right) \right) \right) .%
\end{array}%
\end{equation*}

Now compare the two nice bases $\left( \Delta ,A\right) $ and $\left( \delta
\left( \Delta \right) ,\delta \left( A\right) \right) $ of $L/K$. We have%
\begin{equation*}
L=K\left( \Delta \right) [A]=K\left( \delta \left( \Delta \right) \right)
[\delta \left( A\right) ]=\delta \left( L\right)
\end{equation*}%
and%
\begin{equation*}
\lbrack L:K\left( \Delta \right) ]=\left[ \delta \left( L\right) :K\left(
\delta \left( \Delta \right) \right) \right] \leq +\infty
\end{equation*}
since $A$ and $\delta \left( A\right) $ both are linear bases of $L=\delta
\left( L\right) $ as a linear space over the subfields $K\left( \Delta
\right) $ and $K\left( \delta \left( \Delta \right) \right) $, respectively.

From the assumption that $L$ is tame Galois with respect to $\left( \Delta
,A\right) $, i.e.,
\begin{equation*}
L=K\left( \Delta \right) [A]
\end{equation*}
is algebraic Galois over $K\left( \Delta \right) $, it is seen that
\begin{equation*}
L=K\left( \delta \left( \Delta \right) \right) [\delta \left( A\right) ]
\end{equation*}
is also algebraic Galois over $K\left( \delta \left( \Delta \right) \right) $
since via the automorphism $\delta \in D_{L/K}\left( \Delta ,A\right) $, the
normality and separability of the field $L$ over $K\left( \Delta \right) $
are exactly translated into the ones of $L=\delta \left( L\right) $ over $%
K\left( \delta \left( \Delta \right) \right) =\delta \left( K\left( \Delta
\right) \right) $, respectively. It follows that $L$ is $\sigma $-tame
Galois over $K$ inside a subset%
\begin{equation*}
\Omega _{0}=\{\delta \left( \Delta \right) :\delta \in D_{L/K}\left( \Delta
,A\right) \}
\end{equation*}
of $\Omega _{L/K}$. Then we have
\begin{equation*}
\begin{array}{l}
K\left( \Delta \right) \\
=L^{Aut\left( L/K\left( \Delta \right) \right) } \\
=L^{Aut\left( L/K\left( \delta \left( \Delta \right) \right) \right) } \\
=K\left( \delta \left( \Delta \right) \right)%
\end{array}%
\end{equation*}%
and
\begin{equation*}
\begin{array}{l}
L=K\left( \Delta \right) \left[ A\right] \\
=K\left( \delta \left( \Delta \right) \right) \left[ \delta \left( A\right) %
\right] \\
=K\left( \Delta \right) \left[ \delta \left( A\right) \right]%
\end{array}%
\end{equation*}%
by \emph{Lemma 4.13} which states that the full algebraic Galois subgroup $%
\pi _{a}\left( L/K\right) \left( \Delta ,A\right) $ corresponds to the
unique highest transcendental subfield $K\left( \Delta \right) $ in $L/K$.

Put
\begin{equation*}
E=K\left( \Delta \right)
\end{equation*}
and
\begin{equation*}
F=span_{K}\left( L\setminus K\left( \Delta \right) \right) =span_{K}\left(
L\setminus K\left( \delta \left( \Delta \right) \right) \right) .
\end{equation*}
We have
\begin{equation*}
\delta \left( E\right) =\delta \left( K\left( \Delta \right) \right)
=K\left( \delta \left( \Delta \right) \right) =E.
\end{equation*}

On the other hand, applying the same procedure of \emph{Steps 1-3} in $%
\left( ii\right) $ to the automorphism $\delta $ here, for the linear basis $%
A_{0}\cup A_{1}$ of the linear subspace $F$ over $K$, there is%
\begin{equation*}
\delta \left( F\right) \subseteq F
\end{equation*}
since%
\begin{equation*}
\begin{array}{l}
\delta \left( A_{0}\cup A_{1}\right) \\
=\delta \left( A_{0}\right) \cup \delta \left( A_{1}\right) \\
\subseteq \delta \left( L\right) \setminus K\left( \delta \left( \Delta
\right) \right) \\
=L\setminus K\left( \Delta \right) \subseteq F.%
\end{array}%
\end{equation*}

Hence, the restrictions
\begin{equation*}
\delta _{E}=\delta |_{E}\in GL_{K}\left( E\right)
\end{equation*}
and
\begin{equation*}
\delta _{F}=\delta |_{F}\in GL_{K}\left( F\right)
\end{equation*}
are well-defined. This completes the proof.
\end{proof}

\subsection{Factorisation of a decomposition group}

For any nice basis $\left( \Delta ,A\right) $ of the extension $L/K$, let $%
D_{L/K}\left( \Delta ,A\right) $, $\pi _{a}\left( L/K\right) \left( \Delta
,A\right) $ and $\pi _{t}\left( L/K\right) \left( \Delta ,A\right) $ denote
the decomposition group, full algebraic Galois subgroup and highest
transcendental Galois subgroup of $L/K$ at $\left( \Delta ,A\right) $,
respectively.

\begin{theorem}
Let $L$ be an extension over a field $K$ of finite transcendence degree.
Suppose $L$ is $\sigma $-tame Galois over $K$ (respectively, inside $\Omega
_{0}$). Here, $\Omega _{0}$ is a given subset of some certain transcendence
bases of $L/K$. Fixed a basis $\left( \Delta ,A\right) $ of the extension $%
L/K$ (respectively, with $\Delta \in \Omega _{0}$).

Then $\pi _{a}\left( L/K\right) \left( \Delta ,A\right) $ and $\pi
_{t}\left( L/K\right) \left( \Delta ,A\right) $ both are normal subgroups of
the decomposition group $D_{L/K}\left( \Delta ,A\right) $; moreover, the
decomposition group is a direct product of the two normal subgroups
\begin{equation*}
D_{L/K}\left( \Delta ,A\right) =\pi _{a}\left( L/K\right) \left( \Delta
,A\right) \cdot \pi _{t}\left( L/K\right) \left( \Delta ,A\right) .
\end{equation*}

In particular, each element $\sigma \in D_{L/K}\left( \Delta ,A\right) $ has
a unique factorisation
\begin{equation*}
\sigma =\sigma _{a}\cdot \sigma _{t}=\sigma _{t}\cdot \sigma _{a}
\end{equation*}%
in $Aut\left( L/K\right) $ with $\sigma _{a}\in \pi _{a}\left( L/K\right)
\left( \Delta ,A\right) $ and $\sigma _{t}\in \pi _{t}\left( L/K\right)
\left( \Delta ,A\right) .$
\end{theorem}

\begin{proof}
Fixed a basis $\left( \Delta ,A\right) $ of the extension $L/K$ with $\Delta
\in \Omega _{0}$. We will proceed in several steps.

\emph{Step 1}. Prove $\pi _{a}\left( L/K\right) \left( \Delta ,A\right) $
and $\pi _{t}\left( L/K\right) \left( \Delta ,A\right) $ both are subgroups
of the decomposition group $D_{L/K}\left( \Delta ,A\right) $.

In fact, take an element $\sigma $ of the subgroup $\pi _{a}\left(
L/K\right) \left( \Delta ,A\right) =Aut\left( L/K\left( \Delta \right)
\right) $. It is seen that $\sigma \left( \Delta \right) =\Delta \subseteq
K\left( \Delta \right) $ and $\sigma \left( A\right) $ is a linear basis of $%
L$ over $K\left( \Delta \right) $; then
\begin{equation*}
\pi _{a}\left( L/K\right) \left( \Delta ,A\right) =\pi _{a}\left( L/K\right)
\left( \sigma \left( \Delta \right) ,\sigma \left( A\right) \right) .
\end{equation*}%
It follows that
\begin{equation*}
\left( \Delta ,A\right) \thicksim _{\pi }\left( \sigma \left( \Delta \right)
,\sigma \left( A\right) \right)
\end{equation*}%
are $\pi $-equivalent nice bases of $L/K$; hence, $\sigma \in D_{L/K}\left(
\Delta ,A\right) $. This proves
\begin{equation*}
\pi _{a}\left( L/K\right) \left( \Delta ,A\right) \subseteq D_{L/K}\left(
\Delta ,A\right) .
\end{equation*}

In the same way, take any $\sigma \in \pi _{t}\left( L/K\right) \left(
\Delta ,A\right) $. Immediately, $\sigma \left( A\right) =A$. For the
restriction to $K\left( \Delta \right) $ we have
\begin{equation*}
\sigma \left( K\left( \Delta \right) \right) =K\left( \sigma \left( \Delta
\right) \right) =K\left( \Delta \right) .
\end{equation*}%
Then
\begin{equation*}
\pi _{t}\left( L/K\right) \left( \Delta ,A\right) =\pi _{t}\left( L/K\right)
\left( \sigma \left( \Delta \right) ,\sigma \left( A\right) \right) .
\end{equation*}%
Hence,
\begin{equation*}
\left( \Delta ,A\right) \thicksim _{\pi }\left( \sigma \left( \Delta \right)
,\sigma \left( A\right) \right) .
\end{equation*}%
This proves $\sigma \in D_{L/K}\left( \Delta ,A\right) $ and
\begin{equation*}
\pi _{a}\left( L/K\right) \left( \Delta ,A\right) \subseteq D_{L/K}\left(
\Delta ,A\right)
\end{equation*}%
holds.

\emph{Step 2}. As in \emph{Proposition 5.3}, take the field $L$ as the
direct decomposition
\begin{equation*}
L=E\oplus F
\end{equation*}%
of $K$-linear subspaces. Here,
\begin{equation*}
\begin{array}{l}
E:=K\left( \Delta \right) \subseteq L; \\
F:=span_{K}\left( L\backslash K\left( \Delta \right) \right) \subseteq L.%
\end{array}%
\end{equation*}

For any $x\in L$ there is a unique expression
\begin{equation*}
x=x_{E}+x_{F}
\end{equation*}%
with $x_{E}\in E$ and $x_{F}\in F$. Denote $x\in L$ by a column $2$-vector
\begin{equation*}
x=\left(
\begin{array}{c}
x_{E} \\
x_{F}%
\end{array}%
\right) .
\end{equation*}

From \emph{Proposition 5.3}, for any $\sigma \in \pi _{a}\left( L/K\right)
\left( \Delta ,A\right) ,\rho \in \pi _{t}\left( L/K\right) \left( \Delta
,A\right) $ and $\delta \in D_{L/K}\left( \Delta ,A\right) $ we have
\begin{equation*}
\sigma =\left(
\begin{array}{cc}
1_{E} & 0 \\
0 & \sigma _{F}%
\end{array}%
\right) ,
\end{equation*}

\begin{equation*}
\rho =\left(
\begin{array}{cc}
\rho _{E} & 0 \\
0 & 1_{F}%
\end{array}%
\right)
\end{equation*}%
and
\begin{equation*}
\delta =\left(
\begin{array}{cc}
\delta _{E} & 0 \\
0 & \delta _{F}%
\end{array}%
\right)
\end{equation*}%
as $2\times 2$ matrices, respectively.

\emph{Step 3}. Prove $\pi _{a}\left( L/K\right) \left( \Delta ,A\right) $
and $D_{L/K}\left( \Delta ,A\right) $ have the same $F$-components
\begin{equation*}
\{\sigma _{F}:\sigma \in \pi _{a}\left( L/K\right) \left( \Delta ,A\right)
\}=\{\delta _{F}:\delta \in D_{L/K}\left( \Delta ,A\right) \}.
\end{equation*}

In deed, from \emph{Step 1} it is seen that $\{\sigma _{F}:\sigma \in \pi
_{a}\left( L/K\right) \left( \Delta ,A\right) \}$ is a subset of $\{\delta
_{F}:\delta \in D_{L/K}\left( \Delta ,A\right) \}$.

Conversely, take an element $t\in \{\delta _{F}:\delta \in D_{L/K}\left(
\Delta ,A\right) \}$ such that
\begin{equation*}
t=\delta _{F}
\end{equation*}%
for some $\delta \in D_{L/K}\left( \Delta ,A\right) $. Put
\begin{equation*}
\tau =\left(
\begin{array}{cc}
1_{E} & 0 \\
0 & t%
\end{array}%
\right) .
\end{equation*}%
It reduces to show
\begin{equation*}
\tau \in \pi _{a}\left( L/K\right) \left( \Delta ,A\right) .
\end{equation*}

Take any
\begin{equation*}
x=x_{E}+x_{F}
\end{equation*}
and
\begin{equation*}
y=y_{E}+y_{F}
\end{equation*}
in $L$ with $x_{E},y_{E}\in E$ and $x_{F},y_{F}\in F$. As $E,F\subseteq L$
we have
\begin{equation*}
\begin{array}{l}
x+y \\
=\left( x_{E}+y_{E}\right) +\left( x_{F}+y_{F}\right) ; \\
x\cdot y \\
=\left( x_{E}+x_{F}\right) \cdot \left( y_{E}+y_{F}\right) \\
=\left( x_{E}\cdot y_{E}\right) +\left( x_{E}\cdot y_{F}+x_{F}\cdot
y_{E}+x_{F}\cdot y_{F}\right) .%
\end{array}%
\end{equation*}

Obviously,
\begin{equation*}
\begin{array}{l}
x_{E}+y_{E},x_{E}\cdot y_{E}\in E; \\
x_{F}+y_{F},x_{E}\cdot y_{F}+x_{F}\cdot y_{E}+x_{F}\cdot y_{F}\in F.%
\end{array}%
\end{equation*}
Then
\begin{equation*}
\begin{array}{l}
\delta \left( x+y\right) \\
=\delta \left( x_{E}+y_{E}\right) +\delta \left( x_{F}+y_{F}\right) \\
=\delta \left( x_{E}+y_{E}\right) +t\left( x_{F}+y_{F}\right) ; \\
\delta \left( x\cdot y\right) \\
=\delta \left( x_{E}\cdot y_{E}\right) + \delta\left( x_{E}\cdot
y_{F}+x_{F}\cdot y_{E}+x_{F}\cdot y_{F}\right) \\
=\delta \left( x_{E}\cdot y_{E}\right) +t\left( x_{E}\cdot y_{F}+x_{F}\cdot
y_{E}+x_{F}\cdot y_{F}\right) .%
\end{array}%
\end{equation*}

From $t=\delta |_{F}$ we have
\begin{equation*}
\begin{array}{l}
\tau \left( x+y\right) \\
=\left(
\begin{array}{cc}
1_{E} & 0 \\
0 & t%
\end{array}%
\right) \left(
\begin{array}{c}
x_{E}+y_{E} \\
x_{F}+y_{F}%
\end{array}%
\right) \\
=\left(
\begin{array}{c}
x_{E}+y_{E} \\
t\left( x_{F}+y_{F}\right)%
\end{array}%
\right) \\
=\left(
\begin{array}{c}
x_{E} \\
t\left( x_{F}\right)%
\end{array}%
\right) +\left(
\begin{array}{c}
y_{E} \\
t\left( y_{F}\right)%
\end{array}%
\right) \\
=\tau \left( x\right) +\tau \left( y\right) ; \\
\tau \left( x\cdot y\right) \\
=\left(
\begin{array}{cc}
1_{E} & 0 \\
0 & t%
\end{array}%
\right) \left(
\begin{array}{c}
x_{E}\cdot y_{E} \\
x_{E}\cdot y_{F}+x_{F}\cdot y_{E}+x_{F}\cdot y_{F}%
\end{array}%
\right) \\
=\left(
\begin{array}{c}
x_{E}\cdot y_{E} \\
t\left( x_{E}\cdot y_{F}+x_{F}\cdot y_{E}+x_{F}\cdot y_{F}\right)%
\end{array}%
\right) \\
=\left(
\begin{array}{c}
x_{E} \\
t\left( x_{F}\right)%
\end{array}%
\right) \cdot \left(
\begin{array}{c}
y_{E} \\
t\left( y_{F}\right)%
\end{array}%
\right) \\
=\tau \left( x\right) \cdot \tau \left( y\right) .%
\end{array}%
\end{equation*}
Particularly,
\begin{equation*}
\tau \left( x_{E}\right) =x_{E}.
\end{equation*}

Hence,
\begin{equation*}
\tau \in Aut\left( L/K\left( \Delta \right) \right) =\pi _{a}\left(
L/K\right) \left( \Delta ,A\right) .
\end{equation*}
This proves the two subgroups have the same $F$-components.

\emph{Step 4}. Prove $\pi _{t}\left( L/K\right) \left( \Delta ,A\right) $
and $D_{L/K}\left( \Delta ,A\right) $ have the same $E$-components
\begin{equation*}
\{\sigma _{E}:\sigma \in \pi _{t}\left( L/K\right) \left( \Delta ,A\right)
\}=\{\delta _{E}:\delta \in D_{L/K}\left( \Delta ,A\right) \}.
\end{equation*}

Immediately from \emph{Step 1} it is seen that $\{\sigma _{E}:\sigma \in \pi
_{a}\left( L/K\right) \left( \Delta ,A\right) \}$ is a subset of $\{\delta
_{E}:\delta \in D_{L/K}\left( \Delta ,A\right) \}$. From \emph{Definition 4.7%
} we have
\begin{equation*}
\{\sigma _{E}:\sigma \in \pi _{t}\left( L/K\right) \left( \Delta ,A\right)
\}=Aut\left( K\left( \Delta \right) /K\right) .
\end{equation*}

Let $\delta \in D_{L/K}\left( \Delta ,A\right) $. It reduces to prove
\begin{equation*}
\delta _{E}\in Aut\left( K\left( \Delta \right) /K\right) .
\end{equation*}

As $L/K$ is $\sigma $-tame Galois, from $\left( ii\right) $ of \emph{Lemma
4.12} it is seen that
\begin{equation*}
K\left( \Delta \right) =K\left( \delta \left( \Delta \right) \right)
\end{equation*}%
holds since
\begin{equation*}
\begin{array}{l}
Aut\left( L/K\left( \Delta \right) \right) \\
=\pi _{a}\left( L/K\right) \left( \Delta ,A\right) \\
=\pi _{a}\left( L/K\right) \left( \delta \left( \Delta \right) ,\delta
\left( A\right) \right) \\
=Aut\left( L/K\left( \delta \left( \Delta \right) \right) \right) .%
\end{array}%
\end{equation*}

Hence,
\begin{equation*}
\delta _{E}\in Aut\left( K\left( \Delta \right) /K\right) .
\end{equation*}
This proves the two subgroups have the same $E$-components.

\emph{Step 5}. Prove $\pi _{a}\left( L/K\right) \left( \Delta ,A\right) $ is
a normal subgroup of $D_{L/K}\left( \Delta ,A\right) $.

In deed, take any $\sigma \in \pi _{a}\left( L/K\right) \left( \Delta
,A\right) ,\delta \in D_{L/K}\left( \Delta ,A\right) $. As
\begin{equation*}
\sigma =\left(
\begin{array}{cc}
1_{E} & 0 \\
0 & \sigma _{F}%
\end{array}%
\right)
\end{equation*}%
and
\begin{equation*}
\delta =\left(
\begin{array}{cc}
\delta _{E} & 0 \\
0 & \delta _{F}%
\end{array}%
\right) ,
\end{equation*}%
we have
\begin{equation*}
\begin{array}{l}
\delta ^{-1}\cdot \sigma \cdot \delta \\
=\left(
\begin{array}{cc}
\delta _{E}^{-1} & 0 \\
0 & \delta _{F}^{-1}%
\end{array}%
\right) \cdot \left(
\begin{array}{cc}
1_{E} & 0 \\
0 & \sigma _{F}%
\end{array}%
\right) \cdot \left(
\begin{array}{cc}
\delta _{E} & 0 \\
0 & \delta _{F}%
\end{array}%
\right) \\
=\left(
\begin{array}{cc}
1_{E} & 0 \\
0 & \delta _{F}^{-1}\cdot \sigma _{F}\cdot \delta _{F}%
\end{array}%
\right) .%
\end{array}%
\end{equation*}%
Hence, for any $\delta \in D_{L/K}\left( \Delta ,A\right) $
\begin{equation*}
\delta ^{-1}\cdot \pi _{a}\left( L/K\right) \left( \Delta ,A\right) \cdot
\delta \subseteq \pi _{a}\left( L/K\right) \left( \Delta ,A\right)
\end{equation*}%
holds.

In the same way, it is seen that $\pi _{t}\left( L/K\right) \left( \Delta
,A\right) $ is a normal subgroup of the decomposition group $D_{L/K}\left(
\Delta ,A\right) $.

\emph{Step 6}. Let $\sigma \in \pi _{a}\left( L/K\right) \left( \Delta
,A\right) $ and $\tau \in \pi _{t}\left( L/K\right) \left( \Delta ,A\right) $%
. We have
\begin{equation*}
\begin{array}{l}
\sigma \cdot \tau \\
=\left(
\begin{array}{cc}
1_{E} & 0 \\
0 & \sigma _{F}%
\end{array}%
\right) \cdot \left(
\begin{array}{cc}
\tau _{E} & 0 \\
0 & 1_{F}%
\end{array}%
\right) \\
=\left(
\begin{array}{cc}
\tau _{E} & 0 \\
0 & \sigma _{F}%
\end{array}%
\right) \\
=\left(
\begin{array}{cc}
\tau _{E} & 0 \\
0 & 1_{F}%
\end{array}%
\right) \cdot \left(
\begin{array}{cc}
1_{E} & 0 \\
0 & \sigma _{F}%
\end{array}%
\right) \\
=\tau \cdot \sigma .%
\end{array}%
\end{equation*}

Then
\begin{equation*}
\pi _{a}\left( L/K\right) \left( \Delta ,A\right) \cdot \pi _{t}\left(
L/K\right) \left( \Delta ,A\right) =\pi _{t}\left( L/K\right) \left( \Delta
,A\right) \cdot \pi _{a}\left( L/K\right) \left( \Delta ,A\right) .
\end{equation*}

It follows that the product
\begin{equation*}
\pi _{a}\left( L/K\right) \left( \Delta ,A\right) \cdot \pi _{t}\left(
L/K\right) \left( \Delta ,A\right)
\end{equation*}%
is a subgroup of $D_{L/K}\left( \Delta ,A\right) $. From \emph{Steps 3-4} we
have
\begin{equation*}
D_{L/K}\left( \Delta ,A\right) =\pi _{a}\left( L/K\right) \left( \Delta
,A\right) \cdot \pi _{t}\left( L/K\right) \left( \Delta ,A\right)
\end{equation*}%
which is a direct product of the two normal subgroups since
\begin{equation*}
\pi _{a}\left( L/K\right) \left( \Delta ,A\right) \bigcap \pi _{t}\left(
L/K\right) \left( \Delta ,A\right) =\{1\}.
\end{equation*}%
This completes the proof.
\end{proof}

\subsection{Invariant subfield of a decomposition group}

Now  consider the invariant subfield of $L/K$ under a decomposition group.

\begin{theorem}
Let $L$ be an extension over a field $K$ of finite transcendence degree.
Suppose $L$ is $\sigma $-tame Galois over $K$ (respectively, inside $\Omega
_{0}$). Here, $\Omega _{0}$ is a given subset of some certain transcendence
bases of $L/K$. Then we have
\begin{equation*}
L^{D_{L/K}\left( \Delta ,A\right) }=L^{Aut\left( L/K\right) }=K
\end{equation*}%
for any nice basis $\left( \Delta ,A\right) $ of $L/K$ (respectively, with $%
\Delta \in \Omega _{0}$).
\end{theorem}

\begin{proof}
Without loss of generality, suppose $\Omega _{0}$ is the set of all the
transcendence bases of the extension $L/K$. We will proceed in several steps.

\emph{Step 1}. Consider the decomposition group $D_{L/K}\left( \Delta
,A\right) $, full algebraic Galois subgroup $\pi _{a}\left( L/K\right)
\left( \Delta ,A\right) $ and highest transcendental Galois subgroup $\pi
_{t}\left( L/K\right) \left( \Delta ,A\right) $ of $L/K$ at $\left( \Delta
,A\right) $, respectively. From \emph{Theorem 7.4} we have two towers of
subgroups
\begin{equation*}
\pi _{a}\left( L/K\right) \left( \Delta ,A\right) \subseteq D_{L/K}\left(
\Delta ,A\right) \subseteq Aut\left( L/K\right)
\end{equation*}%
and
\begin{equation*}
\pi _{t}\left( L/K\right) \left( \Delta ,A\right) \subseteq D_{L/K}\left(
\Delta ,A\right) \subseteq Aut\left( L/K\right) .
\end{equation*}

It follows that for their invariant subfields in $L/K$ we have
\begin{equation*}
L^{\pi _{a}\left( L/K\right) \left( \Delta ,A\right) }\supseteq
L^{D_{L/K}\left( \Delta ,A\right) }\supseteq L^{Aut\left( L/K\right)
}\supseteq K.
\end{equation*}

Particularly,
\begin{equation*}
\begin{array}{l}
L^{\pi _{a}\left( L/K\right) \left( \Delta ,A\right) \cup \pi _{t}\left(
L/K\right) \left( \Delta ,A\right) } \\
\supseteq L^{D_{L/K}\left( \Delta ,A\right) \cup \pi _{t}\left( L/K\right)
\left( \Delta ,A\right) }=L^{D_{L/K}\left( \Delta ,A\right) } \\
\supseteq L^{Aut\left( L/K\right) \cup \pi _{t}\left( L/K\right) \left(
\Delta ,A\right) }=L^{Aut\left( L/K\right) } \\
\supseteq K.%
\end{array}%
\end{equation*}%
\emph{Step 2}. Put
\begin{equation*}
S=\pi _{a}\left( L/K\right) \left( \Delta ,A\right) ;
\end{equation*}
\begin{equation*}
T=\pi _{t}\left( L/K\right) \left( \Delta ,A\right) .
\end{equation*}%
Then
\begin{equation*}
\begin{array}{l}
L^{S\cup T} \\
=\{x\in L:\sigma \left( x\right) =x,\forall \sigma \in S\cup T\} \\
=\{x\in L^{S}:\sigma \left( x\right) =x,\forall \sigma \in T\} \\
=\left( L^{S}\right) ^{T} \\
=\left( L^{T}\right) ^{S} \\
=\{x\in L:\sigma \left( x\right) =x,\forall \sigma \in S\}\cap \{x\in
L:\sigma \left( x\right) =x,\forall \sigma \in T\} \\
=L^{S}\cap L^{T}.%
\end{array}%
\end{equation*}%
\emph{Step 3}. Since $L$ is $\sigma $-tame Galois over $K$, we have
\begin{equation*}
L^{S}=L^{Aut\left( L/K\left( \Delta \right) \right) }=K\left( \Delta \right)
.
\end{equation*}

From \emph{Theorem 3.4} it is seen that $K\left( \Delta \right) $ is Galois
over $K$; it follows that
\begin{equation*}
K\left( \Delta \right) ^{Aut\left( K\left( \Delta \right) /K\right) }=K
\end{equation*}%
holds. On the other hand, we have
\begin{equation*}
Aut\left( K\left( \Delta \right) /K\right) =\{\sigma |_{K\left( \Delta
\right) }:\sigma \in T\}.
\end{equation*}%
Then
\begin{equation*}
\begin{array}{l}
\left( L^{S}\right) ^{T} \\
=K\left( \Delta \right) ^{T} \\
=\{x\in K\left( \Delta \right) :\sigma \left( x\right) =x,\forall \sigma \in
T\} \\
=\{x\in K\left( \Delta \right) :\sigma \left( x\right) =x,\forall \sigma \in
Aut\left( K\left( \Delta \right) /K\right) \} \\
=K\left( \Delta \right) ^{Aut\left( K\left( \Delta \right) /K\right) } \\
=K.%
\end{array}%
\end{equation*}%
Hence,
\begin{equation*}
K=L^{S\cup T}\supseteq L^{D_{L/K}\left( \Delta ,A\right) }\supseteq
L^{Aut\left( L/K\right) }\supseteq K.
\end{equation*}%
This completes the proof.
\end{proof}

\subsection{Galois correspondence for $\sigma $-tame Galois, (III)}

From the above two theorems we have the following lemma on Galois
correspondence.

\begin{lemma}
\emph{(Galois correspondence for $\sigma $-tame Galois, III)} Let $L$ be an
extension over a field $K$ of finite transcendence degree. Suppose $L$ is $%
\sigma $-tame Galois over $K$ (respectively, inside $\Omega _{0}$). Here, $%
\Omega _{0}$ is a given subset of some certain transcendence bases of $L/K$.

Then there is a bijection between the set $\mathfrak{N}\left( L/K\right) $
of $\pi $-equivalence classes of nice bases $\left( \Delta ,A\right) $ of $%
L/K$ (respectively, with $\Delta \in \Omega _{0}$) and the set $\mathfrak{D}%
\left( L/K\right) $ of decomposition groups $D_{L/K}\left( \Delta ,A\right) $
of $L/K$ at nice bases $\left( \Delta ,A\right) $ (respectively, with $%
\Delta \in \Omega _{0}$) given by
\begin{equation*}
\left[ \Delta ,A\right] _{\pi }\in \mathfrak{N}\left( L/K\right) \longmapsto
D_{L/K}\left( \Delta ,A\right) \in \mathfrak{D}\left( L/K\right)
\end{equation*}%
such that
\begin{equation*}
D_{L/K}\left( \Delta ,A\right) =\pi _{a}\left( L/K\right) \left( \Delta
,A\right) \cdot \pi _{t}\left( L/K\right) \left( \Delta ,A\right) .
\end{equation*}

Here, $\pi _{a}\left( L/K\right) \left( \Delta ,A\right) $ and $\pi
_{t}\left( L/K\right) \left( \Delta ,A\right) $ denote the full algebraic
and highest transcendental Galois subgroups of $L/K$ at a nice basis $\left(
\Delta ,A\right) $, respectively.
\end{lemma}

\begin{proof}
Immediately from \emph{Lemmas 4.12-13} and \emph{Theorems 5.4-5}.
\end{proof}

\section{Noether Solutions and Transcendental Galois Theory}

In this section we will discuss the properties on the interplay between
solutions of Noether solutions and transcendental Galois theory. The results
obtained here will be applied to prove two key lemmas in \emph{\S 7} and one
main theorem in \emph{\S 9}.

\subsection{Galois correspondence and Noether solutions}

For a transcendental extension $L$ over a field $K$, recall that (See \emph{%
Definition 1.4}) a subgroup $G$ of the Galois group $Aut\left( L/K\right) $
is said to be a \textbf{Noether solution} of the extension $L/K$ if the two
conditions are satisfied:

$\left( i\right) $ $L$ is an algebraic extension over the $G$-invariant
subfield $L^{G}$ of $L$;

$\left( ii\right) $ The $G$-invariant subfield $L^{G}$ is a purely
transcendental extension over $K$.

\begin{theorem}
\emph{(Galois correspondence and solutions of Noether's problem)} Let $L$ be
an extension over a field $K$ of finite transcendence degree. Suppose $L$ is
$\sigma $-tame Galois over $K$ (respectively, inside $\Omega _{0}$). Here, $%
\Omega _{0}$ is a given subset of some certain transcendence bases of $L/K$.

$\left( i\right) $ There is a bijection between the subset $\mathfrak{S}%
\left( L/K\right) $ of Noether solutions $G$ of $L/K$ (respectively, with $G$
being a $\Omega _{0}$-subgroup of $Aut\left( L/K\right) $) and the subset $%
\mathfrak{N}\left( L/K\right) $ of $\pi $-equivalence classes of nice bases $%
\left( \Delta ,A\right) $ of $L/K$ (respectively, with $\Delta \in \Omega
_{0}$) given by
\begin{equation*}
G\in \mathfrak{S}\left( L/K\right) \longmapsto \left[ \Delta ,A\right] _{\pi
}\in \mathfrak{N}\left( L/K\right)
\end{equation*}%
such that $G=\pi _{a}\left( L/K\right) \left( \Delta ,A\right) $ holds.

$\left( ii\right) $ There is a bijection between the subset $\mathfrak{S}%
\left( L/K\right) $ of Noether solutions $G$ of $L/K$ (respectively, with $G$
being a $\Omega _{0}$-subgroup of $Aut\left( L/K\right) $) and the subset $%
\mathfrak{D}\left( L/K\right) $ of the decomposition groups $D_{L/K}\left(
\Delta ,A\right) $ of $L/K$ at nice bases $\left( \Delta ,A\right) $
(respectively, with $\Delta \in \Omega _{0}$) given by
\begin{equation*}
G\in \mathfrak{S}\left( L/K\right) \longmapsto D_{L/K}\left( \Delta
,A\right) \in \mathfrak{D}\left( L/K\right)
\end{equation*}%
such that there is a factorisation for the decomposition groups
\begin{equation*}
D_{L/K}\left( \Delta ,A\right) =G\cdot \pi _{t}\left( L/K\right) \left(
\Delta ,A\right) .
\end{equation*}
\end{theorem}

\begin{proof}
Without loss of generality, suppose $\Omega _{0}$ is the set of all the
transcendence bases of the extension $L/K$.

$\left( i\right) $ Let $G\subseteq Aut\left( L/K\right) $ be a Noether
solution of the extension $L/K$. As $L$ is $\sigma $-tame Galois over $K$,
we have $G=Aut\left( L^{G}/K\right) $ and $L^{G}$ is purely transcendental
over $K$. It follows that we obtain a nice basis $\left( \Delta ,A\right) $
of $L/K$ and then $\left[ \Delta ,A\right] _{\pi }\in \mathfrak{N}\left(
L/K\right) $. Here, $A$ is a linear basis of the linear space $L$ over the
subfield $L^{G}$ and $\Delta $ is a transcendence base of the purely
transcendental extension $L^{G}/K$. Hence, there is a map
\begin{equation*}
\tau :G\in \mathfrak{S}\left( L/K\right) \longmapsto \left[ \Delta ,A\right]
\in \mathfrak{N}\left( L/K\right) .
\end{equation*}

It is seen that the map $\tau $ is well-defined from the Galois
correspondence for $\sigma $-tame Galois.

Conversely, each $\left[ \Delta ,A\right] _{\pi }\in \mathfrak{N}\left(
L/K\right) $ determines a unique subgroup
\begin{equation*}
G=Aut\left( L/K\left( \Delta \right) \right) =\pi _{a}\left( L/K\right)
\left( \Delta ,A\right) .
\end{equation*}

Hence, $\tau $ is a bijection between the set $\mathfrak{S}\left( L/K\right)
$ of Noether solutions of $L/K$ and the set $\mathfrak{D}\left( L/K\right) $
of decomposition groups of $L/K$.

$\left( ii\right) $ Immediately from the above $\left( i\right) $ and \emph{Lemma 5.6}.
\end{proof}

\subsection{Conjugacy classes of decomposition groups at nice bases}

Let $L/K$ be a transcendental extension and let $\mathfrak{N}\left(
L/K\right) $ denote the set of $\pi $-equivalence classes $\left[ \Delta ,A%
\right] _{\pi }$ of nice bases $\left( \Delta ,A\right) $.

\begin{definition}
Two $\left[ \Delta _{1},A_{1}\right] _{\pi }$ and $\left[ \Delta _{2},A_{2}%
\right] _{\pi }$ in $\mathfrak{N}\left( L/K\right) $ are $G$\textbf{%
-equivalent}, denoted by
\begin{equation*}
\left[ \Delta _{1},A_{1}\right] _{\pi }\thicksim _{G}\left[ \Delta _{2},A_{2}%
\right] _{\pi },
\end{equation*}%
if there is some $g\in Aut\left( L/K\right) $ such that
\begin{equation*}
\Delta _{2}=g\left( \Delta _{1}\right) \text{ and }A_{2}=g\left(
A_{1}\right) .
\end{equation*}

Denote by
\begin{equation*}
Aut\left( L/K\right) \diagdown \mathfrak{N}\left( L/K\right)
\end{equation*}%
the quotient space of the set $\mathfrak{N}\left( L/K\right) $ by the $G$%
-equivalence relation $\thicksim _{G}$.
\end{definition}

\begin{proposition}
(\emph{Conjugacy classes of decomposition groups}) Let $L$ be an extension
over a field $K$ of finite transcendence degree. Suppose $L$ is $\sigma $%
-tame Galois over $K$ (respectively, inside $\Omega _{0}$). Here, $\Omega
_{0}$ is a given subset of some certain transcendence bases of $L/K$.

$\left( i\right) $ There is a natural left action of the Galois group $%
Aut\left( L/K\right) $ on the set $\mathfrak{N}\left( L/K\right) $ of $\pi $%
-equivalence classes of nice bases $\left[ \Delta ,A\right] $ of $L/K$
(respectively, with $\Delta \in \Omega _{0}$) given by
\begin{equation*}
g\cdot \left[ \Delta ,A\right] _{\pi }=\left[ g\left( \Delta \right)
,g\left( A\right) \right] _{\pi }
\end{equation*}%
for any $g\in Aut\left( L/K\right) $ and $\left[ \Delta ,A\right] _{\pi }\in
\mathfrak{N}\left( L/K\right) $ (respectively, with $\Delta \in \Omega _{0}$%
).

$\left( ii\right) $ Let $\left[ \Delta ,A\right] _{\pi },\left[ \Lambda ,B%
\right] _{\pi }\in \mathfrak{N}\left( L/K\right) $ (respectively, with $%
\Delta ,\Lambda \in \Omega _{0}$). If there is some $g\in Aut\left(
L/K\right) $ such that $\Lambda =g\left( \Delta \right) $ and $B=g\left(
A\right) $, then we have
\begin{equation*}
D_{L/K}\left( \Lambda ,B\right) =g\cdot D_{L/K}\left( \Delta ,A\right) \cdot
g^{-1}
\end{equation*}%
for the decomposition groups.

$\left( iii\right) $ For any $\left[ \Delta ,A\right] _{\pi }\in \mathfrak{N}%
\left( L/K\right) $ (respectively, with $\Delta \in \Omega _{0}$), the
decomposition group $D_{L/K}\left( \Delta ,A\right) $ of $L/K$ at $\left(
\Delta ,A\right) $ is the stable subgroup of the element $\left[ \Delta ,A%
\right] _{\pi }$ under the left action of $Aut\left( L/K\right) $ on $%
\mathfrak{N}\left( L/K\right) $.

$\left( iv\right) $ There is a bijection between the $Aut\left( L/K\right) $%
-orbit of $\left[ \Delta ,A\right] _{\pi }$ (respectively, with $\Delta \in
\Omega _{0}$) and the set of conjugacy classes of decomposition groups $%
D_{L/K}\left( \Delta ,A\right) $ (respectively, with $\Delta \in \Omega _{0}$%
) in the Galois group $Aut\left( L/K\right) .$
\end{proposition}

\begin{proof}
For sake of convenience, suppose $\Omega _{0}$ is the set of all the
transcendence bases of the extension $L/K$.

$\left( i\right) $ Let $g\in Aut\left( L/K\right) $ and $\left[ \Delta ,A%
\right] _{\pi }\in \mathfrak{N}\left( L/K\right) $. It is seen that $g\left(
\Delta \right) $ is still an algebraically independent set over $K$ and $%
g\left( A\right) $ is a linearly independent set over $K\left( g\left(
\Delta \right) \right) $. As $L=K\left( \Delta \right) \left[ A\right] $, we
have
\begin{equation*}
L=g\left( L\right) =K\left( g\left( \Delta \right) \right) \left[ g\left(
A\right) \right] .
\end{equation*}

It follows that $\left( g\left( \Delta \right) ,g\left( A\right) \right) $
is a nice basis of the extension $L/K$ and hence
\begin{equation*}
\left[ g\left( \Delta \right) ,g\left( A\right) \right] _{\pi }\in \mathfrak{%
N}\left( L/K\right) .
\end{equation*}

$\left( ii\right) $ Immediately from the left action of the Galois group $%
Aut\left( L/K\right) $ on the set $\mathfrak{N}\left( L/K\right) $.

$\left( iii\right) $ Trivial from definition for decomposition group.

$\left( iv\right) $ Fixed a $\left[ \Delta ,A\right] _{\pi }\in \mathfrak{N}%
\left( L/K\right) $. From $\left( i\right) -\left( ii\right) $ we have
\begin{equation*}
\left[ g\left( \Delta \right) ,g\left( A\right) \right] _{\pi }\in \mathfrak{%
N}\left( L/K\right)
\end{equation*}%
and
\begin{equation*}
D_{L/K}\left( g\left( \Delta \right) ,g\left( A\right) \right) =g\cdot
D_{L/K}\left( \Delta ,A\right) \cdot g^{-1}
\end{equation*}%
for any $g\in Aut\left( L/K\right) $.

Then the assignment
\begin{equation*}
g\in Aut\left( L/K\right) \longmapsto g\cdot D_{L/K}\left( \Delta ,A\right)
\cdot g^{-1}
\end{equation*}%
defines a map $\tau $ from the $Aut\left( L/K\right) $-orbit of $\left[
\Delta ,A\right] _{\pi }$ to the set of conjugacy classes of $D_{L/K}\left(
\Delta ,A\right) $ in $Aut\left( L/K\right) .$ It is seen that the map $\tau
$ is bijective from \emph{Lemma 5.6} that there is a bijection between the
set $\mathfrak{N}\left( L/K\right) $ and the set $\mathfrak{D}\left(
L/K\right) $ of decomposition groups of $L/K$ at nice bases. This competes
the proof.
\end{proof}

\subsection{Properties for Noether solutions}

Let $L\ $be an arbitrary transcendental extension over a field $K$ of finite
transcendence degree. Suppose $G$ is a subgroup of $Aut\left( L/K\right) $
such that $L$ is algebraic over the invariant subfield $L^{G}$ of $L$ under
the subgroup $G$.

\begin{theorem}
\emph{(Properties for Noether solutions)} If $G$ is a Noether solution of $%
L/K$, then there is a unique infinite subgroup $H$ of $Aut\left( L/K\right) $
satisfying the three properties:

$\left( N1\right) $ There is $G\cap H=\{1\}$ and $\sigma \cdot \delta
=\delta \cdot \sigma $ holds in $Aut\left( L/K\right) $ for any $\sigma \in
G $ and $\delta \in H$.

$\left( N2\right) $ $K$ is the invariant subfield $L^{\left\langle G\cup
H\right\rangle }$ of $L$ under the subgroup $\left\langle G\cup
H\right\rangle $ generated in $Aut\left( L/K\right) $ by the subset $G\cup H$%
.

$\left( N3\right) $ The set of restrictions $\sigma |_{L^{G}}$ of $\sigma $
contained in $H$ is the Galois group $Aut\left( L^{G}/K\right) $ of the
subextension $L^{G}/K$.

In such a case, $G=\pi _{a}\left( L/K\right) \left( \Delta ,A\right) $, $%
H=\pi _{t}\left( L/K\right) \left( \Delta ,A\right) $ and $\left\langle
G\cup H\right\rangle =D_{L/K}\left( \Delta ,A\right) $ hold and there is a
direct decomposition $\left\langle G,H\right\rangle =GH=HG$ such that $G$
and $H$ are normal subgroups of $\left\langle G,H\right\rangle $, where $%
\left( \Delta ,A\right) $ is any given nice basis of $L/K$ with $K\left(
\Delta \right) =L^{G}$.
\end{theorem}

\begin{proof}
Assume $G$ is a Noether solution of the extension $L/K$, i.e., $L$ is
algebraic Galois over the invariant subfield $L^{G}$ and $L^{G}$ is purely
transcendental over $K$. Fixed any nice basis $\left( \Delta ,A\right) $ of $%
L/K$ with $L^{G}=K\left( \Delta \right) $ and $L=L^{G}\left[ A\right] $.
From \emph{Proposition 4.5} it is seen that $G$ is the full algebraic Galois
subgroup $\pi _{a}\left( L/K\right) \left( \Delta ,A\right) $ of $L/K$ at $%
\left( \Delta ,A\right) $. That is, $G=\pi _{a}\left( L/K\right) \left(
\Delta ,A\right) =Aut\left( L/K\left( \Delta \right) \right) $.

Let $\Omega _{0}$ be the set of all the transcendence bases $\Delta $ of $%
L^{G}/K$ with $L^{G}=K\left( \Delta \right) $. Then $L$ is $\sigma $-tame
Galois over $K$ inside $\Omega _{0}$. Suppose $H$ is the highest
transcendental Galois subgroup $\pi _{t}\left( L/K\right) \left( \Delta
,A\right) $ of $L/K$ at $\left( \Delta ,A\right) $, i.e., $H=\pi _{t}\left(
L/K\right) \left( \Delta ,A\right) \cong Aut\left( L/K\left( \Delta \right)
\right) $. It is seen that $\pi _{t}\left( L/K\right) \left( \Delta
,A\right) $ satisfies the properties $\left( N1\right) -\left( N3\right) $
by \emph{Propositions 4.6-7} and \emph{Theorem 6.1}. This gives the
existence of $H$.

On the other hand, the uniqueness of $H$ is obtained from \emph{Lemmas 4.13}
and \emph{5.6}. This completes the proof.
\end{proof}

\subsection{Galois invariance of Noether numbers}

Now we introduce a quantity to account for the solutions of Noether's
problem on rationality. Let $L$ be a transcendental extension over a field $%
K $.

A positive integer $a\in \mathbb{Z}$ is called a \textbf{Noether number} of
the extension $L/K$ if there a finite subgroup $G$ of the Galois group $%
Aut\left( L/K\right) $ such that $G$ is a Noether solution of $L/K$ and $%
a=\sharp G$ is the order of $G$.

We denote by $\mathfrak{A}\left( L/K\right) $ the set of all Noether numbers
of the extension $L/K$.

\begin{proposition}
(\emph{Invariance of Noether numbers under Galois action}) Let $L$ be a
finitely generated transcendental extension over a field $K$. Suppose $L$ is
$\sigma $-tame Galois over $K$ (respectively, inside $\Omega _{0}$). Here, $%
\Omega _{0}$ is a given subset of some certain transcendence bases of $L/K$.

$\left( i\right) $ Let $\left( \Delta ,A\right) $ be a nice basis of $L/K$
(respectively, with $\Delta \in \Omega _{0}$). Set
\begin{equation*}
a_{L/K}\left( \Delta ,A\right) =\sharp A
\end{equation*}%
to be the number of elements in $A$. Then
\begin{equation*}
a_{L/K}\left( \Delta ,A\right) =\sharp \pi _{a}\left( L/K\right) \left(
\Delta ,A\right)
\end{equation*}%
is the order of the full algebraic Galois subgroup $\pi _{a}\left(
L/K\right) \left( \Delta ,A\right) $ of $L/K$ at $\left( \Delta ,A\right) $.

In particular, $a_{L/K}\left( \Delta ,A\right) $ is a Noether number of $L/K$%
, called the \textbf{Noether number} of $L/K$ \textbf{at a nice basis} $%
\left( \Delta ,A\right) $.

$\left( ii\right) $ For any $g\in Aut\left( L/K\right) $ and $\left[ \Delta
,A\right] _{\pi }\in \mathfrak{N}\left( L/K\right) $ (respectively, with $%
\Delta $ and $g\left( \Delta \right) $ being contained in $\Omega _{0}$),
there is an equality
\begin{equation*}
a_{L/K}\left( \Delta ,A\right) =a_{L/K}\left( g\left( \Delta \right)
,g\left( A\right) \right) .
\end{equation*}
\end{proposition}

\begin{proof}
Immediately from \emph{Proposition 6.3}.
\end{proof}

\begin{proposition}
Let $L$ be a finitely generated transcendental extension over a field $K$.
Suppose $L$ is $\sigma $-tame Galois over $K$ (respectively, inside $\Omega
_{0}$). Here, $\Omega _{0}$ is an $Aut\left( L/G\right) $\emph{-invariant
subset} of some certain transcendence bases of $L/K$, i.e., $\sigma \left(
\Delta \right) \in \Omega _{0}$ holds for any $\Delta \in \Omega _{0}$ and $%
\sigma \in Aut\left( L/K\right) $.

$\left( i\right) $ The Noether number $a_{L/K}\left( \Delta ,A\right) $
defines a function $a_{L/K}$ from $\mathfrak{N}\left( L/K\right) $
(respectively, the restriction of $a_{L/K}$ to the subset $\mathfrak{N}%
\left( L/K\right) |_{\Omega _{0}}$ of $\pi $-equivalence classes of nice
bases $\left( \Delta ,A\right) $ of $L/K$ with $\Delta \in \Omega _{0}$)
into $\mathbb{N}$, given by
\begin{equation*}
\left[ \Delta ,A\right] _{\pi }\longmapsto a_{L/K}\left( \Delta ,A\right) .
\end{equation*}

$\left( ii\right) $ For each $\left[ \Delta ,A\right] _{\pi }$ in $\mathfrak{%
N}\left( L/K\right) $ (respectively, $\mathfrak{N}\left( L/K\right)
|_{\Omega _{0}}$), the Noether number $a_{L/K}$ is constant on the $%
Aut\left( L/K\right) $-orbit of $\left[ \Delta ,A\right] _{\pi }$ in $%
Aut\left( L/K\right) \diagdown \mathfrak{N}\left( L/K\right) $
(respectively, $Aut\left( L/K\right) \diagdown \mathfrak{N}\left( L/K\right)
|_{\Omega _{0}}$).

$\left( iii\right) $ For any $\left[ \Delta ,A\right] _{\pi }\in \mathfrak{N}%
\left( L/K\right) $, set
\begin{equation*}
\mathfrak{A}\left( L/K\right) \left( \Delta ,A\right)
\end{equation*}%
to be the set of values of the function $a_{L/K}$ on the $Aut\left(
L/K\right) $-orbits of $\left[ \Delta ,A\right] _{\pi }$ in $Aut\left(
L/K\right) \diagdown \mathfrak{N}\left( L/K\right) $.

Then
\begin{equation*}
\begin{array}{l}
\mathfrak{A}\left( L/K\right) \\
=\bigcup\limits_{\left[ \Delta ,A\right] _{\pi }\in \mathfrak{N}\left(
L/K\right) }\mathfrak{A}\left( L/K\right) \left( \Delta ,A\right) \\
=\bigcup\limits_{\left[ \Delta ,A\right] _{\pi }\in \mathfrak{N}\left(
L/K\right) }\{a_{L/K}\left( \Delta ,A\right) \}.%
\end{array}%
\end{equation*}
\end{proposition}

\begin{proof}
$\left( i\right) $ Trivial.

$\left( ii\right) $ Let $\left[ \Delta ,A\right] _{\pi }\in \mathfrak{N}%
\left( L/K\right) $ and $a_{L/K}\left( \Delta ,A\right) =n_{0}$. Then
\begin{equation*}
a_{L/K}\left( \Lambda ,B\right) =n_{0}
\end{equation*}%
holds for any $\left[ \Lambda ,B\right] _{\pi }\in \mathfrak{N}\left(
L/K\right) $ such that
\begin{equation*}
\left[ \Delta ,A\right] _{\pi }\thicksim _{G}\left[ \Lambda ,B\right] _{\pi
}.
\end{equation*}%
Hence, $a_{L/K}$ is a constant function on each $Aut\left( L/K\right) $%
-orbit.

$\left( iii\right) $ Immediately from $\left( ii\right) $.
\end{proof}

\subsection{Distribution of Noether numbers}

For the property on distribution of Noether numbers, there are the following
results such as the following.

\begin{lemma}
Let $L$ be an extension over a field $K$ of finite transcendence degree.
Suppose $L$ is $\sigma $-tame Galois over $K$ (respectively, inside $\Omega
_{0}$). Here, $\Omega _{0}$ is a given subset of some certain transcendence
bases of $L/K$. Then each finite subgroup $H$ of the Galois group $Aut\left(
L/K\right) $ is contained in a full algebraic Galois subgroup $G$ of the
extension $L/K$ (respectively, with $G$ being a $\Omega _{0}$-subgroup of $%
Aut\left( L/K\right) $).
\end{lemma}

\begin{proof}
Let $H\subseteq Aut\left( L/K\right) $ be a finite subgroup. Let $M=L^{H}$
be the $H$-invariant subfield of $L$. Immediately $L$ is finite and Galois
over $M$ and $M$ is transcendental over $K$.

For the intermediate subfield $M$, there are two cases.

$Case$ $\left( i\right) $. If $M$ is purely transcendental over $K$, then $H$
is a full algebraic Galois subgroups of $L/K$.

$Case$ $\left( ii\right) $. Suppose $M$ is not a purely transcendental
extension over $K$. By taking a transcendence base $\Lambda $ of the
extension $M/K$ and taking a linear basis $B$ of the linear space $L$ over
the field $K\left( \Lambda \right) $, we obtain a nice basis $\left( \Lambda
,B\right) $ of the extension $L/K$. Consider a tower of extensions of fields
\begin{equation*}
K\subseteq K\left( \Lambda \right) \subseteq M\subseteq L.
\end{equation*}

From the assumption that $L$ is $\sigma $-tame Galois over $K$, it is seen
that $L$ is algebraic and Galois over $K\left( \Lambda \right) $. It follows
that
\begin{equation*}
\pi _{a}\left( L/K\right) \left( \Lambda ,B\right) =Aut\left( L/K\left(
\Lambda \right) \right)
\end{equation*}
is a full algebraic Galois subgroup of $L/K$.

Evidently, $H=Aut\left( L/M\right) $ is a subgroup of $\pi _{a}\left(
L/K\right) \left( \Lambda ,B\right) $. This completes the proof.
\end{proof}

\begin{lemma}
Let $L$ be an extension over a field $K$ of finite transcendence degree.
Suppose $L$ is $\sigma $-tame Galois over $K$ (respectively, of vertical
type). Then for any full algebraic Galois subgroup $P$ of the extension $L/K$
(respectively, of vertical type), there is a full algebraic Galois subgroup $%
G$ of $L/K$ (respectively, of vertical type) such that $P\subsetneqq G$ is a
proper subgroup of $G$.
\end{lemma}

\begin{proof}
Let $P=\pi _{a}\left( L/K\right) \left( \Delta ,A\right) $ be a full
algebraic Galois subgroup of the extension $L/K$ at a nice basis $\left(
\Delta ,A\right) $.

Set
\begin{equation*}
\Delta =\{t_{1},t_{2},\cdots ,t_{m}\}
\end{equation*}
where
\begin{equation*}
m=tr.\deg L/K.
\end{equation*}

Consider the field
\begin{equation*}
M=K\left( t_{1},t_{2},\cdots ,t_{m-1},t_{m}^{n_{0}}\right)
\end{equation*}
for a given integer $n_{0}$ with $2\leq n_{0}\in \mathbb{N}$.

Then
\begin{equation*}
K\subseteq M\subsetneqq K\left( \Delta \right)
\end{equation*}
and $K\left( \Delta \right) $ is a finite extension over $M$.

On the other hand, let $H=Aut\left( L/M\right) $ be the Galois group of the
extension $L/M$. It is seen that $P\subsetneqq H$ is a proper subgroup of $H$%
.

By applying the procedure in the proof of above \emph{Lemma 6.6} to the case
of the subgroup $H\subseteq Aut\left( L/K\right) $ here, we obtain a full
algebraic Galois subgroup
\begin{equation*}
\pi _{a}\left( L/K\right) \left( \Lambda ,B\right)
\end{equation*}%
of $L/K$ at a nice basis $\left( \Lambda ,B\right) $ such that
\begin{equation*}
H\subseteq \pi _{a}\left( L/K\right) \left( \Lambda ,B\right) .
\end{equation*}%
Here, $K\left( \Lambda \right) \subseteq M$. Hence, we have
\begin{equation*}
P\subsetneqq H\subseteq G=\pi _{a}\left( L/K\right) \left( \Lambda ,B\right)
.
\end{equation*}%
This completes the proof.
\end{proof}

\begin{lemma}
Let $L$ be an extension over a field $K$ of finite transcendence degree and $%
L$ $\sigma $-tame Galois over $K$. Suppose $P$ and $Q$ are two full
algebraic Galois subgroups of the extension $L/K$ satisfying either of the
conditions:

$\left( i\right) $ $P$ or $Q$ is a finite group;

$\left( ii\right) $ $L$ is algebraic over the invariant subfield $%
L^{\left\langle P,Q\right\rangle }$ under $\left\langle P,Q\right\rangle $.
Here $\left\langle P,Q\right\rangle $ denotes the subgroup generated in $%
Aut\left( L/K\right) $ by $P$ and $Q$.

Then there is a full algebraic Galois subgroup $G$ of $L/K$ such that $%
P\subseteq G$ and $Q\subseteq G$ are subgroups of $G$.
\end{lemma}

\begin{proof}
Let $P=\pi _{a}\left( L/K\right) \left( \Delta _{1},A_{1}\right) $ and $%
Q=\pi _{a}\left( L/K\right) \left( \Delta _{2},A_{2}\right) $ be two full
algebraic Galois subgroups of the extension $L/K$ at nice bases $\left(
\Delta _{1},A_{1}\right) $ and $\left( \Delta _{2},A_{2}\right) $ of $L/K$,
respectively. We will proceed in several steps.

\emph{Step 1}. Assume that the condition $\left( i\right) $ is satisfied.
Without loss of generality, let the subgroup $Q$ is finite. We have
\begin{equation*}
\begin{array}{l}
L^{P\cup Q} \\
=\{x\in L:\sigma \left( x\right) =x\text{ for }\forall \sigma \in P\cup Q\}
\\
=\{x\in L^{P}:\sigma \left( x\right) =x\text{ for }\forall \sigma \in Q\} \\
=\left( L^{P}\right) ^{Q} \\
=K\left( \Delta _{1}\right) ^{Q}%
\end{array}%
\end{equation*}
since
\begin{equation*}
L^{P}=K\left( \Delta _{1}\right) .
\end{equation*}

As $Q$ is a finite group, it is seen that $K\left( \Delta _{1}\right) $ is a
finite extension over the subfield
\begin{equation*}
M:=K\left( \Delta _{1}\right) ^{Q}.
\end{equation*}%
Hence, $L$ is algebraic over the subfield $M$.

Let $\Lambda $ be a transcendence base of $L/K$ with $K\left( \Lambda
\right) \subseteq M$. From the assumption that $L$ is $\sigma $-tame Galois
over $K$, there is a full algebraic Galois subgroup
\begin{equation*}
G=\pi _{a}\left( L/K\right) \left( \Lambda ,B\right)
\end{equation*}%
of $L/K$ at a nice basis $\left( \Lambda ,B\right) $ such that $P$ and $Q$
are contained in $G$.

\emph{Step 2}. Suppose that the condition $\left( ii\right) $ is satisfied.
As
\begin{equation*}
P\subseteq \left\langle P,Q\right\rangle \text{ and } Q\subseteq
\left\langle P,Q\right\rangle ,
\end{equation*}
for the invariant subfields we have
\begin{equation*}
K\left( \Delta _{1}\right) =L^{P}\supseteq L^{\left\langle P,Q\right\rangle }
\end{equation*}%
and
\begin{equation*}
K\left( \Delta _{2}\right) =L^{Q}\supseteq L^{\left\langle P,Q\right\rangle }
\end{equation*}%
since $L$ is $\sigma $-tame Galois over $K$.

Then
\begin{equation*}
M:=K\left( \Delta _{1}\right) \cap K\left( \Delta _{2}\right) \supseteq
L^{\left\langle P,Q\right\rangle }.
\end{equation*}

From the assumption that $L$ is algebraic over $L^{\left\langle
P,Q\right\rangle }$, it is seen that $L$ is algebraic over the subfield $M$.
Let $\Lambda $ be a transcendence base of $L/K$ with $K\left( \Lambda
\right) \subseteq M$. We have a full algebraic Galois subgroup
\begin{equation*}
G=\pi _{a}\left( L/K\right) \left( \Lambda ,B\right)
\end{equation*}%
of $L/K$ at a nice basis $\left( \Lambda ,B\right) $ such that $P$ and $Q$
are contained in $G$.
\end{proof}

\begin{theorem}
\emph{(Distribution of Noether numbers) }Let $L$ be a finitely generated
transcendental extension over a field $K$. Suppose $L$ is $\sigma $-tame
Galois over $K$ (respectively, of vertical type). For Noether numbers of $%
L/K $, there are the following statements.

$\left( i\right) $ For any positive integer $m\in \mathbb{N}$ there is a
Noether number $\omega $ of $L/K$ such that $m$ divides $\omega $.

$\left( ii\right) $ For any Noether number $\alpha $ of $L/K$, there is
another Noether number $\beta $ of $L/K$ such that $\alpha <\beta $ and $%
\alpha $ divides $\beta $.

$\left( iii\right) $ For any two Noether numbers $\alpha $ and $\beta $ of $%
L/K$, there is another Noether number $\gamma $ of $L/K$ such that $\alpha $
and $\beta $ divide $\gamma $, respectively.

$\left( iv\right) $ The set $\mathfrak{A}\left( L/K\right) $ of Noether
numbers of $L/K$ has no upper bound.
\end{theorem}

\begin{proof}
Let $\Omega _{0}$ be the subset of vertical transcendence bases of the
extension $L/K$. It is seen that $\Omega _{0}$ is invariant under the Galois
group $Aut\left( L/K\right) $, i.e., $\sigma \left( \Omega _{0}\right)
\subseteq \Omega _{0}$ holds for each $\sigma \in Aut\left( L/K\right) $.

Each full algebraic Galois subgroup of the extension $L/K$ is a finite group
since $L$ is finitely generated over $K$. Hence,

$\left( i\right) -\left( iii\right) $ Immediately from \emph{Lemmas 6.6-9}.

$\left( iv\right) $ Immediately from $\left( ii\right) $.
\end{proof}

\subsection{Galois-complemented groups as an essential condition to Noether
solutions}

In this subsection we will undertake a further discussion on the two
conditions $\left( N1\right) -\left( N2\right) $.

Let $L$ be a purely transcendental extension over a field $K$ of finite
transcendence degree. Suppose $G$ is a subgroup of $Aut\left( L/K\right) $
such that $L$ is algebraic over the $G$-invariant subfield $L^{G}$, i.e., $G$
is a finite subgroup of $Aut\left( L/K\right) $. Fixed a subgroup $H$ of $%
Aut\left( L/K\right) $. In the following, consider the two conditions $%
\left( N1\right) -\left( N2\right) $:

$\ \left( N1\right) $ There is $G\cap H=\{1\}$ and $\sigma \cdot \delta
=\delta \cdot \sigma $ holds in $Aut\left( L/K\right) $ for any $\sigma \in
G $ and $\delta \in H$.

$\ \left( N2\right) $ $K$ is the invariant subfield $L^{\left\langle G\cup
H\right\rangle }$ of $L$ under the subgroup $\left\langle G\cup
H\right\rangle $ generated in $Aut\left( L/K\right) $ by the subset $G\cup H$%
.

\begin{proposition}
Assume the two conditions $\left( N1\right) -\left( N2\right) $ are
satisfied. Let $D=\left\langle G\cup H\right\rangle $ be the subgroup
generated in $Aut\left( L/K\right) $ by the union $G\cup H$. Then there are
the following statements.

$\left( i\right) $ $D=G\cdot H$ holds. In particular, $G$ and $H$ are normal
subgroups in $D$.

$\left( ii\right) $ $L^{D}=\left( L^{G}\right) ^{H}=\left( L^{H}\right) ^{G}$%
; $L^{D}=L^{G\cup H}=L^{G}\cap L^{H}.$

$\left( iii\right) $ $L=\left( L\setminus L^{G}\right) \cup \left(
L\setminus L^{H}\right) \cup K$.

$\left( iv\right) $ $L^{G}\cap L^{H}=K$; $\left( L^{G}\setminus K\right)
\cap \left( L^{H}\setminus K\right) =\emptyset $. In particular, there is a
union of disjoint subsets
\begin{equation*}
L=\left( L\setminus L^{G}\right) \cup \left( L\setminus L^{H}\right) \cup K.
\end{equation*}

$\left( v\right) $ $L^{G}\setminus K=L\setminus L^{H}$; $L^{H}\setminus
K=L\setminus L^{G}$. In particular,
\begin{equation*}
L^{G}=\left( L\setminus L^{H}\right) \cup K;
\end{equation*}
\begin{equation*}
L^{H}=\left( L\setminus L^{G}\right) \cup K.
\end{equation*}

$\left( vi\right) $ If $L\setminus L^{G}\not=\emptyset $, then
\begin{equation*}
\sigma \left( L\setminus L^{G}\right) =L\setminus L^{G}
\end{equation*}%
holds for any $\sigma \in G$ and
\begin{equation*}
\delta \left( L\setminus L^{H}\right) =L\setminus L^{H}
\end{equation*}%
holds for any $\delta \in H$.

$\left( vii\right) $ Fixed any $x\in L$ and $\sigma \in G$. Then
\begin{equation*}
\sigma \left( x\right) \in L\setminus L^{G}\Longleftrightarrow x\in
L\setminus L^{G}.
\end{equation*}

$\left( viii\right) $ Fixed any $x\in L$ and $\delta \in H$. Then
\begin{equation*}
\delta \left( x\right) \in L\setminus L^{H}\Longleftrightarrow x\in
L\setminus L^{H}.
\end{equation*}
\end{proposition}

\begin{proof}
If $G=\{1\}$, then $H=Aut\left( L/K\right) $, $L^{H}=L^{Aut\left( L/K\right)
}=K$ (by \emph{Theorem 3.1}) and $L\setminus L^{G}=L\setminus L=\emptyset $.
We have $\sigma \left( \emptyset \right) =\emptyset $ for any $\sigma \in
Aut\left( L/K\right) $, i.e., $\left( i\right) -\left( vi\right) $ hold.

In the following, assume $G\not=\{1\}$. As $L$ is algebraic over $L^{G}$, it
is seen that $G\not=Aut\left( L/K\right) $ and $H\not=\{1\}$.

$\left( i\right) $ From the assumption $\left( N1\right) $ it is seen that $%
D=G\cdot H$ and $G$ and $H$ both are normal subgroups of $D$.

$\left( ii\right) $ For the invariant subfields of $L$ under the subgroups
we have
\begin{equation*}
L^{D}\subseteq L^{G}\text{ and }L^{D}\subseteq L^{H}
\end{equation*}%
since
\begin{equation*}
D\supseteq G\text{ and }D\supseteq H.
\end{equation*}%
As each element $\tau \in D$ has an expression that $\tau =\sigma \cdot
\delta $ with $\sigma \in G$ and $\delta \in H$, then
\begin{equation*}
\begin{array}{l}
L^{G}\cap L^{H} \\
\supseteq L^{D} \\
=\{x\in L:\tau \left( x\right) =x,\forall \tau \in D\} \\
=\{x\in L:\sigma \cdot \delta \left( x\right) =x,\forall \sigma \in
G,\forall \delta \in H\} \\
=\{x\in L:\delta \left( x\right) =\sigma ^{-1}\left( x\right) ,\forall
\sigma \in G,\forall \delta \in H\} \\
\supseteq \{x\in L:\delta \left( x\right) =\sigma ^{-1}\left( x\right)
=x,\forall \sigma \in G,\forall \delta \in H\} \\
=\{x\in L:\sigma ^{-1}\left( x\right) =x,\delta \left( x\right) =x,\forall
\sigma \in G,\forall \delta \in H\} \\
=L^{G\cup H} \\
=\{x\in L:\sigma ^{-1}\left( x\right) =x,\forall \sigma \in G\}\cap \{x\in
L:\delta \left( x\right) =x,\forall \delta \in H\} \\
=L^{G}\cap L^{H}; \\
\\
L^{D} \\
=\{x\in L^{G}:\delta \left( x\right) =x,\forall \delta \in H\} \\
=\left( L^{G}\right) ^{H} \\
=\{x\in L^{H}:\sigma \left( x\right) =x,\forall \sigma \in G\} \\
=\left( L^{H}\right) ^{G} \\
=K%
\end{array}%
\end{equation*}%
from the condition $\left( N2\right) $.

$\left( iii\right) $ From $\left( ii\right) $ we have
\begin{equation*}
L^{G}\cap L^{H}=K
\end{equation*}%
and then
\begin{equation*}
\begin{array}{l}
L\setminus K \\
=L\setminus \left( L^{G}\cap L^{H}\right) \\
=\left( L\setminus L^{G}\right) \cup \left( L\setminus L^{H}\right) .%
\end{array}%
\end{equation*}

Hence, there is a union of subsets
\begin{equation*}
L=\left( L\setminus L^{G}\right) \cup \left( L\setminus L^{H}\right) \cup K.
\end{equation*}

$\left( iv\right) $ For the invariant subfields of $L$ we have
\begin{equation*}
\begin{array}{l}
L^{G}\cap L^{H} \\
=L^{G\cup H} \\
=L^{\left\langle G\cup H\right\rangle } \\
=L^{D} \\
=K.%
\end{array}%
\end{equation*}

Then
\begin{equation*}
\left( L^{G}\setminus K\right) \cap \left( L^{H}\setminus K\right) =\left(
L^{G}\cap L^{H}\right) \setminus K=\emptyset .
\end{equation*}

$\left( v\right) $ Prove $L^{G}\setminus K=L\setminus L^{H}$. In deed, as
\begin{equation*}
L=\left( L\setminus L^{G}\right) \cup L^{G}
\end{equation*}
holds and from $\left( iii\right) -\left( iv\right) $ there is a union of
disjoint subsets
\begin{equation*}
L=\left( L\setminus L^{G}\right) \cup \left( L\setminus L^{H}\right) \cup K,
\end{equation*}
we must have
\begin{equation*}
L^{G}=\left( L\setminus L^{H}\right) \cup K.
\end{equation*}%
Hence,
\begin{equation*}
L^{G}\setminus K=L\setminus L^{H}.
\end{equation*}

In the same way, we have
\begin{equation*}
L^{H}=\left( L\setminus L^{G}\right) \cup K
\end{equation*}
since
\begin{equation*}
L=\left( L\setminus L^{H}\right) \cup L^{H}.
\end{equation*}

This proves
\begin{equation*}
L^{H}\setminus K=L\setminus L^{G}.
\end{equation*}

$\left( vi\right) $ Prove $\sigma \left( L\setminus L^{G}\right) =L\setminus
L^{G}$ holds for any $\sigma \in G$. In deed, fixed any $\sigma _{0}\in G$
and $x_{0}\in L\setminus L^{G}$. We must have $\sigma _{0}\left(
x_{0}\right) \not\in L^{G}$. Otherwise, if $y_{0}:=\sigma _{0}\left(
x_{0}\right) \in L^{G}$, then for every $\sigma \in G$ we have
\begin{equation*}
\sigma \left( y_{0}\right) =\sigma \left( \sigma _{0}\left( x_{0}\right)
\right) =\sigma _{0}\left( x_{0}\right) =y_{0}\in L^{G}
\end{equation*}
in the $G$-invariant subfield $L^{G}$; in particular, $\sigma
_{0}^{-1}\left( y_{0}\right) =y_{0}$ holds since $\sigma _{0}^{-1}\in G$; it
follows that
\begin{equation*}
x_{0}=\sigma _{0}^{-1}\left( y_{0}\right) =y_{0}\in L^{G},
\end{equation*}
which will be in contradiction. This proves
\begin{equation*}
\sigma \left( L\setminus L^{G}\right) \subseteq L\setminus L^{G}
\end{equation*}%
holds for any $\sigma \in G$.

Particularly, there is
\begin{equation*}
\sigma ^{-1}\left( L\setminus L^{G}\right) \subseteq L\setminus L^{G}
\end{equation*}%
for any $\sigma \in G$.

Hence,
\begin{equation*}
\sigma \left( L\setminus L^{G}\right) =L\setminus L^{G}
\end{equation*}%
holds for any $\sigma \in G$.

In the same way, by replacing $G$ by $H$, we have
\begin{equation*}
\delta \left( L\setminus L^{H}\right) =L\setminus L^{H}
\end{equation*}%
for any $\delta \in H$.

$\left( vii\right) $ Fixed an $x_{0}\in L$. From $\left( iv\right) $ for $L$
there is a union of disjoint subsets
\begin{equation*}
L=\left( L\setminus L^{G}\right) \cup \left( L\setminus L^{H}\right) \cup K.
\end{equation*}%
There are three cases for $x_{0}$.

Case $\left( i\right) $ If $x_{0}\in L\setminus L^{G}$, then $\sigma \left(
x_{0}\right) \in L\setminus L^{G}$ from $\left( vi\right) $; conversely, if $%
y_{0}=\sigma \left( x_{0}\right) \in L\setminus L^{G}$, then $x_{0}=\sigma
^{-1}\left( y_{0}\right) \in L\setminus L^{G}$ from $\left( vi\right) $
again since $\sigma ^{-1}\in G$ also holds.

Case $\left( ii\right) $ Let $x_{0}\in K$. Trivial since $\sigma \left(
x_{0}\right) =x_{0}$ for any $\sigma \in G$.

Case $\left( iii\right) $ Suppose $x_{0}\in L\setminus L^{H}$. From $\left(
iv\right) $ we have $L\setminus L^{H}=L^{G}\setminus K$ and then $x_{0}\in
L^{G}\setminus K\subseteq L^{G}$. Trivial since $\sigma \left( x_{0}\right)
=x_{0}$ for any $\sigma \in G$.

In the same way, prove $\left( viii\right) $. This completes the proof.
\end{proof}

Consider the direct decompositions of the field $L$ as we have done in \emph{%
\S 5.2}.

\begin{proposition}
(\emph{Direct decompositions}) Assume the conditions $\left( N1\right)
-\left( N2\right) $ are satisfied. Put
\begin{equation*}
E_{1}:=span_{K}\left( L^{G}\setminus K\right) ,F_{1}:=span_{K}\left(
L\setminus L^{G}\right) ;
\end{equation*}%
\begin{equation*}
E_{2}:=span_{K}\left( L\setminus L^{H}\right) ,F_{2}:=span_{K}\left(
L^{H}\setminus K\right) .
\end{equation*}

Then there are the direct sums of $K$-linear subspaces for $L$
\begin{equation*}
L=E_{1}\oplus K\oplus F_{1};
\end{equation*}%
\begin{equation*}
L=E_{2}\oplus K\oplus F_{2}.
\end{equation*}%
And for each $x\in L$ we have
\begin{equation*}
x=\left(
\begin{array}{c}
x_{E_{1}} \\
x_{K} \\
x_{F_{1}}%
\end{array}%
\right)
\end{equation*}%
and
\begin{equation*}
x=\left(
\begin{array}{c}
x_{E_{2}} \\
x_{K} \\
x_{F_{2}}%
\end{array}%
\right)
\end{equation*}%
as $3$-column vectors. Here, for $i=1,2$, the components $x_{E_{i}}\in E_{i}$%
, $x_{K_{i}}\in K$ and $x_{F_{i}}\in F_{i}$ are uniquely determined by $x\in
L$. In such a case,
\begin{equation*}
x_{K_{1}}=x_{K_{2}}
\end{equation*}%
holds for any $x\in L$.
\end{proposition}

\begin{proof}
From \emph{Proposition 6.11} it is seen that there are the direct sums of
linear subspaces for the $K$-vector space $L$.
\end{proof}

There are the following linear decompositions of automorphisms, i.e., the
triple decompositions (c.f. the double decompositions in \emph{Proposition
5.3}).

\begin{proposition}
(\emph{Linear decompositions of automorphisms}) Assume the two conditions $%
\left( N1\right) -\left( N2\right) $ are satisfied. Let $\sigma \in G$ and $%
\delta \in H$. Then there are the linear decompositions of automorphisms
\begin{equation*}
\sigma \left( x\right) =\left(
\begin{array}{ccc}
1_{E_{1}} & 0 & 0 \\
0 & 1_{K} & 0 \\
0 & 0 & \sigma _{F_{1}}%
\end{array}%
\right) \left(
\begin{array}{c}
x_{E_{1}} \\
x_{K} \\
x_{F_{1}}%
\end{array}%
\right)
\end{equation*}%
and
\begin{equation*}
\delta \left( x\right) =\left(
\begin{array}{ccc}
\delta _{E_{2}} & 0 & 0 \\
0 & 1_{K} & 0 \\
0 & 0 & 1_{F_{2}}%
\end{array}%
\right) \left(
\begin{array}{c}
x_{E_{2}} \\
x_{K} \\
x_{F_{2}}%
\end{array}%
\right)
\end{equation*}%
for any
\begin{equation*}
x=\left(
\begin{array}{c}
x_{E_{1}} \\
x_{K} \\
x_{F_{1}}%
\end{array}%
\right) =\left(
\begin{array}{c}
x_{E_{2}} \\
x_{K} \\
x_{F_{2}}%
\end{array}%
\right) \in L,
\end{equation*}%
where the components of $\sigma $ and $\delta $ are given respectively by
the restrictions of maps
\begin{equation*}
\begin{array}{l}
1_{E_{1}}=\sigma |_{E_{1}}:E_{1}\rightarrow E_{1}, \\
1_{K}=\sigma |_{K}:K\rightarrow K, \\
\sigma _{F_{1}}=\sigma |_{F_{1}}:F_{1}\rightarrow F_{1}%
\end{array}%
\end{equation*}%
and
\begin{equation*}
\begin{array}{l}
\delta _{E_{2}}=\delta |_{E_{2}}:E_{2}\rightarrow E_{2}, \\
1_{K}=\delta |_{K}:K\rightarrow K, \\
1_{F_{2}}=\delta |_{F_{2}}:F_{2}\rightarrow F_{2}.%
\end{array}%
\end{equation*}
\end{proposition}

\begin{proof}
In deed, from \emph{Propositions 6.11-2} we have the linear decompositions
given by the square matrices for any $\sigma \in G$ and $\delta \in H$.
\end{proof}

\begin{proposition}
Suppose the two conditions $\left( N1\right) -\left( N2\right) $ are
satisfied. Then we have $E_{1}=E_{2}$ and $F_{1}=F_{2}.$
\end{proposition}

\begin{proof}
Immediately from $\left( v\right) $ of \emph{Proposition 6.11}.
\end{proof}

Now we have the following theorem on linear decompositions of automorphisms
in the triple case. Note that here the $G$-invariant subfield $L^{G}$ is not
necessarily a purely transcendental extension over $K$.

\begin{theorem}
\emph{(Linear decompositions of automorphisms)} Let $L $ be a purely
transcendental extension over a field $K$ of finite transcendence degree.
Assume $G$ is a subgroup of $Aut\left( L/K\right) $ such that $L$ is
algebraic over the $G$-invariant subfield $L^{G}$. Let $H$ be a subgroup of $%
Aut\left( L/K\right) $. Suppose the two conditions $\left( N1\right) -\left(
N2\right) $ are satisfied:

$\ \left( N1\right) $ There is $G\cap H=\{1\}$ and $\sigma \cdot \delta
=\delta \cdot \sigma $ holds in $Aut\left( L/K\right) $ for any $\sigma \in
G $ and $\delta \in H$.

$\ \left( N2\right) $ $K$ is the invariant subfield $L^{\left\langle G\cup
H\right\rangle }$ of $L$ under the subgroup $\left\langle G\cup
H\right\rangle $ generated in $Aut\left( L/K\right) $ by the subset $G\cup H$%
.

Put
\begin{equation*}
E=span_{K}\left( L^{G}\setminus K\right) ;
\end{equation*}%
\begin{equation*}
F=span_{K}\left( L^{H}\setminus K\right) .
\end{equation*}%
Then there are the following statements.

$\left( i\right) $ For each $x\in L$ we have
\begin{equation*}
x=\left(
\begin{array}{c}
x_{E} \\
x_{K} \\
x_{F}%
\end{array}%
\right)
\end{equation*}%
where $x_{E}\in E,x_{K}\in K$ and $x_{F}\in F$ are called the $E$\textbf{%
-component}, $K$\textbf{-component} and $F$\textbf{-component} of $x$ in the
triple decomposition, respectively.

$\left( ii\right) $ Let $\sigma \in G$. For each $x\in L$ we have
\begin{equation*}
\sigma \left( x\right) =\left(
\begin{array}{ccc}
1_{E} & 0 & 0 \\
0 & 1_{K} & 0 \\
0 & 0 & \sigma _{F}%
\end{array}%
\right) \left(
\begin{array}{c}
x_{E} \\
x_{K} \\
x_{F}%
\end{array}%
\right) .
\end{equation*}

Here, $\sigma _{E}=1_{E},\sigma _{K}=1_{K}$ and $\sigma _{F}$ are called the
$E$\textbf{-component}, $K$\textbf{-component} and $F$\textbf{-component} of
$\sigma $ in the triple decomposition, respectively.

$\left( iii\right) $ Let $\delta \in H$. For each $x\in L$ we have
\begin{equation*}
\delta \left( x\right) =\left(
\begin{array}{ccc}
\delta _{E} & 0 & 0 \\
0 & 1_{K} & 0 \\
0 & 0 & 1_{F}%
\end{array}%
\right) \left(
\begin{array}{c}
x_{E} \\
x_{K} \\
x_{F}%
\end{array}%
\right) .
\end{equation*}

Here, $\delta _{E}=1_{E},\delta _{K}=1_{K}$ and $\delta _{F}$ are called the
$E$\textbf{-component}, $K$\textbf{-component} and $F$\textbf{-component} of
$\delta $ in the triple decomposition, respectively.

$\left( iv\right) $ For the subgroup $\left\langle G\cup H\right\rangle $ in
the Galois group $Aut\left( L/K\right) $, there is a direct decomposition $%
\left\langle G\cup H\right\rangle =GH=HG$ in the Galois group $Aut\left(
L/K\right) $. In particular, $G$ and $H$ are normal subgroups of $%
\left\langle G\cup H\right\rangle $.
\end{theorem}

\begin{proof}
Immediately from \emph{Propositions 6.11-4}.
\end{proof}

For the existence of the above subgroup $H$ in $Aut\left( L/K\right) $, we
have the following theorem.

\begin{theorem}
Let $L$ be a purely transcendental extension over a field $K$ of finite
transcendence degree. Suppose $G$ is a subgroup of $Aut\left( L/K\right) $
such that $L$ is algebraic over the $G$-invariant subfield $L^{G}$.

If either the $G$-invariant subfield $L^{G}$ is purely transcendental over $%
K $ or there is a transcendence base $\Delta $ of $L^{G}/K$ such that $L$
and $L^{G}$ both are Galois over $K\left( \Delta \right) $, then there
exists a subgroup $H$ of $Aut\left( L/K\right) $ such that the two
conditions $\left( N1\right) -\left( N2\right) $ are satisfied:

$\ \left( N1\right) $ There is $G\cap H=\{1\}$ and $\sigma \cdot \delta
=\delta \cdot \sigma $ holds in $Aut\left( L/K\right) $ for any $\sigma \in
G $ and $\delta \in H$.

$\ \left( N2\right) $ $K$ is the invariant subfield $L^{\left\langle G\cup
H\right\rangle }$ of $L$ under the subgroup $\left\langle G\cup
H\right\rangle $ generated in $Aut\left( L/K\right) $ by the subset $G\cup H$%
.
\end{theorem}

\begin{proof}
$\left( i\right) $ Assume $L^{G}$ is purely transcendental over $K$. Let $%
H=\pi _{t}\left( L/K\right) \left( \Lambda ,B\right) $ be the highest
transcendental Galois subgroup of $L/K$. Here, $\left( \Lambda ,B\right) $
is a nice basis of $L/K$ such that $K\left( \Lambda \right) =L^{G}$. From
\emph{Theorem 6.4} it is seen that the two conditions $\left( N1\right)
-\left( N2\right) $ are satisfied.

$\left( ii\right) $ Suppose there is a transcendence base $\Delta $ of $%
L^{G}/K$ such that $L$ and $L^{G}$ both are Galois over $K\left( \Delta
\right) $. Consider a tower of algebraic subextensions in $L/K$
\begin{equation*}
K\left( \Delta \right) \subseteq L^{G}\subseteq L.
\end{equation*}

Put
\begin{equation*}
E_{1}=span_{K}\left( L^{G}\setminus K\left( \Delta \right) \right) ;
\end{equation*}%
\begin{equation*}
E_{2}=span_{K}K\left( \Delta \right) ;
\end{equation*}%
\begin{equation*}
F=span_{K}\left( L\setminus L^{G}\right) .
\end{equation*}

It is seen that $L$, as a $K$-linear space, has a direct decomposition $%
L=E\oplus F$, where $E:=E_{2}\oplus E_{1}=L^{G}$. Hence, each $\widetilde{%
\delta }:L^{G}\rightarrow L^{G}$ of the Galois group $Aut\left(
L^{G}/K\left( \Delta \right) \right) $ has a unique extension $\delta
:L\rightarrow L$ given by
\begin{equation*}
\delta \left( x\right) =\left(
\begin{array}{cc}
\widetilde{\delta } & 0 \\
0 & 1_{F}%
\end{array}%
\right) \left(
\begin{array}{c}
x_{E} \\
x_{F}%
\end{array}%
\right)
\end{equation*}%
for any $x\in L$ with the $E$-component $x_{E}$ and the $F$-component $x_{F}$
in the double decomposition (See \emph{Proposition 5.3}). As $L$ is
algebraic Galois over $L^{G}$, we have a subgroup contained in $Aut\left(
L/K\right) $
\begin{equation*}
G_{0}:=\{\delta \in Aut\left( L/K\left( \Delta \right) \right) :\widetilde{%
\delta }\in Aut\left( L^{G}/K\left( \Delta \right) \right) ,\delta
_{F}=1_{F}\}
\end{equation*}

such that the Galois group $Aut\left( L/K\left( \Delta \right) \right) $ has
a direct decomposition
\begin{equation*}
Aut\left( L/K\left( \Delta \right) \right) =G\cdot G_{0}=G_{0}\cdot G
\end{equation*}%
where $G$ and $G_{0}$ are normal subgroups of $Aut\left( L/K\left( \Delta
\right) \right) $.

On the other hand, let $H_{0}=\pi _{t}\left( L/K\right) \left( \Delta
,A\right) $ be the highest transcendental Galois subgroup of $L/K$, where $%
\left( \Delta ,A\right) $ is a nice basis of $L/K$. Suppose
\begin{equation*}
H=\left\langle G_{0}\cup H_{0}\right\rangle
\end{equation*}%
is the subgroup in the Galois group $Aut\left( L/K\right) $ generated by the
subset $G_{0}\cup H_{0}$. Then there is
\begin{equation*}
H=G_{0}\cdot H_{0}
\end{equation*}%
in $Aut\left( L/K\right) $ from \emph{Theorem 6.4} and it follows that the
two conditions $\left( N1\right) -\left( N2\right) $ are satisfied. This
completes the proof.
\end{proof}

\begin{definition}
Let $L$ be an extension over a field $K$. A subgroup $G$ of $Aut\left(
L/K\right) $ is said to be \textbf{Galois-complemented} in $L/K$ if there is
a subgroup $H$ of $Aut\left( L/K\right) $ such that the two conditions $%
\left( N1\right) -\left( N2\right) $ are satisfied. In such a case, the
subgroup $H$ is said to be a \textbf{Galois-complement} of $G$ in $L/K$.

In particular, if $G$ is a Noether solution of $L/K$, the highest
transcendental Galois subgroup $\pi _{t}\left( L/K\right) \left( \Delta
,A\right) $ of $L/K$ is a Galois-complement of $G$ in $L/K$, called the
\textbf{natural Galois-complement} of $G$ in $L/K$, where $\left( \Delta
,A\right) $ is a nice basis $\left( \Delta ,A\right) $ of $L/K$ with $%
K\left( \Delta \right) =L^{G}$. For simplicity, we denote by $H_{L/K}\left(
G\right) $ the natural Galois-complement of $G$ in $L/K$.
\end{definition}

This turns out an essential condition to Noether solutions of $L/K$.

\begin{remark}
(\emph{An essential condition to Noether solutions}) Assume $L$ is a purely
transcendental extension over a field $K$ of finite transcendence degree.
Let $G$ be a subgroup of $Aut\left( L/K\right) $ such that $L$ is algebraic
over the $G$-invariant subfield $L^{G}$.

$\left( i\right) $ If $G$ is a Noether solution, then $G$ must be
Galois-complemented in $L/K$, i.e., the two conditions $\left( N1\right)
-\left( N2\right) $ must be satisfied. This is an essential condition to
Noether solutions of $L/K$.

In such a case, $G$ has a unique natural Galois-complement $H_{G}$ in $L/K$
which is maximal among all possible Galois-complements of $G$ in $L/K$.

In general, $G$ is not necessarily Galois-complemented in $L/K$. In deed,
that's one of the key difficulties in Noether's problem on rationality.

$\left( ii\right) $ Let $G$ be Galois-complemented in $L/K$. $G$ does not
necessarily have a natural Galois-complement in $L/K$. In deed, consider a
tower of subextensions $K\subseteq M:=K\left( \Lambda \right) \subseteq
M_{0}:=L^{G}\subseteq L$, where $\Lambda $ is a transcendence base of $L/K$
and $L^{G}$ is Galois over $M$. Then $P:=Aut\left( M_{0}/M\right) $ is a
Noether solution of $M_{0}/K$.

Let $H_{0}$ be a Galois-complement of $G$ in $L/K$ and $H_{M_{0}/M}\left(
P\right) $ the natural Galois-complement of $P$ in $M_{0}/M$. Suppose $%
H_{M_{0}}$ and $H_{M}$ are subgroups of $H_{0}$ such that the restrictions $%
H_{M_{0}}|_{M_{0}}=P$ and $H_{M}|_{M_{0}}=H_{M_{0}/M}\left( P\right) $. Then
$H:=H_{M_{0}}\cdot H_{M}$ is a Galois-complement of $G$ in $L/K$.
\end{remark}

\subsection{Galois-complemented groups and the first characteristic of
transcendental Galois theory}

Consider a Galois extension $M$ over a field $K$. If $M$ is algebraic over $%
K $ of degree $\left[ M:K\right] =+\infty $, then there is no finite
subgroup $G$ of the Galois group $Aut\left( M/K\right) $ such that $K$ is
the $G$-invariant subfield $M^{G}$ of $M$. But there is another distinct
phenomenon for a case of a transcendental extension $M/K$ (See \emph{Remark
6.21}).

\begin{lemma}
Let $L$ be a purely transcendental extension over a field $K$ of finite
transcendence degree. Suppose $G$ is a finite subgroup of the Galois group $%
Aut\left( L/K\right) $.

If $G$ is Galois-complemented in $L/K$, then each Galois-complement $H_{G}$
of $G$ is also Galois-complemented in $L/K$. In such a case, $G$ is a
Galois-complement of $H_{G}$ in $L/K$.
\end{lemma}

\begin{proof}
Let $G$ be Galois-complemented in $L/K$ and $H_{G}$ a Galois-complement of $%
G $. By definition, the two conditions $\left( N1\right) -\left( N2\right) $
hold for $G$ and $H_{G}$. Hence, $H_{G}$ is Galois-complemented in $L/K$ and
$G$ is a Galois-complement of $H_{G}$ in $L/K$.
\end{proof}

\begin{lemma}
Let $L$ be a purely transcendental extension over a field $K$ of finite
transcendence degree. Suppose $G$ is a finite subgroup of the Galois group $%
Aut\left( L/K\right) $.

If $G$ is Galois-complemented in $L/K$ and $H$ is a Galois-complement of $G$
in $L/K$, then the $H$-invariant subfield $L^{H}$ of $L$ is Galois over $K$
and the Galois group $Aut\left( L^{H}/K\right) $ contains a finite group $%
\widetilde{G}:=\{\sigma |_{L^{H}}:\sigma \in G\}$ consisting of the
restrictions $\sigma |_{L^{H}}$ of $\sigma \in G$ such that
\begin{equation*}
\left( L^{H}\right) ^{Aut\left( L^{H}/K\right) }=\left( L^{H}\right)
^{G}=\left( L^{H}\right) ^{\widetilde{G}}=K
\end{equation*}%
hold for the invariant subfields of $L^{H}$.
\end{lemma}

\begin{proof}
Let $G$ be Galois-complemented in $L/K$ and $H$ a Galois-complement of $G$.
By definition, the two conditions $\left( N1\right) -\left( N2\right) $ hold
for $G$ and $H$. Consider the $H$-invariant subfield $L^{H}$ of $L$ and the
\emph{transcendental} subextension $L^{H}/K$.

Prove $L^{H}$ is Galois over $K$. In deed, from $\left( v\right) $ of \emph{%
Proposition 6.11} there is
\begin{equation*}
L^{H}=\left( L\setminus L^{G}\right) \cup K.
\end{equation*}%
If $G=\{1\}$, then $H=Aut\left( L/K\right) $ and $L^{H}=K$ is Galois over $K$
itself.

Assume $G\not=\{1\}$. As $L\setminus L^{G}\not=\emptyset $, from $\left(
vi\right) $ of \emph{Proposition 6.11} we have
\begin{equation*}
\sigma \left( L\setminus L^{G}\right) =L\setminus L^{G}
\end{equation*}%
for any $\sigma \in G\subseteq Aut\left( L/K\right) $ and then
\begin{equation*}
\sigma \left( L^{H}\right) =L^{H}
\end{equation*}%
holds for any $\sigma \in G\subseteq Aut\left( L/K\right) $.

It follows that the group
\begin{equation*}
\widetilde{G}:=\{\sigma |_{L^{H}}:\sigma \in G\},
\end{equation*}%
which consists of the restrictions $\sigma |_{L^{H}}$ of $\sigma \in G$, is
a \emph{finite} subgroup of the Galois group $Aut\left( L^{H}/K\right) $ of
the subextension $L^{H}/K$.

On the other hand, as the condition $\left( N2\right) $ holds, from $\left(
ii\right) $ of \emph{Proposition 6.11} we have
\begin{equation*}
K\subseteq \left( L^{H}\right) ^{Aut\left( L^{H}/K\right) }\subseteq \left(
L^{H}\right) ^{\widetilde{G}}=\left( L^{H}\right) ^{G}=L^{D}=K.
\end{equation*}%
This proves $L^{H}$ is Galois over $K$ and $\widetilde{G}\subseteq Aut\left(
L^{H}/K\right) $ holds.
\end{proof}

\begin{remark}
(\emph{The first characteristic of transcendental Galois theory}) Let $L$ be
a purely transcendental extension over a field $K$ of finite transcendence
degree. Let $G$ be a finite subgroup of the Galois group $Aut\left(
L/K\right) $.

$\left( i\right) $ If $G=\left\langle \tau _{L/K}\right\rangle $ is the
subgroup in $Aut\left( L/K\right) $ generated by the reciprocal translation $%
\tau _{L/K}$ of $L/K$, then $G$ is of order two and $K$ is the $G$-invariant
subfield $L^{G}$ of $L$, i.e., $L^{G}=L^{Aut\left( L/K\right) }=K$. (See
\emph{Theorem 3.1}).

$\left( ii\right) $ If $G$ is Galois-complemented in $L/K$ and $H$ is a
Galois-complement of $G$ in $L/K$, then the $H$-invariant subfield $L^{H}$
of $L$ is transcendental over $K$ and $K$ is the $G$-invariant subfield $%
\left( L^{H}\right) ^{G}$ of $L^{H}$, i.e., $\left( L^{H}\right) ^{G}=\left(
L^{H}\right) ^{Aut\left( L^{H}/K\right) }=K$. (See \emph{Lemma 6.20}).

It follows that for a transcendental extension $M/K$, there can be finite
groups $G$ of the Galois group $Aut\left( M/K\right) $ such that $K$ is the $%
G$-invariant subfield $M^{G}$ of $M$. However, if $M/K$ is an infinite
algebraic extension, this phenomenon never occurs.
\end{remark}

\begin{example}
Let $L=K\left( t_{1},t_{2},\cdots ,t_{n}\right) $ be a purely transcendental
extension over a field $K$ with $t_{1},t_{2},\cdots ,t_{n}$ being
algebraically independent variables. Let $G=\Sigma _{n}$ be the symmetric
group on $n$ letters, as a subgroup of the Galois group $Aut\left(
L/K\right) $, acting on $t_{1},t_{2},\cdots
,t_{n} $ in an evident manner (See \emph{Definition 7.7}).

$\left( i\right) $ It is seen that $L$ is algebraic Galois over the $G$%
-invariant subfield $L^{G}$ of $L$. (See \emph{Lemma 7.6}).

$\left( ii\right) $ From \emph{Theorem 6.16} it is seen that $G$ is
Galois-complemented in $L/K$.

$\left( iii\right) $ Let $H$ be the natural Galois-complement of $G$ in $L/K$. Then
the $H$-invariant subfield $L^{H}$ is purely
transcendental over $K$. (See \emph{Lemma 7.6}).

$\left( iv\right) $ The reciprocal translation $\tau _{L^{H}/K}$ of $L^{H}/K$
is not contained in the subgroup $\widetilde{G}:=\{\sigma |_{L^{H}}:\sigma
\in G\}$ of the Galois group $Aut\left( L^{H}/K\right) $, i.e., $\tau
_{L^{H}/K}\not\in \widetilde{G}$.

$\left( v\right) $ Let $M=L^{H}$. From \emph{Lemma 6.20} we obtain a new
case of a finite subgroup $\widetilde{G}$ of the Galois group $Aut\left(
M/K\right) $ for a purely transcendental extension $M/K$ such that $M^{%
\widetilde{G}}=M^{G}=K$ hold, i.e., $K$ is the $\widetilde{G}$-invariant
subfield $M^{\widetilde{G}}$ of $M$.
\end{example}

\subsection{Galois-complemented groups and the second characteristic of
transcendental Galois theory}

In the following we discuss the levels of hyperplanes in a purely
transcendental extension which play the same role as the codimensions of
subvarieties in an algebraic variety (particularly in an excellent algebraic
scheme). See \emph{[}An 2019\emph{]}\ for further results.

Let $L$ be a purely transcendental extension over a field $K$ of finite
transcendence degree and $\Omega _{L/K}$ the set of all the transcendence
bases of $L/K$.

\begin{definition}
By a \textbf{hyperplane} in $L/K$ we understand a subfield $M=K\left(
\Lambda \right) $ for $\Lambda \in \Omega _{L/K}$. For a hyperplane $M$ in $%
L/K$, the number $d_{L/K}\left( M\right) :=\left[ L:M\right] $ is called the
\textbf{level} of $M$ in $L/K$.

Let $M_{1}\subseteq M_{2}$ be two hyperplanes in $L/K$. The number $%
dist\left( M_{1},M_{2}\right) $ of the hyperplanes $M$ in $L/K$ with $%
M_{1}\subsetneqq M\subseteq M_{2}$ is called the \textbf{distance} between $%
M_{1}$ and $M_{2}$, i.e.,
\begin{equation*}
dist\left( M_{1},M_{2}\right) :=\sharp \{M\text{ is a hyperplane in }%
L/K:M_{1}\subsetneqq M\subseteq M_{2}\}.
\end{equation*}
\end{definition}

\begin{lemma}
Let $L$ be a purely transcendental extension over a field $K$ of finite
transcendence degree. Then for any $\sigma \in Aut\left( L/K\right) $ and
hyperplane $M$ in $L/K$, the image $\sigma \left( M\right) $ is still a
hyperplane in $L/K$.

In particular, each $\sigma \in Aut\left( L/K\right) $ preserves the levels
of hyperplanes in $L/K$ and the distances of hyperplanes in $L/K$,
respectively.
\end{lemma}

\begin{proof}
Let $\sigma \in Aut\left( L/K\right) $ and let $M=K\left( \Delta \right) $
be a hyperplane in $L/K$ of level $d:=d_{L/K}\left( M\right) $. Here, $%
\Delta \in \Omega _{L/K}$. Immediately, $\sigma \left( M\right) =K\left(
\sigma \left( \Delta \right) \right) $ is a hyperplane in $L/K$.

Fixed a nice basis $\left( \Delta ,A\right) $ of $L/K$. It is seen that $%
\left( \sigma \left( \Delta \right) ,\sigma \left( A\right) \right) $ is
also a nice basis of $L/K$. In particular, we have
\begin{equation*}
\left[ L:K\left( \Delta \right) \right] =\sharp A=\sharp \sigma \left(
A\right) =\left[ L:K\left( \sigma \left( \Delta \right) \right) \right] .
\end{equation*}

Hence, any $\sigma \in Aut\left( L/K\right) $ preserves the levels of $M$
and $\sigma \left( M\right) $. It follows that each $\sigma \in Aut\left(
L/K\right) $ preserves the distance of hyperplanes.
\end{proof}

From \emph{Lemma 6.24} we obtain the second characteristic of transcendental
Galois theory.

\begin{remark}
(\emph{The second characteristic of transcendental Galois theory}) Let $L$
be a purely transcendental extension over a field $K$ of finite
transcendence degree. Each automorphism $\sigma $ of $L$ over $K$ preserves
the levels and distances of hyperplanes in $L/K$. In particular, via $\sigma
$, an upper hyperplane is mapped to an upper one; a lower hyperplane is
mapped to a lower one. Such a picture offers us the following \emph{Lemma
6.26}. Note that \emph{Lemma 6.26} is not true for an algebraic extension $%
L/K$.
\end{remark}

The following lemma will play a key role in the proofs in \emph{\S 7} and in
\emph{\S 9} for \emph{Lemmas 7.16-8} and \emph{Theorem 1.8}, respectively.

\begin{lemma}
\emph{(Noether solutions: Galois preserved by quotients)} Let $L$ be a purely
transcendental extension over a field $K$ of finite transcendence degree.
Fixed two Noether solutions $G$ and $P$ of $L/K$.

If $P$ is contained in $G$, then $L^{P}$ is algebraic Galois over $L^{G}$
and we have
\begin{equation*}
G=P\cdot Q=Q\cdot P
\end{equation*}%
for the group $G$, where $Q$ is the intersection of subgroups in $Aut\left(
L/K\right) $
\begin{equation*}
Q=G\cap H_{P}
\end{equation*}
and $H_{P}$ is the natural Galois-complement of $P$ in $L/K$.

In such a case, for any $\sigma \in G$ we have
\begin{equation*}
\sigma \left( L\setminus L^{P}\right) =L\setminus L^{P};
\end{equation*}
\begin{equation*}
\sigma \left( L^{P}\setminus L^{G}\right) =L^{P}\setminus L^{G}.
\end{equation*}
\end{lemma}

\begin{proof}
Assume $G\supseteq P$. $P$ has a unique natural Galois-complement $H_{P}$ in
$L/K$ from \emph{Theorem 6.4}. We will proceed in several steps.

\emph{Step 1}. Fixed a $\sigma _{0}\in G$. Prove $\sigma _{0}\left(
L^{P}\setminus L^{G}\right) \subseteq L^{P}\setminus L^{G}$ holds.

Hypothesize there is some $x_{0}\in L^{P}\setminus L^{G}$. As $L$ is finite
and Galois over $L^{G}$, we have $\sigma _{0}\left( x_{0}\right) \in
L\setminus L^{G}.$ Let $\Lambda $ be a transcendence base of $L/K$ such that
$L^{P}=K\left( \Lambda \right) $.

As $x_{0}\in L$ is algebraic over $K\left( \Lambda \right) =L^{P}$, for the
degrees of the subextensions we have
\begin{equation*}
\left[ K\left( \Lambda \right) \left( x_{0}\right) :K\left( \Lambda \right) %
\right] =\left[ K\left( \sigma _{0}\left( \Lambda \right) \right) \left(
\sigma _{0}\left( x_{0}\right) \right) :K\left( \sigma _{0}\left( \Lambda
\right) \right) \right]
\end{equation*}%
by taking the minimal polynomials $f\left( X\right) $ of $x_{0}$ over $%
L^{P}=K\left( \Lambda \right) $ and $g\left( X\right) $ of $\sigma
_{0}\left( x_{0}\right) $ over $\sigma _{0}\left( L^{P}\right) =K\left(
\sigma _{0}\left( \Lambda \right) \right) $, respectively. In deed, put
\begin{equation*}
f\left( X\right) =\sum_{0\leq i\leq h}a_{i}X^{i}\in K\left( \Lambda \right)
\left[ X\right]
\end{equation*}%
where for each $0\leq i\leq h$ we have
\begin{equation*}
\Lambda =\{t_{1},t_{2},\cdots ,t_{n}\};
\end{equation*}%
\begin{equation*}
a_{i}=\frac{p_{i}\left( t_{1},t_{2},\cdots ,t_{n}\right) }{q_{i}\left(
t_{1},t_{2},\cdots ,t_{n}\right) }\in K\left( \Lambda \right) ;
\end{equation*}%
$p_{i}\left( X_{1},X_{2},\cdots ,X_{n}\right) $ and $q_{i}\left(
X_{1},X_{2},\cdots ,X_{n}\right) \not=0$ are polynomials over $K$ of the
indeterminates $X_{1},X_{2},\cdots ,X_{n}$.

It follows that we have
\begin{equation*}
g\left( X\right) =\sigma _{0}\left( f\right) \left( X\right) =\sum_{0\leq
i\leq h}\sigma _{0}\left( a_{i}\right) X^{i}\in K\left( \sigma _{0}\left(
\Lambda \right) \right) \left[ X\right]
\end{equation*}%
where for each $0\leq i\leq h$ we have
\begin{equation*}
\sigma _{0}\left( a_{i}\right) =\frac{p_{i}\left( \sigma _{0}\left(
t_{1}\right) ,\sigma _{0}\left( t_{2}\right) ,\cdots ,\sigma _{0}\left(
t_{n}\right) \right) }{q_{i}\left( \sigma _{0}\left( t_{1}\right) ,\sigma
_{0}\left( t_{2}\right) ,\cdots ,\sigma _{0}\left( t_{n}\right) \right) }\in
K\left( \sigma _{0}\left( \Lambda \right) \right) .
\end{equation*}%
Hence, $x_{0}$ is a root of
\begin{equation*}
f\left( X\right) =0
\end{equation*}%
if and only if $\sigma _{0}\left( x_{0}\right) $ is a root of
\begin{equation*}
g\left( X\right) =\sigma _{0}\left( f\right) \left( X\right) =0
\end{equation*}%
This proves
\begin{equation*}
\begin{array}{l}
\deg f\left( X\right) \\
=\left[ K\left( \Lambda \right) \left( x_{0}\right) :K\left( \Lambda \right) %
\right] \\
=\deg g\left( X\right) \\
=\left[ K\left( \sigma _{0}\left( \Lambda \right) \right) \left( \sigma
_{0}\left( x_{0}\right) \right) :K\left( \sigma _{0}\left( \Lambda \right)
\right) \right] .%
\end{array}%
\end{equation*}

On the other hand, as $x_{0}\in L^{P}\setminus L^{G}$, we have
\begin{equation*}
\left[ K\left( \Lambda \right) \left( x_{0}\right) :K\left( \Lambda \right) %
\right] =1;
\end{equation*}%
as $\sigma _{0}\left( x_{0}\right) \in L\setminus L^{P}$, we have
\begin{equation*}
2\leq \left[ K\left( \sigma _{0}\left( \Lambda \right) \right) \left( \sigma
_{0}\left( x_{0}\right) \right) :K\left( \sigma _{0}\left( \Lambda \right)
\right) \right] \leq \left[ G:P\right] ,
\end{equation*}%
which is in contradiction. This proves
\begin{equation*}
\sigma \left( L^{P}\setminus L^{G}\right) \subseteq L^{P}\setminus L^{G}
\end{equation*}%
holds for any $\sigma \in G$.

Particularly, there is
\begin{equation*}
\sigma ^{-1}\left( L^{P}\setminus L^{G}\right) \subseteq L^{P}\setminus L^{G}
\end{equation*}%
since $\sigma ^{-1}\in G$. Hence, we have
\begin{equation*}
\sigma \left( L^{P}\setminus L^{G}\right) =L^{P}\setminus L^{G}
\end{equation*}%
for any $\sigma \in G$.

\emph{Step 2}. Let $x_{0}\in L^{P}\setminus L^{G}$. As $L$ is algebraic
Galois over $L^{G}$, it is seen that every $L^{G}$-conjugate of the subfield
$L^{G}\left[ x_{0}\right] $ is of the form $L^{G}\left[ \sigma \left(
x_{0}\right) \right] $ with $\sigma \in G$; then $L^{G}\left[ \sigma \left(
x_{0}\right) \right] \subseteq L^{P}$ since $\sigma \left( x_{0}\right) \in
L^{P}\setminus L^{G}$ from $\left( i\right) $. This proves $L^{P}$ is
algebraic Galois over $L^{G}$.

\emph{Step 3}. Let $\widetilde{Q}=\{\delta |_{L^{P}}:\delta \in Q\}$. Then $%
\widetilde{Q}=Aut\left( L^{P}/L^{G}\right) $. From $\left( iii\right)
-\left( iv\right) $ of \emph{Proposition 5.3} we have
\begin{equation*}
Q=G\cap H_{P};
\end{equation*}%
\begin{equation*}
G=P\cdot Q=Q\cdot P,
\end{equation*}%
where $H_{P}$ is the natural Galois-complement of $P$ in $L/K$; $P$ and $Q$
both are normal subgroups of $G$.

\emph{Step 4}. From $\left( vi\right) $ of \emph{Proposition 6.11} for any $%
\sigma \in P$ there is
\begin{equation*}
\sigma \left( L\setminus L^{P}\right) =L\setminus L^{P}
\end{equation*}%
and for any $\delta \in Q$, from $\left( iii\right) $ of \emph{Theorem 6.15}
we have
\begin{equation*}
\delta \left( L\setminus L^{P}\right) =L\setminus L^{P}.
\end{equation*}%
Hence,
\begin{equation*}
\tau \left( L\setminus L^{P}\right) =L\setminus L^{P}
\end{equation*}%
holds for any $\tau \in G$. This completes the proof.
\end{proof}

\begin{remark}
(\emph{Noether solutions: Galois not preserved by towers}) Let $L$ be a
purely transcendental extension over a field $K$ of finite transcendence
degree. Let $G$ be a Noether solution $L/K$ and $P$ a Noether solution of $%
L^{G}/K$ with the natural Galois-complements $H_{L/K}\left( G\right) $ in $%
L/K$ and $H_{L^{G}/K}\left( P\right) $ in $L^{G}/K$, respectively. Then
\begin{equation*}
H_{L/K}\left( G\right) |_{L^{G}}\bigcap P=P;
\end{equation*}%
\begin{equation*}
H_{L/K}\left( G\right) |_{L^{G}}\bigcap H_{L^{G}/K}\left( P\right)
=H_{L^{G}/K}\left( P\right) .
\end{equation*}%
Here, $H_{L/K}\left( G\right) |_{L^{G}}:=\{\delta |_{L^{G}}:\delta \in
H_{L/K}\left( G\right) \}$ denotes the restriction of $H_{L/K}\left(
G\right) $ to $L^{G}$. In general, $L$ is not necessarily algebraic Galois
over the $P$-invariant subfield $L^{P}$, due to the same reason as for the
case of algebraic extensions: We have
\begin{equation*}
\left( H_{L/K}\left( G\right) \bigcap Aut\left( L/L^{P}\right) \right)
|_{L^{G}}\supseteq P
\end{equation*}%
for the restriction of the subgroup to $L^{G}$; but for the other side
\begin{equation*}
\left( H_{L/K}\left( G\right) \bigcap Aut\left( L/L^{P}\right) \right)
|_{L^{G}}\subseteq P,
\end{equation*}%
it is not necessarily true, in general.
\end{remark}

\begin{remark}
(\emph{Distances of hyperplanes reflexing properties of Galois groups}) Let $%
L$ be a purely transcendental extension over a field $K$ of finite
transcendence degree. Let $G$ be a Noether solution of $L/K$. Consider the
distance
\begin{equation*}
\rho \left( G\right) :=dist\left( L^{G},L\right)
\end{equation*}
between the hyperplanes $L^{G}$ and $L$ in the extension $L/K$. If $\sharp
G=2$, then $\rho \left( G\right) =1$ holds. Particularly, if $G$ is the full
permutation subgroup $\Sigma $ in $L/K$, then we have $\rho \left( G\right)
=1$ from \emph{Lemma 7.18}.

Is there any other $G$ in $Aut\left( L/K\right) $ of the property that $\rho
\left( G\right) =1$ holds? In deed, these subgroups $G$ in $Aut\left(
L/K\right) $ are called \textbf{distance-one subgroups} in $L/K$. For
details see \emph{[}An 2019\emph{]}.
\end{remark}

\section{$G$-Symmetric Functions and Noether Solutions}

In this section we will discuss the invariant subfields of $G$-symmetric
functions under the permutation subgroups $G$ of the Galois group $Aut\left(
L/K\right) $ for a purely transcendental extension $L$ over a field $K$.
Here, as a general result, \emph{Lemma 7.16} is  a generalisation
of the first counter-example to Noether's problem given by Swan in 1969 (See \emph{[}%
Swan 1969\emph{]}).

\subsection{Least nice number for a transcendental extension}

For sake of convenience, we have a quantity to describe a purely
transcendental sub-extension in a given transcendental extension.

\begin{definition}
Let $E$ be a finitely generated extension over a field $F$. The integer
\begin{equation*}
\mathfrak{ni}\left( \frac{E}{F}\right) :=\min \{\sharp A-1\in \mathbb{Z}%
:\left( \Delta ,A\right) \text{ is a nice basis of }E/F\}
\end{equation*}%
is called the \textbf{least nice number} of the extension $E/F$.
\end{definition}

\begin{remark}
Let $L$ be a purely transcendental extension over a field $K$ of finite
transcendence degree $n$. Let $G$ be a finite subgroup of the Galois group $%
Aut\left( L/K\right) $. Consider the $G$-invariant subfield $L^{G}$ of $L$.

$\left( i\right) $ Evidently, $G$ is a Noether solution of $L/K$ if and only
if $\mathfrak{ni}\left( \frac{L^{G}}{K}\right) $ $=0$ holds for the least
nice number of the subextension $L^{G}/K$. So, in general, it is very hard
for one to finds what the least nice number $\mathfrak{ni}\left( \frac{L^{G}%
}{K}\right) $ really is. However, the least nice number $m=\mathfrak{ni}%
\left( \frac{L^{G}}{K}\right) $ offers to us the existence of the maximal
elements $\Delta _{\max }$ of all the possible transcendence bases $\Delta $
of the subextension $L^{G}/K$ such that $\left[ L^{G}:K\left( \Delta _{\max
}\right) \right] =m$ holds. (See \emph{Remark 1.6}).

$\left( ii\right) $ If the subgroup $G$ is not Galois-complemented in $L/K$,
then $\mathfrak{ni}\left( \frac{L^{G}}{K}\right) $ $\geq 1$ holds for the
least nice number of the subextension $L^{G}/K$ (from \emph{Theorem 6.16});
in such a case, $G$ is not a Noether solution of $L/K$. This is the key
observation for us to obtain the correct pictures for the following \emph{%
Lemmas 7.16-8} in the present section, which are viewed as a generalisation
of the main theorem in \emph{[}Swan 1969\emph{]}.
\end{remark}

\subsection{$G$-symmetric polynomials}

Let $G$ be a \emph{permutation group} on $m$ letters, i.e., a subgroup of
the \emph{symmetric group} $\Sigma _{m}$ on $m$ letters.

\begin{definition}
($G$\emph{-symmetric polynomials}) Let $K\left[ X_{1},X_{2},\cdots ,X_{m}%
\right] $ be the ring of polynomials of indeterminates $X_{1},X_{2},\cdots
,X_{m}$ with coefficients in a field $K$.

$\left( i\right) $ An element $\sigma \in G$ induces a $\sigma $\textbf{-map}
\begin{equation*}
\sigma _{\ast }:K\left[ X_{1},X_{2},\cdots ,X_{m}\right] \rightarrow K\left[
X_{1},X_{2},\cdots ,X_{m}\right]
\end{equation*}
given by
\begin{equation*}
f\left( X_{1},X_{2},\cdots ,X_{m}\right) \longmapsto  f_{\sigma}
\left( X_{1},X_{2},\cdots ,X_{m}\right)
\end{equation*}
where
\begin{equation*}
f_{\sigma} \left( X_{1},X_{2},\cdots ,X_{m}\right) :=f\left(
X_{\sigma \left( 1\right) },X_{\sigma \left( 2\right) },\cdots ,X_{\sigma
\left( m\right) }\right) .
\end{equation*}

For simplicity, we will identity $\sigma _{\ast }$ with $\sigma $.

$\left( ii\right) $ A polynomial $f\left( X_{1},X_{2},\cdots ,X_{m}\right) $
of indeterminates $X_{1},X_{2},\cdots ,X_{m}$ with coefficients in $K$ is
said to be $G$\textbf{-symmetric} if there is
\begin{equation*}
f\left( X_{1},X_{2},\cdots ,X_{m}\right) =f_{\sigma} \left(
X_{1},X_{2},\cdots ,X_{m}\right)
\end{equation*}%
for any $\sigma \in G$. Denote by
\begin{equation*}
\Lambda ^{G}\left( K\left[ X_{1},X_{2},\cdots ,X_{m}\right] \right)
\end{equation*}%
the set of $G$-symmetric polynomials of indeterminates $X_{1},X_{2},\cdots
,X_{m}$ over $K$.
\end{definition}

\begin{remark}
Let $\sigma \in G$. The $\sigma $-map $\sigma _{\ast }$ preserves the
algebraic operations in the polynomial ring $K\left[ X_{1},X_{2},\cdots
,X_{m}\right] $.

In deed, for any $c\in K$ and $f,g\in K\left[ X_{1},X_{2},\cdots ,X_{m}%
\right] $ there are
\begin{equation*}
\begin{array}{l}
\sigma _{\ast }\left( c\cdot f\right) =c\cdot \sigma _{\ast }\left( f\right)
; \\
\sigma _{\ast }\left( f\cdot g\right) =\sigma _{\ast }\left( f\right) \cdot
\sigma _{\ast }\left( g\right) ; \\
\sigma _{\ast }\left( f+g\right) =\sigma _{\ast }\left( f\right) +\sigma
_{\ast }\left( g\right) .%
\end{array}%
\end{equation*}

The set $\Lambda ^{G}\left( K\left[ X_{1},X_{2},\cdots ,X_{m}\right] \right)
$ of $G$-symmetric polynomials of indeterminates $X_{1},X_{2},\cdots ,X_{m}$
over $K$ is a subring of the polynomial ring $K\left[ X_{1},X_{2},\cdots
,X_{m}\right] $.
\end{remark}

\begin{remark}
(\emph{Reviews on elementary symmetric polynomials}) The \emph{elementary
symmetric polynomials} of $m$ indeterminates $X_{1},X_{2},\cdots ,X_{m}$
over a field $K$ are homogeneous polynomials given as the following:
\begin{equation*}
e_{1}\left( X_{1},X_{2},\cdots ,X_{m}\right) ,e_{2}\left( X_{1},X_{2},\cdots
,X_{m}\right) ,\cdots ,e_{m}\left( X_{1},X_{2},\cdots ,X_{m}\right)
\end{equation*}%
where
\begin{equation*}
e_{k}\left( X_{1},X_{2},\cdots ,X_{m}\right) :=\sum_{1\leq
i_{1}<i_{2}<\cdots <i_{k}\leq m}X_{i_{1}}X_{i_{2}}\cdots X_{i_{k}}
\end{equation*}%
for each $1\leq k\leq m$.

$\left( i\right) $ Let $\Sigma _{m}$ denote the symmetric group on $m$
letters. Then the elementary symmetric polynomials
\begin{equation*}
e_{1}\left( X_{1},X_{2},\cdots ,X_{m}\right) ,e_{2}\left( X_{1},X_{2},\cdots
,X_{m}\right) ,\cdots ,e_{m}\left( X_{1},X_{2},\cdots ,X_{m}\right)
\end{equation*}%
are $\Sigma _{m}$-symmetric polynomials.

$\left( ii\right) $ Let $Y,X_{1},X_{2},\cdots ,X_{m}$ be distinct
indeterminates over a field $K$. By comparing the coefficients of a
polynomial, it is seen that the elementary symmetric polynomials are
coefficients of the following polynomial of indeterminate $Y$ over the ring $%
K\left[ X_{1},X_{2},\cdots ,X_{m}\right] $
\begin{equation*}
\prod_{1\leq i\leq m}\left( Y-X_{i}\right) =\sum_{0\leq k\leq m}\left(
-1\right) ^{k}\cdot e_{k}\left( X_{1},X_{2},\cdots ,X_{m}\right) \cdot
Y^{m-k}
\end{equation*}%
where
\begin{equation*}
e_{0}\left( X_{1},X_{2},\cdots ,X_{m}\right) :=1.
\end{equation*}
\end{remark}

\begin{lemma}
\emph{(}[\emph{Bourbaki}]\emph{)} Let $K\left[ X_{1},X_{2},\cdots ,X_{m}\right] $
be the ring of polynomials of indeterminates $X_{1},X_{2},\cdots ,X_{m}$
over a field $K$. Let $\Lambda ^{\Sigma _{m}}\left( K\left[
X_{1},X_{2},\cdots ,X_{m}\right] \right) $ be the set of $\Sigma _{m}$%
-symmetric polynomials of indeterminates $X_{1},X_{2},\cdots ,X_{m}$ over $K$%
. Here, $\Sigma _{m}$ denotes the symmetric group on $m$ letters. Then there
are the following statements.

$\left( i\right) $ The elementary symmetric polynomials
\begin{equation*}
e_{1}\left( X_{1},X_{2},\cdots ,X_{m}\right) ,e_{2}\left( X_{1},X_{2},\cdots
,X_{m}\right) ,\cdots ,e_{m}\left( X_{1},X_{2},\cdots ,X_{m}\right)
\end{equation*}
are algebraically independent over $K$.

$\left( ii\right) $ $\Lambda ^{\Sigma _{m}}\left( K\left[ X_{1},X_{2},\cdots
,X_{m}\right] \right) $ is a $K$-algebra generated by the elementary
symmetric polynomials
\begin{equation*}
e_{1}\left( X_{1},X_{2},\cdots ,X_{m}\right) ,e_{2}\left( X_{1},X_{2},\cdots
,X_{m}\right) ,\cdots ,e_{m}\left( X_{1},X_{2},\cdots ,X_{m}\right) .
\end{equation*}

$\left( iii\right) $ $K\left[ X_{1},X_{2},\cdots ,X_{m}\right] $ is a
finitely generated and free module over the subring $\Lambda ^{\Sigma
_{m}}\left( K\left[ X_{1},X_{2},\cdots ,X_{m}\right] \right) $ of rank $%
\sharp \Sigma _{m}=m!$.
\end{lemma}

\begin{proof}
The proof is based on basic calculations on polynomials over a field by
comparing the coefficients of monomials in a lexicographic order (See \emph{%
Theorem 1}, \emph{Page 58}, \emph{chap iv}, \emph{[}Bourbaki\emph{]}).
\end{proof}

\subsection{$G$-symmetric functions}

Permutation group as a subgroup of the Galois group

Let $t_{1},t_{2},\cdots ,t_{m}$ be algebraically independent variables over
a field $K$. Consider the field $L=K\left( t_{1},t_{2},\cdots ,t_{m}\right) $%
.

\begin{definition}
Let $G\subseteq \Sigma _{m}$ be a permutation group on $m$ letters.

$\left( i\right) $ An element $\sigma \in G$ induces a $\sigma $\textbf{-map
}(relative to the variables $t_{1},t_{2},\cdots ,t_{m}$), saying $\sigma
_{\ast }:K\left( t_{1},t_{2},\cdots ,t_{m}\right) \rightarrow K\left(
t_{1},t_{2},\cdots ,t_{m}\right) $, given by
\begin{equation*}
\xi =\frac{f\left( t_{1},t_{2},\cdots ,t_{m}\right) }{g\left(
t_{1},t_{2},\cdots ,t_{m}\right) }\longmapsto \sigma _{\ast }\left( \xi
\right) =\frac{f_{\sigma }\left( t_{1},t_{2},\cdots ,t_{m}\right) }{%
g_{\sigma }\left( t_{1},t_{2},\cdots ,t_{m}\right) }
\end{equation*}%
where
\begin{equation*}
\frac{f_{\sigma }\left( t_{1},t_{2},\cdots ,t_{m}\right) }{g_{\sigma }\left(
t_{1},t_{2},\cdots ,t_{m}\right) }:=\frac{f\left( t_{\sigma \left( 1\right)
},t_{\sigma \left( 2\right) },\cdots ,t_{\sigma \left( m\right) }\right) }{%
g\left( t_{\sigma \left( 1\right) },t_{\sigma \left( 2\right) },\cdots
,t_{\sigma \left( m\right) }\right) };
\end{equation*}%
\begin{equation*}
f\left( X_{1},X_{2},\cdots ,X_{m}\right) ,g\left( X_{1},X_{2},\cdots
,X_{m}\right) \in K\left[ X_{1},X_{2},\cdots ,X_{m}\right]
\end{equation*}%
are polynomials with $g\left( X_{1},X_{2},\cdots ,X_{m}\right) \not=0$.

For simplicity, we identity $\sigma _{\ast }$ with $\sigma $ by the
isomorphism
\begin{equation*}
\ast :G\rightarrow \ast \left( G\right), \sigma \mapsto \sigma _{\ast }
\end{equation*}
of groups.

$\left( ii\right) $ Via the isomorphism $\ast $, $G$ is taken as a subgroup
of the Galois group $Aut\left( L/K\right) $, called a \textbf{permutation
subgroup} of the Galois group $Aut\left( L/K\right) $ on $m$ letters
(relative to the variables $t_{1},t_{2},\cdots ,t_{m}$). In particular, the
symmetric group $\Sigma _{m}$ is called the \textbf{full permutation subgroup%
} of the Galois group $Aut\left( L/K\right) $ on $m$ letters (relative to
the variables $t_{1},t_{2},\cdots ,t_{m}$).
\end{definition}

\begin{remark}
Let $G$ be a permutation group on $m$ letters. For each $\sigma \in G$, the $%
\sigma $-map (relative to the variables $t_{1},t_{2},\cdots ,t_{m}$), also
denoted by $\sigma $, preserves the algebraic operations in the function
field $K\left( t_{1},t_{2},\cdots ,t_{m}\right) $.

In deed, for any $c\in K$ and $f,g\in K\left( t_{1},t_{2},\cdots
,t_{m}\right) $ there are
\begin{equation*}
\begin{array}{l}
\sigma \left( c\cdot f\right) =c\cdot \sigma \left( f\right) ; \\
\sigma \left( f\cdot g\right) =\sigma \left( f\right) \cdot \sigma \left(
g\right) ; \\
\sigma \left( f+g\right) =\sigma \left( f\right) +\sigma \left( g\right) ;
\\
\sigma \left( \frac{f}{g}\right) =\frac{\sigma \left( f\right) }{\sigma
\left( g\right) }\text{ whenever }g\not=0.%
\end{array}%
\end{equation*}
\end{remark}

\begin{definition}
($G$\emph{-symmetric functions}) Let $G$ be a permutation subgroup of $%
Aut\left( L/K\right) $ on $m$ letters (relative to the variables $%
t_{1},t_{2},\cdots ,t_{m}$). An element
\begin{equation*}
\xi =\frac{f\left( t_{1},t_{2},\cdots ,t_{m}\right) }{g\left(
t_{1},t_{2},\cdots ,t_{m}\right) }
\end{equation*}%
of the function field $K\left( t_{1},t_{2},\cdots ,t_{m}\right) $ is said to
be $G$\textbf{-symmetric} (relative to the variables $t_{1},t_{2},\cdots
,t_{m}$) if for any $\sigma \in G$ there is
\begin{equation*}
\frac{f\left( t_{1},t_{2},\cdots ,t_{m}\right) }{g\left( t_{1},t_{2},\cdots
,t_{m}\right) }=\frac{f_{\sigma }\left( t_{1},t_{2},\cdots ,t_{m}\right) }{%
g_{\sigma }\left( t_{1},t_{2},\cdots ,t_{m}\right) }
\end{equation*}%
where
\begin{equation*}
f\left( X_{1},X_{2},\cdots ,X_{m}\right) ,g\left( X_{1},X_{2},\cdots
,X_{m}\right) \in K\left[ X_{1},X_{2},\cdots ,X_{m}\right]
\end{equation*}%
are polynomials with $g\left( X_{1},X_{2},\cdots ,X_{m}\right) \not=0$.

We denote by
\begin{equation*}
\Lambda ^{G}\left( K\left( t_{1},t_{2},\cdots ,t_{m}\right) \right)
\end{equation*}%
the set of $G$-symmetric elements in the function field $K\left(
t_{1},t_{2},\cdots ,t_{m}\right) $.
\end{definition}

\begin{remark}
Let $G$ be a permutation subgroup of $Aut\left( L/K\right) $ on $m$ letters
(relative to the variables $t_{1},t_{2},\cdots ,t_{m}$).

$\left( i\right) $ Take an element
\begin{equation*}
\xi =\frac{f\left( t_{1},t_{2},\cdots ,t_{m}\right) }{g\left(
t_{1},t_{2},\cdots ,t_{m}\right) }\in K\left( t_{1},t_{2},\cdots
,t_{m}\right)
\end{equation*}
where
\begin{equation*}
f\left( X_{1},X_{2},\cdots ,X_{m}\right) ,g\left( X_{1},X_{2},\cdots
,X_{m}\right) \in K\left[ X_{1},X_{2},\cdots ,X_{m}\right]
\end{equation*}
are polynomials with $g\left( X_{1},X_{2},\cdots ,X_{m}\right) \not=0$.

If $f\left( X_{1},X_{2},\cdots ,X_{m}\right) $ and $g\left(
X_{1},X_{2},\cdots ,X_{m}\right) $ are $G$-symmetric polynomials, then $\xi $
is a $G$-symmetric element.

$\left( ii\right) $ The set $\Lambda ^{G}\left( K\left( t_{1},t_{2},\cdots
,t_{m}\right) \right) $ of $G$-symmetric elements (relative to the variables
$t_{1},t_{2},\cdots ,t_{m}$) in $K\left( t_{1},t_{2},\cdots ,t_{m}\right) $
is a subfield of the field $K\left( t_{1},t_{2},\cdots ,t_{m}\right) $.
\end{remark}

\subsection{The main theorem on the relations between permutation subgroups
and Noether solutions}

Let $L:=K\left( t_{1},t_{2},\cdots ,t_{m}\right) $ be a purely
transcendental extension over a field $K$ of transcendence degree $m$.

\begin{definition}
Let $G$ be a subgroup of the Galois group $Aut\left( L/K\right) $. If for
any $1\leq i<j\leq m$, there is an element $\sigma \in G$ such that $\sigma
\left( t_{i}\right) =t_{j}$ holds, then $G$ is said to \textbf{have a
transitive action on the variables} $t_{1},t_{2},\cdots ,t_{m}$.
\end{definition}

Let $\Sigma _{m}$ denote the full permutation subgroup of the Galois group $%
Aut\left( L/K\right) $ on $m$ letters (relative to the variables $%
t_{1},t_{2},\cdots ,t_{m}$). The full permutation subgroup $\Sigma _{m}$ has
a transitive action on the variables $t_{1},t_{2},\cdots ,t_{m}$.

Now we have the following theorem on the subfield of symmetric functions
under a permutation subgroup of the Galois group for a given purely
transcendental extension, which is one of the main theorems in the paper.
Here, $\sharp G$ denotes the order of a finite group $G$.

\begin{theorem}
\emph{(Main Theorem: Permutation subgroups and Noether solutions)} Let $%
L=K\left( t_{1},t_{2},\cdots ,t_{m}\right) $ be a purely transcendental
extension over a field $K$ of finite transcendence degree $m$. Let $G$ be a
permutation subgroup of $\Sigma _{m}\subseteq Aut\left( L/K\right) $ on $m$
letters (relative to the variables $t_{1},t_{2},\cdots ,t_{m}$).

Consider the set $\Lambda ^{G}\left( L\right) $ of $G$-symmetric elements
(relative to $t_{1},t_{2},\cdots ,t_{m}$) in the function field $L$. There
are the following statements.

$\left( i\right) $ For any $m\geq 1$, $L$ is algebraic Galois over the
subfield $\Lambda ^{G}\left( L\right) $ such that
\begin{equation*}
\left[ L:\Lambda ^{G}\left( L\right) \right] =\sharp G
\end{equation*}%
and
\begin{equation*}
Aut\left( L/\Lambda ^{G}\left( L\right) \right) =G
\end{equation*}%
hold, i.e., $\Lambda ^{G}\left( L\right) $ is the $G$-invariant subfield $%
L^{G}$ of $L$ under the subgroup $G$.

$\left( ii\right) $ Let $m\geq 1$. Then the least nice number of $\Lambda
^{G}\left( L\right) $ over $K$ satisfies the equality
\begin{equation*}
\mathfrak{ni}\left( \frac{\Lambda ^{G}\left( L\right) }{K}\right) =0
\end{equation*}%
if and only if $G$ is the full permutation subgroup $\Sigma _{m}$ of the
Galois group $Aut\left( L/K\right) $ on $m$ letters (relative to $%
t_{1},t_{2},\cdots ,t_{m}$).

In such a case, the subfield $\Lambda ^{G}\left( L\right) $ is purely
transcendental over $K$, i.e., $G=\Sigma _{m}$ is a Noether solution of $L/K$%
.

$\left( iii\right) $ Suppose $m\geq 3$ and $\{1\}\subsetneqq G\subsetneqq
\Sigma _{m}$. Then there is an inequality
\begin{equation*}
\mathfrak{ni}\left( \frac{\Lambda ^{G}\left( L\right) }{K}\right) \geq 1
\end{equation*}%
for the least nice number of $\Lambda ^{G}\left( L\right) $ over $K$.

In such a case, the subfield $\Lambda ^{G}\left( L\right) $ is not purely
transcendental over $K$, i.e., $G$ is not a Noether solution of $L/K$.
\end{theorem}

We will prove \emph{Theorem 7.12} in \emph{\S 7.6} after we obtain the
following preparatory lemmas on subfields of $G$-symmetric functions under
permutation subgroups in the Galois group of the extension $L=K\left(
t_{1},t_{2},\cdots ,t_{m}\right) $ over $K$.

\subsection{Lemmas on subfields of $G$-symmetric functions under
permutation subgroups}

Let $t_{1},t_{2},\cdots ,t_{m}$ be a set of algebraically independent
variables over a field $K$. Put $L:=K\left( t_{1},t_{2},\cdots ,t_{m}\right)
$.

Let $\Sigma _{m}$ be the full permutation subgroup of the Galois group $%
Aut\left( L/K\right) $ on $m$ letters (relative to the variables $%
t_{1},t_{2},\cdots ,t_{m}$) (See \emph{Definition 7.7}). Denote by $\Lambda
^{\Sigma _{m}}\left( L\right) $ the set of $\Sigma _{m}$-symmetric functions
(relative to the variables $t_{1},t_{2},\cdots ,t_{m}$) in the function
field $L=K\left( t_{1},t_{2},\cdots ,t_{m}\right) $.

Suppose $G\subseteq \Sigma _{m}$ is a subgroup which has a transitive action
on the variables $t_{1},t_{2},\cdots ,t_{m}$. Denote by $\Lambda ^{G}\left(
L\right) $ the set of $G$-symmetric functions (relative to the variables $%
t_{1},t_{2},\cdots ,t_{m}$) in the function field $L=K\left(
t_{1},t_{2},\cdots ,t_{m}\right) $.

\begin{lemma}
\emph{(}$G$\emph{-symmetric functions as }$G$\emph{-invariant elements)} Let
$\Lambda ^{G}\left( L\right) $ be the set of $G$-symmetric functions
(relative to the variables $t_{1},t_{2},\cdots ,t_{m}$) in the function
field $L=K\left( t_{1},t_{2},\cdots ,t_{m}\right) $. Then $\Lambda
^{G}\left( L\right) $ is the $G$-invariant subfield $L^{G}$ of $L$ under the
subgroup $G$, i.e., we have
\begin{equation*}
\Lambda ^{G}\left( L\right) =L^{G}.
\end{equation*}

In particular, for the case $G=\Sigma _{m}$, let $\Lambda ^{\Sigma
_{m}}\left( L\right) $ be the set of $\Sigma _{m}$-symmetric functions
(relative to $t_{1},t_{2},\cdots ,t_{m}$) in the function field $L$. Then $%
\Lambda ^{\Sigma _{m}}\left( L\right) $ is the $\Sigma _{m}$-invariant
subfield $L^{\Sigma _{m}}$ of $L$ under the subgroup $\Sigma _{m}$, i.e.,
\begin{equation*}
\Lambda ^{\Sigma _{m}}\left( L\right) =L^{\Sigma _{m}}.
\end{equation*}
\end{lemma}

\begin{proof}
Fixed any $\sigma \in G$ and $\xi \in L=K\left( t_{1},t_{2},\cdots
,t_{m}\right) $. From \emph{Definition 7.7} we have
\begin{equation*}
\sigma \left( \xi \right) =\frac{f\left( t_{\sigma \left( 1\right)
},t_{\sigma \left( 2\right) },\cdots ,t_{\sigma \left( m\right) }\right) }{%
g\left( t_{\sigma \left( 1\right) },t_{\sigma \left( 2\right) },\cdots
,t_{\sigma \left( m\right) }\right) }
\end{equation*}%
where
\begin{equation*}
\xi =\frac{f\left( t_{1},t_{2},\cdots ,t_{m}\right) }{g\left(
t_{1},t_{2},\cdots ,t_{m}\right) };
\end{equation*}%
\begin{equation*}
f\left( X_{1},X_{2},\cdots ,X_{m}\right) ,g\left( X_{1},X_{2},\cdots
,X_{m}\right) \in K\left[ X_{1},X_{2},\cdots ,X_{m}\right]
\end{equation*}%
are polynomials with $g\left( X_{1},X_{2},\cdots ,X_{m}\right) \not=0$.
Hence, we have
\begin{equation*}
\sigma \left( \xi \right) =\xi
\end{equation*}%
if and only if
\begin{equation*}
\frac{f\left( t_{\sigma \left( 1\right) },t_{\sigma \left( 2\right) },\cdots
,t_{\sigma \left( m\right) }\right) }{g\left( t_{\sigma \left( 1\right)
},t_{\sigma \left( 2\right) },\cdots ,t_{\sigma \left( m\right) }\right) }=%
\frac{f\left( t_{1},t_{2},\cdots ,t_{m}\right) }{g\left( t_{1},t_{2},\cdots
,t_{m}\right) }
\end{equation*}%
holds. This completes the proof.
\end{proof}

\begin{lemma}
\emph{(The full permutation subgroup and Noether solution)} Let $\Lambda ^{\Sigma _{m}}\left( L\right) $ be the subfield of $\Sigma _{m}$%
-symmetric functions (relative to $t_{1},t_{2},\cdots ,t_{m}$) in the
function field $L=K\left( t_{1},t_{2},\cdots ,t_{m}\right) $. Here, $m\geq 1$%
. There are the following statements.

$\left( i\right) $ $\Lambda ^{\Sigma _{m}}\left( L\right) $ is
purely transcendental  over $K$ of transcendence degree $m$.

$\left( ii\right) $ $L$ is algebraic Galois over the subfield $\Lambda
^{\Sigma _{m}}\left( L\right) $ of degree
\begin{equation*}
\left[ L:\Lambda ^{\Sigma _{m}}\left( L\right) \right] =\sharp \Sigma
_{m}=m!.
\end{equation*}%
In particular, $\Sigma _{m}$ is a Noether solution of $L/K$.

$\left( iii\right) $ Fixed the elementary symmetric polynomials
\begin{equation*}
e_{1}\left( X_{1},X_{2},\cdots ,X_{m}\right) ,e_{2}\left( X_{1},X_{2},\cdots
,X_{m}\right) ,\cdots ,e_{m}\left( X_{1},X_{2},\cdots ,X_{m}\right)
\end{equation*}%
of $m$ indeterminates $X_{1},X_{2},\cdots ,X_{m}$ over $K$.

Then the subfield $\Lambda ^{\Sigma _{m}}\left( L\right) $ is generated over
$K$ by the evaluations
\begin{equation*}
e_{1}\left( t_{1},t_{2},\cdots ,t_{m}\right) ,e_{2}\left( t_{1},t_{2},\cdots
,t_{m}\right) ,\cdots ,e_{m}\left( t_{1},t_{2},\cdots ,t_{m}\right)
\end{equation*}%
of the elementary symmetric polynomials at the variables $t_{1},t_{2},\cdots
,t_{m}\in L$, i.e.,
\begin{equation*}
\begin{array}{l}
\Lambda ^{\Sigma _{m}}\left( K\left( t_{1},t_{2},\cdots ,t_{m}\right) \right)
\\
=K\left( e_{1}\left( t_{1},t_{2},\cdots ,t_{m}\right) ,e_{2}\left(
t_{1},t_{2},\cdots ,t_{m}\right) ,\cdots ,e_{m}\left( t_{1},t_{2},\cdots
,t_{m}\right) \right) .%
\end{array}%
\end{equation*}
\end{lemma}

\begin{proof}
Consider the subring $A=K\left[ t_{1},t_{2},\cdots ,t_{m}\right] $ of $%
L=K\left( t_{1},t_{2},\cdots ,t_{m}\right) $. Let $B=\Lambda ^{\Sigma
_{m}}\left( A\right) $ be the subring of $\Sigma _{m}$-symmetric elements
(relative to the variables $t_{1},t_{2},\cdots ,t_{m}$) in $L$. From \emph{%
Definition 7.9} it is seen that $\Lambda ^{\Sigma _{m}}\left( L\right) $ is
the fractional field of $B$, i.e., $\Lambda ^{\Sigma _{m}}\left( L\right)
=Fr\left( B\right) \subseteq Fr\left( A\right) =L$. Put%
\begin{equation*}
e_{1}:=e_{1}\left( t_{1},t_{2},\cdots ,t_{m}\right) ,e_{2}:=e_{2}\left(
t_{1},t_{2},\cdots ,t_{m}\right) ,\cdots ,e_{m}:=e_{m}\left(
t_{1},t_{2},\cdots ,t_{m}\right) .
\end{equation*}%
We will proceed in several steps.

\emph{Step 1}. From $\left( i\right) -\left( ii\right) $ of \emph{Lemma 7.6}
it is seen that
\begin{equation*}
\Lambda ^{\Sigma _{m}}\left( L\right) =K\left( e_{1},e_{2},\cdots
,e_{m}\right)
\end{equation*}%
is generated over $K$ by the $\Sigma _{m}$-symmetric functions
\begin{equation*}
e_{1},e_{2},\cdots ,e_{m}
\end{equation*}%
in $L$ which are algebraically independent over $K$. It is seen that $%
\Lambda ^{\Sigma _{m}}\left( L\right) $ is a purely transcendental extension
over $K$ of transcendence degree $m$. This proves $\left( i\right) $ \& $%
\left( iii\right) $.

\emph{Step 2}. From $\left( iii\right) $ of \emph{Lemma 7.6} we have
elements
\begin{equation*}
f_{1},f_{2},\cdots ,f_{m!}\in A
\end{equation*}%
such that
\begin{equation*}
A=span_{B}\{f_{1},f_{2},\cdots ,f_{m!}\}
\end{equation*}%
holds. That is, $f_{1},f_{2},\cdots ,f_{m!}$ are a $B$-linear basis of the
free $B$-module $A$. It follows that
\begin{equation*}
A=B\left[ f_{1},f_{2},\cdots ,f_{m!}\right]
\end{equation*}%
is a ring generated over the subring $B$ by the elements $f_{1},f_{2},\cdots
,f_{m!}$. Hence, $L=Fr\left( A\right) $ is a field generated over the
subfield
\begin{equation*}
\Lambda ^{\Sigma _{m}}\left( L\right) =Fr\left( B\right)
\end{equation*}%
by a finite number of elements
\begin{equation*}
f_{1},f_{2},\cdots ,f_{m!};
\end{equation*}%
$L=Fr\left( A\right) $ is a vector space over the subfield $Fr\left(
B\right) $ generated by the elements $f_{1},f_{2},\cdots ,f_{m!}$.

\emph{Step 3}. Prove $f_{1},f_{2},\cdots ,f_{m!}$ are linearly independent
over the subfield $Fr\left( B\right) $. Hypothesize $f_{1},f_{2},\cdots
,f_{m!}$ are not linearly independent over $Fr\left( B\right) $. For
simplicity, let $h=m!$. Without loss of generality, suppose
\begin{equation*}
f_{h}=c_{1}\cdot f_{1}+c_{2}\cdot f_{2}+\cdots +c_{h-1}\cdot f_{h-1}
\end{equation*}
such that
\begin{equation*}
c_{1},c_{2},\cdots ,c_{h-1}\in Fr\left( B\right)
\end{equation*}
are not all zero.

By taking the common denominator $d\in B$ of the fractions $%
c_{1},c_{2},\cdots ,c_{h-1}$, we have
\begin{equation*}
d\cdot f_{h}=d_{1}\cdot f_{1}+d_{2}\cdot f_{2}+\cdots +d_{h-1}\cdot f_{h-1}
\end{equation*}
where
\begin{equation*}
d_{1}=d\cdot c_{1},d_{2}=d\cdot c_{2},\cdots ,d_{h-1}=d\cdot c_{h-1}\in B.
\end{equation*}
Then
\begin{equation*}
d_{1}\cdot f_{1}+d_{2}\cdot f_{2}+\cdots +d_{h-1}\cdot f_{h-1}-d\cdot
f_{h}=0.
\end{equation*}

On the other hand,
\begin{equation*}
d_{1},d_{2},\cdots ,d_{h-1},-d\in B
\end{equation*}
are not all zero; it follows that $f_{1},f_{2},\cdots ,f_{h}\in A$ are not
linearly independent over the subring $B$, which is in contradiction. This
proves $L$ is a vector space over $Fr\left( B\right) $ of dimension $m!$.

\emph{Step 4}. From \emph{Lemma 7.13} it is seen that $\Lambda ^{\Sigma
_{m}}\left( L\right) =Fr\left( B\right) $ is the invariant subfield $%
L^{\Sigma _{m}}$ of $L$ under the subgroup $\Sigma _{m}\subseteq Aut\left(
L/K\right) $ and that
\begin{equation*}
\Sigma _{m}=Aut\left( L/L^{\Sigma _{m}}\right)
\end{equation*}%
is the Galois group. Hence, $L$ is algebraic Galois over the subfield $%
L^{\Sigma _{m}}=\Lambda ^{\Sigma _{m}}\left( L\right) $ since there are
equalities
\begin{equation*}
\sharp \Sigma _{m}=\left[ L:\Lambda ^{\Sigma _{m}}\left( L\right) \right] =m!
\end{equation*}%
from \emph{Step 3}. This completes the proof.
\end{proof}

\begin{lemma}
\emph{(}[Suzuki 1982]\emph{)} \emph{(Properties for alternating subgroups)} Let $%
\Sigma _{m}$ be the full permutation subgroup of the Galois group $Aut\left(
L/K\right) $ on $m$ letters (relative to the variables $t_{1},t_{2},\cdots
,t_{m}$). Let $A_{m}$ denote the alternating subgroup $A_{m}$ of $\Sigma
_{m} $. There are the following statements.

$\left( i\right) $ Let $m\geq 3$. Then $A_{m}$ is a characteristic subgroup
of $\Sigma _{m}$ and $A_{m}$ is generated in $\Sigma _{m}$ by all the $3$%
-cycles.

$\left( ii\right) $ For $m=4$, the normal subgroups of $\Sigma _{4}$ are $%
\{1\},A_{4},\Sigma _{4}$.

$\left( iii\right) $ Let $m\geq 5$. If $G\not=\{1\}$ is a normal subgroup of
$\Sigma _{m}$, then $A_{m}$ must be contained in $G$.

$\left( iv\right) $ For any $1\leq m\not=4$, $A_{m}$ is the unique minimal
nomal subgroup of $\Sigma _{m}$.
\end{lemma}

\begin{proof}
$\left( i\right) -\left( iv\right) $ See \emph{Chapter 3, \S 2, Pages 291-295%
} in \emph{[}Suzuki 1982\emph{]}.
\end{proof}

\begin{lemma}
\emph{(The cyclic permutation subgroup and Noether solution)} Let $C_{m}$
denote the cyclic permutation subgroup of $\Sigma _{m}\subseteq Aut\left(
L/K\right) $ (relative to the variables $t_{1},t_{2},\cdots ,t_{m}$) of
order $m$, i.e., $C_{m}$ is induced from the cyclic subgroup given by the $m$-cycle
$
\left( 1,2,\cdots ,m\right)
$
in the symmetric group on $m$ letters. Here, $m\geq 1$.

Consider the set $\Lambda ^{C_{m}}\left( L\right) $ of $C_{m}$-symmetric
functions (relative to $t_{1},t_{2},\cdots ,t_{m}$) in the function field $L$%
. There are the following statements.

$\left( i\right) $ $C_{m}$ is a cyclic subgroup in $Aut\left( L/K\right) $
of order $\sharp C_{m}=m$; $C_{m}$ has a transitive action on $%
t_{1},t_{2},\cdots ,t_{m}$.

$\left( ii\right) $ $L$ is algebraic Galois over the subfield $\Lambda
^{C_{m}}\left( L\right) $. In particular, we have%
\begin{equation*}
\Lambda ^{C_{m}}\left( L\right) =L^{C_{m}};
\end{equation*}%
\begin{equation*}
C_{m}=Aut\left( L/\Lambda ^{C_{m}}\left( L\right) \right) ;
\end{equation*}
\begin{equation*}
\left[ L:\Lambda ^{C_{m}}\left( L\right) \right] =\sharp C_{m}=m.
\end{equation*}

$\left( iii\right) $ The variables $t_{1},t_{2},\cdots ,t_{m}\in L$ are $%
\Lambda ^{C_{m}}\left( L\right) $-conjugate elements in $L$.

$\left( iv\right) $ The minimal polynomial $\phi \left( X\right) $ of the
variables $t_{1},t_{2},\cdots ,t_{m}\in L$ over $\Lambda ^{C_{m}}\left(
L\right) $ is given by%
\begin{equation*}
\phi \left( X\right) :=\prod_{1\leq i\leq m}\left( X-t_{i}\right)
=\sum_{0\leq k\leq m}\left( -1\right) ^{k}\cdot e_{k}\left(
t_{1},t_{2},\cdots ,t_{m}\right) \cdot X^{m-k}
\end{equation*}%
where
\begin{equation*}
e_{0}\left( t_{1},t_{2},\cdots ,t_{m}\right) :=1;
\end{equation*}%
\begin{equation*}
e_{1}\left( X_{1},X_{2},\cdots ,X_{m}\right) ,e_{2}\left( X_{1},X_{2},\cdots
,X_{m}\right) ,\cdots ,e_{m}\left( X_{1},X_{2},\cdots ,X_{m}\right)
\end{equation*}%
are the elementary symmetric polynomials of $m$ indeterminates $%
X_{1},X_{2},\cdots ,X_{m}$ over $K$.

$\left( v\right) $ If $m=1$, then $C_{m}=\Sigma _{1}$ and for the least nice
number there is
\begin{equation*}
\mathfrak{ni}\left( \frac{\Lambda ^{C_{m}}\left( L\right) }{K}\right) =0.
\end{equation*}

In such a case, $\Lambda ^{C_{1}}\left( L\right) =L$ is purely
transcendental over $K$, i.e., $C_{1}$ is a Noether solution of $L/K$.

$\left( vi\right) $ If $m=2$, then $C_{m}=\Sigma _{2}$ and $L$ is algebraic
Galois over the subfield $\Lambda ^{C_{m}}\left( L\right) $ such that
\begin{equation*}
\mathfrak{ni}\left( \frac{\Lambda ^{C_{m}}\left( L\right) }{K}\right) =0.
\end{equation*}

In such a case, $\Lambda ^{C_{2}}\left( L\right) $ is purely transcendental
over $K$, i.e., $C_{2}$ is a Noether solution of $L/K$.

$\left( vii\right) $ If $m\geq 3$, then $\{1\}\subsetneqq C_{m}\subsetneqq
\Sigma _{m}$ and there is an inequality
\begin{equation*}
\mathfrak{ni}\left( \frac{\Lambda ^{C_{m}}\left( L\right) }{K}\right) \geq 1
\end{equation*}%
for the least nice number of the subfield $\Lambda ^{C_{m}}\left( L\right) $
over $K$.

In such a case, $\Lambda ^{C_{m}}\left( L\right) $ is not purely
transcendental over $K$, i.e., $C_{m}$ is not a Noether solution of $L/K$.
\end{lemma}

\begin{proof}
$\left( i\right) -\left( iii\right) $ Trivial from elementary calculations
on polynomials over a field.

$\left( v\right) -\left( vi\right) $ Immediately from the above $\left(
i\right) -\left( ii\right) $ of \emph{Lemma 7.14}.

$\left( vii\right) $ Let $m\geq 3$. It is seen that $C_{m}$ is a proper
subgroup of the full permutation subgroup $\Sigma _{m}$ of the Galois group $%
Aut\left( L/K\right) $ since there is at least one element, called \emph{%
transposition}, contained in the difference $\Sigma _{m}\setminus C_{m}$
(See basic facts on permutation subgroups in any standard textbook, for
instance, \emph{[}Suzuki 1982\emph{]}).

Hypothesize the subfield $\Lambda ^{C_{m}}\left( L\right) $ is a purely
transcendental extension over $K$. It is seen that the subgroup $%
C_{m}=Aut\left( L/\Lambda ^{C_{m}}\left( L\right) \right) $ is a Noether
solution of $L/K$ since $\Lambda ^{C_{m}}\left( L\right) =L^{C_{m}}$ is the $%
C_{m}$-invariant subfield of $L$ from the above $\left( ii\right) $.

By applying \emph{Theorem 6.4} to the case for the subgroup $C_{m}$ here, it
is seen that $C_{m}$ is Galois-complemented in $L/K$ and there is a unique
natural Galois-complement $H_{C_{m}}$ of $C_{m}$ in $L/K$. Put
\begin{equation*}
Q:=C_{m}\cap H_{C_{m}}.
\end{equation*}%
From \emph{Lemma 6.26} on the quotients of Noether solutions, we have
\begin{equation*}
\Sigma _{m}=C_{m}\cdot Q=Q\cdot C_{m}
\end{equation*}%
where $C_{m}$ and $Q$ are normal subgroups of $\Sigma _{m}$ for any $m\geq 3$%
.

Now consider the cyclic permutation group $C_{m}$ for $m\geq 3$. There are
three cases for $m\geq 3$.

\emph{Case }$\left( i\right) $ Assume $m\geq 5$. From $\left( iii\right) $
of \emph{Lemma 7.15}\ it is seen that $C_{m}$ and $Q$ both contain the
alternating subgroup $A_{m}$ of $\Sigma _{m}$; then
\begin{equation*}
\{1\}\subsetneqq A_{m}\subseteq C_{m}\cap Q\subseteq C_{m}\cap H_{C_{m}},
\end{equation*}%
which is in contradiction with the condition $\left( N1\right) $ in \emph{%
Definition 6.17} that
\begin{equation*}
C_{m}\cap H_{C_{m}}=\{1\}
\end{equation*}%
Holds.

\emph{Case }$\left( ii\right) $ If $m=4$, from $\left( ii\right) $ of \emph{%
Lemma 7.15}\ it is seen that $Q$ must be one of the normal subgroups $%
\{1\},A_{4},\Sigma _{4}$ in $\Sigma _{4}$; then the order $\sharp Q$ of $Q$
is one of the integers $1,12,24$, which is in contradiction with the fact
\begin{equation*}
\sharp Q=\frac{\sharp \Sigma _{4}}{\sharp C_{4}}=\frac{4!}{4}=6.
\end{equation*}

\emph{Case }$\left( iii\right) $ Suppose $m=3$. From $\left( iv\right) $ of
\emph{Lemma 7.15}\ it is seen that the alternating subgroup $A_{3}$ is
contained in the normal subgroup $Q$ of $\Sigma _{3}$; then for the order $%
\sharp Q$ of $Q$ we have
\begin{equation*}
\sharp Q\geq \sharp A_{3}=\frac{1}{2}3!=3,
\end{equation*}%
which is in contradiction with the fact
\begin{equation*}
\sharp Q=\frac{\sharp \Sigma _{3}}{\sharp C_{3}}=\frac{3!}{3}=2.
\end{equation*}

Hence, the subfield $\Lambda ^{C_{m}}\left( L\right) $ is not purely
transcendental over $K$, i.e., the least nice number of the extension $%
\Lambda ^{C_{m}}\left( L\right) $ over $K$ has a lower bound at least $1$.
This completes the proof.
\end{proof}

\begin{lemma}
\emph{(The alternating subgroup and Noether solution)} Let $A_{m}$ be the
alternating subgroup of $\Sigma _{m}\subseteq Aut\left( L/K\right) $
(relative to the variables $t_{1},t_{2},\cdots ,t_{m}$), i.e., $A_{m}$ is
induced from the alternating subgroup in the symmetric group on $m$ letters.
Here, $m\geq 1$.

Consider the set $\Lambda ^{A_{m}}\left( L\right) $ of $A_{m}$-symmetric
functions (relative to $t_{1},t_{2},\cdots ,t_{m}$) in the function field $L$%
. There are the following statements.

$\left( i\right) $ If $1\leq m\not=2$, then $A_{m}$ has a transitive action
on $t_{1},t_{2},\cdots ,t_{m}$.

$\left( ii\right) $ $L$ is algebraic Galois over the subfield $\Lambda
^{A_{m}}\left( L\right) $. In particular, we have%
\begin{equation*}
\Lambda ^{A_{m}}\left( L\right) =L^{A_{m}};
\end{equation*}%
\begin{equation*}
A_{m}=Aut\left( L/\Lambda ^{A_{m}}\left( L\right) \right) ;
\end{equation*}%
\begin{equation*}
\left[ L:\Lambda ^{A_{m}}\left( L\right) \right] =\sharp A_{m}=\frac{1}{2}m!.
\end{equation*}

$\left( iii\right) $ If $m=1$, then $A_{m}=\Sigma _{1}$ and $\Lambda
^{A_{m}}\left( L\right) =L$ is purely transcendental over $K$; $A_{1}$ is a
Noether solution of $L/K$.

$\left( iv\right) $ If $m=2$, then $A_{m}=\{1\}$ and $\Lambda ^{A_{m}}\left(
L\right) =L$ is purely transcendental over $K$, i.e., $A_{2}$ is a Noether
solution of $L/K$.

$\left( v\right) $ If $m\geq 3$, then $\{1\}\subsetneqq A_{m}\subsetneqq
\Sigma _{m}$ and there is an inequality
\begin{equation*}
\mathfrak{ni}\left( \frac{\Lambda ^{A_{m}}\left( L\right) }{K}\right) \geq 1
\end{equation*}%
for the least nice number of the subfield $\Lambda ^{A_{m}}\left( L\right) $
over $K$.

In such a case, $\Lambda ^{A_{m}}\left( L\right) $ is not purely
transcendental over $K$, i.e., $A_{m}$ is not a Noether solution of $L/K$.
\end{lemma}

\begin{proof}
$\left( i\right) $ Assume $1\leq m\not=2$. We proceed in two cases.

\emph{Case} $\left( i\right) $ Let $m=0$ or $1$ $\left(\text{mod } 4\right) $%
. Consider the natural linear involution $\sigma _{m}$ of $L/K$ (relative to
$t_{1},t_{2},\cdots ,t_{m}$) (See \emph{Definition 8.1}). We have
\begin{equation*}
\mathfrak{sgn}\left( \sigma _{m}\right) =\frac{1}{2}m\left( m-1\right)
\end{equation*}%
for the sign of the $m$-permutation $\sigma _{m}\in \Sigma _{m}$. We have $%
\sigma _{m}\in A_{m}$ and the single $\sigma _{m}$ has a transitive action
on $t_{1},t_{2},\cdots ,t_{m}$. So does $A_{m}$.

\emph{Case} $\left( ii\right) $ Let $m=2$ or $3$ $\left(\text{mod } 4\right)
$. By taking the $3$-cycle $\left( m-2,m-1,m\right) $ in $\Sigma _{m}$ and
the natural linear involution $\sigma _{m-1}$ of $K\left( t_{1},t_{2},\cdots
,t_{m-1}\right) $ over (relative to $t_{1},t_{2},\cdots ,t_{m-1}$) with $%
\sigma _{m-1}$ fixing $t_{m}$, from \emph{Case} $\left( i\right) $\ it is
seen that $A_{m}$ has a transitive action on $t_{1},t_{2},\cdots ,t_{m}$.

$\left( ii\right) $ Immediately from \emph{Lemma 7.13}.

$\left( iii\right) -\left( iv\right) $ Trivial.

$\left( v\right) $ Let $m\geq 3$. Hypothesise $A_{m}$ is a Noether solution
of $L/K$. From \emph{Lemma 6.26} we have a subgroup $Q=\Sigma _{m}\cap
H_{A_{m}}$ such that%
\begin{equation*}
\Sigma _{m}=A_{m}\cdot Q=Q\cdot A_{m}
\end{equation*}%
hold, where $H_{A_{m}}$ is the natural Galois-complement of $A_{m}$ in $L/K$
and $Q$ is also a normal subgroup in $\Sigma _{m}$ of order%
\begin{equation*}
\sharp Q=\frac{\sharp \Sigma _{m}}{\sharp A_{m}}=2.
\end{equation*}

On the other hand, if $m=4$, then rom $\left( iii\right) $ of \emph{Lemma
7.15} we have
\begin{equation*}
\sharp Q\in \{1,12,24\},
\end{equation*}%
which is in contradiction.

If $m=3$ or $m\geq 5$, then rom $\left( iv\right) $ of \emph{Lemma 7.15} we
have
\begin{equation*}
\sharp Q\geq \sharp A_{m}=\frac{1}{2}m!\geq 3,
\end{equation*}%
which is in contradiction.

Therefore, the subfield $\Lambda ^{A_{m}}\left( L\right) $ is not purely
transcendental over $K$, i.e., the least nice number of the extension $%
\Lambda ^{A_{m}}\left( L\right) $ over $K$ has a lower bound at least $1$.
This completes the proof.
\end{proof}

\begin{lemma}
\emph{(Any permutation subgroup and Noether solution)} Suppose
\begin{equation*}
\{1\}\subseteq G\subseteq \Sigma _{m}
\end{equation*}%
is a permutation subgroup in $Aut\left( L/K\right) $ (relative to $%
t_{1},t_{2},\cdots ,t_{m}$). Here, $m\geq 1$.

Consider the set $\Lambda ^{G}\left( L\right) $ of $G$-symmetric functions
(relative to $t_{1},t_{2},\cdots ,t_{m}$) in the function field $L$. There
are the following statements.

$\left( i\right) $ $L$ is algebraic Galois over the subfield $\Lambda
^{G}\left( L\right) =L^{G}$ of degree
\begin{equation*}
\left[ L:\Lambda ^{G}\left( L\right) \right] =\sharp G.
\end{equation*}

$\left( ii\right) $ If $m\geq 3$ and $\{1\}\subsetneqq G\subsetneqq \Sigma
_{m}$, then the least nice number of the subextension $\Lambda ^{G}\left(
L\right) $ over $K$ satisfies an inequality
\begin{equation*}
\mathfrak{ni}\left( \frac{\Lambda ^{G}\left( L\right) }{K}\right) \geq 1.
\end{equation*}%
In such a case, $\Lambda ^{G}\left( L\right) $ is not purely transcendental
over $K$, i.e., $G$ is not a Noether solution of $L/K$.

$\left( iii\right) $ Suppose $m\geq 1$. Then $\Lambda ^{G}\left( L\right) $
is purely transcendental over $K$ if and only if $G=\Sigma _{m}$ holds. In
such a case, $G$ is a Noether solution of $L/K$.
\end{lemma}

\begin{proof}
$\left( i\right) $ We have $\left[ L:\Lambda ^{G}\left( L\right) \right] =m!$
from \emph{Lemmas 7.5-6} and then $L$ is algebraic Galois over $\Lambda
^{G}\left( L\right) $ since $\sharp G=m!$.

$\left( ii\right) $ Let $m\geq 3$ and $\{1\}\subsetneqq G\subsetneqq \Sigma
_{m}$. Hypothesise $G$ is a Noether solution of $L/K$. From \emph{Lemma 6.26}
we have a subgroup $Q=\Sigma _{m}\cap H_{G}$ such that%
\begin{equation*}
\Sigma _{m}=G\cdot Q=Q\cdot G
\end{equation*}%
hold. Here, $H_{G}$ denotes the natural Galois-complement of $G$ in $L/K$;
the subgroups $G$ and $Q$ both are normal subgroups in $\Sigma _{m}$; for
the orders of the finite subgroups we have%
\begin{equation*}
\sharp \Sigma _{m}=\sharp G\cdot \sharp Q.
\end{equation*}

On the other hand, if $m=4$, then rom $\left( iii\right) $ of \emph{Lemma
7.15} we have
\begin{equation*}
\sharp G,\sharp Q\in \{1,12,24\}.
\end{equation*}%
Then
\begin{equation*}
\{1\}\subsetneqq Q\subsetneqq \Sigma _{m}
\end{equation*}%
since%
\begin{equation*}
\{1\}\subsetneqq G\subsetneqq \Sigma _{m}.
\end{equation*}

It follows that we have
\begin{equation*}
\sharp G=\sharp Q=12,
\end{equation*}%
where there is a contradiction that
\begin{equation*}
24=\sharp \Sigma _{4}=\sharp G\cdot \sharp Q=12\cdot 12.
\end{equation*}

If $m=3$ or $m\geq 5$, then rom $\left( iv\right) $ of \emph{Lemma 7.15} we
have
\begin{equation*}
\sharp Q\geq \sharp A_{m}=\frac{1}{2}m!\geq 3,
\end{equation*}%
where there is a contradiction that
\begin{equation*}
m!=\sharp \Sigma _{m}=\sharp G\cdot \sharp Q=\frac{1}{2}m!\cdot \frac{1}{2}%
m!.
\end{equation*}

This proves that the subfield $\Lambda ^{G}\left( L\right) $ is not purely
transcendental over $K$, i.e., $G$ is not a Noether solution of $L/K$.

$\left( iii\right) $ Immediately from the above $\left( ii\right) $ for the
case that $m\geq 3$ and from \emph{Lemma 7.14} for the case that $m=1$ or $2$%
. This completes the proof.
\end{proof}

\subsection{Proof of Theorem 7.12}

Now we give the proof for the main theorem in the present section.

\begin{proof}
(\textbf{Proof of Theorem 7.12})

$\left( i\right) $ Let $L=K\left( t_{1},t_{2},\cdots ,t_{m}\right) $ and let
$G$ be a permutation subgroup of the Galois group $Aut\left( L/K\right) $ on
$m$ letters (relative to the variables $t_{1},t_{2},\cdots ,t_{m}$). It is
seen that the function field $L$ is an algebraic Galois extension over $K$
of degree $\sharp G$ from \emph{Lemmas 7.13} \& \emph{7.18}.

$\left( ii\right) -\left( iii\right) $ Immediately from \emph{Lemmas 7.14}
\& \emph{7.18}.
\end{proof}

\subsection{Concluding remark}

There are the following concluding remarks for the present section.

\begin{remark}
Let $m\geq 3$. Consider the function field $L=K\left( t_{1},t_{2},\cdots
,t_{m}\right) $. Suppose $C_{m}\subseteq \Sigma _{m}$ is the cyclic subgroup
of order $m$. Then $C_{m}$ has a transitive action on the variables $%
t_{1},t_{2},\cdots ,t_{m}$. The above \emph{Lemma 7.16} says that the $C_{m}$%
-invariant subfield $L^{C_{m}}$ is not purely transcendental over $K$, which
is viewed as a generalisation of the main theorem in \emph{[}Swan 1969\emph{]%
}.
\end{remark}

\begin{remark}
(\emph{Question on least nice numbers}) Continue to consider the function
field $L=K\left( t_{1},t_{2},\cdots ,t_{m}\right) $. Let $A_{m}\subseteq
\Sigma _{m}$ and $C_{m}\subseteq \Sigma _{m}$ be the alternating subgroup
and the cyclic subgroup of order $m$, respectively. Then for what $r,s\in
\mathbb{N}$ we have
\begin{equation*}
\mathfrak{ni}\left( \frac{L^{A_{m}}}{K}\right) =r\geq 1
\end{equation*}%
and
\begin{equation*}
\mathfrak{ni}\left( \frac{L^{C_{m}}}{K}\right) =s\geq 1
\end{equation*}%
for the least nice numbers of the subextensions in $L/K$? This would give to
us another aspect for Noether's problem on rationality.
\end{remark}

Unfortunately, we can not find the two numbers $r$ and $s$.

\section{Lifting of Isomorphisms and Noether Solutions}

In this section we will discuss the linear involutions in a purely
transcendental extension of a field and the normalised lifting of
isomorphisms of the subfields with application to spatial isomorphisms,
where new particular cases of Noether's problem will be obtained. $K$-theory has a natural connection to Noether solutions.

\subsection{Linear involutions of a purely transcendental extension}

We need these facts on linear involutions of  a purely
transcendental extension for us to prove \emph{Theorem 1.8} which is stated in \emph{\S 1}.

Let $L=K\left( t_{1},t_{2},\cdots ,t_{n}\right) $ be a purely transcendental
extension over a field $K$ of transcendence degree $n\geq 1$.

\begin{definition}
A \textbf{linear involution} of $L$ over $K$ (relative to the variables $%
t_{1},t_{2},\cdots ,t_{n}$) is a map $\sigma \in Aut\left( L/K\right) $ of
\emph{homogeneous linear type} (See \emph{Remark 1.19} for details)
satisfying the property:
\begin{equation*}
\sigma :\left(
\begin{array}{c}
t_{\widetilde{\sigma }\left( 1\right) } \\
t_{\widetilde{\sigma }\left( 2\right) } \\
\vdots \\
t_{\widetilde{\sigma }\left( n\right) }%
\end{array}%
\right) \longmapsto A\cdot \left(
\begin{array}{c}
t_{\widetilde{\sigma }\left( 1\right) } \\
t_{\widetilde{\sigma }\left( 2\right) } \\
\vdots \\
t_{\widetilde{\sigma }\left( n\right) }%
\end{array}%
\right)
\end{equation*}%
where
\begin{equation*}
A=\left(
\begin{array}{cccc}
&  &  & 1 \\
&  & 1 &  \\
& \cdots &  &  \\
1 &  &  &
\end{array}%
\right)
\end{equation*}%
is the $n\times n$ matrix and $\widetilde{\sigma }$ is a given $n$%
-permutation on $\mathcal{N}:=\{1,2,\cdots ,n\}$.

In particular, the \textbf{natural involution} of $L$ over $K$ (relative to
the variables $t_{1},t_{2},\cdots ,t_{n}$) is the linear involution $\sigma $
of $L$ over $K$ such that $\widetilde{\sigma }=id_{\mathcal{N}}$ is the
identity map on $\mathcal{N}$.
\end{definition}

\begin{remark}
Let $L=K\left( t_{1},t_{2},\cdots ,t_{n}\right) $ be a purely transcendental
extension over a field $K$.

$\left( i\right) $ There are infinitely many linear involutions of $L$ over $%
K$. In deed, relative to the variables $t_{1},t_{2},\cdots ,t_{n}$, there
are exactly $n!$ linear involutions of $L/K$. On the other hand, for any $%
c_{1},c_{2},\cdots ,c_{n}\in K$ we have
\begin{equation*}
L=K\left( t_{1},t_{2},\cdots ,t_{n}\right) =K\left(
t_{1}+c_{1},t_{2}+c_{2},\cdots ,t_{n}+c_{n}\right) ;
\end{equation*}
also there are $n!$ linear involutions of $L/K$ relative to the variables
\begin{equation*}
t_{1}+c_{1},t_{2}+c_{2},\cdots ,t_{n}+c_{n}.
\end{equation*}

$\left( ii\right) $ Consider the set $\Omega _{ver}$ of vertical
transcendence bases of $L/K$ relative to the variables $t_{1},t_{2},\cdots
,t_{n}$. There are automorphisms of the corresponding sub-extensions of $L/K$
which are of monomial type (See \emph{Remark 1.19} for definition) relative
to $t_{1},t_{2},\cdots ,t_{n}$. By \emph{Lemma 2.5} it is seen that these
automorphisms can be extended to $L/K$ and there are the same properties as
linear involutions of $L/K$.
\end{remark}

\begin{proposition}
\emph{(Properties for linear involutions)} Let $\sigma $ be a linear
involution of $L$ over $K$ (relative to the variables $t_{1},t_{2},\cdots
,t_{n}$). Let $I_{\sigma }$ denote the subgroup in the Galois group $%
Aut\left( L/K\right) $ generated by $\sigma $. There are the following
statements.

$\left( i\right) $ $I_{\sigma }$ is a subgroup of order $2$.

$\left( ii\right) $ If $n\geq 3$, then $I_{\sigma }$ is one of the smallest
subgroups in $Aut\left( L/K\right) $ which have a transitive action on the
variables $t_{1},t_{2},\cdots ,t_{n}$.

$\left( iii\right) $ If $n\geq 3$, then the $I_{\sigma }$-invariant subfield
$L^{\sigma }$ of $L$ is not purely transcendental over $K$.
\end{proposition}

\begin{proof}
$\left( i\right) $ Trivial.

$\left( ii\right) $ Immediately from $\left( i\right) $.

$\left( iii\right) $ For any $n\geq 3$, we have $\{1\}\not=I_{\sigma
}\subsetneqq \Sigma _{n}$. From \emph{Lemma 7.18} it is seen that the $%
I_{\sigma }$-invariant subfield $L^{\sigma }$ is not purely transcendental
over $K$.
\end{proof}

From \emph{Theorem 7.12}, \emph{Lemma 7.16}\ and the above \emph{Proposition
8.3}, there is the following remark on the comparison between algebraic
Galois theory and transcendental Galois theory.

\begin{remark}
(\emph{The third characteristic of transcendental Galois theory: Patterns
producing roots}) Let $L=K\left( t_{1},t_{2},\cdots ,t_{n}\right) $ be a
purely transcendental extension over a field $K$. Let $\Sigma _{n}\subseteq
Aut\left( L/K\right) $ denote the full permutation subgroup of $L/K$
(relative to the variables $t_{1},t_{2},\cdots ,t_{n}$).

Fixed any subgroup $G\subseteq \Sigma _{n}$. For instance, $G$ is a subgroup
in $Aut\left( L/K\right) $ generated by a linear involution of $L/K$.

$\left( i\right) $ The $G$-invariant subfield $L^{G}$ of $L$ is the $G$%
-symmetric functions $\Lambda ^{G}\left( L\right) $ of $L$; the function
field $L$ is algebraic Galois over the $G$-invariant subfield $L^{G}$. As $%
L=L^{G}\left[ \theta \right] $ is a simple extension over the subfield $%
L^{G} $ for some $\theta \in L$, there is a minimal polynomial $f\left(
X\right) $ of $\theta $ over the subfield $L^{G}$ such that $f\left( \theta
\right) =0$.

$\left( ii\right) $ Each permutation subgroup $G$ of $\Sigma _{n}$
establishes a \textbf{pattern} for the variables $t_{1},t_{2},\cdots ,t_{n}$
during the mapping and translating in the extension $L/K$. So, \emph{%
patterns }$G$\emph{\ produce roots }$\theta $. This gives to us the third
characteristic of transcendental Galois theory, which is distinct from
algebraic Galois theory.

$\left( iii\right) $ The unusual complexities of Galois groups $H$ in
transcendental Galois theory are due to the entanglements of the permutation
subgroups $G$ (relative to the given variables) with the usual Galois groups
$P$ in algebraic Galois theory. For the \emph{separated cases} that $%
H=G\cdot P$, there are some results on new particular cases of Noether's
problem, which will be discussed in the next subsection \emph{\S 8.3}.
\end{remark}

\begin{example}
Let $t_{1},t_{2},t_{3,}t_{4},t_{5}$ be algebraically independent variables
over $\mathbb{Q}$ and $L=\mathbb{Q}\left(
t_{1},t_{2},t_{3,}t_{4},t_{5}\right) $. Let $G\subseteq Aut\left( L/\mathbb{Q%
}\right) $ be a subgroup of order $2$. Consider the following two cases for
the subgroup $G$.

\emph{Case }$\left( i\right) $. Assume $G$ is of \emph{horizontal type} (See
\emph{Remark 1.19} for details), i.e., $G$ is a subgroup of the full
permutation group $\Sigma _{5}$ (relative to the variables $%
t_{1},t_{2},t_{3,}t_{4},t_{5}$), for instance, $G$ is generated by a linear
involution of $L$ over $\mathbb{Q}$.

Then the $G$-invariant subfield $L^{G}$ of $L$ is not purely transcendental
over $\mathbb{Q}$. In such a case, $G$ is not a Noether solution of $L/%
\mathbb{Q}$.

\emph{Case }$\left( ii\right) $. Suppose $G$ is of \emph{vertical type} (See
\emph{Remark 1.19} for details), i.e., $G$ is the Galois group of the
quadratic extension $L/M$, where $M$ is an intermediate subfield in $L/%
\mathbb{Q}$ of the form $M=\mathbb{Q}\left( t_{1},\cdots ,t_{i}^{2},\cdots
,t_{5}\right) $ with $1\leq i\leq 5$.

Then the $G$-invariant subfield $L^{G}$ of $L$ is purely transcendental over
$\mathbb{Q}$. In such a case, $G$ is a Noether solution of $L/\mathbb{Q}$.
\end{example}

\subsection{Normalised lifting of Galois groups of subextensions}

Consider a purely transcendental extension $L=K\left( t_{1},t_{2},\cdots
,t_{n}\right) $ over a field $K$.

\begin{remark}
(\emph{Normalised lifting}) Fixed a nice basis $\left( \Delta .A\right) $ of
the extension $L/K$. Consider the highest transcendental Galois subgroup
\begin{equation*}
\pi _{t}\left( L/K\right) \left( \Delta ,A\right) \subseteq Aut\left(
L/K\right)
\end{equation*}%
of $L/K$ at $\left( \Delta ,A\right) $. From \emph{Proposition 4.7} it is
seen that there is an isomorphism
\begin{equation*}
j_{\left( \Delta ,A\right) }:Aut\left( K\left( \Delta \right) /K\right)
\rightarrow \pi _{t}\left( L/K\right) \left( \Delta ,A\right)
\end{equation*}%
of groups, called the \textbf{normalised lifting} of $L/K$ at $\left( \Delta
,A\right) $, given by
\begin{equation*}
\sigma |_{K\left( \Delta \right) }\longmapsto \sigma
\end{equation*}%
for any
\begin{equation*}
\sigma \in \pi _{t}\left( L/K\right) \left( \Delta ,A\right) .
\end{equation*}
\end{remark}

Here, if $L$ is Galois over $K\left( \Delta \right) $ and $L$ is pure
transcendental over $K$, the normalised lifting of $Aut\left( K\left( \Delta
\right) /K\right) $ is nothing other than the natural Galois-complemnt of $%
Aut\left( L/K\left( \Delta \right) \right) $ in $L/K$.

Heuristically from the above \emph{Example 8.3}, we have the following new
particular cases for Noether's problem.

\begin{proposition}
(\emph{Case: }$\deg =2n$) Let $t_{1},t_{2},\cdots ,t_{n}$ be algebraically
independent variables over a field $K$ and let $L=K\left( t_{1},t_{2},\cdots
,t_{n}\right) $ be the function field. Consider the intermediate subfield
\begin{equation*}
M=K\left( t_{1}^{2},t_{2}^{2},\cdots ,t_{n}^{2}\right)
\end{equation*}%
in the extension $L/K$ and the Galois group $G=Aut\left( L/M\right) $.
Suppose $j_{\left( \Delta ,A\right) }$ is the normalised lifting of the
Galois group \thinspace $Aut\left( M/K\right) $ onto the highest
transcendental Galois subgroup $\pi _{t}\left( L/K\right) \left( \Delta
,A\right) $ of $L/K$ at $\left( \Delta ,A\right) $. Here, $\Delta
=\{t_{1}^{2},t_{2}^{2},\cdots ,t_{n}^{2}\}$, $A\subseteq L$ and $\left(
\Delta ,A\right) $ is a nice basis of $L/K$.

$\left( i\right) $ Assume $\Sigma _{n}^{M}\subseteq Aut\left( M/K\right) $
is the full permutation subgroup of $M/K$ on $n$ letters relative to the
variables $t_{1}^{2},t_{2}^{2},\cdots ,t_{n}^{2}\in M$. Then the product
\begin{equation*}
G\cdot j_{\left( \Delta ,A\right) }\left( \Sigma _{n}^{M}\right)
\end{equation*}%
is a subgroup of the Galois group $Aut\left( L/K\right) $ with $G$ and $%
j_{\left( \Delta ,A\right) }\left( \Sigma _{n}^{M}\right) $ being normal
subgroups; moreover, the subgroups
\begin{equation*}
G\text{ and }G\cdot j_{\left( \Delta ,A\right) }\left( \Sigma _{n}^{M}\right)
\end{equation*}
both are Noether solutions of the extension $L/K$.

$\left( ii\right) $ Let $C_{n}^{M}\subseteq \Sigma _{n}^{M}$ be the cyclic
permutation subgroup of $M/K$ of order $n$ relative to the variables $%
t_{1}^{2},t_{2}^{2},\cdots ,t_{n}^{2}\in M$. Then the product
\begin{equation*}
G\cdot j_{\left( \Delta ,A\right) }\left( C_{n}^{M}\right)
\end{equation*}%
is a subgroup of the Galois group $Aut\left( L/K\right) $ with $G$ and $%
j_{\left( \Delta ,A\right) }\left( C_{n}^{M}\right) $ being normal
subgroups; furthermore, if $n\geq 3$, then neither of the invariant
subfields of $L$ respectively under the subgroups
\begin{equation*}
j_{\left( \Delta ,A\right) }\left( C_{n}^{M}\right) \text{ and } G\cdot
j_{\left( \Delta ,A\right) }\left( C_{n}^{M}\right)
\end{equation*}
is purely transcendental over $K$.

$\left( iii\right) $ Let $A_{n}^{M}\subseteq \Sigma _{n}^{M}$ be the
alternating subgroup of $M/K$ of order $\frac{1}{2}\cdot n!$ relative to the
variables $t_{1}^{2},t_{2}^{2},\cdots ,t_{n}^{2}\in M$. Then the product
\begin{equation*}
G\cdot j_{\left( \Delta ,A\right) }\left( A_{n}^{M}\right)
\end{equation*}%
is a subgroup of the Galois group $Aut\left( L/K\right) $ with $G$ and $%
j_{\left( \Delta ,A\right) }\left( A_{n}^{M}\right) $ being normal
subgroups; furthermore, if $n\geq 3$, then neither of the invariant
subfields of $L$ respectively under the subgroups
\begin{equation*}
j_{\left( \Delta ,A\right) }\left( A_{n}^{M}\right) \text{ and } G\cdot
j_{\left( \Delta ,A\right) }\left( A_{n}^{M}\right)
\end{equation*}
is purely transcendental over $K$.

$\left( iv\right) $ Let $I_{\sigma }^{M}$ be the subgroup in $Aut\left(
M/K\right) $ generated by a linear involution $\sigma $ of $M$ over $K$
relative to the variables $t_{1}^{2},t_{2}^{2},\cdots ,t_{n}^{2}\in M$. Then
the product
\begin{equation*}
G\cdot j_{\left( \Delta ,A\right) }\left( I_{\sigma }^{M}\right)
\end{equation*}%
is a subgroup of the Galois group $Aut\left( L/K\right) $ with $G$ and $%
j_{\left( \Delta ,A\right) }\left( I_{\sigma }^{M}\right) $ being normal
subgroups; furthermore, if $n\geq 3$, then neither of the invariant
subfields of $L$ respectively under the subgroups
\begin{equation*}
j_{\left( \Delta ,A\right) }\left( I_{\sigma }^{M}\right) \text{ and }
G\cdot j_{\left( \Delta ,A\right) }\left( I_{\sigma }^{M}\right)
\end{equation*}
is purely transcendental over $K$.
\end{proposition}

\begin{proof}
$\left( i\right) $ It is seen that $L$ is algebraic Galois over $M$ of
degree $\left[ L:M\right] =2^{n}$ and $G$ is a Noether solution of $L/M$.
From \emph{Lemma 7.18} it is seen that $\Sigma _{n}^{M}$ is a Noether
solution of the sub-extension $M/K$.

Prove $L$ is algebraic Galois over the $\Sigma _{n}^{M}$-invariant subfield $%
M^{\Sigma _{n}^{M}}$ of $M$. In deed, consider the linear decompositions of
automorphisms at the given nice basis $\left( \Delta ,A\right) $ of $L/K$,
where
\begin{equation*}
\pi _{a}\left( L/K\right) \left( \Delta ,A\right) =G\text{ and }\pi
_{t}\left( L/K\right) \left( \Delta ,A\right) \supseteq j_{\left( \Delta
,A\right) }\left( \Sigma _{n}^{M}\right)
\end{equation*}%
holds. (See \emph{Proposition 5.3} for details). By the diagonal action on
the extension $L/K$, it is seen that the subgroup
\begin{equation*}
G\cdot j_{\left( \Delta ,A\right) }\left( \Sigma _{n}^{M}\right) \subseteq
Aut\left( L/K\right)
\end{equation*}%
is contained in the decomposition group
\begin{equation*}
D_{L/K}\left( \Delta ,A\right) =\pi _{a}\left( L/K\right) \left( \Delta
,A\right) \cdot \pi _{t}\left( L/K\right) \left( \Delta ,A\right)
\end{equation*}%
of the order
\begin{equation*}
\sharp \left( G\cdot j_{\left( \Delta ,A\right) }\left( \Sigma
_{n}^{M}\right) \right) =2n\cdot n!,
\end{equation*}%
which is equal to the degree
\begin{equation*}
\left[ L:M^{\Sigma _{n}^{M}}\right] =\left[ L:M\right] \cdot \left[
M:M^{\Sigma _{n}^{M}}\right] =2n\cdot n!
\end{equation*}%
of the extension $L/M^{\Sigma _{n}^{M}}$. It follows that $L$ is algebraic
Galois over $M^{\Sigma _{n}^{M}}$ since we have
\begin{equation*}
L^{G\cdot j_{\left( \Delta ,A\right) }\left( \Sigma _{n}^{M}\right) }=\left(
L^{G}\right) ^{\Sigma _{n}^{M}}=M^{\Sigma _{n}^{M}}
\end{equation*}%
for the invariant subfield $L^{G\cdot j_{\left( \Delta ,A\right) }\left(
\Sigma _{n}^{M}\right) }$ of $L$ under the subgroup $G\cdot j_{\left( \Delta
,A\right) }\left( \Sigma _{n}^{M}\right) $. From \emph{Lemma 6.26} it is
seen that
\begin{equation*}
G\cdot j_{\left( \Delta ,A\right) }\left( \Sigma _{n}^{M}\right) =Aut\left(
L/M^{\Sigma _{n}^{M}}\right)
\end{equation*}%
is the Galois group of the extension $L/M^{\Sigma _{n}^{M}}$ such that $G$
and $j_{\left( \Delta ,A\right) }\left( \Sigma _{n}^{M}\right) $ both are
normal subgroups of $Aut\left( L/M^{\Sigma _{n}^{M}}\right) $; the subgroup $%
G\cdot j_{\left( \Delta ,A\right) }\left( \Sigma _{n}^{M}\right) $ is a
Noether solution of the extension $L/K$.

$\left( ii\right) -\left( iv\right) $ Immediately from $\left( i\right) $
and \emph{Lemma 7.18} since we have
\begin{equation*}
j_{\left( \Delta ,A\right) }\left( C_{n}^{M}\right) ,j_{\left( \Delta
,A\right) }\left( A_{n}^{M}\right) ,j_{\left( \Delta ,A\right) }\left(
I_{\sigma }^{M}\right) \subseteq j_{\left( \Delta ,A\right) }\left( \Sigma
_{n}^{M}\right)
\end{equation*}%
as subgroups. This completes the proof.
\end{proof}

\begin{proposition}
(\emph{Case: }$\deg =2$) Let $t_{1},t_{2},\cdots ,t_{n}$ be algebraically
independent variables over a field $K$ and let $L=K\left( t_{1},t_{2},\cdots
,t_{n}\right) $ be the function field. Fixed any $1\leq i\leq n$. Consider
the intermediate subfield
\begin{equation*}
M\left( i\right) =K\left( t_{1},\cdots ,t_{i}^{2},\cdots ,t_{n}\right)
\end{equation*}%
in the extension $L/K$ (with $L/M\left( i\right) $ being a quadratic
sub-extension) with the Galois group
\begin{equation*}
G\left( i\right) =Aut\left( L/M\left( i\right) \right) .
\end{equation*}
Suppose $j_{\left( \Delta _{i},A_{i}\right) }^{M\left( i\right) }$ is the
normalised lifting of the Galois group \thinspace $Aut\left( M\left(
i\right) /K\right) $ onto the highest transcendental Galois subgroup $\pi
_{t}\left( L/K\right) \left( \Delta _{i},A_{i}\right) $ of $L/K$ at $\left(
\Delta _{i},A_{i}\right) $. Here, $\Delta _{i}=\{t_{1},\cdots
,t_{i}^{2},\cdots ,t_{n}\}$, $A_{i}\subseteq L$ and $\left( \Delta
_{i},A_{i}\right) $ is a nice basis of $L/K$.

$\left( i\right) $ Let $\Sigma _{n}^{M\left( i\right) }\subseteq Aut\left(
M\left( i\right) /K\right) $ denote the full permutation subgroup of $%
M\left( i\right) /K$ on $n$ letters relative to the variables $t_{1},\cdots
,t_{i}^{2},\cdots ,t_{n}\in M\left( i\right) $. Then the product
\begin{equation*}
G\left( i\right) \cdot j_{\left( \Delta _{i},A_{i}\right) }^{M\left(
i\right) }\left( \Sigma _{n}^{M\left( i\right) }\right)
\end{equation*}%
is a subgroup of $Aut\left( L/K\right) $ such that $G\left( i\right) $ and $%
j_{\left( \Delta _{i},A_{i}\right) }^{M\left( i\right) }\left( \Sigma
_{n}^{M\left( i\right) }\right) $ are its normal subgroups.

Furthermore, the subgroups
\begin{equation*}
G\left( i\right) \text{ and } G\left( i\right) \cdot j_{\left( \Delta
_{i},A_{i}\right) }\left( \Sigma _{n}^{M}\right)
\end{equation*}
both are Noether solutions of the extension $L/K$.

$\left( ii\right) $ Let $C_{n}^{M\left( i\right) }\subseteq \Sigma
_{n}^{M\left( i\right) }$ be the cyclic permutation subgroup of $M\left(
i\right) /K$ of order $n$ relative to the variables $t_{1},\cdots
,t_{i}^{2},\cdots ,t_{n}\in M\left( i\right) $. Then the product
\begin{equation*}
G\left( i\right) \cdot j_{\left( \Delta _{i},A_{i}\right) }^{M\left(
i\right) }\left( C_{n}^{M\left( i\right) }\right)
\end{equation*}%
is a subgroup of $Aut\left( L/K\right) $ such that $G\left( i\right) $ and $%
j_{\left( \Delta _{i},A_{i}\right) }^{M\left( i\right) }\left(
C_{n}^{M\left( i\right) }\right) $ are its normal subgroups.

Furthermore, if $n\geq 3$, then neither of the invariant subfields of $L$
respectively under the subgroups
\begin{equation*}
j_{\left( \Delta _{i},A_{i}\right) }^{M\left( i\right) }\left(
C_{n}^{M\left( i\right) }\right) \text{ and } G\left( i\right) \cdot
j_{\left( \Delta _{i},A_{i}\right) }^{M\left( i\right) }\left(
C_{n}^{M\left( i\right) }\right)
\end{equation*}
is purely transcendental over $K$.

$\left( iii\right) $ Let $A_{n}^{M\left( i\right) }\subseteq \Sigma
_{n}^{M\left( i\right) }$ be the alternating subgroup of $M\left( i\right)
/K $ of order $\frac{1}{2}\cdot n!$ relative to the variables $t_{1},\cdots
,t_{i}^{2},\cdots ,t_{n}\in M\left( i\right) $. Then the product
\begin{equation*}
G\left( i\right) \cdot j_{\left( \Delta _{i},A_{i}\right) }^{M\left(
i\right) }\left( A_{n}^{M\left( i\right) }\right)
\end{equation*}%
is a subgroup of $Aut\left( L/K\right) $ such that $G\left( i\right) $ and $%
j_{\left( \Delta _{i},A_{i}\right) }^{M\left( i\right) }\left(
A_{n}^{M\left( i\right) }\right) $ are its normal subgroups.

Furthermore, if $n\geq 3$, then neither of the invariant subfields of $L$
respectively under the subgroups
\begin{equation*}
j_{\left( \Delta _{i},A_{i}\right) }^{M\left( i\right) }\left(
A_{n}^{M\left( i\right) }\right) \text{ and } G\left( i\right) \cdot
j_{\left( \Delta _{i},A_{i}\right) }^{M\left( i\right) }\left(
A_{n}^{M\left( i\right) }\right)
\end{equation*}
is purely transcendental over $K$.

$\left( iv\right) $ Let $I_{\sigma }^{M\left( i\right) }$ be the subgroup in
$Aut\left( M\left( i\right) /K\right) $ generated by a linear involution $%
\sigma $ of $M\left( i\right) $ over $K$ relative to the variables $%
t_{1},\cdots ,t_{i}^{2},\cdots ,t_{n}\in M\left( i\right) $. Then the
product
\begin{equation*}
G\left( i\right) \cdot j_{\left( \Delta _{i},A_{i}\right) }^{M\left(
i\right) }\left( I_{\sigma }^{M\left( i\right) }\right)
\end{equation*}%
is a subgroup of $Aut\left( L/K\right) $ such that $G\left( i\right) $ and $%
j_{\left( \Delta _{i},A_{i}\right) }^{M\left( i\right) }\left( I_{\sigma
}^{M\left( i\right) }\right) $ are its normal subgroups.

Furthermore, if $n\geq 3$, then neither of the invariant subfields of $L$
respectively under the subgroups
\begin{equation*}
j_{\left( \Delta _{i},A_{i}\right) }^{M\left( i\right) }\left( I_{\sigma
}^{M\left( i\right) }\right) \text{ and } G\left( i\right) \cdot j_{\left(
\Delta _{i},A_{i}\right) }^{M\left( i\right) }\left( I_{\sigma }^{M\left(
i\right) }\right)
\end{equation*}
is purely transcendental over $K$.
\end{proposition}

\begin{proof}
Apply the same procedure as in the proof for the above \emph{Proposition 8.7}
with $G\left( i\right) $ and $M\left( i\right) $ being in place of $G$ and $%
M $, respectively.
\end{proof}

Let $L=K\left( t_{1},t_{2},\cdots ,t_{n}\right) $ be a purely transcendental
extension over a field $K$. Recall that a transcendence base $\Delta $ of $%
L/K$ is said to be \textbf{vertical} (relative to the variables $%
t_{1},t_{2},\cdots ,t_{n}$) if $\Delta =\{t_{1}^{m_{1}},t_{2}^{m_{2}},\cdots
,t_{n}^{m_{n}}\}$ for some positive integers $m_{1},m_{2},\cdots ,m_{n}\in
\mathbb{Z}$. Let $\Omega _{ver}$ denote the set of all the vertical
transcendence bases $\Delta $ of $L/K$ (relative to the variables $%
t_{1},t_{2},\cdots ,t_{n}$). If $L$ is $\sigma $-tame Galois over $K$ inside
$\Omega _{ver}$, $L=K\left( t_{1},t_{2},\cdots ,t_{n}\right) $ is said to be
$\sigma $\textbf{-tame Galois of vertical type} over $K$. (See \emph{%
Definition 3.8}).

\begin{remark}
Here, for a purely transcendental extension $L/K$, we attempt to use $\sigma
$-tame Galois of vertical type to understand that $L$ has enough roots of
unity over $K$ in \emph{[}Fisher 1915\emph{]}.
\end{remark}

Enlightened by a result for the case of \textquotedblleft enough roots of
unity\textquotedblright\ in \emph{[}Fisher 1915\emph{]}, we have such a
result.

\begin{proposition}
(\emph{Case of vertical type:} $\deg =a$) Let $L=K\left( t_{1},t_{2},\cdots
,t_{n}\right) $ be a purely transcendental extension over a field $K$ such
that $L$ is $\sigma $-tame Galois of vertical type over $K$ (relative to the
variables $t_{1},t_{2},\cdots ,t_{n}$). Fixed an integer $1\leq a\in \mathbb{%
Z}$ and an $n$\emph{-partition} $\left( a:n\right) $ of $a$
\begin{equation*}
1\leq a_{1}\leq a_{2}\leq \cdots \leq a_{n}\leq m
\end{equation*}%
such that $a=a_{1}\cdot a_{2}\cdot \cdots \cdot a_{n}$ holds, where each $%
a_{i}\in \mathbb{Z}$. Consider the intermediate subfield
\begin{equation*}
M\left( a:n\right) =K\left( t_{1}^{a_{1}},\cdots ,t_{i}^{a_{i}},\cdots
,t_{n}^{a_{n}}\right)
\end{equation*}%
in $L/K$ with the Galois group
\begin{equation*}
G\left( a:n\right) :=Aut\left( L/M\left( a:n\right) \right) .
\end{equation*}%
Suppose
\begin{equation*}
j_{\left( \Delta _{\left( a:n\right) },A_{\left( a:n\right) }\right)
}^{M\left( a:n\right) }
\end{equation*}%
is the normalised lifting of the Galois group \thinspace $Aut\left( M\left(
a:n\right) /K\right) $ onto the highest transcendental Galois subgroup $\pi
_{t}\left( L/K\right) \left( \Delta _{\left( a:n\right) },A_{\left(
a:n\right) }\right) $ of $L/K$ at $\left( \Delta _{\left( a:n\right)
},A_{\left( a:n\right) }\right) $. Here, $\Delta _{\left( a:n\right)
}=\{t_{1}^{a_{1}},\cdots ,t_{i}^{a_{i}},\cdots ,t_{n}^{a_{n}}\}$, $A_{\left(
a:n\right) }\subseteq L$ and $\left( \Delta _{\left( a:n\right) },A_{\left(
a:n\right) }\right) $ is a nice basis of $L/K$.

$\left( i\right) $ Let
\begin{equation*}
\Sigma _{n}^{M\left( a:n\right) }\subseteq Aut\left( M\left( a:n\right)
/K\right)
\end{equation*}
denote the full permutation subgroup of $M\left( a:n\right) /K$ on $n$
letters relative to the variables $t_{1}^{a_{1}},\cdots
,t_{i}^{a_{i}},\cdots ,t_{n}^{a_{n}}\in M\left( a:n\right) $. Then the
product
\begin{equation*}
G\left( a:n\right) \cdot j_{\left( \Delta _{\left( a:n\right) },A_{\left(
a:n\right) }\right) }^{M\left( a:n\right) }\left( \Sigma _{n}^{M\left(
a:n\right) }\right)
\end{equation*}%
is a subgroup of $Aut\left( L/K\right) $ such that $G\left( a:n\right) $ and
$j_{\left( \Delta _{\left( a:n\right) },A_{\left( a:n\right) }\right)
}^{M\left( a:n\right) }\left( \Sigma _{n}^{M\left( a:n\right) }\right) $ are
its normal subgroups.

Furthermore, the subgroups
\begin{equation*}
G\left( a:n\right) \text{ and } G\left( a:n\right) \cdot j_{\left( \Delta
_{\left( a:n\right) },A_{\left( a:n\right) }\right) }^{M\left( a:n\right)
}\left( \Sigma _{n}^{M\left( a:n\right) }\right)
\end{equation*}
both are Noether solutions of the extension $L/K$.

$\left( ii\right) $ Let
\begin{equation*}
C_{n}^{M\left( a:n\right) }\subseteq \Sigma _{n}^{M\left( a:n\right) }
\end{equation*}
denote the cyclic permutation subgroup of $M\left( a:n\right) /K$ of order $%
n $ relative to the variables $t_{1}^{a_{1}},\cdots ,t_{i}^{a_{i}},\cdots
,t_{n}^{a_{n}}\in M\left( a:n\right) $. Then the product
\begin{equation*}
G\left( a:n\right) \cdot j_{\left( \Delta _{\left( a:n\right) },A_{\left(
a:n\right) }\right) }^{M\left( a:n\right) }\left( C_{n}^{M\left( a:n\right)
}\right)
\end{equation*}%
is a subgroup of $Aut\left( L/K\right) $ such that $G\left( a:n\right) $ and
$j_{\left( \Delta _{\left( a:n\right) },A_{\left( a:n\right) }\right)
}^{M\left( a:n\right) }\left( C_{n}^{M\left( a:n\right) }\right) $ are its
normal subgroups.

Furthermore, if $n\geq 3$, then neither of the invariant subfields of $L$
respectively under the subgroups
\begin{equation*}
j_{\left( \Delta _{\left( a:n\right) },A_{\left( a:n\right) }\right)
}^{M\left( a:n\right) }\left( C_{n}^{M\left( a:n\right) }\right)
\end{equation*}
and
\begin{equation*}
G\left( a:n\right) \cdot j_{\left( \Delta _{\left( a:n\right) },A_{\left(
a:n\right) }\right) }^{M\left( a:n\right) }\left( C_{n}^{M\left( a:n\right)
}\right)
\end{equation*}
is purely transcendental over $K$.

$\left( iii\right) $ Let
\begin{equation*}
A_{n}^{M\left( a:n\right) }\subseteq \Sigma _{n}^{M\left( a:n\right) }
\end{equation*}
denote the alternating subgroup of $M\left( a:n\right) /K$ of order $\frac{1%
}{2}\cdot n!$ relative to the variables $t_{1}^{a_{1}},\cdots
,t_{i}^{a_{i}},\cdots ,t_{n}^{a_{n}}\in M\left( a:n\right) $. Then the
product
\begin{equation*}
G\left( a:n\right) \cdot j_{\left( \Delta _{\left( a:n\right) },A_{\left(
a:n\right) }\right) }^{M\left( a:n\right) }\left( A_{n}^{M\left( a:n\right)
}\right)
\end{equation*}
is a subgroup of $Aut\left( L/K\right) $ such that $G\left( a:n\right) $ and
$j_{\left( \Delta _{\left( a:n\right) },A_{\left( a:n\right) }\right)
}^{M\left( a:n\right) }\left( A_{n}^{M\left( a:n\right) }\right) $ are its
normal subgroups.

Furthermore, if $n\geq 3$, then neither of the invariant subfields of $L$
respectively under the subgroups
\begin{equation*}
j_{\left( \Delta _{\left( a:n\right) },A_{\left( a:n\right) }\right)
}^{M\left( a:n\right) }\left( A_{n}^{M\left( a:n\right) }\right)
\end{equation*}
and
\begin{equation*}
G\left( a:n\right) \cdot j_{\left( \Delta _{\left( a:n\right) },A_{\left(
a:n\right) }\right) }^{M\left( a:n\right) }\left( A_{n}^{M\left( a:n\right)
}\right)
\end{equation*}
is purely transcendental over $K$.

$\left( iv\right) $ Suppose $I_{\sigma }^{M\left( a:n\right) }$ is the
subgroup in $Aut\left( M\left( a:n\right) /K\right) $ generated by a linear
involution $\sigma $ of $M\left( a:n\right) $ over $K$ relative to the
variables
\begin{equation*}
t_{1}^{a_{1}},\cdots ,t_{i}^{a_{i}},\cdots ,t_{n}^{a_{n}}\in M\left(
a:n\right).
\end{equation*}
Then the product
\begin{equation*}
G\left( a:n\right) \cdot j_{\left( \Delta _{\left( a:n\right) },A_{\left(
a:n\right) }\right) }^{M\left( a:n\right) }\left( I_{\sigma }^{M\left(
a:n\right) }\right)
\end{equation*}
is a subgroup of $Aut\left( L/K\right) $ such that $G\left( a:n\right) $ and
$j_{\left( \Delta _{\left( a:n\right) },A_{\left( a:n\right) }\right)
}^{M\left( a:n\right) }\left( I_{\sigma }^{M\left( a:n\right) }\right) $ are
its normal subgroups.

Furthermore, if $n\geq 3$, then neither of the invariant subfields of $L$
respectively under the subgroups
\begin{equation*}
j_{\left( \Delta _{\left( a:n\right) },A_{\left( a:n\right) }\right)
}^{M\left( a:n\right) }\left( I_{\sigma }^{M\left( a:n\right) }\right)
\end{equation*}
and
\begin{equation*}
G\left( a:n\right) \cdot j_{\left( \Delta _{\left( a:n\right) },A_{\left(
a:n\right) }\right) }^{M\left( a:n\right) }\left( I_{\sigma }^{M\left(
a:n\right) }\right)
\end{equation*}
is purely transcendental over $K$.
\end{proposition}

\begin{proof}
Apply the same procedure as in the proof for the above \emph{Proposition 8.7}
with $G\left( a:n\right) $ and $M\left( a:n\right) $ being in place of $G$
and $M$, respectively.
\end{proof}

\subsection{Subgroups of separated type at nice bases}

Let $L$ be a purely transcendental extension over a field $K$ of finite
transcendence degree.

\begin{definition}
A subgroup $G$ of the Galois group $Aut\left( L/K\right) $ is said to have a
\textbf{separated factorisation}
\begin{equation*}
G=P\cdot j_{\left( \Delta ,A\right) }\left( Q\right)
\end{equation*}%
in $L/K$ if it holds in $Aut\left( L/K\right) $ for some vertical
transcendence base $\Delta $ of $L/K$, where $P$ is a subgroup of $Aut\left(
L/K\left( \Delta \right) \right) $, $Q$ is a subgroup of $Aut\left( K\left(
\Delta \right) /K\right) $, $A$ is a $K\left( \Delta \right) $-linear basis
of $L$ as a vector space and $j_{\left( \Delta ,A\right) }$ is the
normalised lifting of $L/K$ at the nice basis $\left( \Delta ,A\right) $.

In such a case, $G$ is said to be \textbf{of separated type} in $L/K$ at $%
\left( \Delta ,A\right) $; $P$ and $Q$ are called the \textbf{algebraic part}
and \textbf{transcendental part} of $G$ at $\left( \Delta ,A\right) $,
respectively.
\end{definition}

\begin{proposition}
Let $K$ be a field and $L$ a purely transcendental extension over $K$ of
finite transcendence degree $n$. Let $G$ be a subgroup of $Aut\left(
L/K\right) $. Fixed a nice basis $\left( \Delta ,A\right) $ of $L/K$. Here, $%
\Delta $ is a vertical transcendental base of $L/K$.

$\left( i\right) $ If $G$ is of separated type at $\left( \Delta ,A\right) $%
, then the conjugate subgroup $\sigma \cdot G\cdot \sigma ^{-1}$ must be of
separated type at $\left( \sigma \left( \Delta \right) ,\sigma \left(
A\right) \right) $ for any $\sigma \in Aut\left( L/K\right) $.

$\left( ii\right) $ Suppose $G$ is of separated type at $\left( \Delta
,A\right) $ such that $K\left( \Delta \right) $ is the invariant subfield of
$L$ under the algebraic part of $G$ and the transcendental part of $G$ is
the full permutation subgroup $\Sigma _{n}^{K\left( \Delta \right) }$ of $%
K\left( \Delta \right) /K$ on $n$ letters relative to $\Delta $. Then $G$ is
a Noether solution of $L/K$.
\end{proposition}

\begin{proof}
$\left( i\right) $ Let $\sigma \in Aut\left( L/K\right) ,A^{\prime }=\sigma
\left( A\right) $ and $\Delta ^{\prime }=\sigma \left( \Delta \right) $. It
is seen that $\left( \Delta ^{\prime },A^{\prime }\right) $ is also a nice
basis of $L/K$ since
\begin{equation*}
L=K\left( \Delta \right) \left[ A\right] =K\left( \Delta ^{\prime }\right) %
\left[ A^{\prime }\right] .
\end{equation*}

Assume $G$ is of separated type at $\left( \Delta ,A\right) $ with $P$ and $%
Q $ being the algebraic part and transcendental part of $G$, respectively.
Then%
\begin{equation*}
B^{\prime }:=\sigma \left( B\right) ;
\end{equation*}%
\begin{equation*}
\sigma \cdot Q\cdot \sigma ^{-1}=Q^{\prime }:=Aut\left( K\left( \Delta
^{\prime }\right) /K\right) ;
\end{equation*}%
\begin{equation*}
K\left( \Delta ^{\prime }\right) \subseteq L^{P^{\prime }}:=K\left( \Delta
^{\prime }\right) \left[ B^{\prime }\right] =L;
\end{equation*}%
\begin{equation*}
\sigma \cdot P\cdot \sigma ^{-1}=P^{\prime }=Aut\left( L/L^{P^{\prime
}}\right) \subseteq Aut\left( L/K\left( \Delta ^{\prime }\right) \right) ,
\end{equation*}

where%
\begin{equation*}
Q:=Aut\left( K\left( \Delta \right) /K\right) ;
\end{equation*}
\begin{equation*}
K\left( \Delta \right) \subseteq L^{P}:=K\left( \Delta \right) \left[ B%
\right] =L;
\end{equation*}
\begin{equation*}
P=Aut\left( L/L^{P}\right) \subseteq Aut\left( L/K\left( \Delta \right)
\right) .
\end{equation*}

From the linear decomposition of automorphisms in \emph{Proposition 5.3}, we
have%
\begin{equation*}
\sigma \cdot j_{\left( \Delta ,A\right) }\left( Q\right) \cdot \sigma
^{-1}=j_{\left( \Delta ^{\prime },A^{\prime }\right) }\left( Q^{\prime
}\right)
\end{equation*}

and then
\begin{equation*}
\sigma \cdot G\cdot \sigma ^{-1}=G^{\prime }
\end{equation*}
holds, where
\begin{equation*}
G=P\cdot j_{\left( \Delta ,A\right) }\left( Q\right) ;
\end{equation*}
\begin{equation*}
G^{\prime }=P^{\prime }\cdot j_{\left( \Delta ^{\prime },A^{\prime }\right)
}\left( Q^{\prime }\right) .
\end{equation*}
This proves the conjugate subgroup $G^{\prime }=\sigma \cdot G\cdot \sigma
^{-1}$ is of separated type at $\left( \Delta ^{\prime },A^{\prime }\right) $%
.

$\left( ii\right) $ We have
\begin{equation*}
\begin{array}{l}
L^{G} \\
=L^{Aut\left( L/K\left( \Delta \right) \right) \cdot j_{\left( \Delta
,A\right) }\left( \Sigma _{n}^{K\left( \Delta \right) }\right) } \\
=\left( L^{Aut\left( L/K\left( \Delta \right) \right) }\right) ^{j_{\left(
\Delta ,A\right) }\left( \Sigma _{n}^{K\left( \Delta \right) }\right) } \\
=\left( K\left( \Delta \right) \right) ^{j_{\left( \Delta ,A\right) }\left(
\Sigma _{n}^{K\left( \Delta \right) }\right) } \\
=\left( K\left( \Delta \right) \right) ^{\Sigma _{n}^{K\left( \Delta \right)
}}.%
\end{array}%
\end{equation*}
This proves $L^{G}$ is purely transcendental over $K$.
\end{proof}

From the above \emph{Proposition 8.12} we have the converse question.

\begin{question}
Let $L=K\left( t_{1},t_{2},\cdots ,t_{n}\right) $ be a purely transcendental
extension of a field $K$ of transcendence degree $n$. Let $G$ be a Noether
solution of $L/K$.

\emph{Is it true that }$G$\emph{\ must be of separated type at a nice basis }%
$\left( \Delta ,A\right) $\emph{\ of }$L/K$\emph{?} Here, $\Delta $ is a
vertical transcendence base of $L/K$. In other words, \emph{does there exist
any subgroup }$G$\emph{\ of }$Aut\left( L/K\right) $\emph{\ such that }$G$%
\emph{\ is a Noether solution of }$L/K$\emph{\ but not of separated type?}
\end{question}

\subsection{Lifting properties of isomorphisms between purely transcendental
sub-extensions}

Now consider the congruency classes of Galois groups of purely
transcendental sub-extensions in a given purely transcendental extension.

Let $L$ and $M$ be two extensions over a common field $K$. Recall that two
subgroups $G\subseteq Aut\left( L/K\right) $ and $H\subseteq Aut\left(
M/K\right) $ are $\left( L,M\right) $\emph{-spatially isomorphic} if the
three conditions are satisfied: $\left( i\right) $ There is an isomorphism $%
\psi :H\rightarrow G$ of groups. $\left( ii\right) $ There is a $K$%
-isomorphism $\phi :L\rightarrow M$ of fields such that the actions of $G$
on $L$ and $H$ on $M$ are $\left( \psi ,\phi \right) $\emph{-compatible},
i.e., $g\left( x\right) =\phi ^{-1}\circ \psi \left( g\right) \circ \phi
\left( x\right) $ holds for any $g\in G$ and $x\in L$. $\left( iii\right) $
Fixed any $g\in G$ and $x\in L$. There is $g\left( x\right) =x$ in $L$ if
and only if $\psi \left( g\right) \circ \phi \left( x\right) =\phi \left(
x\right) $ holds in $M$. (See \emph{Definition 1.9}).

\begin{proposition}
(\emph{Congruency classes and spatial isomorphisms}) Let $L$ be an extension
over a field $K$. Let $G$ and $H$ be two subgroups of the Galois group $%
Aut\left( L/K\right) $. Consider the case $M=L$.

Then $G$ and $H$ are $\left( L,L\right) $-spatially isomorphic if and only
if $G$ and $H$ are congruent in $Aut\left( L/K\right) $, i.e., $H=\sigma
^{-1}\cdot G\cdot \sigma $ holds for some $\sigma \in Aut\left( L/K\right) $.
\end{proposition}

\begin{proof}
Assume $G$ and $H$ are $\left( L,L\right) $-spatially isomorphic. Let $\tau
_{\phi }\left( g\right) :=\phi \circ g\circ \phi ^{-1}$. Then $\tau _{\phi
}\left( G\right) =H$.

Conversely, suppose $H=\phi \circ G\circ \phi ^{-1}$ for some $\phi \in
Aut\left( L/K\right) $. Let $\psi =\tau _{\phi }$. Then $G$ and $H$ are $%
\left( L,L\right) $-spatially isomorphic.
\end{proof}

\begin{remark}
Let $L$ be a purely transcendental extension over a field $K$ of
transcendence degree $n<+\infty $. Let $\left( \Delta ,A\right) $ and $%
\left( \Lambda ,B\right) $ be nice bases of $L/K$ such that $L$ is Galois
over $K\left( \Delta \right) $ and $K\left( \Lambda \right) $, respectively.
Suppose $\sharp A=\sharp B$.

In general, a $K$-isomorphism $f:K\left( \Delta \right) \rightarrow K\left(
\Lambda \right) $ is not necessarily extended to be a $K$-isomorphism $%
\sigma :L\rightarrow L$ of fields.

For instance, consider the case $\sharp A=\sharp B=2$ and $n=5$, where $%
\Delta $ is vertical transcendence base of $L/K$ with $L/K\left( \Delta
\right) $ being a quadratic extension, i.e., $L=K\left( t_{1},t_{2},\cdots
,t_{n}\right) $ and $\Delta =\{t_{1}^{2},t_{2},\cdots ,t_{n}\}$; $K\left(
\Lambda \right) $ is the invariant subfield of $L$ under a linear involution
$\delta $ of $L/K$ and $\delta $ is given by the square matrix
\begin{equation*}
\left(
\begin{array}{ccccc}
0 & 1 & 0 & 0 & 0 \\
1 & 0 & 0 & 0 & 0 \\
0 & 0 & 1 & 0 & 0 \\
0 & 0 & 0 & 1 & 0 \\
0 & 0 & 0 & 0 & 1%
\end{array}%
\right)
\end{equation*}%
of order $5$ relative to $\Lambda $.

Then there exists no $\sigma \in Aut\left( L/K\right) $ such that the
restriction $\sigma |_{K\left( \Delta \right) }$ is the given $K$%
-isomorphism $f:K\left( \Delta \right) \rightarrow K\left( \Lambda \right) $.
\end{remark}

\begin{lemma}
\emph{(Lifting of isomorphisms of purely transcendental subextensions, }$%
\emph{I}$\emph{)} Let $L$ be a purely transcendental extension over a field $%
K$ of finite transcendence degree. Fixed two transcendence bases $\Delta $
and $\Lambda $ of $L/K$ such that $L$ is Galois over $K\left( \Delta \right)
$ and $K\left( \Lambda \right) $, respectively. Suppose $f:K\left( \Delta
\right) \rightarrow K\left( \Lambda \right) $ is a $K$-isomorphism of fields.

Then $f$ can be extended to be a $K$-isomorphism $\sigma :L\rightarrow L$ of
fields such that $f$ is the restriction $\sigma |_{K\left( \Delta \right) }$
if and only if the Galois groups $Aut\left( L/K\left( \Delta \right) \right)
$ and $Aut\left( L/K\left( \Lambda \right) \right) $ are conjugate in $%
Aut\left( L/K\right) $.

In such a case, if $\left( \Delta ,A\right) $ and $\left( \Lambda ,B\right) $
are nice bases of $L/K$, then $\sharp A=\sharp B$ holds.
\end{lemma}

\begin{proof}
Consider the nice bases $\left( \Delta ,A\right) $ and $\left( \Lambda
,B\right) $ of $L/K$, respectively. If $K\left( \Delta \right) =K\left(
\Lambda \right) $, particularly $f$ is the identity map $id_{K\left( \Delta
\right) }$ on $K\left( \Delta \right) $, then it is trivial from the
decomposition group $D_{L/K}\left( \Delta ,A\right) $ of $L/K$ since $%
Aut\left( L/K\left( \Delta \right) \right) =Aut\left( L/K\left( \Lambda
\right) \right) $ holds in such a case. In the following we suppose $f$ is
not an identity map.

Prove $\Rightarrow $. Assume there is an automorphism $\sigma \in Aut\left(
L/K\right) $ such that the restriction $\sigma |_{K\left( \Delta \right) }$
is the map $f$. Let $\Delta ^{\prime }=\sigma \left( \Delta \right) $ and $%
A^{\prime }=\sigma \left( A\right) $. We have
\begin{equation*}
K\left( \Lambda \right) =K\left( \Delta ^{\prime }\right)
\end{equation*}%
and
\begin{equation*}
L=K\left( \Lambda \right) \left[ B\right] =K\left( \Delta ^{\prime }\right) %
\left[ A^{\prime }\right] .
\end{equation*}%
Then
\begin{equation*}
\begin{array}{l}
\pi _{a}\left( L/K\right) \left( \Lambda ,B\right) \\
=Aut\left( L/K\left( \Lambda \right) \right) \\
=Aut\left( L/K\left( \Delta ^{\prime }\right) \right) \\
=\pi _{a}\left( L/K\right) \left( \Delta ^{\prime },A^{\prime }\right) ; \\
\\
\pi _{t}\left( L/K\right) \left( \Lambda ,B\right) \\
=j_{\left( \Lambda ,B\right) }\left( Aut\left( K\left( \Lambda \right)
/K\right) \right) \\
=j_{\left( \Delta ^{\prime },A^{\prime }\right) }\left( Aut\left( K\left(
\Delta ^{\prime }\right) /K\right) \right) \\
=\pi _{t}\left( L/K\right) \left( \Delta ^{\prime },A^{\prime }\right) .%
\end{array}%
\end{equation*}

Consider the decomposition groups%
\begin{equation*}
D_{L/K}\left( \Delta ,A\right) =\pi _{a}\left( L/K\right) \left( \Delta
,A\right) \cdot \pi _{t}\left( L/K\right) \left( \Delta ,A\right)
\end{equation*}%
and%
\begin{equation*}
D_{L/K}\left( \Delta ^{\prime },A^{\prime }\right) =\pi _{a}\left(
L/K\right) \left( \Delta ^{\prime },A^{\prime }\right) \cdot \pi _{t}\left(
L/K\right) \left( \Delta ^{\prime },A^{\prime }\right)
\end{equation*}%
of $L/K$ at nice bases $\left( \Delta ,A\right) $ and $\left( \Delta
^{\prime },A^{\prime }\right) $, respectively. By \emph{Proposition 6.3} we
have%
\begin{equation*}
D_{L/K}\left( \Delta ^{\prime },A^{\prime }\right) =\sigma \cdot
D_{L/K}\left( \Delta ,A\right) \cdot \sigma ^{-1};
\end{equation*}
from the linear decomposition of automorphisms in \emph{Proposition 5.3} we
have
\begin{equation*}
\begin{array}{l}
\pi _{a}\left( L/K\right) \left( \Lambda ,B\right) \\
=\pi _{a}\left( L/K\right) \left( \Delta ^{\prime },A^{\prime }\right) \\
=\sigma \cdot \pi _{a}\left( L/K\right) \left( \Delta ,A\right) \cdot \sigma
^{-1}.%
\end{array}%
\end{equation*}

Hence, the Galois groups $Aut\left( L/K\left( \Delta \right) \right) $ and $%
Aut\left( L/K\left( \Lambda \right) \right) $ are conjugate in $Aut\left(
L/K\right) $, i.e.,
\begin{equation*}
\sigma \cdot Aut\left( L/K\left( \Delta \right) \right) \cdot \sigma
^{-1}=Aut\left( L/K\left( \Lambda \right) \right) .
\end{equation*}

As $L$ is algebraic Galois over $K\left( \Delta \right) $ and $K\left(
\Lambda \right) $, respectively, we have
\begin{equation*}
\begin{array}{l}
\sharp A \\
=\sharp Aut\left( L/K\left( \Delta \right) \right) \\
=\sharp Aut\left( L/K\left( \Lambda \right) \right) \\
=\sharp B.%
\end{array}%
\end{equation*}

Prove $\Leftarrow $. Let $Aut\left( L/K\left( \Delta \right) \right) $ and $%
Aut\left( L/K\left( \Lambda \right) \right) $ be conjugate in $Aut\left(
L/K\right) $, i.e., there is some $\sigma \in Aut\left( L/K\right) $ such
that
\begin{equation*}
\sigma \cdot Aut\left( L/K\left( \Delta \right) \right) \cdot \sigma
^{-1}=Aut\left( L/K\left( \Lambda \right) \right) .
\end{equation*}%
From \emph{Proposition 8.14} it is seen that the two Galois groups $%
Aut\left( L/K\left( \Delta \right) \right) $ and $Aut\left( L/K\left(
\Lambda \right) \right) $ are $\left( L,L\right) $-spatially isomorphic.
Then the following three conditions are satisfied:

$\left( i\right) $ There is an isomorphism $\psi :H\rightarrow G$ of groups;

$\left( ii\right) $ There is a $K$-isomorphism $\phi :L\rightarrow M=L$ of
fields such that the actions of $G$ on $L$ and $H$ on $M=L$ are $\left( \psi
,\phi \right) $\emph{-compatible}, i.e., $g\left( x\right) =\phi ^{-1}\circ
\psi \left( g\right) \circ \phi \left( x\right) $ holds for any $g\in G$ and
$x\in L$;

$\left( iii\right) $ Fixed any $g\in G$ and $x\in L$. There is $g\left(
x\right) =x$ in $L$ if and only if $\psi \left( g\right) \circ \phi \left(
x\right) =\phi \left( x\right) $ holds in $M=L$.

It follows that $\phi \in Aut\left( L/K\right) $ and
\begin{equation*}
\psi \left( g\right) =\tau _{\phi }\left( g\right) :=\phi ^{-1}\cdot g\cdot
\phi
\end{equation*}
holds in $Aut\left( L/K\right) $ for any $g\in Aut\left( L/K\left( \Delta
\right) \right) $. Put $\Delta ^{\prime }=f\left( \Delta \right) $; $%
g^{\prime }=\psi \left( g\right) $ for any $g\in Aut\left( L/K\left( \Delta
\right) \right) $; $x^{\prime }=\phi \left( x\right) $ for any $x\in L$. We
have $K\left( \Delta ^{\prime }\right) =K\left( \Lambda \right) $.

On the other hand, as $L$ is finite and Galois over $K\left( \Delta \right) $%
, we have some element $\theta \in L$ such that
\begin{equation*}
L=K\left( \Delta \right) \left[ \theta \right]
\end{equation*}
is a simple extension over $K\left( \Delta \right) $.

Prove $L=K\left( \Delta ^{\prime }\right) \left[ \theta ^{\prime }\right] $
is a simple extension over $K\left( \Delta ^{\prime }\right) $. Here, $%
\Delta ^{\prime }=f\left( \Delta \right) $ and $\theta ^{\prime }=\phi
\left( \theta \right) $. In deed, put
\begin{equation*}
n=\sharp Aut\left( L/K\left( \Delta \right) \right)
\end{equation*}
and
\begin{equation*}
Aut\left( L/K\left( \Delta \right) \right) =\{g_{0},g_{1},\cdots ,g_{n-1}\}
\end{equation*}
with $\sigma _{0}$ being the identity map $id_{L}$ on $L$. There is
\begin{equation*}
Aut\left( L/K\left( \Delta ^{\prime }\right) \right) =\{g_{0}^{\prime
},g_{1}^{\prime },\cdots ,g_{n-1}^{\prime }\}
\end{equation*}
where $g_{i}^{\prime }=\psi \left( g_{i}\right) $ for each $0\leq i\leq n-1$.

Consider the $K\left( \Delta \right) $-linear space $L$. As $0\not=\theta
\in L$, we have $0\not=g_{i}\left( \theta \right) \in L$ for $0\leq i\leq
n-1 $; as
\begin{equation*}
g_{0},g_{1},\cdots ,g_{n-1}
\end{equation*}
are $L$-linearly independent from \emph{Proposition 3.19} on Dedekind
independence, it is seen that
\begin{equation*}
g_{0}\left( \theta \right) ,g_{1}\left( \theta \right) ,\cdots
,g_{n-1}\left( \theta \right) \in L
\end{equation*}
are $K\left( \Delta \right) $-linearly independent; otherwise, if
\begin{equation*}
\sum_{i=0}^{n-1}c_{i}\cdot g_{i}\left( \theta \right) =0
\end{equation*}
holds in $L$ for some
\begin{equation*}
c_{0},c_{1},\cdots ,c_{n-1}\in K\left( \Delta \right)
\end{equation*}
which are not all zero, then
\begin{equation*}
\sum_{i=0}^{n-1}c_{i}\cdot g_{i}=0
\end{equation*}
and it follows that for any $x\in L$ there is an equality
\begin{equation*}
\left( \sum_{i=0}^{n-1}c_{i}\cdot g_{i}\right) \left( x\right) =0
\end{equation*}
as a $K\left( \Delta \right) $-linear mapping on $L$, which will be in
contradiction. Hence,
\begin{equation*}
L=span_{K\left( \Delta \right) }\{g_{0}\left( \theta \right) ,g_{1}\left(
\theta \right) ,\cdots ,g_{n-1}\left( \theta \right) \}.
\end{equation*}

In the same way, consider the $K\left( \Delta ^{\prime }\right) $-linear
space $L$. We have
\begin{equation*}
L=span_{K\left( \Delta ^{\prime }\right) }\{g_{0}^{\prime }\left( \theta
^{\prime }\right) ,g_{1}^{\prime }\left( \theta ^{\prime }\right) ,\cdots
,g_{n-1}^{\prime }\left( \theta ^{\prime }\right) \},
\end{equation*}
where
\begin{equation*}
\theta ^{\prime }\in L\setminus K\left( \Delta ^{\prime }\right)
\end{equation*}
holds from the above condition $\left( iii\right) $ since
\begin{equation*}
\theta \in L\setminus K\left( \Delta \right)
\end{equation*}
holds, i.e., there is
\begin{equation*}
g_{i}\left( \theta \right) \not=\theta
\end{equation*}
for any $1\leq i\leq n-1$. This proves
\begin{equation*}
L=K\left( \Delta ^{\prime }\right) \left[ \theta ^{\prime }\right] .
\end{equation*}

Prove there is an $K$-automorphism $\sigma \in Aut\left( L/K\right) $ which
is an extension of $f$ from $K\left( \Delta \right) $ to $L$. In deed, put
\begin{equation*}
\Delta =\{t_{1},t_{2},\cdots ,t_{m}\}
\end{equation*}
and
\begin{equation*}
\Delta ^{\prime }=\{t_{1}^{\prime },t_{2}^{\prime },\cdots ,t_{m}^{\prime }\}
\end{equation*}
where $t_{i}^{\prime }=f\left( t_{i}\right) $ for any $1\leq i\leq m$.
Consider the extensions $L/K\left( \Delta \right) $ and $L/K\left( \Delta
^{\prime }\right) $. We have
\begin{equation*}
L=K\left( \Delta \right) \left[ \theta \right] =\{P\left( \theta \right) \in
L:P\left( X\right) \in K\left( \Delta \right) \left[ X\right] \}
\end{equation*}
and
\begin{equation*}
L=K\left( \Delta ^{\prime }\right) \left[ \theta ^{\prime }\right]
=\{Q\left( \theta ^{\prime }\right) \in L:Q\left( X\right) \in K\left(
\Delta ^{\prime }\right) \left[ X\right] \}
\end{equation*}
where
\begin{equation*}
P\left( X\right) =\sum_{i=0}^{r}a_{i}X^{i}\in K\left( \Delta \right) \left[ X%
\right]
\end{equation*}
with each $a_{i}\in K\left( \Delta \right) $ and
\begin{equation*}
Q\left( X\right) =\sum_{i=0}^{s}b_{i}X^{i}\in K\left( \Delta ^{\prime
}\right) \left[ X\right]
\end{equation*}
with each $b_{i}\in K\left( \Delta ^{\prime }\right) $.

Set a mapping
\begin{equation*}
\sigma :L\rightarrow L
\end{equation*}
of sets given by
\begin{equation*}
P\left( \theta \right) =\sum_{i=0}^{r}a_{i}\cdot \theta ^{i}\longmapsto
Q\left( \theta ^{\prime }\right) =\sum_{i=0}^{r}f\left( b_{i}\right) \cdot
\theta ^{\prime i}
\end{equation*}
for each polynomial
\begin{equation*}
P\left( X\right) \in K\left( \Delta \right) \left[ X\right] .
\end{equation*}

Then $\sigma $ is the desired $K$-automorphism of the field $L$ such that $%
\sigma |_{K\left( \Delta \right) }=f$ holds. This completes the proof.
\end{proof}

\begin{lemma}
\emph{(Lifting of isomorphisms of purely transcendental subextensions, }$%
\emph{II}$\emph{)} Let $L$ and $M$ be purely transcendental extensions over
a field $K$ of the same finite transcendence degree. Fixed transcendence
bases $\Delta $ of $L/K$ and $\Lambda $ of $M/K$ such that $L/K\left( \Delta
\right) $ and $M/K\left( \Lambda \right) $ both are Galois extensions,
respectively. Suppose $f:K\left( \Delta \right) \rightarrow K\left( \Lambda
\right) $ is a $K$-isomorphism of fields.

Then there is a $K$-isomorphism $\sigma :L\rightarrow M$ of fields such that
$f$ is the restriction $\sigma |_{K\left( \Delta \right) }$ if and only if
the Galois groups $Aut\left( L/K\left( \Delta \right) \right) $ and $%
Aut\left( M/K\left( \Lambda \right) \right) $ are $\left( L,M\right) $%
-spatially isomorphic.
\end{lemma}

\begin{proof}
Immediately from the same procedure as in the proof of \emph{Lemma 8.16} by
taking $\left( \Delta ,A\right) $ and $\left( \Delta ^{\prime },A^{\prime
}\right) $ as nice bases of the extensions $L/K$ and $M/K$ (with $K\left(
\Delta ^{\prime }\right) =K\left( \Lambda \right) $), respectively.
\end{proof}

\begin{lemma}
\emph{(Lifting of isomorphisms of purely transcendental subextensions, }$%
\emph{III}$\emph{)} Let $L$ and $M$ be purely transcendental extensions over
a field $K$ of the same finite transcendence degree. Then the Galois groups $%
Aut\left( L/K\right) $ of the extension $L/K$ and $Aut\left( M/K\right) $ of
$M/K$ are (automatically) $\left( L,M\right) $-spatially isomorphic.
\end{lemma}

\begin{proof}
Fixed a $K$-isomorphism $\phi :L\rightarrow M$. For any $x\in L$ and $g\in
Aut\left( L/K\right) $, let $y=\phi \left( x\right) $ and put
\begin{equation*}
\psi \left( g\right) \left( y\right) =\phi \circ g\circ \phi ^{-1}\left(
y\right) .
\end{equation*}%
There is $\psi \left( g\right) \in Aut\left( M/K\right) $ and then
\begin{equation*}
\psi \left( Aut\left( L/K\right) \right) =Aut\left( M/K\right)
\end{equation*}%
holds. On the other hand, take any $g,g^{\prime }\in Aut\left( L/K\right) $.
For any $y\in M$ we have
\begin{equation*}
\begin{array}{l}
x=\phi ^{-1}\left( y\right) \in L; \\
\\
\psi \left( g\circ g^{\prime }\right) \left( y\right) \\
=\phi \circ \left( g\circ g^{\prime }\right) \circ \phi ^{-1}\left( y\right)
\\
=\phi \circ g\circ g^{\prime }\left( x\right) \\
=\left( \phi \circ g\circ \phi ^{-1}\right) \left( \phi \circ g^{\prime
}\left( x\right) \right) \\
=\left( \phi \circ g\circ \phi ^{-1}\right) \circ \left( \phi \circ
g^{\prime }\circ \phi ^{-1}\right) \left( y\right) .%
\end{array}%
\end{equation*}%
This proves
\begin{equation*}
\psi :Aut\left( L/K\right) \rightarrow Aut\left( M/K\right)
\end{equation*}%
is a $K$-isomorphism of fields.
\end{proof}

\begin{remark}
Let $L$ and $M$ be purely transcendental extensions over a field $K$ of the
same finite transcendence degree. Fixed transcendence bases $\Delta $ of the
extension $L/K$ and $\Lambda $ of $M/K$ such that $L/K\left( \Delta \right) $
and $M/K\left( \Lambda \right) $ both are Galois extensions, respectively.

$\left( i\right) $ From \emph{Lemma 8.18}, the Galois groups $Aut\left(
K\left( \Delta \right) /K\right) $ and $Aut\left( K\left( \Lambda \right)
/K\right) $ are $\left( K\left( \Delta \right) ,K\left( \Lambda \right)
\right) $-spatially isomorphic.

$\left( ii\right) $ Consider the highest transcendental subgroups $\pi
_{t}\left( L/K\right) \left( \Delta ,A\right) $ of $L/K$ and $\pi _{t}\left(
M/K\right) \left( \Lambda ,B\right) $ of $M/K$, where $\left( \Delta
,A\right) $ and $\left( \Lambda ,B\right) $ are nice bases of $L/K$ and $M/K$%
, respectively. We have
\begin{equation*}
\pi _{t}\left( L/K\right) \left( \Delta ,A\right) =j_{\left( \Delta
,A\right) }\left( Aut\left( K\left( \Delta \right) /K\right) \right);
\end{equation*}
\begin{equation*}
\pi _{t}\left( L/K\right) \left( \Lambda ,B\right) =j_{\left( \Lambda
,B\right) }\left( Aut\left( K\left( \Lambda \right) /K\right) \right) .
\end{equation*}
However, $\pi _{t}\left( L/K\right) \left( \Delta ,A\right) $ and $\pi
_{t}\left( M/K\right) \left( \Lambda ,B\right) $ are not necessarily $\left(
L,M\right) $-spatially isomorphic in general.
\end{remark}

Recall that for an extension $L$ over a field $K$, $D_{L/K}\left( \Delta
,A\right) $, $\pi _{a}\left( L/K\right) \left( \Delta ,A\right) $ and $\pi
_{t}\left( L/K\right) \left( \Delta ,A\right) $ denote the decomposition
group, full algebraic subgroup and highest transcendental subgroup of $L/K$
at a nice basis $\left( \Delta ,A\right) $, respectively. See \emph{\S \S 4-5%
} for details.

\begin{theorem}
\emph{(Lifting of isomorphisms and spatial isomorphisms)} Let $L$ and $M$ be
purely transcendental extensions over a field $K$ of the same finite
transcendence degree. Fixed nice bases $\left( \Delta ,A\right) $ of $L/K$
and $\left( \Lambda ,B\right) $ of $M/K$ such that $L/K\left( \Delta \right)
$ and $M/K\left( \Lambda \right) $ both are Galois extensions, respectively.
The following statements are equivalent.

$\left( i\right) $ $Aut\left( L/K\left( \Delta \right) \right) $ and $%
Aut\left( M/K\left( \Lambda \right) \right) $ are $\left( L,M\right) $%
-spatially isomorphic.

$\left( ii\right) $ $\pi _{a}\left( L/K\right) \left( \Delta ,A\right) $ and
$\pi _{a}\left( M/K\right) \left( \Lambda ,B\right) $ are $\left( L,M\right)
$-spatially isomorphic.

$\left( iii\right) $ $\pi _{t}\left( L/K\right) \left( \Delta ,A\right) $
and $\pi _{t}\left( M/K\right) \left( \Lambda ,B\right) $ are $\left(
L,M\right) $-spatially isomorphic.

$\left( iv\right) $ $D_{L/K}\left( \Delta ,A\right) $ and $D_{M/K}\left(
\Lambda ,B\right) $ are $\left( L,M\right) $-spatially isomorphic.

$\left( v\right) $ For each $K$-isomorphism $f:K\left( \Delta \right)
\rightarrow K\left( \Lambda \right) $ of fields, there are finitely many $K$%
-isomorphism $\sigma :L\rightarrow M$ of fields whose restrictions $\sigma
|_{K\left( \Delta \right) }$ to $K\left( \Delta \right) $ are $f$.

$\left( vi\right) $ There is some $K$-isomorphism $f:K\left( \Delta \right)
\rightarrow K\left( \Lambda \right) $ of fields such that there exists a $K$%
-isomorphism $\sigma :L\rightarrow M$ of fields with $f=\sigma |_{K\left(
\Delta \right) }$ being the restriction.
\end{theorem}

\begin{proof}
$\left( i\right) \Leftrightarrow \left( ii\right) $. It is from the fact
that $Aut\left( L/K\left( \Delta \right) \right) =\pi _{a}\left( L/K\right)
\left( \Delta ,A\right) $ and $Aut\left( M/K\left( \Lambda \right) \right)
=\pi _{a}\left( M/K\right) \left( \Lambda ,B\right) $ hold.

$\left( i\right) \Leftrightarrow \left( v\right) $. It is from \emph{Lemma
8.17}. In such a case, fixed a $K$-isomorphism $f:K\left( \Delta \right)
\rightarrow K\left( \Lambda \right) $ of fields and a $K$-isomorphism $%
\sigma :L\rightarrow M$ of fields with $f=\sigma |_{K\left( \Delta \right) }$%
. It is seen that the Galois group $Aut\left( L/K\left( \Delta \right)
\right) $ is a finite group and for each $g\in Aut\left( L/K\left( \Delta
\right) \right) $ the composite $\sigma \circ g:L\rightarrow M$ is a $K$%
-isomorphism of fields with $f=\sigma |_{K\left( \Delta \right) }$.

$\left( ii\right) \Rightarrow \left( iii\right) $. Assume $\pi _{a}\left(
L/K\right) \left( \Delta ,A\right) $ and $\pi _{a}\left( M/K\right) \left(
\Lambda ,B\right) $ are $\left( L,M\right) $-spatially isomorphic. There is
an isomorphism
\begin{equation*}
\psi :\pi _{a}\left( L/K\right) \left( \Delta ,A\right) \rightarrow \pi
_{a}\left( M/K\right) \left( \Lambda ,B\right)
\end{equation*}
of groups and a $K$-isomorphism $\phi :L\rightarrow M$ of fields such that $%
g\left( x\right) =\phi ^{-1}\circ \psi \left( g\right) \circ \phi \left(
x\right) $ holds for any $g\in \pi _{a}\left( L/K\right) \left( \Delta
,A\right) $ and $x\in L$ and that for any $g\in \pi _{a}\left( L/K\right)
\left( \Delta ,A\right) $ and $x\in L$ there is $g\left( x\right) =x$ in $L$
if and only if $\psi \left( g\right) \left( \phi \left( x\right) \right)
=\phi \left( x\right) $ holds in $M$.

Consider the linear decomposition of isomorphisms in \emph{\S 5.2} and apply
\emph{Proposition 5.3} to the case here. Put
\begin{equation*}
E=K\left( \Delta \right) ,F=span_{K}\left( L\setminus K\left( \Delta \right)
\right) ;
\end{equation*}
\begin{equation*}
E^{\prime }=K\left( \Lambda \right) ,F^{\prime }=span_{K}\left( M\setminus
K\left( \Lambda \right) \right) .
\end{equation*}
We will proceed in several steps.

\emph{Step 1}. Prove $\phi \left( E\right) =E^{\prime }$. In deed, $\phi
\left( E\right) \subseteq E^{\prime }$ holds; otherwise, if there is some
nonzero $x_{0}\in E$ with $\phi \left( x_{0}\right) \in F^{\prime }$, then
there is some $g_{0}\in \pi _{a}\left( L/K\right) \left( \Delta ,A\right) $
such that $\psi \left( g_{0}\right) \left( \phi \left( x_{0}\right) \right)
\not=\phi \left( x_{0}\right) $ since $\psi \left( \pi _{a}\left( L/K\right)
\left( \Delta ,A\right) \right) =\pi _{a}\left( M/K\right) \left( \Lambda
,B\right) $ holds and $K\left( \Lambda \right) $ is the invariant subfield
of $M$ under the Galois group $\pi _{a}\left( M/K\right) \left( \Lambda
,B\right) $; it follows that $g_{0}\left( x_{0}\right) \not=x_{0}$ holds,
which will be in contradiction. Hence, $\phi \left( E\right) \subseteq
E^{\prime }$. In the same, by interchanging $L$ and $M$, consider the $K$%
-isomorphism $\phi ^{-1}:M\rightarrow L$ and we have $\phi ^{-1}\left(
E^{\prime }\right) \subseteq E$. This proves $\phi \left( E\right)
=E^{\prime }$.

\emph{Step 2}. Prove $\phi \left( F\right) =F^{\prime }$. In deed, $\phi
\left( F\right) \subseteq F^{\prime }$ holds; otherwise, if there is some
nonzero $z_{0}\in E$ with $\phi \left( z_{0}\right) \in E^{\prime }$, then
for every $h^{\prime }\in \pi _{a}\left( M/K\right) \left( \Lambda ,B\right)
$ there is $h^{\prime }\left( \phi \left( z_{0}\right) \right) =\phi \left(
z_{0}\right) $ since $\psi \left( \pi _{a}\left( L/K\right) \left( \Delta
,A\right) \right) =\pi _{a}\left( M/K\right) \left( \Lambda ,B\right) $
holds and $K\left( \Lambda \right) $ is the invariant subfield of $M$ under
the Galois group $\pi _{a}\left( M/K\right) \left( \Lambda ,B\right) $; it
follows that $h\left( x_{0}\right) =x_{0}$ holds for any $h=\psi ^{-1}\left(
h^{\prime }\right) \in \pi _{a}\left( L/K\right) \left( \Delta ,A\right) $,
which will be in contradiction. Hence, $\phi \left( F\right) \subseteq
F^{\prime }$. In the same, by interchanging $L$ and $M$, consider the $K$%
-isomorphism $\phi ^{-1}:M\rightarrow L$ and we have $\phi ^{-1}\left(
F^{\prime }\right) \subseteq F$. This proves $\phi \left( F\right)
=F^{\prime }$.

It follows that the restriction $\phi |_{E}:K\left( \Delta \right)
\rightarrow K\left( \Lambda \right) $ is a $K$-isomorphism of fields.

\emph{Step 3}. Prove $\pi _{t}\left( L/K\right) \left( \Delta ,A\right) $
and $\pi _{t}\left( M/K\right) \left( \Lambda ,B\right) $ are $\left(
L,M\right) $-spatially isomorphic. In deed, consider the given $K$%
-isomorphism $\phi :L\rightarrow M$ of fields. Put $\psi \left( h\right)
=\phi \circ h\circ \phi ^{-1}$ for any $h\in \pi _{t}\left( L/K\right)
\left( \Delta ,A\right) $. It is seen that
\begin{equation*}
\psi :\pi _{t}\left( L/K\right) \left( \Delta ,A\right) \rightarrow \pi
_{t}\left( M/K\right) \left( \Lambda ,B\right)
\end{equation*}
is an isomorphism of groups with $\psi $ and $\phi $ being compatible. On
the other hand, take any $h\in \pi _{t}\left( L/K\right) \left( \Delta
,A\right) $ and $z\in L$. From \emph{Step 2} we have $h\left( z\right) =z$
in $L$ if and only if $\psi \left( g\right) \left( \phi \left( z\right)
\right) =\phi \left( z\right) $ holds in $M$ in virtue of \emph{Proposition
5.3}. This gives the desired $\left( L,M\right) $-spatial isomorphism.

$\left( ii\right) \Rightarrow \left( iv\right) $. Consider the $K$%
-isomorphism $\phi :L\rightarrow M$ of fields in above \emph{Step 3}. Put $%
\psi \left( h\right) =\phi \circ h\circ \phi ^{-1}$ for any $h\in
D_{L/K}\left( \Delta ,A\right) $. It is seen that $D_{L/K}\left( \Delta
,A\right) $ and $D_{M/K}\left( \Lambda ,B\right) $ are $\left( L,M\right) $%
-spatially isomorphic.

$\left( iv\right) \Rightarrow \left( ii\right) $. Just applying the same
procedure in above \emph{Step 3} in virtue of \emph{Steps 1-2} and \emph{%
Theorem 5.4}.

$\left( iii\right) \Rightarrow \left( vi\right) $. Suppose $\pi _{t}\left(
L/K\right) \left( \Delta ,A\right) $ and $\pi _{t}\left( M/K\right) \left(
\Lambda ,B\right) $ are $\left( L,M\right) $-spatially isomorphic. There is
an isomorphism
\begin{equation*}
\psi :\pi _{t}\left( L/K\right) \left( \Delta ,A\right) \rightarrow \pi
_{t}\left( M/K\right) \left( \Lambda ,B\right)
\end{equation*}
of groups and a $K$-isomorphism $\phi :L\rightarrow M$ of fields such that $%
g\left( z\right) =\phi ^{-1}\circ \psi \left( h\right) \circ \phi \left(
z\right) $ holds for any $h\in \pi _{t}\left( L/K\right) \left( \Delta
,A\right) $ and $z\in L$ and that for any $h\in \pi _{t}\left( L/K\right)
\left( \Delta ,A\right) $ and $z\in L$ there is $h\left( z\right) =z$ in $L$
if and only if $\psi \left( h\right) \left( \phi \left( z\right) \right)
=\phi \left( z\right) $ holds in $M$.

Consider the restriction $\phi |_{K\left( \Delta \right) }:K\left( \Delta
\right) \rightarrow M$.

Prove $\phi \left( K\left( \Delta \right) \right) =K\left( \Lambda \right) $%
. In deed, given any $\sigma \in \pi _{t}\left( L/K\right) \left( \Delta
,A\right) $. From \emph{Proposition 5.3} again there is the $E$-component $%
\sigma _{E}$ and $F$-component $1_{F}$ of $\sigma $, where $%
E=span_{K}K\left( \Delta \right) $ and $F=span_{K}\left( L\setminus K\left(
\Delta \right) \right) $. In the same way, given any $\sigma ^{\prime }\in
\pi _{t}\left( M/K\right) \left( \Lambda ,B\right) $ and there is the $%
E^{\prime }$-component $\sigma _{E^{\prime }}^{\prime }$ and $F^{\prime }$%
-component $1_{F^{\prime }}$ of $\sigma ^{\prime }$, where $E^{\prime
}=span_{K}K\left( \Lambda \right) $ and $F^{\prime }=span_{K}\left(
M\setminus K\left( \Lambda \right) \right) $. Let $z=z_{E}+0_{F}\in \Delta
\subseteq K\left( \Delta \right) $ and $z^{\prime }=\phi \left( z\right) \in
M$. There is some $g\in \pi _{t}\left( L/K\right) \left( \Delta ,A\right) $
such that $g\left( z\right) \not=z$. It follows that $\psi \left( g\right)
\in \pi _{t}\left( M/K\right) \left( \Lambda ,B\right) $ and $\psi \left(
g\right) \left( z^{\prime }\right) \not=z^{\prime }$. Then we must have $%
z^{\prime }=z_{E^{\prime }}^{\prime }+0_{F^{\prime }}\in E^{\prime }=K\left(
\Lambda \right) $; otherwise, if $z^{\prime }\not\in E^{\prime }$, the $%
F^{\prime }$-component $z_{F^{\prime }}^{\prime }$ of $z^{\prime }$ will not
be zero;

It follows that $\phi |_{K\left( \Delta \right) }:K\left( \Delta \right)
\rightarrow K\left( \Lambda \right) $ is a $K$-isomorphism of fields. Hence,
$\phi $ is an extension of $\phi |_{K\left( \Delta \right) }$.

$\left( vi\right) \Rightarrow \left( i\right) $. It is from \emph{Lemma 8.17}%
.

$\left( v\right) \Rightarrow \left( vi\right) $. Trivial. This completes the
proof.
\end{proof}

\subsection{Noether solutions and spatial isomorphisms}

As application of the above \emph{Theorem 8.20} we have the following
theorems on Noether solutions and spatial isomorphisms, which will be
applied for us to prove the main theorems in \emph{\S 9}.

\begin{theorem}
\emph{(Noether solutions and spatial isomorphisms, I)} Let $L$ and $M$ be
purely transcendental extensions over a field $K$ of the same finite
transcendence degree. Suppose $G$ is a subgroup of $Aut\left( M/K\right) $.

Then $G$ is a Noether solution of $M/K$ if and only if there is a nice basis
$\left( \Delta ,A\right) $ of $L/K$ satisfying the two conditions: $\left(
i\right) $ $L$ is Galois over $K\left( \Delta \right) $; $\left( ii\right) $
$Aut\left( L/K\left( \Delta \right) \right) $ and $G$ are $\left( L,M\right)
$-spatially isomorphic groups.
\end{theorem}

\begin{proof}
$\Rightarrow $ Assume $G$ is a Noether solution of $M/K$, i.e., $M$ is
algebraic Galois over the invariant subfield $M^{G}$ and $M^{G}$ is purely
transcendental over $K$. Fixed a $K$-isomorphism $\phi :M\rightarrow L$ of
fields. Let $\left( \Lambda ,B\right) $ be a nice basis of $M/K$ with $%
M^{G}=K\left( \Lambda \right) $.

Prove $\left( \phi \left( \Lambda \right) ,\phi \left( B\right) \right) $ is
a nice basis of $L/K$. In deed, $\phi \left( B\right) $ is a linear basis of
$L=\phi \left( M\right) $ as a vector space over the subfield $K\left( \phi
\left( \Lambda \right) \right) $. As each $v$ of $M$ is a linear combination
of $B$ with coefficients in $K\left( \Lambda \right) $, it is seen that $%
\phi \left( v\right) \in K\left( \phi \left( \Lambda \right) \right) \left[
\phi \left( B\right) \right] $ and then $L=K\left( \phi \left( \Lambda
\right) \right) \left[ \phi \left( B\right) \right] $ holds by the
isomorphism $\phi $.

Prove $L$ is algebraic Galois over $K\left( \phi \left( \Lambda \right)
\right) $. In deed, $L$ is Galois over $K\left( \Delta \right) $, via the
isomorphism $\phi $ it is seen that $L$ is normal and separable over $%
K\left( \phi \left( \Lambda \right) \right) $.

On the other hand, as $\phi $ is the extension of the restriction
\begin{equation*}
\phi |_{K\left( \Lambda \right) }:K\left( \Lambda \right) \rightarrow
K\left( \phi \left( \Lambda \right) \right) ,
\end{equation*}%
it is seen that $Aut\left( L/K\left( \phi \left( \Lambda \right) \right)
\right) $ and $G=Aut\left( M/K\left( \Lambda \right) \right) $ are $\left(
L,M\right) $-spatially isomorphic groups from \emph{Theorem 8.20}. Here, $%
\Delta =\phi \left( \Lambda \right) $.

$\Leftarrow $ Suppose there is a nice basis $\left( \Delta ,A\right) $ of $%
L/K$ such that $L$ is Galois over $K\left( \Delta \right) $ and that $%
H:=Aut\left( L/K\left( \Delta \right) \right) $ and $G$ are $\left(
L,M\right) $-spatially isomorphic groups. That is, the three conditions are
satisfied:

$\left( a\right) $ There is an isomorphism $\psi :H\rightarrow G$ of groups;

$\left( b\right) $ There is a $K$-isomorphism $\phi :L\rightarrow M$ of
fields such that
\begin{equation*}
g\left( x\right) =\phi ^{-1}\circ \psi \left( g\right) \circ \phi \left(
x\right)
\end{equation*}
holds for any $g\in H$ and $x\in L$;

$\left( c\right) $ Fixed any $g\in H$ and $x\in L$. There is $g\left(
x\right) =x$ in $L$ if and only if $\psi \left( g\right) \circ \phi \left(
x\right) =\phi \left( x\right) $ holds in $M$.

In the same way as above, it is seen that $\left( \phi \left( \Delta \right)
,\phi \left( A\right) \right) $ is a nice basis of $M/K$ and that $M$ is
algebraic Galois over $K\left( \phi \left( \Delta \right) \right) $. Put $%
\Lambda =\phi \left( \Delta \right) $.

Prove $Aut\left( L/K\left( \Delta \right) \right) $ and $Aut\left( M/K\left(
\Lambda \right) \right) $ are $\left( L,M\right) $-spatially isomorphic. In
deed, from the assumption that $\psi \left( g\right) =\phi \circ g\circ \phi
^{-1} $ holds for any $g\in Aut\left( L/K\left( \Delta \right) \right) $, it
is seen that $\psi :Aut\left( L/K\left( \Delta \right) \right) \rightarrow
Aut\left( M/K\left( \Lambda \right) \right) $ is an isomorphism of groups.
Let $x\in L$ and $y=\phi \left( x\right) \in M$. Let $g\in Aut\left(
L/K\left( \Delta \right) \right) $ and $h=\psi \left( g\right) $. Then $%
g\left( x\right) =x$ holds if and only if $h\left( y\right) =y$ holds.

In particular, we have $G=\psi \left( Aut\left( L/K\left( \Delta \right)
\right) \right) =Aut\left( M/K\left( \Lambda \right) \right) . $ As $M$ is
algebraic Galois over $K\left( \Lambda \right) $, it is seen that $K\left(
\Lambda \right) $ is the $G$-invariant subfield $M^{G}$ of $M$. Hence, $%
M^{G}=M^{Aut\left( M/K\left( \Lambda \right) \right) }=K\left( \Lambda
\right) $ hold and it follows that $M^{G}$ is purely transcendental over $K$%
. This proves that $G$ is a Noether solution of $M/K$.
\end{proof}

\begin{theorem}
\emph{(Noether solutions and spatial isomorphisms, II)} Let $L$ and $M$ be
purely transcendental extensions over a field $K$ of the same finite
transcendence degree. Suppose $G$ is a subgroup of $Aut\left( M/K\right) $.

Then $G$ is a Noether solution of $M/K$ if and only if there is a subgroup $%
H $ of the Galois group $Aut\left( M/K\right) $ satisfying the three
conditions:

$\ \left( N1\right) $ There is $G\cap H=\{1\}$ and $\sigma \cdot \delta
=\delta \cdot \sigma $ holds in $Aut\left( L/K\right) $ for any $\sigma \in
G $ and $\delta \in H$.

$\ \left( N2\right) $ $K$ is the invariant subfield $L^{\left\langle G\cup
H\right\rangle }$ of $L$ under the subgroup $\left\langle G\cup
H\right\rangle $ generated in $Aut\left( L/K\right) $ by the subset $G\cup H$%
.

$\ \left( N3\right) $ There is a nice basis $\left( \Delta ,A\right) $ of $%
L/K$ having the two properties: $\left( a\right) $ $L$ is Galois over $%
K\left( \Delta \right) $; $\left( b\right) $ The highest transcendental
Galois subgroup $\pi _{t}\left( L/K\right) \left( \Delta ,A\right) $ and $H$
are $\left( L,M\right) $-spatially isomorphic groups.

In such a case, $H$ is the highest transcendental Galois subgroup $\pi
_{t}\left( M/K\right) \left( \Lambda ,B\right) $ of $M/K$. Here, $\left(
\Lambda ,B\right) $ is any nice basis of $M/K$ such that $K\left( \Lambda
\right) $ is the $G$-invariant subfield $M^{G}$ of $M$.
\end{theorem}

\begin{proof}
$\Rightarrow $ Assume $G$ is a Noether solution of $M/K$. Fixed a nice basis
$\left( \Lambda ,B\right) $ of $M/K$ such that $K\left( \Lambda \right) $ is
the $G$-invariant subfield $M^{G}$ of $M$. Let $\phi :L\rightarrow M$ be a $%
K $-isomorphism of fields and $\left( \Delta ,A\right) $ a nice basis of $%
L/K $ with $K\left( \Delta \right) =\phi ^{-1}\left( M^{G}\right) $. Suppose
$\left( \Lambda ,B\right) $ is a nice basis of $M/K$ with $K\left( \Lambda
\right) =M^{G}$.

Then $G$ is the full algebraic Galois subgroup $\pi _{a}\left( M/K\right)
\left( \Lambda ,B\right) $ and $H$ is the highest transcendental Galois
subgroup $\pi _{t}\left( M/K\right) \left( \Lambda ,B\right) $ of $M/K$
satisfying the three conditions $\left( N1\right) -\left( N3\right) $.

$\Leftarrow $ Suppose there is a subgroup $H$ of $Aut\left( M/K\right) $
satisfying the three conditions $\left( N1\right) -\left( N3\right) $, where
$\left( \Delta ,A\right) $ is a nice basis of $L/K$ such that $L$ is Galois
over $K\left( \Delta \right) $ and that $\pi _{t}\left( L/K\right) \left(
\Delta ,A\right) $ and $H$ are $\left( L,M\right) $-spatially isomorphic via
$\left( \psi ,\phi \right) $, that is,

$\left( a\right) $ There is an isomorphism $\psi :\pi _{t}\left( L/K\right)
\left( \Delta ,A\right) \rightarrow H$ of groups;

$\left( b\right) $ There is a $K$-isomorphism $\phi :L\rightarrow M$ of
fields such that $g\left( x\right) =\phi ^{-1}\circ \psi \left( g\right)
\circ \phi \left( x\right) $ holds for any $g\in \pi _{t}\left( L/K\right)
\left( \Delta ,A\right) $ and $x\in L$;

$\left( c\right) $ Fixed any $g\in \pi _{t}\left( L/K\right) \left( \Delta
,A\right) $ and $x\in L$. There is $g\left( x\right) =x$ in $L$ if and only
if $\psi \left( g\right) \circ \phi \left( x\right) =\phi \left( x\right) $
holds in $M$.

Put $\Lambda =\phi \left( \Delta \right) $ and $B=\phi \left( A\right) $. It
is seen that $\left( \Lambda ,B\right) $ is a nice basis of $M/K$ and that $%
M $ is algebraic Galois over $K\left( \Lambda \right) $. We have
\begin{equation*}
H=\psi \left( \pi _{t}\left( L/K\right) \left( \Delta ,A\right) \right) =\pi
_{t}\left( M/K\right) \left( \Lambda ,B\right) .
\end{equation*}

Consider the groups $\pi _{a}\left( L/K\right) \left( \Delta ,A\right) $ and
$\pi _{t}\left( L/K\right) \left( \Delta ,A\right) $ for $L/K$ and the
groups $G$ and $H$ for $M/K$, respectively. From \emph{Lemma 6.11} and \emph{%
Theorem 6.15} it is seen that $\pi _{a}\left( L/K\right) \left( \Delta
,A\right) $ and $G$ are $\left( L,M\right) $-spatially isomorphic groups via
the $\left( L,M\right) $-spatial isomorphism $\left( \psi ,\phi \right) $.
It follows that we have
\begin{equation*}
M^{G}=M^{\pi _{a}\left( M/K\right) \left( \Lambda ,B\right) }=K\left(
\Lambda \right)
\end{equation*}%
where $G=\psi \left( \pi _{a}\left( L/K\right) \left( \Delta ,A\right)
\right) =\pi _{a}\left( M/K\right) \left( \Lambda ,B\right) . $ This proves
the $G$-invariant subfield $M^{G}$ is purely transcendental over $K$.
\end{proof}

\begin{theorem}
\emph{(Field isomorphisms and spatial isomorphisms)} Let $L$ and $M$ be
purely transcendental extensions over a field $K$ of the same finite
transcendence degree. Fixed a $K$-isomorphism $\phi :L\rightarrow M$ of
fields and a Noether solution $G$ of $L/K$. Suppose $G_{\phi }:=Aut\left(
M/\phi \left( L^{G}\right) \right) $ is the Galois group of $M$ over the
image $\phi \left( L^{G}\right) $ of the $G$-invariant subfield $L^{G}$
under the mapping $\phi $.

Then $G_{\phi }=\{\phi \circ g\circ \phi ^{-1}:g\in G\}$ holds and the two
Galois groups $G$ and $G_{\phi }$ are $\left( L,M\right) $-spatially
isomorphic. In particular, $G_{\phi }$ is also a Noether solution of $M/K$.
\end{theorem}

\begin{proof}
$\left( i\right) $ Assume $\left( \Delta ,A\right) $ is a nice basis of $L/K$
such that $L^{G}=K\left( \Delta \right) $. Let $\Delta ^{\prime }=\phi
\left( \Delta \right) $ and $A^{\prime }=\phi \left( A\right) $. We have $%
M=K\left( \Delta ^{\prime }\right) \left[ A^{\prime }\right] $ since $%
L=K\left( \Delta \right) \left[ A\right] $ holds. In particular, as $L$ is
algebraic Galois over $K\left( \Delta \right) $, it is seen that $M$ is an
algebraic, normal and separable extension over $K\left( \Delta ^{\prime
}\right) $; then $M$ is algebraic Galois over $K\left( \Delta ^{\prime
}\right) $. It follows that we have $G_{\phi }=Aut\left( M/M^{G_{\phi
}}\right) ; $ $M^{G_{\phi }}=K\left( \Delta ^{\prime }\right) =K\left( \phi
\left( \Delta \right) \right) =\phi \left( K\left( \Delta \right) \right)
=\phi \left( L^{G}\right) . $ Hence, $G_{\phi }$ is a Noether solution of $%
M/K$.

$\left( ii\right) $ Let $G^{\prime }=\{\phi \circ g\circ \phi ^{-1}:g\in G\}$%
. It is seen that $G^{\prime }\subseteq G_{\phi }$ holds. On the other hand,
put $\psi \left( g\right) =\phi \circ g\circ \phi ^{-1}$ for any $g\in G$.
For any two $g,g^{\prime }\in G$, we have $\psi \left( g\right) =\psi \left(
g^{\prime }\right) $ if and only if $g=g^{\prime }$ holds. It follows that
we have
\begin{equation*}
\sharp G_{\phi }=\left[ M:K\left( \Delta ^{\prime }\right) \right] =\left[
L:K\left( \Delta \right) \right] =\sharp G=\sharp G^{\prime }.
\end{equation*}
Hence, $G^{\prime }=G_{\phi }$ holds.

$\left( iii\right) $ As $\phi $ is the extension of the mapping $\phi
|_{K\left( \Delta \right) }:K\left( \Delta \right) \rightarrow K\left(
\Delta ^{\prime }\right) $, from \emph{Theorem 8.20} it is seen that $%
G=Aut\left( L/K\left( \Delta \right) \right) $ and $G_{\phi }=Aut\left(
M/K\left( \Delta ^{\prime }\right) \right) $ are $\left( L,M\right) $%
-spatially isomorphic via the spatial isomorphism $\left( \psi ,\phi \right)
$, where $\psi :G\rightarrow G_{\phi }$ is an isomorphism of groups given by
$g\mapsto \psi \left( g\right) $ for any $g\in G$.
\end{proof}

\subsection{The inner action of the Galois group on Noether solutions}

Let $L $ be a purely transcendental extension over a field $K$ of finite
transcendence degree. Let $G$ be a subgroup of the Galois group $Aut\left(
L/K\right) $.

\begin{remark}
(\emph{The inner action of the Galois group on its subgroups}) For a $K$%
-isomorphism $\phi :L\rightarrow L$ of fields, put
\begin{equation*}
G_{\phi }:=\{\phi \circ g\circ \phi ^{-1}\in Aut\left( L/K\right) :g\in G\}.
\end{equation*}

$\left( i\right) $ Immediately, $G_{\phi }$ is a subgroup of $Aut\left(
L/K\right) $. It offers an \emph{inner action} of the Galois group $%
Aut\left( L/K\right) $ on its subgroups $G\subseteq Aut\left( L/K\right) $.

$\left( ii\right) $ Let $Aut\left( L/K\right) ^{G}$ denote the \emph{%
isotropy subgroup} of $Aut\left( L/K\right) $ under the inner action of the
Galois group $Aut\left( L/K\right) $ on the subgroup $G$. Then $Aut\left(
L/K\right) ^{G}$ is the set of the $G$-invariant isomorphisms $\phi $ in $%
Aut\left( L/K\right) $.

Here, a $K$-isomorphism $\phi \in Aut\left( L/K\right) $ is said to be $G$%
\textbf{-invariant} if $G=G_{\phi }$ holds in $Aut\left( L/K\right) $;
otherwise, $\phi \in Aut\left( L/K\right) $ is said to be $G$\textbf{-variant%
} if there is $G\not=G_{\phi }$ in $Aut\left( L/K\right) $.

$\left( iii\right) $ Let $G^{Aut\left( L/K\right) }$ denote the \emph{orbit}
of the subgroup $G$ under the inner action of the Galois group $Aut\left(
L/K\right) $ on $G$. As usual, $G^{Aut\left( L/K\right) }$ is a set
bijective to the quotient set $Aut\left( L/K\right) /Aut\left( L/K\right)
^{G}$ of the left cosets of the subgroup $Aut\left( L/K\right) ^{G}$ in the
Galois group $Aut\left( L/K\right) $.
\end{remark}

\begin{proposition}
Assume $G$ is a Noether solution of $L/K$. Fixed a nice basis $\left( \Delta
,A\right) $ of $L/K$ with $L^{G}=K\left( \Delta \right) $. Let $\phi
_{1},\phi _{2}\in Aut\left( L/K\right) $. Then
\begin{equation*}
G=G_{\phi _{2}^{-1}\circ \phi _{1}}
\end{equation*}%
holds if and only if there is
\begin{equation*}
G_{\phi _{1}}=G_{\phi _{2}}.
\end{equation*}
\end{proposition}

\begin{proof}
Just check the definition.
\end{proof}

\begin{proposition}
Assume $G$ is a Noether solution of $L/K$. Fixed a nice basis $\left( \Delta
,A\right) $ of $L/K$ with $L^{G}=K\left( \Delta \right) $. Let $\phi
_{1},\phi _{2}\in Aut\left( L/K\right) $. The following statements are
equivalent.

$\left( i\right) $ There is
\begin{equation*}
G_{\phi _{1}}=G_{\phi _{2}}
\end{equation*}%
for the subgroups in $Aut\left( L/K\right) $.

$\left( ii\right) $ There is
\begin{equation*}
Aut\left( L/K\left( \phi _{1}\left( \Delta \right) \right) \right)
=Aut\left( L/K\left( \phi _{2}\left( \Delta \right) \right) \right)
\end{equation*}%
for the Galois groups.

$\left( iii\right) $ There is
\begin{equation*}
\pi _{a}\left( L/K\right) \left( \phi _{1}\left( \Delta \right) ,\phi
_{1}\left( A\right) \right) =\pi _{a}\left( L/K\right) \left( \phi
_{1}\left( \Delta \right) ,\phi _{1}\left( A\right) \right)
\end{equation*}%
for the full algebraic Galois subgroups.

$\left( iv\right) $ There is
\begin{equation*}
\pi _{t}\left( L/K\right) \left( \phi _{1}\left( \Delta \right) ,\phi
_{1}\left( A\right) \right) =\pi _{t}\left( L/K\right) \left( \phi
_{1}\left( \Delta \right) ,\phi _{1}\left( A\right) \right)
\end{equation*}%
for the highest transcendental Galois subgroups.

$\left( v\right) $ There is
\begin{equation*}
D_{L/K}\left( \phi _{1}\left( \Delta \right) ,\phi _{1}\left( A\right)
\right) =D_{L/K}\left( \phi _{1}\left( \Delta \right) ,\phi _{1}\left(
A\right) \right)
\end{equation*}%
for the decomposition groups.

$\left( vi\right) $ There is
\begin{equation*}
K\left( \phi _{1}\left( \Delta \right) \right) =K\left( \phi _{2}\left(
\Delta \right) \right)
\end{equation*}%
for the subfield.

$\left( vii\right) $ There is
\begin{equation*}
K\left( \phi _{1}\left( \Lambda \right) \right) =K\left( \phi _{2}\left(
\Lambda \right) \right)
\end{equation*}%
for every nice basis $\left( \Lambda ,B\right) $ of $L/K$ with $K\left(
\Lambda \right) =L^{G}$.
\end{proposition}

\begin{proof}
$\left( i\right) \Leftrightarrow \left( ii\right) \Leftrightarrow \left(
iii\right) \Leftrightarrow \left( vi\right) \Leftrightarrow \left(
vii\right) $. It is immediately from \emph{Proposition 8.25} and \emph{%
Theorem 8.23} by considering the algebraic Galois extensions $L/K\left( \phi
_{1}\left( \Delta \right) \right) $ and $L/K\left( \phi _{2}\left( \Delta
\right) \right) $ and their Galois groups $Aut\left( L/K\left( \phi
_{1}\left( \Delta \right) \right) \right) $ and $Aut\left( L/K\left( \phi
_{2}\left( \Delta \right) \right) \right) $, respectively. Here, $K\left(
\phi _{1}\left( \Delta \right) \right) =L^{G_{\phi _{1}}}$ and $K\left( \phi
_{2}\left( \Delta \right) \right) =L^{G_{\phi _{2}}}$ are the invariant
subfields of $L$, respectively.

$\left( iii\right) \Leftrightarrow \left( iv\right) $. It is immediately
from \emph{Lemmas 4.12-3}.

$\left( v\right) \Leftrightarrow \left( vi\right) $. It is immediately from
\emph{Lemma 5.6} since $\left( iii\right) \Leftrightarrow \left( iv\right) $
holds.
\end{proof}

\begin{proposition}
Assume $G$ is a Noether solution of $L/K$.

$\left( i\right) $ Let $Z\left( G\right) $ denote the center of the group $G$%
. Then the quotient group $G/Z\left( G\right) $ of $G$ by $Z\left( G\right) $
is isomorphic to the inner automorphism group $Inn\left( G\right) $ of $G$.

In particular, $G$ is a subgroup of the isotropy subgroup $Aut\left(
L/K\right) ^{G}$.

$\left( ii\right) $ Let $Z_{Aut\left( L/K\right) }\left( G\right) $ denote
the centralizer of the subgroup $G$ in $Aut\left( L/K\right) $. Then $%
Z_{Aut\left( L/K\right) }\left( G\right) $ is a normal subgroup of the
isotropy subgroup $Aut\left( L/K\right) ^{G}$.

In particular, the quotient group $Aut\left( L/K\right) ^{G}/Z_{Aut\left(
L/K\right) }\left( G\right) $ of $Aut\left( L/K\right) ^{G}$ by $%
Z_{Aut\left( L/K\right) }\left( G\right) $ is a finite group.
\end{proposition}

\begin{proof}
Let $\phi \in Aut\left( L/K\right) $. Put $\tau _{\phi }\left( x\right)
=\phi \cdot x\cdot \phi ^{-1}$ for any $x\in G$. The mapping $\tau _{\phi
}:G\rightarrow G_{\phi }$ is an isomorphism of groups.

In particular, for any $\phi \in Aut\left( L/K\right) ^{G}$, the mapping
\begin{equation*}
\tau _{\phi }:G\rightarrow G=G_{\phi }
\end{equation*}%
is an automorphism of the group $G$, i.e., $\tau _{\phi }\in Aut\left(
G\right) $. We have a homomorphism
\begin{equation*}
\tau :Aut\left( L/K\right) ^{G}\rightarrow Aut\left( G\right)
\end{equation*}%
of groups given by $\phi \longmapsto \tau _{\phi }$. Then
\begin{equation*}
Z_{Aut\left( L/K\right) }\left( G\right) =\ker \left( \tau \right)
\end{equation*}%
holds, i.e., $Z_{Aut\left( L/K\right) }\left( G\right) $ is the kernel $\ker
\left( \tau \right) $ of the homomorphism $\tau $. It follows that $Z\left(
G\right) \subseteq \ker \left( \tau \right) $ and $G\subseteq Aut\left(
L/K\right) ^{G}$ hold.

On the other hand, as $G$ is a finite subgroup of $Aut\left( L/K\right) $,
it is seen that $Aut\left( G\right) $ is a finite group and then the
quotient group
\begin{equation*}
\frac{Aut\left( L/K\right) ^{G}}{\ker \left( \tau \right) }=\frac{Aut\left(
L/K\right) ^{G}}{Z_{Aut\left( L/K\right) }\left( G\right) }
\end{equation*}%
is also a finite group. This completes the proof.
\end{proof}

\begin{proposition}
(\emph{Left actions on decomposition groups induced from the inner action of
the Galois group}) The inner action of the Galois group $Aut\left(
L/K\right) $ on Noether solutions $G$ of $L/K$ induces left actions of $%
Aut\left( L/K\right) $ on decomposition groups, full algebraic Galois groups
and highest transcendental Galois groups of $L/K$ in an evident manner:

Let $G$ be a Noether solution of $L/K$. Fixed a nice basis $\left( \Delta
,A\right) $ of $L/K$ with $K\left( \Delta \right) =L^{G}$. For any $\phi \in
Aut\left( L/K\right) $, define
\begin{equation*}
\begin{array}{l}
\phi \cdot D_{L/K}\left( \Delta ,A\right) \\
\quad \quad \quad \quad :=D_{L/K}\left( \phi \left( \Delta \right) ,\phi
\left( A\right) \right) ; \\
\phi \cdot \pi _{a}\left( L/K\right) \left( \Delta ,A\right) \\
\quad \quad \quad \quad :=\pi _{a}\left( L/K\right) \left( \phi \left(
\Delta \right) ,\phi \left( A\right) \right) ; \\
\phi \cdot \pi _{t}\left( L/K\right) \left( \Delta ,A\right) \\
\quad \quad \quad \quad :=\pi _{t}\left( L/K\right) \left( \phi \left(
\Delta \right) ,\phi \left( A\right) \right) .%
\end{array}%
\end{equation*}

Define the orbits
\begin{equation*}
\begin{array}{l}
D_{L/K}\left( \Delta ,A\right) ^{Aut\left( L/K\right) } \\
\quad \quad \quad \quad :=\{\phi \cdot D_{L/K}\left( \Delta ,A\right) :\phi
\in Aut\left( L/K\right) \}; \\
\pi _{a}\left( L/K\right) \left( \Delta ,A\right) ^{Aut\left( L/K\right) }
\\
\quad \quad \quad \quad :=\{\phi \cdot \pi _{a}\left( L/K\right) \left(
\Delta ,A\right) :\phi \in Aut\left( L/K\right) \}; \\
\pi _{t}\left( L/K\right) \left( \Delta ,A\right) ^{Aut\left( L/K\right) }
\\
\quad \quad \quad \quad :=\{\phi \cdot \pi _{t}\left( L/K\right) \left(
\Delta ,A\right) :\phi \in Aut\left( L/K\right) \}.%
\end{array}%
\end{equation*}

Define the isotropy subgroups
\begin{equation*}
\begin{array}{l}
Aut\left( L/K\right) ^{D_{L/K}\left( \Delta ,A\right) } \\
\quad \quad \quad \quad :=\{\phi \in Aut\left( L/K\right) :\phi \cdot
D_{L/K}\left( \Delta ,A\right) =D_{L/K}\left( \Delta ,A\right) \}; \\
Aut\left( L/K\right) ^{\pi _{a}\left( L/K\right) \left( \Delta ,A\right) }
\\
\quad \quad \quad \quad :=\{\phi \in Aut\left( L/K\right) :\phi \cdot \pi
_{a}\left( L/K\right) \left( \Delta ,A\right) =\pi _{a}\left( L/K\right)
\left( \Delta ,A\right) \}; \\
Aut\left( L/K\right) ^{\pi _{t}\left( L/K\right) \left( \Delta ,A\right) }
\\
\quad \quad \quad \quad :=\{\phi \in Aut\left( L/K\right) :\phi \cdot \pi
_{t}\left( L/K\right) \left( \Delta ,A\right) =\pi _{t}\left( L/K\right)
\left( \Delta ,A\right) \}.%
\end{array}%
\end{equation*}

Then there are bijective sets
\begin{equation*}
\begin{array}{l}
G^{Aut\left( L/K\right) } \\
=\pi _{a}\left( L/K\right) \left( \Delta ,A\right) ^{Aut\left( L/K\right) }
\\
\cong \frac{Aut\left( L/K\right) }{Aut\left( L/K\right) ^{G}} \\
=\frac{Aut\left( L/K\right) }{Aut\left( L/K\right) ^{\pi _{a}\left(
L/K\right) \left( \Delta ,A\right) }} \\
\cong \frac{Aut\left( L/K\right) }{Aut\left( L/K\right) ^{\pi _{t}\left(
L/K\right) \left( \Delta ,A\right) }} \\
\cong \pi _{t}\left( L/K\right) \left( \Delta ,A\right) ^{Aut\left(
L/K\right) } \\
\cong \frac{Aut\left( L/K\right) }{Aut\left( L/K\right) ^{D_{L/K}\left(
\Delta ,A\right) }} \\
\cong D_{L/K}\left( \Delta ,A\right) ^{Aut\left( L/K\right) }.%
\end{array}%
\end{equation*}
\end{proposition}

\begin{proof}
Immediately from \emph{Propositions 8.25-6}.
\end{proof}

\begin{proposition}
Let $G$ be a Noether solution of $L/K$ and $\left( \Delta ,A\right) $ a nice
basis $\left( \Delta ,A\right) $ of $L/K$ with $K\left( \Delta \right)
=L^{G} $. Fixed any $\phi \in Aut\left( L/K\right) $.

$\left( i\right) $ For the decomposition groups of $L/K$,
\begin{equation*}
D_{L/K}\left( \phi \left( \Delta \right) ,\phi \left( A\right) \right) =\phi
\cdot D_{L/K}\left( \Delta ,A\right) \cdot \phi ^{-1}
\end{equation*}%
holds in the Galois group $Aut\left( L/K\right) $.

In particular, the decomposition groups $D_{L/K}\left( \Delta ,A\right) $
and $D_{L/K}\left( \phi \left( \Delta \right) ,\phi \left( A\right) \right) $
are $\left( L,L\right) $-spatially isomorphic.

$\left( ii\right) $ For the full algebraic Galois subgroups of $L/K$,
\begin{equation*}
\pi _{a}\left( L/K\right) \left( \phi \left( \Delta \right) ,\phi \left(
A\right) \right) =\phi \cdot \pi _{a}\left( L/K\right) \left( \Delta
,A\right) \cdot \phi ^{-1};
\end{equation*}%
holds in the Galois group $Aut\left( L/K\right) $.

In particular, $\pi _{a}\left( L/K\right) \left( \Delta ,A\right) $ and $\pi
_{a}\left( L/K\right) \left( \phi \left( \Delta \right) ,\phi \left(
A\right) \right) $ are $\left( L,L\right) $-spatially isomorphic.

$\left( iii\right) $ For the highest transcendental Galois subgroups of $L/K$%
,
\begin{equation*}
\pi _{t}\left( L/K\right) \left( \phi \left( \Delta \right) ,\phi \left(
A\right) \right) =\phi \cdot \pi _{t}\left( L/K\right) \left( \Delta
,A\right) \cdot \phi ^{-1}.
\end{equation*}%
holds in the Galois group $Aut\left( L/K\right) $.

In particular, $\pi _{t}\left( L/K\right) \left( \Delta ,A\right) $ and $\pi
_{t}\left( L/K\right) \left( \phi \left( \Delta \right) ,\phi \left(
A\right) \right) $ are $\left( L,L\right) $-spatially isomorphic.
\end{proposition}

\begin{proof}
Consider a $K$-isomorphism $\phi :L\rightarrow L$ of fields and its
restriction
\begin{equation*}
\phi :K\left( \Delta \right) \rightarrow K\left( \phi \left( \Delta \right)
\right) .
\end{equation*}
From \emph{Theorem 8.23} we have
\begin{equation*}
\begin{array}{l}
\pi _{a}\left( L/K\right) \left( \phi \left( \Delta \right) ,\phi \left(
A\right) \right) \\
=Aut\left( L/K\left( \phi \left( \Delta \right) \right) \right) \\
=G_{\phi } \\
=\phi \cdot G\cdot \phi ^{-1} \\
=\phi \cdot Aut\left( L/K\left( \Delta \right) \right) \cdot \phi ^{-1} \\
=\phi \cdot \pi _{a}\left( L/K\right) \left( \Delta ,A\right) \cdot \phi
^{-1}%
\end{array}%
\end{equation*}
and it is seen that $\pi _{a}\left( L/K\right) \left( \Delta ,A\right) $ and
$\pi _{a}\left( L/K\right) \left( \phi \left( \Delta \right) ,\phi \left(
A\right) \right) $ are $\left( L,L\right) $-spatially isomorphic. This
proves $\left( ii\right) $.

It follows that $\pi _{t}\left( L/K\right) \left( \Delta ,A\right) $ and $%
\pi _{t}\left( L/K\right) \left( \phi \left( \Delta \right) ,\phi \left(
A\right) \right) $ are $\left( L,L\right) $-spatially isomorphic via the
spatial isomorphism $\left( \tau _{\phi },\phi \right) $ from \emph{Theorem
8.20}. Hence,
\begin{equation*}
\pi _{t}\left( L/K\right) \left( \phi \left( \Delta \right) ,\phi \left(
A\right) \right) =\phi \cdot \pi _{t}\left( L/K\right) \left( \Delta
,A\right) \cdot \phi ^{-1}.
\end{equation*}%
holds in $Aut\left( L/K\right) $. This proves $\left( iii\right) $. In the
same way prove $\left( i\right) $.
\end{proof}

Now for a Noether solution $G$ of $L/K$, put
\begin{equation*}
sp_{L/K}\left( G\right) :=
\end{equation*}%
\begin{equation*}
\sharp \{P\text{ is Noether solution of }L/K:P\text{ is -spatially
isomorphic to }G\};
\end{equation*}%
\begin{equation*}
ind_{L/K}\left( G\right) :=\sharp \frac{Aut\left( L/K\right) }{Z_{Aut\left(
L/K\right) }\left( G\right) },
\end{equation*}%
i.e., the index of the centralizer $Z_{Aut\left( L/K\right) }\left( G\right)
$ in $Aut\left( L/K\right) $;
\begin{equation*}
orb_{L/K}\left( G\right) :=\sharp \frac{Aut\left( L/K\right) }{Aut\left(
L/K\right) ^{G}},
\end{equation*}%
i.e., the index of the isotropy subgroup $Aut\left( L/K\right) ^{G}$ in $%
Aut\left( L/K\right) $.

\begin{theorem}
\emph{(The inner action of the Galois group on Noether solutions)} Let $G$ be
a Noether solution of $L/K$. Consider the inner action of the Galois group $%
Aut\left( L/K\right) $ on the subgroup $G$.

$\left( i\right) $ If either $G=\{1\}$ or $Z_{Aut\left( L/K\right) }\left(
G\right) =Aut\left( L/K\right) $ holds, then we have
\begin{equation*}
ind_{L/K}\left( G\right) =orb_{L/K}\left( G\right) =sp_{L/K}\left( G\right)
=1.
\end{equation*}

$\left( ii\right) $ We have
\begin{equation*}
orb_{L/K}\left( G\right) =sp_{L/K}\left( G\right) \leq ind_{L/K}\left(
G\right) \leq +\infty .
\end{equation*}

$\left( iii\right) $ Fixed a nice basis $\left( \Delta ,A\right) $ of $L/K$
with $L^{G}=K\left( \Delta \right) $. Then we have
\begin{equation*}
sp_{L/K}\left( G\right) =\left[ Aut\left( L/K\right) :Aut\left( L/K\right)
^{G}\right]
\end{equation*}%
and
\begin{equation*}
sp_{L/K}\left( G\right) =\left[ Aut\left( L/K\right) :Aut\left( L/K\right)
^{D_{L/K}\left( \Delta ,A\right) }\right]
\end{equation*}%
for the indexes of the subgroups $Aut\left( L/K\right) ^{G}$ and $Aut\left(
L/K\right) ^{D_{L/K}\left( \Delta ,A\right) }$ in the Galois group $%
Aut\left( L/K\right) $, respectively.

In particular, $sp_{L/K}\left( G\right) $ is finite if and only if there is
\begin{equation*}
\left[ Aut\left( L/K\right) :Aut\left( L/K\right) ^{D_{L/K}\left( \Delta
,A\right) }\right] <+\infty .
\end{equation*}
\end{theorem}

\begin{proof}
$\left( i\right) $ Trivial.

$\left( ii\right) $ Immediately from \emph{Theorem 8.23}, \emph{Remark 8.24}
and \emph{Proposition 8.27}, where we have
\begin{equation*}
orb_{L/K}\left( G\right) =sp_{L/K}\left( G\right) \leq +\infty
\end{equation*}
and%
\begin{equation*}
\begin{array}{l}
orb_{L/K}\left( G\right) :=\sharp \frac{Aut\left( L/K\right) }{Aut\left(
L/K\right) ^{G}} \\
\leq ind_{L/K}\left( G\right) :=\sharp \frac{Aut\left( L/K\right) }{%
Z_{Aut\left( L/K\right) }\left( G\right) } \\
\leq +\infty .%
\end{array}%
\end{equation*}

$\left( iii\right) $ Immediately from \emph{Propositions 8.25-6} and \emph{%
Proposition 8.28} by taking a system $\{\phi _{\lambda }:\lambda \in I\}$ of
representatives of the left cosets of the isotropy subgroup $Aut\left(
L/K\right) ^{G}$ in the Galois group $Aut\left( L/K\right) $, where we have
\begin{equation*}
sp_{L/K}\left( G\right) =orb_{L/K}\left( G\right) =\sharp \{\phi _{\lambda
}:\lambda \in I\};
\end{equation*}%
\begin{equation*}
\left[ Aut\left( L/K\right) :Aut\left( L/K\right) ^{G}\right] =\sharp \{\phi
_{\lambda }:\lambda \in I\};
\end{equation*}%
\begin{equation*}
\left[ Aut\left( L/K\right) :Aut\left( L/K\right) ^{D_{L/K}\left( \Delta
,A\right) }\right] =\sharp \{\phi _{\lambda }:\lambda \in I\}.
\end{equation*}%
This completes the proof.
\end{proof}

\subsection{$K$-theory and Noether solutions}

$K$-theory has a natural connection with Noether solutions. In fact, one aim that $K$-theory comes into Noether's problem here is for us to
depict how many Noether solutions there are in a fixed purely transcendental
extension $L/K$ which contain $G$ as subgroups. The $K$-groups can
dominate the number of Noether solutions in $L/K$ which contains a given
Noether solution. This work is still in progress. (See \emph{[}An 2019\emph{]%
}).

Let $L$ be a purely transcendental extension over a field $K$ of finite
transcendence degree. Fixed a Noether solution $G$ of $L/K$. Recall that the
\emph{height subgroup} $H_{L/K}\left( G\right) $ of $G$ in $L/K$ is the
subgroup generated by%
\begin{equation*}
\bigcup\limits_{G_{\lambda }\in \Psi _{L/K}\left( G\right) }G_{\lambda }
\end{equation*}%
in the Galois group $Aut\left( L/K\right) $, where
\begin{equation*}
\Psi _{L/K}\left( G\right) :=\{G_{\lambda }:\lambda \in I_{L/K}\left(
G\right) \}
\end{equation*}%
denotes the set of all the Noether solutions $G_{\lambda }$ of $L/K$,
indexed by a set $I_{L/K}\left( G\right) $, such that
\begin{equation*}
G\subseteq G_{\lambda }
\end{equation*}%
holds for any $\lambda \in I_{L/K}\left( G\right) $. The \emph{class height
subgroup} of $G$ in $L/K$ is the image
\begin{equation*}
CH_{L/K}\left( G\right) :=\tau _{L/K}\left( H_{L/K}\left( G\right) \right)
\end{equation*}%
of the height subgroup $H_{L/K}\left( G\right) $ under the map $\tau _{L/K}$%
. Here, $\Pi :=GL_{K}\left( L\right) $ is the general linear group of $L$
over $K$; $\Pi ^{ab}:=\Pi /\left[ \Pi ,\Pi \right] $ is the maximal abelian
quotient of $\Pi $ by the commutator subgroup $\left[ \Pi ,\Pi \right] $;
denote by
\begin{equation*}
\tau _{L/K}:\Pi \rightarrow \Pi ^{ab}
\end{equation*}%
the natural homomorphism of the group $\Pi $ onto the quotient group $\Pi
^{ab}$. (See \emph{Definition 1.5}).

See \emph{[}Milnor 1971\emph{]} for definitions and properties of the
Whitehead $K_{1}$-group $K_{1}\left( R\right) $ of a commutative ring $R$
with identity.

\begin{remark}
$\Psi _{L/K}\left( G\right) $ is a partially ordered set by set-inclusion
relation. That is, let $G_{\alpha },G_{\beta }\in \Psi _{L/K}\left( G\right)
$. We say $G_{\alpha }\leq G_{\beta }$ if and only if $G_{\alpha }\subseteq
G_{\beta }$ holds. In such a case, we also say $\alpha \leq \beta $ in the
index set $I_{L/K}\left( G\right) $. That is, with $\leq $, the index set $%
I_{L/K}\left( G\right) $ is a partially ordered set.

The set $\Psi _{L/K}\left( G\right) $ is co-initial with $G$ being the
initial object. But $\Psi _{L/K}\left( G\right) $ is not cofinal, in general.
\end{remark}

\begin{definition}
(See \emph{[}An 2019\emph{]}) A subset $\Phi \subseteq \Psi _{L/K}\left(
G\right) $ is said to be a \textbf{Zorn-component} of $\Psi _{L/K}\left(
G\right) $ if the two conditions are satisfied: $\left( a\right) $ $\Phi $
is a totally ordered subset; $\left( b\right) $ There is no other element $%
P_{0}\in \Psi _{L/K}\left( G\right) $ such that the union $\Phi \cup
\{P_{0}\}$ is still totally ordered.

Two Zorn-components $\Phi _{1}$ and $\Phi _{2}$ of $\Psi _{L/K}\left(
G\right) $ are said to be \textbf{co-disjoint} if the intersection $\Phi
_{1}\cap \Phi _{2}$ is a finite set.
\end{definition}

\begin{definition}
(See \emph{[}An 2019\emph{]}) Let $\Phi $ be a Zorn-component $\Phi $ of $%
\Psi _{L/K}\left( G\right) $. The union $\bigcup\limits_{P\in \Phi }P$
generates a subgroup in $Aut\left( L/K\right) $, denoted by $H_{\infty
}\left( \Phi \right) $, is called a \textbf{Zorn-component} of $%
H_{L/K}\left( G\right) $ given by $\Phi $.

The image $CH_{\infty }\left( \Phi \right) :=\tau _{L/K}\left( H_{\infty
}\left( \Phi \right) \right) $ under the map $\tau _{L/K}:\Pi \rightarrow
\Pi ^{ab}$ is called a \textbf{Zorn-component} of $CH_{L/K}\left( G\right) $
given by $\Phi $.
\end{definition}

Fixed any group $G_{\lambda }\in \Psi _{L/K}\left( G\right) $, i.e,. $%
G_{\lambda }$ is a Noether solution of $L/K$ with $G_{\lambda }\supseteq G$.
Let $V_{\lambda }$ denote the vector space $L$ over the subfield
\begin{equation*}
K_{\lambda }:=K\left( \Delta _{\lambda }\right)
\end{equation*}%
of dimension
\begin{equation*}
n_{\lambda }:=\dim _{K_{\lambda }}V_{\lambda }.
\end{equation*}%
Here, $\Delta _{\lambda }$ is a given transcendence base of $L/K$ such that $%
K\left( \Delta _{\lambda }\right) =L^{G_{\lambda }}$.

\begin{proposition}
Consider the general linear group $GL\left( V_{\lambda }\right) $ of the
vector space over $K_{\lambda }$ and a $K_{\lambda }$-linear basis $%
\{e_{i}^{\lambda }\}_{1\leq i\leq n_{\lambda }}$ of $V_{\lambda }$.

$\left( i\right) $ Fixed an $\alpha _{0}\in I_{L/K}\left( G\right) $. Let $%
\lambda \in I_{L/K}\left( G\right) $ with $\alpha _{0}\leq \lambda $. Under
the linear basis $\{e_{i}^{\alpha _{0}}\}_{1\leq i\leq n_{\lambda }}$ of $%
V_{\alpha _{0}}$, there is an isomorphism
\begin{equation*}
\tau _{\lambda |\alpha _{0}}:G_{\lambda }\rightarrow H_{\lambda |\alpha
_{0}}:=\tau _{\lambda |\alpha _{0}}\left( G_{\lambda }\right) \subseteq
GL\left( n_{\lambda |\alpha _{0}},K_{\alpha _{0}}\right)
\end{equation*}%
of groups, i.e., the subgroup $G_{\lambda }$ of $Aut\left( L/K\right) $ is
isomorphic to a subgroup $H_{\lambda |\alpha _{0}}$ of the general linear
group $GL\left( n_{\lambda |\alpha _{0}},K_{\alpha _{0}}\right) $ consisting
of invertible $n_{\lambda |\alpha _{0}}\times n_{\lambda |\alpha _{0}}$
matrices with coefficients in $K_{\alpha _{0}}$.

$\left( ii\right) $ Under another linear basis $\{e_{i}^{\prime \alpha
_{0}}\}_{1\leq i\leq n_{\lambda }}$ of $V_{\alpha _{0}}$, for the
corresponding subgroup $H_{\lambda |\alpha _{0}}^{\prime }\subseteq GL\left(
n_{\lambda |\alpha _{0}},K_{\alpha _{0}}\right) $ isomorphic to $G_{\lambda
} $, there is a matrix
\begin{equation*}
P_{\lambda |\alpha _{0}}\in GL\left( n_{\lambda |\alpha _{0}},K_{\alpha
_{0}}\right)
\end{equation*}
such that
\begin{equation*}
H_{\lambda |\alpha _{0}}=P_{\lambda |\alpha _{0}}^{-1}\cdot H_{\lambda
|\alpha _{0}}^{\prime }\cdot P_{\lambda |\alpha _{0}}
\end{equation*}%
holds, i.e., $H_{\lambda |\alpha _{0}}$ and $H_{\lambda |\alpha
_{0}}^{\prime }$ are conjugate in the matrix group $GL\left( n_{\lambda
|\alpha _{0}},K_{\alpha _{0}}\right) $.

In particular, for each $\lambda \in I_{L/K}\left( G\right) $ with $\alpha
_{0}\leq \lambda $, the matrix subgroup $H_{\lambda |\alpha _{0}}$ is
uniquely determined by $G_{\lambda }$ up to inner automorphisms of $GL\left(
n_{\lambda |\alpha _{0}},K_{\alpha _{0}}\right) $.

$\left( iii\right) $ For any two $G_{\alpha }\subseteq G_{\beta }$ in $\Psi
_{L/K}\left( G\right) $, we have
\end{proposition}

\begin{proof}
Trivial.
\end{proof}

Fixed an $\alpha _{0}\in I_{L/K}\left( G\right) $. Let $\lambda \in
I_{L/K}\left( G\right) $ with $\alpha _{0}\leq \lambda $. Denote the
commutator subgroup of the matrix group $GL\left( n_{\lambda |\alpha
_{0}},K_{\alpha _{0}}\right) $ by
\begin{equation*}
E_{\lambda |\alpha _{0}}:=\left[ GL\left( n_{\lambda |\alpha _{0}},K_{\alpha
_{0}}\right) ,GL\left( n_{\lambda |\alpha _{0}},K_{\alpha _{0}}\right) %
\right] .
\end{equation*}%
Denote the natural map from the group $G_{\lambda }$ to the quotient group
\begin{equation*}
GL\left( n_{\lambda |\alpha _{0}},K_{\alpha _{0}}\right) /E_{\lambda |\alpha
_{0}}
\end{equation*}
by
\begin{equation*}
\pi _{\lambda |\alpha _{0}}:G_{\lambda }\rightarrow GL\left( n_{\lambda
|\alpha _{0}},K_{\alpha _{0}}\right) /E_{\lambda |\alpha _{0}}.
\end{equation*}

\begin{proposition}
Fixed an $\alpha _{0}\in I_{L/K}\left( G\right) $. Let $\lambda \in
I_{L/K}\left( G\right) $ with $\alpha _{0}\leq \lambda $. Then for any $%
g,g^{\prime }\in G_{\lambda }$ we have
\begin{equation*}
\tau _{L/K}\left( g\right) =\tau _{L/K}\left( g^{\prime }\right)
\end{equation*}%
in $CH_{L/K}\left( G\right) $ if and only if there is some $\lambda \leq
\lambda _{0}$ in $I_{L/K}\left( G\right) $ such that
\begin{equation*}
\pi _{\lambda _{0}|\alpha _{0}}\left( g\right) =\pi _{\lambda _{0}|\alpha
_{0}}\left( g^{\prime }\right)
\end{equation*}%
holds in $GL\left( n_{\lambda _{0}|\alpha _{0}},K_{\alpha _{0}}\right)
/E_{\lambda _{0}|\alpha _{0}}$. In such a case, we say $g$ and $g^{\prime }$
are of the same $M$\textbf{-equivalence class}, denoted by
\begin{equation*}
g\thicksim _{M}g^{\prime }.
\end{equation*}
\end{proposition}

\begin{proof}
Trivial by taking a linear basis of the vector space under which every
finite subgroup of $H_{L/K}\left( G\right) $ is a subgroup consisting of
invertible square matrices of the same order.
\end{proof}

\begin{proposition}
Fixed an $\alpha _{0}\in I_{L/K}\left( G\right) $. Consider $\lambda \in
I_{L/K}\left( G\right) $ with $\alpha _{0}\leq \lambda $. Then
\begin{equation*}
\{G_{\lambda }:\alpha _{0}\leq \lambda \in I_{L/K}\left( G\right) \}
\end{equation*}%
and
\begin{equation*}
\{G_{\lambda }/\thicksim _{M}:\alpha _{0}\leq \lambda \in I_{L/K}\left(
G\right) \}
\end{equation*}%
form inductive systems of groups, respectively. Here, the homomorphisms are
given by set inclusions.

Furthermore, the $M$-equivalence relation on each $G_{\lambda }$ with $%
\alpha _{0}\leq \lambda $ is compatible with the homomorphisms in the above
inductive systems of groups; for the inductive limits there is an
isomorphism
\begin{equation*}
\left( \lim_{\substack{ \longrightarrow  \\ \alpha _{0}\leq \lambda \in
I_{L/K}\left( G\right) }}G_{\lambda }\right) /\thicksim _{M}\quad \cong \lim
_{\substack{ \longrightarrow  \\ \alpha _{0}\leq \lambda \in I_{L/K}\left(
G\right) }}\left( G_{\lambda }/\thicksim _{M}\right)
\end{equation*}%
of groups, where the induced equivalence relation on the inductive limit is
still denoted by   $\thicksim _{M}$.
\end{proposition}

\begin{proof}
Trivial.
\end{proof}

\begin{proposition}
Fixed a Zorn-component $\Phi $ of $\Psi _{L/K}\left( G\right) $, i.e., there
is a subset $I\left( \Phi \right) \subseteq I_{L/K}\left( G\right) $ such
that
\begin{equation*}
\Phi =\{G_{\lambda }:\lambda \in I\left( \Phi \right) \}.
\end{equation*}
$\left( i\right) $ Fixed an $\alpha _{0}\in I\left( \Phi \right) $. Then
there is an isomorphism
\begin{equation*}
\left( \lim_{\substack{ \longrightarrow  \\ \alpha _{0}\leq \lambda \in
I\left( \Phi \right) }}G_{\lambda }\right) / \thicksim _{M}\quad \cong \lim
_{\substack{ \longrightarrow  \\ \alpha _{0}\leq \lambda \in I\left( \Phi
\right) }}\left( G_{\lambda }/\thicksim _{M}\right)
\end{equation*}%
of groups for the inductive limits of the subsystems of groups
\begin{equation*}
\{G_{\lambda }:\alpha _{0}\leq \lambda \in I_{L/K}\left( G\right) \}
\end{equation*}%
and
\begin{equation*}
\{G_{\lambda }/\thicksim _{M}:\alpha _{0}\leq \lambda \in I_{L/K}\left(
G\right) \}.
\end{equation*}%
$\left( ii\right) $ Fixed any $\alpha _{0}\in I\left( \Phi \right) $. Put
\begin{equation*}
H_{\infty }\left( \Phi \right) |_{\alpha _{0}}:=\lim_{\substack{ %
\longrightarrow  \\ \alpha _{0}\leq \lambda \in I\left( \Phi \right) }}%
G_{\lambda };
\end{equation*}
\begin{equation*}
CH_{\infty }\left( \Phi \right) |_{\alpha _{0}}:=\lim_{\substack{ %
\longrightarrow  \\ \alpha _{0}\leq \lambda \in I\left( \Phi \right) }}%
\left( G_{\lambda }/\thicksim _{M}\right) .
\end{equation*}

Then there are natural isomorphisms
\begin{equation*}
\tau _{L/K}\left( H_{\infty }\left( \Phi \right) |_{\alpha _{0}}\right)
\cong CH_{\infty }\left( \Phi \right) |_{\alpha _{0}}\cong K_{1}\left(
L^{G_{\alpha _{0}}};\Phi |_{\alpha _{0}}\right)
\end{equation*}%
of groups. Here, $K_{1}\left( L^{G_{\alpha _{0}}};\Phi |_{\alpha
_{0}}\right) $ is a subgroup of the $K_{1}$-group $K_{1}\left( L^{G_{\alpha
_{0}}}\right) $ of the $G_{\alpha _{0}}$-invariant subfield $L^{G_{\alpha
_{0}}}$ of $L$.
\end{proposition}

\begin{proof}
Immediately from \emph{Lemmas 8.35-6}.
\end{proof}

\begin{proposition}
Let $\Phi =\{G_{\lambda }:\lambda \in I\left( \Phi \right) \}$ be a
Zorn-component of $\Psi _{L/K}\left( G\right) $. Then for any $\alpha ,\beta
\in I\left( \Phi \right) $ there are natural isomorphisms%
\begin{equation*}
\begin{array}{l}
\tau _{L/K}\left( H_{\infty }\left( \Phi \right) |_{\alpha }\right) \\
\cong \tau _{L/K}\left( H_{\infty }\left( \Phi \right) |_{\beta }\right) \\
\cong CH_{\infty }\left( \Phi \right) |_{\alpha } \\
\cong CH_{\infty }\left( \Phi \right) |_{\beta } \\
\cong K_{1}\left( L^{G_{\alpha }};\Phi |_{\alpha }\right) \\
\cong K_{1}\left( L^{G_{\beta }};\Phi |_{\beta }\right)%
\end{array}%
\end{equation*}%
of groups. In particular, we have natural isomorphisms%
\begin{equation*}
\begin{array}{l}
CH_{\infty }\left( \Phi \right) \\
=CH_{\infty }\left( \Phi \right) |_{0} \\
\cong CH_{\infty }\left( \Phi \right) |_{\alpha } \\
\cong K_{1}\left( L^{G_{\alpha }};\Phi |_{\alpha }\right) \\
\cong K_{1}\left( L^{G_{0}};\Phi |_{0}\right)%
\end{array}%
\end{equation*}%
of groups for the Zorn-component of $CH_{L/K}\left( G\right) $ given by $%
\Phi $. Here, $0$ denotes the element of the index subset $I\left( \Phi
\right) $ such that $G=G_{0}$.
\end{proposition}

\begin{proof}
Immediately from the fact that the indexes of the inductive subsystems are
cofinal, respectively.
\end{proof}

Now we have the concluding remark for this subsection on $K$-theory and
Noether solutions.

\begin{remark}
Let $G$ be a Noether solution of $L/K$ and $\Phi =\{G_{\lambda }:\lambda \in
I\left( \Phi \right) \}$ a Zorn-component of $\Psi _{L/K}\left( G\right) $.
Fixed any $\alpha \leq \beta $ in $I\left( \Phi \right) $. From the above
\emph{Proposition 8.39} for the levels $\alpha $ and $\beta $ we have
\begin{equation*}
\begin{array}{l}
CH_{\infty }\left( \Phi \right) \\
\cong K_{1}\left( L^{G_{\alpha }};\Phi |_{\alpha }\right) \\
\cong K_{1}\left( L^{G_{\beta }};\Phi |_{\beta }\right) .%
\end{array}%
\end{equation*}%
It follows us to consider the possible smallest subfield
\begin{equation*}
L_{\infty }\left( \Phi \right) :=\bigcap\limits_{\alpha \in I\left( \Phi
\right) }L^{G_{\alpha }}
\end{equation*}%
of the intersection of all the $G_{\alpha }$-invariant subfields $%
L^{G_{\alpha }}$ for $\alpha \in I\left( \Phi \right) $. Suppose $I\left(
\Phi \right) $ is an infinite set. Is it true that the subfield $L_{\infty
}\left( \Phi \right) $ is still purely transcendental over $K$ and $L$ is
Galois over $L_{\infty }\left( \Phi \right) $? On the other hand, suppose $%
\Phi _{1},\Phi _{2},\cdots ,\Phi _{n}$ are all the Zorn-components of $\Psi
_{L/K}\left( G\right) $ which are mutually co-disjoint. Is it true that
there is a direct decomposition
\begin{equation*}
CH_{L/K}\left( G\right) \cong \bigoplus\limits_{1\leq i\leq n}CH_{\infty
}\left( \Phi _{i}\right)
\end{equation*}%
for the class height group $CH_{L/K}\left( G\right) $ of any given Noether
solution $G$ of $L/K$? We discuss such problems in \emph{[}An 2019\emph{]}
focusing on $K$-theory and Iwasawa theory over function fields of several
variables over $\mathbb{Z}$, where being enlightened by the \emph{Example}
and \emph{Theorem of Rim} (See \emph{Page 29}, \emph{[}Milnor 1971\emph{]}
or \emph{[}Rim 1954\emph{]}), we are attempting to give an explicit
explanation for the Fisher-Swan method which has an ingredient involving the
cyclotomic fields.
\end{remark}

\section{Proofs of Main Theorems}

In this section we will give the proofs of the Main Theorems in \emph{\S 1}.

\subsection{Proof of Theorem 1.6}

Let $L $ be a purely transcendental extension over a field $K$ of finite
transcendence degree.

\begin{proposition}
(\emph{Inner isomorphisms and spatial isomorphisms}) Let $G$ and $G^{\prime }$
be two subgroups of the Galois group $Aut\left( L/K\right) $.

$\left( i\right) $ If $G$ and $G^{\prime }$ are $\left( L,L\right) $%
-spatially isomorphic, then $G$ and $G^{\prime }$ must be conjugate in $%
Aut\left( L/K\right) $.

$\left( ii\right) $ Suppose $G$ is a Noether solution of $L/K$. Then $G$ and
$G^{\prime }$ are $\left( L,L\right) $-spatially isomorphic if and only if $%
G $ and $G^{\prime }$ are conjugate in $Aut\left( L/K\right) $. In such a
case, $G^{\prime }$ is also a Noether solution of $L/K$.
\end{proposition}

\begin{proof}
$\left( i\right) $ Trivial from the definition for spatial isomorphism.

$\left( ii\right) $ Immediately from \emph{Theorem 8.23}.
\end{proof}

Here there is the proof for \emph{Theorem 1.6}.

\begin{proof}
\textbf{(Proof of Theorem 1.6)} Let $G$ be a subgroup of $Aut\left(
L/K\right) $.

$\left( i\right) $ Assume $G$ is a Noether solution of $L/K$. Fixed a
subgroup $G^{\prime }$ of $Aut\left( L/K\right) $. From \emph{Proposition 9.1%
} it is seen that $G$ and $G^{\prime }$ are conjugate in the Galois group $%
Aut\left( L/K\right) $ if and only if $G$ and $G^{\prime }$ are $\left(
L,L\right) $-spatially isomorphic. In such a case, from \emph{Theorem 8.23}
it is seen that $G^{\prime }$ is a Noether $G$-solution of $L/K$.

$\left( ii\right) $ Immediately from $\left( ii\right) $ of \emph{Theorem
8.30}.

$\left( iii\right) $ Immediately from \emph{Proposition 8.39} since for a
field $F$ we have $K_{1}\left( F\right) \equiv F^{\ast }$. This completes
the proof.
\end{proof}

\subsection{Proof of Theorem 1.7}

Let $L $ be a purely transcendental extension over a field $K$ of finite
transcendence degree. Let $G$ be a subgroup of $Aut\left( L/K\right) $ such
that $L$ is algebraic over the $G$-invariant subfield $L^{G}$.

\begin{proposition}
Suppose there is a subgroup $H$ of the Galois group $Aut\left( L/K\right) $
satisfying the three conditions:

$\ \left( N1\right) $ There is $G\cap H=\{1\}$ and $\sigma \cdot \delta
=\delta \cdot \sigma $ holds in $Aut\left( L/K\right) $ for any $\sigma \in
G $ and $\delta \in H$.

$\ \left( N2\right) $ $K$ is the invariant subfield $L^{\left\langle G\cup
H\right\rangle }$ of $L$ under the subgroup $\left\langle G\cup
H\right\rangle $ generated in $Aut\left( L/K\right) $ by the subset $G\cup H$%
.

$\ \left( N3\right) $ There are two transcendence bases $\Delta $ and $%
\Lambda $ of $L/K$ with $K\left( \Delta \right) \subseteq K\left( \Lambda
\right) $ and a subgroup $H_{\Lambda }$ of the Galois group $Aut\left(
K\left( \Lambda \right) /K\right) $ having the three properties: $\left(
a\right) $ $K\left( \Lambda \right) $ is algebraic Galois over $K\left(
\Delta \right) $; $\left( b\right) $ $H$ and $H_{\Lambda }$ are $\left(
L,K\left( \Lambda \right) \right) $-spatially isomorphic; $\left( c\right) $
$Aut\left( K\left( \Delta \right) /K\right) $ is the set of restrictions $%
\sigma |_{K\left( \Delta \right) }$ of $\sigma $ in $H_{\Lambda }$.

Then $H_{\Lambda }$ is the highest transcendental Galois subgroup $\pi
_{t}\left( K\left( \Lambda \right) /K\right) \left( \Delta ,A\right) $ of $%
K\left( \Lambda \right) /K$ at a nice basis $\left( \Delta ,A\right) $ of $%
K\left( \Lambda \right) /K$. In particular, $G$ is a Noether solution of $%
L/K $.
\end{proposition}

\begin{proof}
Assume $G\neq \{1\}$ without loss of generality. Let $M=K\left( \Lambda
\right) $. Consider the linear decompositions for $M/K$ under conditions $%
\left( N1\right) -\left( N2\right) $. From \emph{Proposition 6.11} and \emph{%
Theorem 6.15} it is seen that $H_{\Lambda }$ is the highest transcendental
Galois subgroup $\pi _{t}\left( M/K\right) \left( \Delta ,A\right) $ of $M/K$
at a nice basis $\left( \Delta ,A\right) $ of $M/K$. It follows that the
three conditions $\left( N1\right) -\left( N3\right) $ of \emph{Theorem 8.22}
are all satisfied. Hence, $G$ is a Noether solution of $L/K$.
\end{proof}

Here there is the proof for \emph{Theorem 1.7}.

\begin{proof}
\textbf{(Proof of Theorem 1.7)} $\left( i\right) $ Immediately from \emph{%
Theorem 8.21}, \emph{Proposition 8.29}\ and $\left( iii\right) $ of \emph{%
Theorem 8.30}. Here for a Noether solution $G$ of $L/K$, we have
\begin{equation*}
\begin{array}{l}
con_{L/K}\left( G\right) \\
=sp_{L/K}\left( G\right) \\
=\left[ Aut\left( L/K\right) :Aut\left( L/K\right) ^{G}\right]%
\end{array}%
\end{equation*}%
from \emph{Proposition 9.1}.

$\left( ii\right) $ Assume there is a subgroup $H$ of $Aut\left( L/K\right) $
satisfying the three conditions $\left( N1\right) -\left( N3\right) $. From
\emph{Proposition 9.2} it is seen that $G$ is a Noether solution of $L/K$
and $H_{\Lambda }$ is the highest transcendental Galois subgroup $\pi
_{t}\left( K\left( \Lambda \right) /K\right) \left( \Delta ,A\right) $.

On the other hand, from \emph{Theorem 8.22} it is seen that the subgroup $H$
is the highest transcendental Galois subgroup $\pi _{t}\left( L/K\right)
\left( \Delta ^{\prime },A^{\prime }\right) $, where $\left( \Delta ^{\prime
},A^{\prime }\right) $ is a nice basis of $L/K$ such that $K\left( \Delta
^{\prime }\right) =L^{G}$.

As $H$ and $H_{\Lambda }$ are $\left( L,K\left( \Lambda \right) \right) $%
-spatially isomorphic highest transcendental Galois subgroup, it is seen
that $G=Aut\left( L/K\left( \Delta ^{\prime }\right) \right) $ and $%
Aut\left( K\left( \Lambda \right) /K\left( \Delta \right) \right) $ are also
$\left( L,K\left( \Lambda \right) \right) $-spatially isomorphic groups from
\emph{Theorem 8.20}. Hence, we have
\begin{equation*}
\begin{array}{l}
con_{L/K}\left( G\right) \\
=sp_{L/K}\left( G\right) \\
=\left[ Aut\left( L/K\right) :Aut\left( L/K\right) ^{D_{L/K}\left( \Delta
^{\prime },A^{\prime }\right) }\right] \\
=\left[ Aut\left( L/K\right) :Aut\left( L/K\right) ^{\left\langle G\cup
H\right\rangle }\right]%
\end{array}%
\end{equation*}%
from \emph{Proposition 8.29}, $\left( iii\right) $ of \emph{Theorem 8.30}
and \emph{Proposition 9.1} again, where we have
\begin{equation*}
\left\langle G\cup H\right\rangle =D_{L/K}\left( \Delta ^{\prime },A^{\prime
}\right)
\end{equation*}%
from \emph{Theorem 6.4}.

The uniqueness of $H$ is obtained from \emph{Lemma 4.13}.

Conversely, suppose $G$ is a Noether solution of $L/K$. Let $H_{\Lambda }$
be the highest transcendental Galois subgroup $\pi _{t}\left( K\left(
\Lambda \right) /K\right) \left( \Delta ,A\right) $ of $K\left( \Lambda
\right) /K$ at a nice basis $\left( \Delta ,A\right) $ of $K\left( \Lambda
\right) /K$. Then from \emph{Theorem 8.22} we have a subgroup $H$ of $%
Aut\left( L/K\right) $ satisfying the three conditions $\left( N1\right)
-\left( N3\right) $. This completes the proof.
\end{proof}

\subsection{Proof of Theorem 1.8}

Let $L=K\left( t_{1},t_{2},\cdots ,t_{n}\right) $ be a purely transcendental
extension over a field $K$ of transcendence degree $n$. Here, $%
t_{1},t_{2},\cdots ,t_{n}$ are the variables of $L$ over $K$. Let $\Sigma
_{n}\subseteq Aut\left( L/K\right) $ denote the full permutation subgroup on
$n$ letters (relative to the variables $t_{1},t_{2},\cdots ,t_{n}$).

It's easy for one to obtain the following examples for Noether solutions of $%
L/K$ and for non-solutions of $L/K$, where our aim is to show that even for
two isomorphic subgroups of $Aut\left( L/K\right) $, one is a Noether
solution of $L/K$ but the other is not necessarily a Noether solution of $%
L/K $.

\begin{example}
(\emph{Cases of non-solutions }$G$\emph{\ of order }$\sharp G=2$) Assume $%
n\geq 3$. Let $G$ be the subgroup in $Aut\left( L/K\right) $ generated by a
linear involution $\sigma $ of $L/K$ relative to the variables $%
t_{1},t_{2},\cdots ,t_{n}$. Then $G=\left\langle \sigma \right\rangle $ is a
subgroup of order $\sharp G=2$ and $G$ has a transitive action on the
variables $t_{1},t_{2},\cdots ,t_{n}$. In particular, the $G$-invariant
subfield $L^{G}$ is not purely transcendental over $K$.
\end{example}

\begin{proof}
Immediately from \emph{Proposition 8.3.}
\end{proof}

\begin{example}
(\emph{Cases of Noether solutions }$G$\emph{\ of order }$\sharp G=2$)
Consider a subgroup $G\subseteq Aut\left( L/K\right) $ of order $\sharp G=2$.

$\left( i\right) $ Assume $n\geq 1$. Let $G$ be the Galois group of $L$ over
$K\left( t_{1}^{2},t_{2},\cdots ,t_{n}\right) $. Then $G$ is a Noether
solution of $L/K$.

$\left( ii\right) $ Let $n=2$. Let $G$ be the subgroup in $Aut\left(
L/K\right) $ generated by a linear involution $\sigma $ of $L/K$ relative to
the variables $t_{1},t_{2}$. Then $G=\left\langle \sigma \right\rangle
=\Sigma _{2}$ is a subgroup of order $\sharp G=2$ and $G$ has a transitive
action on the variables $t_{1},t_{2}$. In particular, $G$ is a Noether
solution of $L/K$.
\end{example}

\begin{proof}
$\left( i\right) $ Trivial.

$\left( ii\right) $ Immediately from \emph{Lemma 7.18.}
\end{proof}

\begin{example}
(\emph{Cases of non-solutions }$G$\emph{\ of order }$\sharp G=n$) Assume $%
n\geq 3$. Let $G=C_{n}$ be the cyclic subgroup in $\Sigma _{n}$ of order $%
\sharp G=n$. Then $G$ has a transitive action on the variables $%
t_{1},t_{2},\cdots ,t_{n}$. In particular, the $G$-invariant subfield $L^{G}$
is not purely transcendental over $K$.
\end{example}

\begin{proof}
Immediately from \emph{Lemma 7.18.}
\end{proof}

\begin{example}
(\emph{Cases of Noether solutions }$G$\emph{\ of order }$\sharp G=n$) Assume
$2\leq n=2^{k}$. Let $G$ be the Galois group of $L$ over $K\left(
t_{1}^{2},t_{2}^{2},\cdots ,t_{k}^{2},t_{n+1},\cdots ,t_{n}\right) $. Then $%
G $ is a Noether solution of $L/K$.
\end{example}

\begin{proof}
Trivial.
\end{proof}

Here there is the proof for \emph{Theorem 1.8}.

\begin{proof}
(\textbf{Proof of Theorem 1.8}) $\left( i\right) $ Immediately from \emph{%
Lemmas 7.16-8} and \emph{Propositions 8.7}. Here, as in \emph{Proposition 8.7%
}, put
\begin{equation*}
L=K\left( t_{1},t_{2},\cdots ,t_{n}\right) ;
\end{equation*}
\begin{equation*}
M=K\left( t_{1}^{2},t_{2}^{2},\cdots ,t_{n}^{2}\right).
\end{equation*}
Then the full permutation subgroup $\Sigma _{n}$, $G=Aut\left( L/M\right) $
and $G\cdot j_{\left( \Delta ,A\right) }\left( \Sigma _{n}^{M}\right) $ are
Noether solutions of $L/K$; $\Sigma _{n}$ and $G\cdot j_{\left( \Delta
,A\right) }\left( \Sigma _{n}^{M}\right) $ both have a transitive action on
the variables $t_{1},t_{2},\cdots ,t_{n}$.

$\left( ii\right) $ Immediately from \emph{Lemmas 7.16-8} and \emph{%
Propositions 8.7}. As in \emph{Proposition 8.7}, the alternating subgroup $%
A_{n}$, cyclic subgroup $C_{n}$, the subgroup $I_{\sigma }$ generated by a
linear involution $\sigma $ of $L/K$, $A_{n}\cdot j_{\left( \Delta ,A\right)
}\left( \Sigma _{n}^{M}\right) $, $C_{n}\cdot j_{\left( \Delta ,A\right)
}\left( \Sigma _{n}^{M}\right) $ and $G\cdot j_{\left( \Delta ,A\right)
}\left( I_{\sigma }^{M}\right) $ all are non-solutions of $L/K$, i.e., none
of the invariant subfields under them is purely transcendental over $K$.
These six subgroups all have a transitive action on the variables $%
t_{1},t_{2},\cdots ,t_{n}$.
\end{proof}

\subsection{Concluding remarks}

Let $L$ be a purely transcendental extension over a field $K$ of finite
transcendence degree. Let $\Omega _{L/K}$ denote the set of transcendence
bases of $L/K$.

Recall that an integer $a\in \mathbb{Z}$ is a \emph{Noether number} of a
transcendental extension $L$ over a field $K$ if there is a finite subgroup $%
G$ of the Galois group $Aut\left( L/K\right) $ such that $a$ is the order of
$G$ and $G$ is a Noether solution of $L/K$. (See \emph{\S 6.4} for further
properties).

Recall that a transcendence base $\Delta _{\max }$ of $L/K$ is said to be
\emph{Zorn-maximal} in $L/K$ if there must be $K\left( \Delta \right)
=K\left( \Delta _{\max }\right) $ for any transcendence base $\Delta $ of $%
L/K$ with $K\left( \Delta \right) \supseteq K\left( \Delta _{\max }\right) $%
. Let $G\subseteq Aut\left( L/K\right) $ be a subgroup. A subset $\Omega
_{0}\subseteq \Omega _{L/K}$ is $G$\emph{-invariant} if $\sigma \left(
\Omega _{0}\right) \subseteq \Omega _{0}$ holds for any $\sigma \in G$. A
subset $\Omega _{0}\subseteq \Omega _{L/K}$ is \emph{dense} in $L/K$ if $%
\Omega _{0}$ contains a Zorn-maximal transcendence base $\Delta _{\max }$ of
$L/K$ satisfying the two properties: $\left( a\right) $ For any $\Delta \in
\Omega _{0}$, $K\left( \Delta \right) \subseteq K\left( \Delta _{\max
}\right) $ holds; $\left( b\right) $ For any $\Delta \in \Omega _{0}$, there
is some $\Lambda \in \Omega _{0}$ with $K\left( \Lambda \right) \subseteq
K\left( \Delta \right) $. (See \emph{Definition 2.10}).

Let $\Omega _{0}\subseteq \Omega _{L/K}$ be a subset. Recall that $L$ is $%
\sigma $\emph{-tame Galois} over $K$ \emph{inside} $\Omega _{0}$ if $L$ is
algebraic Galois over $K\left( \Delta \right) $ for each transcendence base $%
\Delta \in \Omega _{0}$. (See \emph{Definition 2.12}).

In the following there is an idealised picture for the relationship between
Noether solutions and non-solutions of $L/K$.

\begin{remark}
Let $\Omega _{0}\subseteq \Omega _{L/K}$ be a subset dense in $L/K$ and
invariant under $Aut\left( L/K\right) $. Suppose $L$ is $\sigma $-tame
Galois over $K$ inside $\Omega _{0}$ or $\Omega _{L/K}$.

$\left( i\right) $ (\emph{Non-solution being enlarged to solution}) Let $G$
be a subgroup of $Aut\left( L/K\right) $ such that $L$ is algebraic over the
$G$-invariant subfield $L^{G}$. Then there is a subgroup $H$ of $Aut\left(
L/K\right) $ satisfying the properties: $G$ is a proper subgroup of $H$; $H$
is a Noether solution of $L/K$.

$\left( ii\right) $ (\emph{Solution being enlarged to non-solution}) Let $%
G\subseteq Aut\left( L/K\right) $ be a Noether solution of $L/K$. Then there
is a subgroup $H$ of $Aut\left( L/K\right) $ satisfying the properties: $G$
is a proper subgroup of $H$; $L$ is not purely transcendental over the $H$%
-invariant subfield $L^{H}$.
\end{remark}

\begin{remark}
(\emph{Distribution of Noether solutions}) Let $\Omega _{0}\subseteq \Omega
_{L/K}$ be a subset dense in $L/K$ and invariant under $Aut\left( L/K\right)
$. Suppose $L$ is $\sigma $-tame Galois over $K$ inside $\Omega _{0}$ or $%
\Omega _{L/K}$. There are the following statements.

$\left( i\right) $ For any Noether number $\alpha $ of $L/K$, there are
infinitely many Noether solutions $G$ of $L/K$ such that $\alpha $ is the
group order $\sharp G$.

$\left( ii\right) $ For any Noether number $\alpha $ of $L/K$, there is
another Noether number $\beta $ of $L/K$ such that $\alpha <\beta $ and $%
\alpha $ divides $\beta $.

$\left( iii\right) $ For any two Noether numbers $\alpha ,\beta $ of $L/K$,
there is another Noether number $\gamma $ of $L/K$ such that $\alpha $ and $%
\beta $ divide $\gamma $, respectively.

$\left( iv\right) $ The set of Noether numbers of $L/K$ has no upper bound.
\end{remark}

\end{document}